\documentclass[a4paper,11pt,reqno]{amsart}
\usepackage[utf8]{inputenc}
\usepackage[T1]{fontenc}
\usepackage{lmodern}
\usepackage[english]{babel}
\usepackage{amsmath,a4wide,wasysym}
\usepackage{xfrac}
\usepackage{esint}
\usepackage{stmaryrd,mathrsfs,bm,amsthm,mathtools,yfonts,amssymb,color,braket,booktabs,graphicx,graphics,amsfonts}
\usepackage{latexsym,microtype,indentfirst,hyperref}
\usepackage{xcolor}
\usepackage{courier}
\usepackage[colorinlistoftodos]{todonotes}

\begingroup
\newtheorem{theorem}{Theorem}[section]
\newtheorem*{theorem*}{Theorem}
\newtheorem{lemma}[theorem]{Lemma}
\newtheorem{proposition}[theorem]{Proposition}
\newtheorem{corollary}[theorem]{Corollary}
\newtheorem{conjecture}[theorem]{Conjecture}
\endgroup

\theoremstyle{definition}
\newtheorem{definition}[theorem]{Definition}
\newtheorem{remark}[theorem]{Remark}
\newtheorem{example}[theorem]{Example}
\newtheorem{ipotesi}[theorem]{Assumption}

\newtheorem{notazioni}[theorem]{Notation}

\numberwithin{equation}{section}

\newenvironment{itemizeb}
{\begin{itemize}\itemsep=2pt}{\end{itemize}}

\DeclareMathAlphabet{\mathpzc}{OT1}{pzc}{m}{it}

\newcommand{\N}{\mathbb{N}} %naturali
\newcommand{\Z}{\mathbb{Z}} %interi
\newcommand{\R}{\mathbb{R}} %reali
\newcommand{\bA}{\mathbf{A}} %norma 2nd ff
\newcommand{\bC}{\mathbf{C}} %cilindri
\newcommand{\bB}{\mathbf{B}} %palle

\newcommand{\p}{\mathbf{p}} %orthogonal projections
\newcommand{\A}{\mathcal{A}} %Q-punti
\newcommand{\G}{\mathcal{G}} %distanza su A_Q
\newcommand{\cB}{\mathcal{B}} %palla in A_Q
\newcommand{\bfeta}{\boldsymbol{\eta}} %media
\newcommand{\Dir}{\mathrm{Dir}} %multivalued Dirichlet functional
\newcommand{\bG}{\mathbf{G}} %grafico
\newcommand{\gr}{\mathrm{Gr}}
\newcommand{\zetab}{{\bm{\zeta}}}

\newcommand{\xib}{{\bm{\xi}}}

\newcommand{\mass}{\mathbf{M}} %mass of a current
\newcommand{\Rc}{\mathscr{R}} %integer rectifiable currents
\newcommand{\In}{\mathscr{I}} %integral currents
\newcommand\res{\mathop{\hbox{\vrule height 7pt width .3pt depth 0pt\vrule height .3pt width 5pt depth 0pt}}\nolimits}

\newcommand{\cone}{\times\!\!\!\!\!\times\,} %cone

\newcommand{\F}{\mathscr{F}} %flat intere

\newcommand{\M}{\mathbf{M}} %massa
\newcommand{\Fl}{\mathcal{F}} %flat modificata
\newcommand{\reg}{\mathrm{Reg}} %insieme regolare
\newcommand{\sing}{\mathrm{Sing}} %insieme singolare

\newcommand{\V}{\mathbf{v}} %varifold associato

\newcommand{\bE}{\mathbf{E}} %eccesso riscalato
\newcommand{\e}{\mathbf{e}} %eccesso non riscalato	
\newcommand{\bd}{\mathbf{d}} %densita` di eccesso
\newcommand{\bmax}{\mathbf{m}} %funzione massimale
\newcommand{\bh}{\mathbf{h}} %height

\newcommand{\cS}{\mathcal{S}} %stratificazione

\newcommand{\modp}{{\rm mod}(p)} %modulo p
\newcommand{\moddue}{{\rm mod}(2)} %modulo 2

\newcommand{\Ha}{\mathcal{H}} %Hausdorff distance

\newcommand{\eps}{\varepsilon} %epsilon
\newcommand{\spt}{\mathrm{spt}} %support (usually of a measure)
\newcommand{\dist}{\mathrm{dist}} %distance
\newcommand{\B}{\mathbf{B}} %geodesic ball
\newcommand{\Lip}{\mathrm{Lip}} %Lipschitz
\renewcommand{\epsilon}{\varepsilon}

\def\XXint#1#2#3{{\setbox0=\hbox{$#1{#2#3}{\int}$ }
\vcenter{\hbox{$#2#3$ }}\kern-.6\wd0}}

\newcommand{\mint}{-\hskip-1.10em\int}

\newcommand{\mres}{\mathbin{\vrule height 1.6ex depth 0pt width %measure restriction
0.13ex\vrule height 0.13ex depth 0pt width 1.3ex}}

\newcommand{\weaks}{\stackrel{*}{\rightharpoonup}} %weak star convergence

\newcommand{\essvar}{{\rm ess}\,{\rm var}} %variazione essenziale

\newcommand{\Iqs}{{\mathcal{A}}_Q(\R^{n})}
\newcommand{\Iq}{{\mathcal{A}}_Q}
\def\a#1{\left\llbracket{#1}\right\rrbracket}
\newcommand{\abs}[1]{\lvert#1\rvert} %absolute value
\newcommand{\Abs}[1]{\left\lvert#1\right\rvert} %large absolute value
\newcommand{\norm}[1]{\left\lVert#1\right\rVert} %norm
\newcommand{\etab}{\boldsymbol{\eta}}

\newcommand{\Iqspec}{{\mathscr{A}_Q (\R^n)}}

\def\a#1{\left\llbracket{#1}\right\rrbracket}
\newcommand{\iso}{\boldsymbol{\iota}}

\newcommand\bmo{{\bm m}_0}
\newcommand\sC{{\mathscr{C}}}
\newcommand\sS{{\mathscr{S}}}
\newcommand\sW{{\mathscr{W}}}
\newcommand\sP{{\mathscr{P}}}
\newcommand\cM{{\mathcal{M}}}
\newcommand\cL{{\mathcal{L}}}
\newcommand\Phii{{\mathbf{\Phi}}}
\newcommand\bGam{{\mathbf{\Gamma}}}
\newcommand\phii{{\mathbf{\varphi}}}
\newcommand\cV{{\mathcal{V}}}
\newcommand\cH{{\mathcal{H}}}
\newcommand\bU{{\mathbf{U}}}
\newcommand\cK{{\mathcal{K}}}
\newcommand\bT{{\mathbf{T}}}
\newcommand\cG{{\mathcal{G}}}
\newcommand\bS{{\mathbf{S}}}

\newcommand\cF{{\mathcal{F}}}
\newcommand\cR{{\mathcal{R}}}
\newcommand\bD{{\mathbf{D}}}
\newcommand\bH{{\mathbf{H}}}
\newcommand\bI{{\mathbf{I}}}
\newcommand\de{{\partial}}
\newcommand\bSigma{{\mathbf{\Sigma}}}
\newcommand\im{{\mathrm{Im}}}
\newcommand\dv{{\mathrm{div}\,}}
\newcommand\bJ{{\mathbf{J}}}
\newcommand\rD{{\mathrm{D}}}
\newcommand\be{{\mathbf{e}}}

\begin{document}

\title{Regularity of area minimizing currents mod $p$}

\author{Camillo De Lellis}
\address{School of Mathematics, Institute for Advanced Study, 1 Einstein Dr., Princeton NJ 05840, USA\\
and Universit\"at Z\"urich}
\email{camillo.delellis@math.ias.edu}

\author{Jonas Hirsch}
\address{Mathematisches Institut, Universit\"at Leipzig, Augustusplatz 10, D-04109 Leipzig, Germany}
\email{hirsch@math.uni-leipzig.de}

\author{Andrea Marchese}
\address{Dipartimento di Matematica, Universit\`a di Trento, via Sommarive 14, I-38123 Povo (TN), Italy}
\email{andrea.marchese@unitn.it}

\author{Salvatore Stuvard}
\address{Department of Mathematics, The University of Texas at Austin, 2515 Speedway, Stop C1200, Austin TX 78712-1202, USA}
\email{stuvard@math.utexas.edu}

\begin{abstract}

We establish a first general partial regularity theorem for area minimizing currents $\modp$, for every $p$, in any dimension and codimension. More precisely, we prove that the Hausdorff dimension of the interior singular set of an $m$-dimensional area minimizing current $\modp$ cannot be larger than $m-1$. Additionally, we show that, when $p$ is odd, the interior singular set is $(m-1)$-rectifiable with locally finite $(m-1)$-dimensional measure.\\


\textsc{AMS Math Subject Classification (2010):} 49Q15, 49Q05, 49N60, 35B65, 35J47.
\end{abstract}

\maketitle

\section{Introduction}

\subsection{Overview and main results}

In this paper we consider currents $\modp$ (where $p\geq 2$ is a fixed positive integer), for which we follow the definitions and the terminology of \cite{Federer69}. In particular, given an open subset $\Omega \subset \R^{m+n}$, we will let $\Rc_m(\Omega)$ and $\F_m(\Omega)$ denote the spaces of $m$-dimensional integer rectifiable currents and $m$-dimensional integral flat chains in $\Omega$, respectively. If $C \subset \R^{m+n}$ is a closed set (or a relatively closed set in $\Omega$), then $\Rc_m (C)$ (resp. $\F_m (C)$) denotes the space of currents $T \in \Rc_m(\R^{m+n})$ (resp. $T \in \F_m(\R^{m+n})$) with compact support $\spt(T)$ contained in $C$. Currents modulo $p$ in $C$ are defined introducing an appropriate family of pseudo-distances on $\F_m(C)$: if 
$S,T \in \F_m (C)$ and $K\subset C$ is compact, then
\begin{align*}
\F^p_K (T-S) &:= \inf \Big\{\mass (R) + \mass (Z)\, : \,R \in \Rc_m (K)\,, Z\in \Rc_{m+1} (K)\;\\
&\qquad\qquad\qquad \mbox{such that}\; T-S= R+\partial Z + p P \; \mbox{for some}\; P \in \F_m (K)\Big\}\, .
\end{align*}
Two flat currents in $C$ are then congruent modulo $p$ if there is a compact set $K\subset C$ such that
$\F^p_K (T-S) =0$. The corresponding congruence class of a fixed flat chain $T$ will be denoted by $[T]$, whereas if $T$ and $S$ are congruent we will write
\[
T = S\, \modp\, .
\]
The symbols $\Rc_m^p(C)$ and $\F_m^p(C)$ will denote the quotient groups obtained from $\Rc_m(C)$ and $\F_m(C)$ via the above equivalence relation. The boundary operator $\partial$ has the obvious property that, if $T=S\, \modp$, then $\partial T = \partial S\, \modp$. 
This allows to define an appropriate notion of boundary $\modp$ as $\partial^p [T] := [\partial T]$. 
Correspondingly, we can define cycles and boundaries $\modp$ in $C$:
\begin{itemize}
\item a current $T\in \F_m (C)$ is a cycle $\modp$ if $\partial T = 0\, \modp$, namely if $\partial^p [T]=0$;
\item a current $T\in \F_m (C)$ is a boundary $\modp$ if $\exists S\in \F_{m+1} (C)$ such that $T= \partial S\, \modp$, namely $[T] = \partial^p [S]$.
\end{itemize}
Note that every boundary $\modp$ is a cycle $\modp$. In what follows, the closed set $C$ will always be sufficiently smooth, more precisely a complete submanifold $\Sigma$ of $\mathbb R^{m+n}$ without boundary and of class $C^1$.

\begin{remark}\label{r:congruence_manifold}
Note that the congruence classes $[T]$ depend on the set $C$, and thus our notation is not precise in this regard. In particular, when two currents are congruent modulo $p$ in $\Sigma\subset \mathbb R^{m+n}$, then they are obviously congruent in $\mathbb R^{m+n}$, but the opposite implication is generally false, see also the discussion in \cite[Remark 3.1]{MS_a}. However, the two properties are equivalent in the particular case of $\Sigma$'s which are Lipschitz deformation retracts of $\mathbb R^{m+n}$, and we will see below that, without loss of generality, we can restrict to the latter case in most of our paper. For this reason we do not keep track of the ambient manifold in the notation regarding the $\modp$ congruence.
\end{remark}

\begin{definition} \label{def:am_modp}
Let $\Omega \subset \R^{m+n}$ be open, and let $\Sigma \subset \R^{m+n}$ be a complete submanifold without boundary of dimension $m+\bar{n}$ and class $C^{1}$. We say that an $m$-dimensional integer rectifiable current $T \in \Rc_{m}(\Sigma)$ is \emph{area minimizing} $\modp$ in $\Sigma \cap \Omega$ if
\begin{equation}\label{e:am_mod_p}
\mass (T) \leq \mass (T + S) \qquad \mbox{for every $S \in \Rc_{m}(\Omega\cap \Sigma)$ which is a boundary $\modp$}.
\end{equation}
\end{definition}

Recalling \cite{Federer69}, it is possible to introduce a suitable notion of mass $\modp$ for classes $[T]$ $\modp$, denoted by $\mass^p$: $\mass^p ([T])$ is the infimum of those $t \in \R \cup \{+\infty\}$ such that for every $\varepsilon >0$ there are a compact set $K\subset \Sigma$ and an $S\in \Rc_m (\Sigma)$ with 
\[
\F^p_K (T-S) < \varepsilon \qquad \mbox{and} \qquad \mass (S) \leq t+\varepsilon\, .
\]
Analogously, \cite{Federer69} defines the support $\modp$ of the current $T\in \Rc_{m}(\Sigma)$, by setting 
\[
\spt^p (T) := \bigcap_{R = T\, \modp} \spt (R) \, .
\]
Clearly, the support depends only upon $[T]$, and we can thus also use the notation $\spt^p ([T])$. 

With the above terminology we can talk about mass minimizing classes $[T]$, because \eqref{e:am_mod_p} can be rewritten as
\[
\mass^p ([T]) \leq \mass^p ([T] + \partial^p [S]) \qquad \mbox{for all $[S]$ with $\spt^p ([S])\subset \Omega \cap \Sigma$.} 
\]
Our paper is devoted to the interior regularity theory for such objects. 

\begin{definition}
Let $T$ be an area-minimizing current $\modp$ in $\Omega \cap \Sigma$. A point $q\in \Omega \cap \spt^p(T)$ is called an {\em interior regular point} if there is a neighborhood $U$ of $q$, a positive integer $Q$ and an oriented $C^1$ embedded submanifold $\Gamma$ of $\Sigma \cap U$ such that
\begin{itemize}
\item[(i)] $T\res U = Q \a{\Gamma}\, \modp$;
\item[(ii)] $\Gamma$ has no boundary in $\Sigma \cap U$.
\end{itemize}
We will denote the set of interior regular points of $T$ by $\reg (T)$. 
\end{definition}

Observe that by definition an interior regular point is necessarily contained in $\spt^p (T)$
and it is necessarily outside $\spt^p (\partial T)$. For this reason, it is natural to define the set of interior singular points of $T$ as
\[
\sing (T) := (\Omega \cap  \spt^p (T)) \setminus (\reg (T) \cup \spt^p (\partial T))\, .
\]
It is very easy to see that $\sing (T)$ cannot be expected to be empty. Probably the following is the best known example: consider the three points $P_j := (\cos \frac{2\pi j}{3}, \sin \frac{2\pi j}{3})\in \mathbb R^2$ for $j= 1,2,3$ and the three oriented segments $\sigma_j$ in $\mathbb R^2$ joining the origin with each of them. Then $T:= \sum_j \a{\sigma_j}$ is area-minimizing ${\rm mod}\, (3)$ in $\mathbb R^2$ and the origin belongs to $\sing (T)$. 

As a first step to a better understanding of the singularities it is therefore desirable to give a bound on the Hausdorff dimension of the singular set.
The present work achieves the best possible bound in the most general case, and in particular it answers a question of White, see \cite[Problem 4.20]{GMT_prob}. 

\begin{theorem}\label{t:main}
Assume that $p\in \mathbb N \setminus \{0,1\}$, that $\Sigma\subset \mathbb R^{m+n}$ is a $C^{3, a_0}$ submanifold
of dimension $m+\bar n$ for some positive $a_0$, that $\Omega \subset \mathbb R^{m+n}$ is open, and that $T\in \Rc_m (\Sigma)$ is area minimizing $\modp$ in $\Omega\cap \Sigma$. Then, $\mathcal{H}^{m-1+\alpha} (\sing (T)) =0$ for every $\alpha>0$. 
\end{theorem}

Prior to the present paper, the state of the art in the literature on the size of the singular set for area minimizing currents $\modp$ was as follows. We start with the results valid in any codimension.
\begin{itemize}
\item[(a)] For $m =1$ it is very elementary to see that $\sing (T)$ is discrete (and empty when $p=2$);
\item[(b)] Under the general assumptions of Theorem \ref{t:main}, $\sing (T)$ is a closed meager set in $(\spt^p (T)\cap \Omega)\setminus \spt^p (\partial T)$ by Allard's interior regularity theory for stationary varifolds, cf. \cite{Allard72} (in fact, in order to apply Allard's theorem it is sufficient to assume that $\Sigma$ is of class $C^2$);
\item[(c)] For $p=2$, $\mathcal{H}^{m-2+\alpha} (\sing (T)) = 0$ for every $\alpha >0$ by Federer's classical work \cite{Federer70}; moreover the same reference shows that $\sing (T)$ consists of isolated points when $m=2$; for $m>2$, Simon \cite{Simon_cylindrical,Simon95} proved that $\sing (T)$ is $(m-2)$-rectifiable and it has locally finite $\mathcal{H}^{m-2}$ measure.
\end{itemize}
We next look at the hypersurface case, namely $\bar n =1$.
\begin{itemize}
\item[(d)] When $p=2$, $\mathcal{H}^{m-2} (\sing (T))=0$ even in the case of minimizers of general uniformly elliptic integrands, see \cite{ASS}; for the area functional, using \cite{NV}, one can conclude additionally that $\sing (T)$ is $(m-7)$-rectifiable and has locally finite $\mathcal{H}^{m-7}$ measure; 
\item[(e)] When $p=3$ and $m=2$, \cite{Taylor} gives a complete description of $\sing (T)$, which consists of $C^{1,\alpha}$ arcs where three regular sheets meet at equal angles;
\item[(f)] When $p$ is odd, \cite{White86} shows that $\mathcal{H}^m (\sing (T)) =0$ for minimizers of a uniformly elliptic integrand, and that $\Ha^{m-1+\alpha}(\sing(T)) = 0$ for every $\alpha > 0$ for minimizers of the area functional;
\item[(g)] When $p=4$, \cite{White79} shows that minimizers of uniformly elliptic integrands are represented by {\em immersed manifolds} outside of a closed set of zero $\mathcal{H}^{m-2}$ measure.

%Moreover, $\sing (T)$ is $(m-1)$-rectifiable and has locally finite $\mathcal{H}^{m-1}$ measure. 
\end{itemize}

\medskip

In view of the examples known so far it is tempting to advance the following

\begin{conjecture}\label{c:flat_points}
Let $T$ be as in Theorem \ref{t:main}. Denote by ${\rm Sing}_{f} (T)$ the subset of \emph{interior flat singular points}, that is those points $q\in \sing (T)$ where there is at least one flat tangent cone; see Sections \ref{s:mono_cones} and \ref{s:strata}. Then
$\mathcal{H}^{m-2+\alpha} ({\rm Sing}_{f} (T)) =0$ for every $\alpha>0$.
\end{conjecture}

Conjecture \ref{c:flat_points} is known to be correct for:
\begin{itemize}
\item[(a)] $m=1$;
\item[(b)] $p=2$ and any $m$ and $\bar{n}$;
\item[(c)] $p$ is odd and the codimension $\bar n = 1$.
\end{itemize} 
In all three cases, however, the conjecture follows from the much stronger fact that ${\rm Sing}_{f} (T)$ is empty: 
\begin{itemize}
\item the case (a) is an instructive exercise in geometric measure theory; 
\item the case (b) follows from Allard's regularity theorem for stationary varifold;
\item the case (c) is a corollary of the main result in \cite{White86}.  
\end{itemize}
Note that in all the other cases we cannot expect ${\rm Sing}_{f} (T)$ to be empty, with the easiest case being $p=4$, $m=2$ and $\bar n =1$, to be discussed in the following
\begin{example} \label{example:flat singularities}
Consider a ball $B \subset \R^2$ as well as two distinct smooth functions $u_1$ and $u_2$ solving the minimal surfaces equation in $B$, and let $S_1$ and $S_2$ denote the integral currents in the cylinder $B \times \R \subset \R^3$ defined by their graphs endowed with the natural orientation. As it is well known, $S_1$ and $S_2$ are then area minimizing, both as integral currents and as currents $\moddue$. Assume, in addition, that the set $\{u_1=u_2\}$ contains a curve $\gamma$ which divides $B$ into two regions $B^{>}$ and $B^{<}$. Explicit $u_1$ and $u_2$ as above are easy to find. The reader could take $B$ to be a suitably small ball centered at the origin, $u_1\equiv 0$, and let $u_2$ be the function which describes Enneper's minimal surface in a neighborhood of the origin. The set $\{u_1=u_2\}$ is then given by $\{(x,y) : x=\pm y\}\cap B$ and $\gamma$ can be taken to be the segment $\{x=y\}\cap B$ while $B^{>}$ and $B^{<}$ would then be $B \cap \{x>y\}$ and $B \cap \{x<y\}$, respectively. We then define the following rectifiable current $T$. Its support is the union of the graphs of $u_1$ and $u_2$, and thus of the supports of $S_1$ and $S_2$. However, while the portions of such graphs lying over $B^{>}$ will be taken with the standard orientation induced by $B$, the portions lying over $B^{<}$ will be taken with the opposite orientation. In $B\times \mathbb R$, the boundary of $T$ is $4 \a{\gamma}$, and $T$ is singular along $\gamma$. Since $T$ can be written as $T=\tilde S_1 + \tilde S_2$, where $\tilde S_k$ are $\moddue$ equivalent to $S_k$ ($k=1,2$), and since $\tilde S_k$ are area minimizing $\moddue$, the structure theorem in \cite{White79} guarantees that $T$ is area minimizing $\mod (4)$. Whenever $u_1$ and $u_2$ are chosen so that $0 \in \gamma$, $u_1 (0) = u_2 (0) =0$, and $\nabla u_1 (0) = \nabla u_2 (0) =0$ (as it is the case in the example above) then $0$ is a singular point of $T$, and the (unique) tangent cone to $T$ at $0$ is the two dimensional horizontal plane $\pi_0 = \{x_3 =0\}$ with multiplicity $2$. In such examples we thus have $0\in {\rm Sing}_{f} (T)$.
\end{example}

In this paper we strengthen the result for $p$ odd by showing that Conjecture \ref{c:flat_points} in fact holds in any codimension. Indeed we prove the following more general theorem. 

\begin{theorem}\label{t:main3}
Let $T$ be as in Theorem \ref{t:main} and $Q< \frac{p}{2}$ a positive integer. Consider the subset $\sing_Q (T)$ of $\spt^p (T)\setminus \spt^p (\partial T)$ which consists of interior singular points of $T$ where the density is $Q$ (see Definition \ref{Q-points}). Then 
$\mathcal{H}^{m-2+\alpha} (\sing_Q (T))=0$ for every $\alpha>0$. 
\end{theorem}

The analysis of tangent cones (cf. Corollary \ref{c:tangent_cones}) implies that if $p$ is odd then 
\[
{\rm Sing}_{f} (T)\subset \bigcup_{Q=1}^{\lfloor \frac{p}{2} \rfloor} \sing_Q (T)\, .
\]
We thus get immediately

\begin{corollary}\label{c:main4}
Conjecture \ref{c:flat_points} holds for every $p$ odd in any dimension $m$ and codimension $\bar n$. 
\end{corollary}

The fact above, combined with the techniques recently introduced in the remarkable work \cite{NV}, allows us to conclude the following theorem. 

\begin{theorem}\label{t:main10}
Let $T$ be as in Theorem \ref{t:main} and assume $p$ is odd. Then $\sing (T)$ is $(m-1)$-rectifiable, and for every compact $K$ with $K\cap \spt^p (\partial T) = \emptyset$ we have $\cH^{m-1} (\sing (T) \cap K) < \infty$.
\end{theorem}

In turn the above theorem implies the following structural result.

\begin{corollary}\label{c:main11}
Let $T$ be as in Theorem \ref{t:main} and assume in addition that $p$ is odd. Denote by $\{\Lambda_i\}_i$ the connected components of $\spt^p (T) \setminus (\spt^p (\partial T) \cup \sing (T))$. Then each $\Lambda_i$ is an orientable smooth minimal submanifold of $\Sigma$ and there is a choice of (smooth) orientations and multiplicities 
$Q_i\in [1, \frac{p}{2}]\cap \mathbb N$ such that the following properties hold for every open $U\Subset \mathbb R^{m+n}\setminus \spt^p (\partial T)$
\begin{itemize}
\item[(a)] Each $T_i = Q_i \a{\Lambda_i}$ is an integral current in $U$ and thus, having chosen an orientation $\vec{S}$ for the rectifiable set $\sing (T)$, 
we have
\[
(\partial T_i)\res U = \Theta_i \vec{S} \mathcal{H}^{m-1} \res (\sing (T)\cap U)\, 
\]
for some integer valued Borel function $\Theta_i$;
\item[(b)] $\sum_i \mass (T_i\res U)<\infty$ and $T \res U=\sum_i T_i \res U$; 
\item[(c)] $\sum_i \mass ((\partial T_i)\res U)<\infty$, $(\partial T)\res U = \sum_i (\partial T_i)\res U$ and
\[
(\partial T)\res U = \sum_i \Theta_i\, \vec{S} \mathcal{H}^{m-1} \res (\sing (T)\cap U)\, ;
\]
in particular $\sum_i \Theta_i (q)$ is an integer multiple of $p$
for $\mathcal{H}^{m-1}$-a.e. $q\in \sing (T)\cap U$.
\end{itemize}
\end{corollary}

It is tempting to advance the following conjecture.

\begin{conjecture}
The conclusions of Theorem \ref{t:main10} hold for $p$ even as well.  
\end{conjecture}

From the latter conjecture one can easily conclude an analogous structure theorem as in Corollary \ref{c:main11}. Note that the conjecture is known to hold for $p=2$ in every codimension (in which case, in fact, we know that $\sing (T)$ has dimension at most $m-2$) and for $p=4$ in codimension $1$. 

\subsection{Plan of the paper} The paper is divided into five parts: the first four parts contain the arguments leading to the proof of Theorems \ref{t:main} and \ref{t:main3}, while the last part is concerned with the proof of the rectifiability Theorem \ref{t:main10} and of Corollary \ref{c:main11}. Each part is further divided into sections. The proof of Theorems \ref{t:main} and \ref{t:main3} is obtained by contradiction, and is inspired by F. Almgren's work on the partial regularity for area minimizing currents in any codimension as revisited by the first-named author and E. Spadaro in \cite{DLS_Lp,DLS_Center,DLS_Blowup}. In particular, Part \ref{part:preliminary} contains the preliminary observations and reductions aimed at stating the contradiction assumption for Theorems \ref{t:main} and \ref{t:main3}, whereas Part \ref{part:Lp}, Part \ref{part:Center}, and Part \ref{part:Blowup} are the counterpart of the papers \cite{DLS_Lp}, \cite{DLS_Center}, and \cite{DLS_Blowup}, respectively. An interesting feature of the regularity theory presented in this work is that Almgren's multiple valued functions minimizing the Dirichlet energy are not the right class of functions to consider when one wants to approximate a minimizing current $\modp$ in a neighborhood of a flat interior singular point whenever the density of the point is precisely $\frac{p}{2}$. Solving this issue requires (even in the codimension $\bar n = 1$ case) the introduction of a class of \emph{special} multiple valued functions minimizing a suitably defined Dirichlet integral. The regularity theory for such maps (which we call \emph{linear theory}) is the content of our paper \cite{DLHMS_linear}. Applications of multivalued functions to flat chains $\modp$ were already envisioned by Almgren in \cite{Almgren_modp}, even though he considered somewhat different objects than those defined in \cite{DLHMS_linear}. Because of this profound interconnection between the two theories, the reading of \cite{DLHMS_linear} is meant to precede that of the present paper. \\

{\bf Acknowledgments.} C.D.L. acknowledges the support of the NSF grants DMS-1946175 and DMS-1854147. A.M. was partially supported by INdAM GNAMPA research projects. The work of S.S. was supported by the NSF grants DMS-1565354, DMS-RTG-1840314 and DMS-FRG-1854344.

\section{Almgren's regularity theory in the $\modp$ setting} \label{sec:guide}
Before entering the main body of the paper, we would like to briefly present an overview of Almgren's regularity theory adapted to the setting of area minimizing currents $\modp$, focusing onto the points where major changes were required in order to overcome the intrinsic difficulties of the problem under consideration. To do so, we restrict our attention to the proof of Theorem \ref{t:main}. For the sake of simplicity, we will assume throughout this discussion that $\Sigma=\R^{m+\bar n}$. \\

Towards a proof by contradiction of Theorem \ref{t:main}, we exploit the classical Almgren's stratification principle for stationary varifolds in order to reduce the contradiction assumption to the following (see Proposition \ref{p:contradiction_sequence}): there exist integers $p \ge 2$ and $Q \le \sfrac{p}{2}$, reals $\alpha,\eta >0$, an open ball $\Omega \ni 0$, and an $m$-dimensional rectifiable current $T$ in $\R^{m+\bar n}$ such that:
\begin{itemize}
\item[(i)] $T$ is area minimizing $\modp$ in $\Omega$ with $(\partial T) \res \Omega =0$ $\modp$ in $\Omega$ and $0\in \sing_Q (T)$;
\item[(ii)] there exist a sequence of radii $r_k\downarrow 0$ and an $m$-dimensional plane $\pi_0$ such that the integral varifolds $\V (T_{0, r_k})$ associated with the rescaled currents centered at 0 converge to a varifold $V = Q\, \mathcal{H}^m \res \pi_0 \otimes \delta_{\pi_0}$;
\item[(iii)] it holds \[\limsup_{k\to\infty} \mathcal{H}^{m-1+\alpha}_\infty ( \sing_Q (T_{0, r_k})\cap \bB_1) \geq \eta\,,\] where $\bB_1$ is the unit open ball in $\R^{m+\bar n}$.
\end{itemize}

The next step in Almgren's strategy would then be to approximate the currents $T_k=T_{0,r_k}$ with graphs of functions $u_k$ defined on $\pi_0$, taking values in the metric space $\mathcal{A}_Q(\pi_0^\perp)$ of $Q$-points in $\pi_0^\perp$ (that is, the space of discrete measures $T=\sum_{i=1}^Q \a{v_i}$ on $\pi_0^\perp$, with positive integer coefficients and total mass $Q$), and minimizing a suitable linearization of the mass functional ($\Dir$-minimizing $Q$-valued functions). In our setting, the main difficulties related to this step occur when $p$ is even and $Q=\sfrac{p}{2}$. In this case, indeed, Almgren's $\Dir$-minimizing $Q$-valued functions are not the correct objects to perform such approximation; see \cite[Example 1.2]{DLHMS_linear}. Notice that the phenomenon responsible of the inadequacy of classical $\Dir$-minimizers in the approximation of area minimizing currents $\modp$ is precisely the existence of flat singular points of density $Q=\sfrac{p}{2}$ discussed in Example \ref{example:flat singularities}. In order to introduce a class of multiple valued functions adapted to our needs, in \cite{DLHMS_linear} we defined the metric space $\mathscr{A}_Q(\R^n)$ of \emph{special} $Q$-points in Euclidean space $\R^n$, and we studied the regularity properties of $\mathscr{A}_Q(\R^n)$-valued functions minimizing a functional representing the natural linearization of the mass $\modp$ (henceforth called $\Dir$-minimizing \emph{special} $Q$-valued functions). For the reader's convenience, we briefly recall here some basic notation introduced in \cite{DLHMS_linear}. The space $\mathscr{A}_Q(\R^n)$ is defined by
\[
\mathscr{A}_Q(\R^n) := \mathcal{A}_Q(\R^n) \sqcup \mathcal{A}_Q(\R^n) / \sim\,,
\]
where $\sim$ is the equivalence relation defined by
\begin{align*}
 (S,1) \sim (T,1) & \iff S=T\,,\\
 (S,-1) \sim (T, -1) & \iff S=T\,,\\
 (S,1) \sim (T,-1) & \iff \exists\, z \in \R^n \, \colon \, S=Q\a{z}=T\,.
\end{align*}

Given a Borel measurable map $u \colon \Omega \subset \R^m \to \mathscr{A}_Q(\R^n)$, there is a canonical decomposition of the domain $\Omega$ into three disjoint sets $\Omega_+$, $\Omega_-$, and $\Omega_0$. More precisely, $\Omega_0$ is the set of points $x \in \Omega$ for which there exists $z \in \R^n$ such that $u(x)=(Q\a{z},1)=(Q\a{z},-1)$; $\Omega_+$ and $\Omega_-$ are, instead, the sets of points $x \in \Omega \setminus \Omega_0$ such that $u(x)= (S,1)$ or $u(x)=(S,-1)$ with $S \in \mathcal{A}_Q(\R^n)$, respectively. Furthermore, we define the functions $u^\pm \colon \Omega \to \mathcal{A}_Q(\R^n)$ by
\[
u^\pm(x) = 
\begin{cases}
S & \mbox{if $x \in \Omega_{\pm}$ and $u(x)=(S,\pm1)$}\,, \\
Q\a{\bfeta(S)} &\mbox{if $x \in \Omega \setminus \Omega_{\pm}$ and $u(x)=(S,\mp 1)$}\,,
\end{cases}
\]
where $\bfeta(S)$ denotes the average of the $Q$-point $S$, see Section \ref{sec:notations}. With these notations at hand, we can define the Dirichlet energy $\Dir(u)$ of a $W^{1,2}$ map $u \colon \Omega \to \mathscr{A}_Q(\R^n)$ by setting
\[
\Dir(u) := \Dir(u^+ \ominus \bfeta \circ u) + \Dir(u^- \ominus \bfeta \circ u) + Q\, \Dir(\bfeta \circ u)\,,
\]
where $\bfeta \circ u$ is the ($\R^n$-valued) average of $u$ and $T \ominus z := \sum_{i=1}^Q \a{v_i - z}$ if $T = \sum_{i=1}^Q \a{v_i} \in \mathcal{A}_Q(\R^n)$ and $z \in \R^n$. Moreover, we can define the integer rectifiable $m$-current $\mathbf{G}_u$ in $\R^{m+n}=\R^{m}\times\R^n$ associated with a Lipschitz function $u \colon \Omega \subset \R^m \to \mathscr{A}_Q(\R^n)$ by:
\[
{\bf G}_u := {\bf G}_{u^+} \mres (\Omega_+ \times \R^n) - {\bf G}_{u^-} \mres (\Omega_- \times \R^n) + Q \, {\bf G}_{\bfeta \circ u} \mres (\Omega_0 \times \R^n)\,,
\]
where ${\bf G}_{u^\pm}$ denotes the current associated with the graph of a classical $Q$-valued function as in \cite[Definition 1.10]{DLS_Currents} and ${\bf G}_{\bfeta \circ u}$ is the current associated with the graph of the average $\bfeta \circ u$. For instance, the current $T$ described in Example \ref{example:flat singularities} coincides with the graph ${\bf G}_u$ of the $\mathscr{A}_2(\R)$-valued function defined by
\[
u(x) = 
\begin{cases}
(\a{u_1(x)} + \a{u_2(x)},1) & \mbox{if $x \in B^{>} \cup \gamma$}\,,\\
(\a{u_1(x)} + \a{u_2(x)},-1) & \mbox{if $x \in B^{<} \cup \gamma$}\,.
\end{cases}
\]

Notice that, as Lipschitz $Q$-valued graphs over a domain $\Omega$ are integer rectifiable currents without boundary in the cylinder $\Omega \times \R^n$, Lipschitz \emph{special} $Q$-valued graphs over $\Omega$ are integer rectifiable currents without boundary $\modp$ in $\Omega \times \R^n$. \\

Now that we have the correct class of approximating functions, we can get back to Almgren's program. The first step is an approximation of each current $T_k$ with the graph of a Lipschitz special $Q$-valued function; see Proposition \ref{p:max}. The proof is based on a BV estimate for the slices of $T_k$ with respect to the plane $\pi_0$, see Lemma \ref{conto_unidimensionale}: while this is classically achieved in \cite{DLS_Lp} testing the current with suitably defined differential forms, our setting requires an \emph{ad hoc} proof due to the fact that $T_k$ may possibly have non-trivial classical boundary. The errors in such an approximation (which, we note in passing, does not use that $T_k$ is area minimizing $\modp$) are controlled linearly by the \emph{excess $\modp$} of $T_k$ with respect to the plane $\pi_0$: this is defined as the difference between the mass of $T_k$ and the mass modulo $p$ of its projection onto $\pi_0$; see Definition \ref{d:excess}. Exploiting the minimality of $T_k$, we can then substantially improve the results of this first Lipschitz approximation in two ways: first, upgrading the control of the errors in terms of a \emph{superlinear} power of the excess $\modp$ (Theorem \ref{thm:almgren_strong_approx}); second, showing that the approximating special $Q$-valued function is close to a $\Dir$-minimizer (Theorem \ref{thm:final_harm_approx}). Finally, we introduce a second notion of excess, called the \emph{nonoriented excess} and smaller than the excess $\modp$ (formula \eqref{e:no_excess}), and we show that all the error estimates in the aforementioned approximation can be upgraded replacing the excess $\modp$ with the nonoriented excess; see Theorem \ref{thm:strong-alm-unoriented}. The nonoriented excess is a more accurate measure of the \emph{local} tilting of a current with respect to a plane regardless of orientations (much like the varifold excess), a feature that is very important in our setting, since area minimizing currents $\modp$ may exhibit changes of orientation even when their boundary $\modp$ vanishes. Furthermore, the flexibility of the nonoriented excess with respect to localization will be of vital importance in the next step of Almgren's program, namely the construction of the center manifold.\\

The latter is arguably the most delicate part in Almgren's proof, and it is motivated by the following issue. Given the sequence of currents $T_k$ converging to $Q \a{\pi_0}$ in the sense of currents $\modp$, and given the sequence $u_k$ of $\mathscr{A}_Q(\pi_0^\perp)$-valued approximating functions, it would be tempting to perform an appropriate (non-homogeneous) rescaling of the functions leading to a new sequence which converges, in the limit as $k \to \infty$, to a \emph{non-trivial} $\Dir$-minimizing special $Q$-valued function $u_\infty$. In view of (iii), if we could prove that the function $u_\infty$ ``inherits'' the singularities of the currents, then we would obtain the desired contradiction by invoking the main result of \cite{DLHMS_linear}, namely the following

\begin{theorem*}[see {\cite[Theorem 10.2]{DLHMS_linear}}]
The singular set of a $\Dir$-minimizing special $Q$-valued function defined on a domain in $\R^m$ has Hausdorff dimension at most $m-1$.
\end{theorem*}

The issue in this plan is that the limit function $u_\infty$ may not exhibit any singularities at all: this happens when, at the natural rescaling rate of the functions $u_k$, the currents $T_k$ are \emph{centered} around a smooth sheet. Conceived precisely to mod-out such a smooth sheet, the center manifold is an $m$-dimensional surface equipped with a special $Q$-valued \emph{section} $N_k$ of its normal bundle which approximately parametrizes $T_k$. The blow-up argument described above then leads to the desired contradiction when performed on the approximations $N_k$. This portion of Almgren's proof is sufficiently robust to go through in our setting without the need of substantial modifications, with all the main necessary estimates being already available from the new Lipschitz approximation and the detailed analysis of $\Dir$-minimizing special $Q$-valued functions contained in \cite{DLHMS_linear}.   

\section{Notation} \label{sec:notations}

We add below a list of standard notation in Geometric Measure Theory, which will be used throughout the paper. More notation will be introduced in the main text when the need arises.
\begin{itemizeb}\leftskip 0.8 cm\labelsep=.3 cm

\item[$\bB_r(x)$] open ball in $\R^{m+n}$ centered at $x \in \R^{m+n}$ with radius $r > 0$;

\item[$\omega_m$] Lebesgue measure of the unit disc in $\R^m$;

\item[$|A|$] Lebesgue measure of $A \subset \R^{m+n}$;

\item[$\Ha^m$] $m$-dimensional Hausdorff measure in $\R^{m+n}$;

\item[$\Lambda_m(\R^{m+n})$] vector space of $m$-vectors in $\R^{m+n}$;

\item[$\mathcal{D}^m(U)$] space of smooth differential $m$-forms with compact support in an open subset $U \subset \R^{m+n}$;

\item[$\mathscr{F}_m$, ($\mathscr{F}_{m}^p$)] integral flat chains (modulo $p$) of dimension $m$;

\item[$\Rc_m$, ($\Rc_{m}^p$)] integer rectifiable currents (modulo $p$) of dimension $m$; we write $T = \llbracket M, \vec{\tau}, \theta \rrbracket$ if $T$ is defined by integration with respect to $\vec{\tau} \, \theta \, \Ha^m \res M$ for a locally $\Ha^m$-rectifiable set $M$ oriented by the Borel measurable unit $m$-vector field $\vec{\tau}$ with multiplicity $\theta$; if $M$ is an oriented submanifold of class $C^1$, then we simply write $\a{M}$ for the associated multiplicity one current;

\item[$\In_m$, ($\In_{m}^p$)]  integral currents (modulo $p$) of dimension $m$;

\item[$\mass$, ($\mass^p$)] mass functional (mass modulo $p$);

\item[$\|T\|$, ($\|T\|_p$)] Radon measure associated to a current $T$ (to a class $\left[ T \right]$) with locally finite mass (mass modulo $p$);

\item[$\vec{T}$] Borel measurable unit $m$-vector field in the polar decomposition $T = \vec{T} \, \|T\|$ of a current with locally finite mass; if $T = \llbracket M, \vec{\tau}, \theta \rrbracket$ is rectifiable, then $\vec{T} = {\rm sgn}(\theta) \, \vec{\tau}$ $\|T\|$-a.e., so that $\vec{T}$ is an orientation of $M$;  

\item[$T \res A$] restriction of the current $T$ to the set $A$: well defined for any Borel $A$ when $T$ has locally finite mass, and for $A$ open if $T$ is any current;

\item[$\langle T, f, z \rangle$] slice of the current $T$ with the function $f$ at the point $z$;

\item[$f_\sharp T$] push-forward of the current $T$ through the map $f$;

\item[$\Theta^m(\mu,x)$] $m$-dimensional density of the measure $\mu$ at the point $x$, given by $\Theta^m(\mu,x) := \lim_{r\to 0^+} \frac{\mu(\bB_r(x))}{\omega_m\, r^m}$ when the limit exists;

\item[$\Theta_T(x)$, $\Theta(T,x)$] same as $\Theta^m(\|T\|,x)$ if $T$ is an $m$-dimensional current with locally finite mass;

\item[$\spt(\mu)$] support of $\mu$, where $\mu$ is a Radon measure on $\R^{m+n}$: it is defined as the set of all points $x \in \R^{m+n}$ such that $\mu(\bB_r(x))>0$ for all $r>0$;

\item[$\spt(T)$] same as $\spt(\|T\|)$ if $T$ is a current with locally finite mass;

\item[$\spt^p(T)$] support $\modp$ of an integer rectifiable current $T$: it only depends on the equivalence class $[T]$;

\item[$\V(M,\Theta)$] rectifiable $m$-varifold defined by $\Theta \, \Ha^m \res M \otimes \delta_{T_{\cdot} M}$ for a locally $\Ha^m$-rectifiable set $M$ and a locally $\Ha^m \res M$-integrable multiplicity $\Theta$;

\item[$\V(T)$] integral varifold associated to an integer rectifiable current $T$: if $T = \llbracket M, \vec{\tau}, \theta \rrbracket$, then $\V(T) = \V(M,\abs{\theta})$;

\item[$\delta V {\left[X\right]} $] first variation of the varifold $V$ in the direction of the vector field $X$;

\item[$A_\Sigma$] second fundamental form of a submanifold $\Sigma \subset \R^{m+n}$;

\item[$H_\Sigma$] mean curvature of a submanifold $\Sigma \subset \R^{m+n}$;

\item[$\Lip(X,Y)$] space of Lipschitz functions $f \colon X \to Y$, where $X,Y$ are metric spaces;

\item[$\Lip(f)$] Lipschitz constant of the Lipschitz function $f$;

\item[$\left(\mathcal{A}_Q(\R^n), \cG \right)$] metric space of classical $Q$-points in $\R^n$;

\item[$\left(\mathscr{A}_Q(\R^n), \cG_s \right)$] metric space of special $Q$-points in $\R^n$;

\item[$\bfeta(S)$] average of the $Q$-point $S$, so that if $S = \sum_{i=1}^Q \a{S_i} \in \mathcal{A}_Q(\R^n)$ then $\bfeta(S) = Q^{-1}\, \sum_{i=1}^Q S_i \in \R^n$;

\item[$\bfeta \circ f$] average of the (possibly special) multiple valued function $f$;

\item[$\gr (u)$] set-theoretical graph of a (possibly multi-valued) function $u$;

\item[$\bT_F$] integer rectifiable current associated (via push-forward) to the image of a (possibly special) multiple valued function;

\item[$\bG_u$] integer rectifiable current associated to the graph of a (possibly special) multiple valued function.

\end{itemizeb}

\newpage

\part{Preliminary observations and blow-up sequence} \label{part:preliminary}

\section{Preliminary reductions}

We recall first that, as specified in \cite[4.2.26]{Federer69}, for any $S\in \Rc_m (\Sigma)$ we can find a {\em representative $\modp$}, namely a $T\in \Rc_m (\Sigma)$ congruent to $S$ $\modp$ such that
\begin{equation}\label{e:rappresentante}
\|T\| (A) \leq \frac{p}{2}\, \mathcal{H}^m (A) \qquad \mbox{for every Borel $A\subset \Sigma$.}
\end{equation}
In particular, such a representative has multiplicity function $\theta$ such that $\abs{\theta} \leq p/2$ at $\|T\|$-a.e. point, and it satisfies $\mass^p ([T\res U]) = \|T\| (U)$ for every open set $U$ and $\spt (T) = \spt^p (T)$ (observe in passing that the restriction to an open set $U$ is defined for every current). It is evident that if $T \in \Rc_m (\Sigma)$ is area minimizing $\modp$ in $\Omega \cap \Sigma$ then $T$ is necessarily representative $\modp$ in $\Omega \cap \Sigma$, in the sense that \eqref{e:rappresentante} holds true for every Borel $A \subset \Omega \cap \Sigma$. For this reason, we shall always assume that $T$ is representative $\modp$, and that the aforementioned properties concerning multiplicity, mass and support of $T$ are satisfied. Note also that such $T$ is area minimizing $\modp$ in any smaller open set $U\subset \Omega$. Moreover $T$ is area minimizing $\modp$ in $\Omega$ if and only if $T\res \Omega$ is area minimizing $\modp$ in $\Omega$. Also, for $\Omega$ sufficiently small the regularity of $\Sigma$ guarantees that $\Sigma \cap \Omega$ is a graph, and thus, if in addition $\Omega$ is a ball, $\overline{\Sigma \cap \Omega}$ is a Lipschitz deformation retract of $\mathbb R^{m+n}$. A current $S\in \Rc_m (\Sigma \cap \Omega)$ is thus a cycle $\modp$ if and only if it is a cycle $\modp$ in $\mathbb R^{m+n}$.
In these circumstances it does not matter what the shape of the ambient manifold $\Sigma$ is outside $\Omega$ and thus, without loss of generality, we can assume that $\Sigma$ is in fact an entire graph. By the same type of arguments we can also assume that $\partial^p [T] = 0$ in $\Omega$. We summarize these reductions in the following assumption (which will be taken as a hypothesis in most of our statements) and in a lemma (which will be used repeatedly). 

\begin{ipotesi}\label{ass:main} 
$\Sigma$ is an entire $C^{3, a_0}$ $(m+\bar n)$-dimensional graph in $\mathbb R^{m+n}$ with $0 < a_0 \leq 1$, and $\Omega \subset \mathbb R^{m+n}$ is an open ball. $T$ is an $m$-dimensional representative $\modp$ in $\Sigma$ that is area minimizing $\modp$ in $\Sigma \cap \Omega$ and such that $(\partial T) \res \Omega =0$ $\modp$ in $\Omega$.  
\end{ipotesi}

\begin{lemma}\label{l:hom_trivial}
Let $\Omega$, $\Sigma$ and $T$ be as in Assumption \ref{ass:main}. Let $T'\in \Rc_m (\Sigma)$ be such that
$\spt (T'-T) \subset \Omega$ and $\partial T' = \partial T\, \modp$. Then
\begin{equation}\label{e:am_2}
\mass (T\res \Omega) \leq \mass (T' \res \Omega)\, .
\end{equation}
\end{lemma}

Theorem \ref{t:main} is then equivalent to

\begin{theorem}\label{t:main2}
Under the Assumption \ref{ass:main} $\sing (T)$ has Hausdorff dimension at most $m-1$. 
\end{theorem}

\section{Stationarity and compactness}
Another important tool that will be used repeatedly in the sequel is the fact that the integral varifold $\V(T)$ induced by an area minimizing representative $\modp$ $T$ is stationary in the open set $\Omega \cap \Sigma \setminus \spt^p(\partial T)$.

\begin{lemma}\label{l:varifold}
Let $\Omega$, $\Sigma$ and $T$ be as in Assumption \ref{ass:main}. Then $\V (T)$ is stationary in $\Sigma \cap \Omega$, namely 
\begin{equation}\label{e:stat}
\delta \V (T) [X] = 0 \qquad \mbox{for all $X\in C_c^1 (\Omega, \mathbb R^{m+n})$ tangent to $\Sigma$.}
\end{equation}
More generally, for $X\in C_c^1 (\Omega, \mathbb R^{m+n})$ we have
\begin{equation}\label{e:stat2}
\delta \V (T) [X] = - \int X \cdot \vec{H}_T (x)\, d\|T\| (x)\,,
\end{equation}
where the mean curvature vector $\vec{H}_T$ can be explicitly computed from the second fundamental form $A_\Sigma$ of
$\Sigma$. More precisely,
if the orienting vector field of $T$ is $\vec{T} (x) = v_1 \wedge \ldots \wedge v_m$ and $v_i$ are orthonormal, then
\begin{equation}\label{e:mean_curv}
\vec{H}_T (x) = \sum_{i=1}^m A_\Sigma (v_i, v_i)\, .
\end{equation}
\end{lemma} 
\begin{proof}
Consider a diffeomorphism $\Phi$ of $\Omega$ such that $\Phi(\Sigma \cap \Omega) \subset \Sigma \cap \Omega$ and $\Phi|_{\Omega \setminus K} \equiv {\rm id}|_{\Omega \setminus K}$ for some compact set $K \subset \Sigma \cap \Omega$. The current $\Phi_{\sharp}T$ satisfies ${\rm spt}(T - \Phi_{\sharp}T) \Subset \Sigma \cap \Omega$. Moreover, since $\partial (\Phi_\sharp T) = \Phi_\sharp (\partial T)$ and $\partial T = 0 \, \modp$, also $\partial (\Phi_\sharp T) = 0 \, \modp$, so that, in particular,
\begin{equation} \label{eq:boundary_cond}
\partial (\Phi_{\sharp}T) = \partial T\, \modp.
\end{equation}
%Indeed, from \cite[4.2.26, p. 425]{Federer69} it follows that for any $\eps > 0$ there exists $Z \in \F_{m-1}(\Sigma)$ such that $Z = \partial T\, \modp$ and $\spt(Z) \subset \bB_{\eps}(\spt^{p}(\partial T))$, where $\bB_{\eps}(\spt^{p}(\partial T))$ denotes the $\eps$-neighborhood of $\spt^{p}(\partial T)$. Therefore, if we choose $\eps$ small enough then $\Phi \equiv {\rm id}$ on $\bB_{\eps}(\spt^{p}(\partial T))$, and thus
%\[
%\Phi_{\sharp}Z = Z = \partial T\, \modp.
%\]
%On the other hand, since $Z = \partial T\, \modp$ clearly also
%\[
%\Phi_{\sharp}Z = \Phi_{\sharp}(\partial T)\, \modp,
%\]
%and since $\Phi_{\sharp}(\partial T) = \partial (\Phi_{\sharp}T)$ equation \eqref{eq:boundary_cond} follows. 
From \eqref{e:am_2}, and setting $V := {\bf v}(T)$, we then get
\[
\|V\| (\Omega) = \mass (T\res \Omega) \leq \mass (\Phi_\sharp T \res \Omega) = \|\Phi_\sharp V\| (\Omega)\, . 
\]
This easily implies that $V$ is stationary in $\Sigma \cap \Omega$. 

The second claim of the Lemma follows then from the stationarity of $V$ in $\Sigma$, see for instance \cite{Simon83}. 
\end{proof}

Consider now an open ball $\bB_R = \Omega\subset \mathbb R^{m+n}$, a sequence of Riemannian manifolds $\Sigma_k$ and a sequence of currents $T_k$ such that each triple $(\Omega, \Sigma_k, T_k)$ satisfies the Assumption \ref{ass:main}. In addition assume that:
\begin{itemize}
\item[(a)] $\Sigma_k$ converges locally strongly in $C^2$ to a Riemannian submanifold $\Sigma$ of $\mathbb R^{m+n}$ which is also an entire graph;
\item[(b)] $\sup_k \|T_k\| (\bB_R) = \sup_k \mass^p (T_k\res \bB_R) < \infty$;
\item[(c)] $\sup_k \mass^p (\partial (T_k\res \bB_R)) < \infty$. 
\end{itemize}
By the compactness theorem for integral currents $\modp$ (cf. \cite[Theorem (4.2.17)$^{\nu}$, p. 432]{Federer69}), we conclude the existence of a subsequence, not relabeled, of a current  $T \in \Rc_m (\mathbb R^{m+n})$ and of a compact set $K\supset \bB_R$ such that
\[
\lim_{k\to\infty} \F_K^p (T_k\res \bB_R - T) = 0\, 
\]
and
\[
(\partial T)\res \bB_R = 0\,\quad \modp\, .
\]
Let $U_\delta$ be the closure of the $\delta$-neighborhood of $\Sigma$ and consider that, for a sufficiently small $\delta>0$, the compact set $K':= \overline{\bB}_R\cap U_\delta$ is a Lipschitz deformation retract of $\mathbb R^{m+n}$. For $k$ sufficiently large, the currents $T_k \res \bB_R$ are supported in $K'$ and \cite[Theorem (4.2.17)$^{\nu}$]{Federer69} implies that $\spt (T) \subset K'$. Since $\delta$ can be chosen arbitrarily small, we conclude that $\spt (T) \subset \Sigma$ and hence that $T\in \Rc_m (\Sigma)$.

At the same time, by Allard's compactness theorem for stationary integral varifolds, we can assume, up to extraction of a subsequence, that $\V (T_k \res \bB_R)$ converges to some integral varifold $V$ in the sense of varifolds. 

\begin{proposition}\label{p:compactness}
Consider $\Omega, \Sigma_k, T_k, \Sigma, T$ and $V$ as above. Then
\begin{itemize}
\item[(i)] $T$ is minimizing $\modp$ in $\Omega\cap \Sigma$, so that, in particular, $T$ is representative $\modp$;
\item[(ii)] $V = \V(T)$ is the varifold induced by $T$.
\end{itemize}
\end{proposition}
\begin{proof}
Let us simplify the notation by writing $T_k$ in place of $T_k \res \bB_R$. Recall that $\F_K^p (T_k - T) \to 0$ for some compact set $K\supset \bB_R$. This means that there are sequences of rectifiable currents $R_k$, $S_k$ and integral currents $Q_k$ \footnote{Although the definition of flat convergence modulo $p$ is given with $Q_k$ flat chains, a simple density argument shows that we can in fact take them integral.} with support in $K$ such that
\begin{equation}\label{e:flat_conv_1}
T_k - T = R_k + \partial S_k + p Q_k
\end{equation}  
and
\begin{equation}\label{e:flat_conv_2}
\lim_{k\to \infty} \left(\mass (R_k) + \mass (S_k)\right) = 0\, .
\end{equation}
As above, denote by $U_\delta$ the closure of the $\delta$-neighborhood of the submanifold $\Sigma$. Observe next that, for every $\delta$ sufficiently small, $K_\delta := U_\delta \cap \overline\bB_R$ is a Lipschitz deformation retract. Moreover, for each $k$ sufficiently large $\spt (T_k) \subset K_\delta$. We can thus assume, without loss of generality, the existence of a $\bar{k} (\delta) \in \mathbb N$ such that 
\begin{equation}\label{e:supports}
\spt (R_k),\, \spt (S_k),\, \spt (Q_k) \subset K_\delta \qquad \forall k\geq \bar{k} (\delta)\, . 
\end{equation}
Next, if we denote by $U_{\delta, k}$ the closures of the $\delta$-neighborhoods of $\Sigma_k$, due to their $C^2$ regularity and $C^2$ convergence to $\Sigma$, for a $\delta>0$ sufficiently small (independent of $k$) the nearest point projections
\[
\p_k : U_{\delta, k} \to \Sigma_k
\]
are well defined. Moreover, 
\begin{equation}\label{e:Lip1}
\lim_{\sigma \downarrow 0} \sup_k \Lip (\left.\p_k\right|_{U_{\sigma, k}}) = 1\, .
\end{equation}

We now show that $T$ is area minimizing $\modp$ in $\bB_R\cap \Sigma$. Assume not: then there is a $\rho<R$ and a 
current $\hat T$ with $\spt (T- \hat T) \subset \overline{\bB}_\rho \cap \Sigma$ such that 
\[
\partial \hat T = \partial T \quad \modp\,  
\]
and, for every $s \in ]\rho, R[$,
\begin{equation}
\varepsilon := \mass (T\res \bB_s) - \mass (\hat{T} \res \bB_s) >0 \, ,
\end{equation}
where $\varepsilon$ is independent of $s$ because of the condition $\spt (T- \hat T) \subset \overline{\bB}_\rho$.

Denote by $d: \mathbb R^{m+n} \to \mathbb R$ the map $x\mapsto |x|$ and consider the slices 
$\langle S_k, d, s\rangle$. By Chebyshev's inequality, for each $k$ we can select an $s_k \in ]\rho, \frac{R+\rho}{2}[$ such that
\begin{equation}\label{e:Cheb}
\mass (\langle S_k, d, s_k\rangle) \leq \frac{2}{R-\rho} \mass (S_k)\, .
\end{equation}
Consider therefore the current:
\begin{equation}
\hat T_k := T_k \res (\R^{m+n} \setminus \bB_{s_k}) - \langle S_k, d, s_k\rangle + R_k \res \bB_{s_k} + \hat{T}\res \bB_{s_k}\, .
\end{equation}
Observe first that $\spt (T_k-\hat{T}_k) \subset \bB_{\frac{R+\rho}{2}}$. Also, note that \eqref{e:flat_conv_1} implies that
$\partial S_k$ has finite mass. Hence, by \cite[Lemma 28.5(2)]{Simon83},
\[
\langle S_k, d, s_k\rangle = \partial (S_k \res \bB_{s_k}) - (\partial S_k) \res \bB_{s_k}\, .
\] 
In particular, combining the latter equality with \eqref{e:flat_conv_1}, we get 
\begin{align*}
\partial \hat{T}_k :&= \partial (T_k \res \R^{m+n}\setminus \bB_{s_k}) + \partial ((T_k - T -R_k - p Q_k) \res \bB_{s_k})
 + \partial (R_k \res \bB_{s_k}) + \partial (\hat{T}\res \bB_{s_k})\\
& = \partial T_k - p \partial (Q_k \res \bB_{s_k}) + \partial (\hat T- T)\,,
\end{align*}
where in the second line we have used that $\spt (\hat{T}-T) \subset \bB_\rho \subset \bB_{s_k}$. Since $\partial (\hat T -T) =0$ $\modp$ in $\Sigma \subset \mathbb R^{m+n}$, we conclude that
$\partial (\hat{T}_k - T_k) = 0$ $\modp$ in $\mathbb R^{m+n}$. However, considering \eqref{e:supports}, for $k$ large enough the currents $\hat{T}_k, S_k, R_k, Q_k, T$ and $\hat{T}$ are all supported in the domain of definition of the retraction $\p_k$. Since $(\p_k)_\sharp T_k = T_k$, we
then have that $\partial (T_k - (\p_k)_\sharp \hat{T}_k) =0$ $\modp$ in $\Sigma_k$. Consider also that, for each $\sigma >0$ fixed, there is a $\bar{k} (\sigma) \in \mathbb N$ such that all the currents above are indeed supported in $U_{\sigma, k}$ when $k\geq \bar{k} (\sigma)$. This implies in particular that, by \eqref{e:Lip1}, 
\[
\liminf_{k\uparrow \infty} \mass ((\p_k)_\sharp \hat{T}_k) = \liminf_{k\uparrow \infty} \mass (\hat{T}_k)\, .
\]
Up to extraction of a subsequence, we can assume that $s_k \to s$ for some $s\in [\rho, \frac{R+\rho}{2}]$. 
Recalling the semicontinuity of the $p$-mass with respect to the flat convergence $\modp$, we easily see that (since the $T_k$'s and $T$ are all representative $\modp$)
\[
\liminf_{k\to \infty} \mass (T_k \res \bB_{s_k}) \geq \mass (T \res \bB_s)\, .
\]
Next, by the estimates \eqref{e:Cheb} and \eqref{e:flat_conv_2} we immediately gain 
\[
\liminf_{k\uparrow \infty} (\mass (\hat{T}_k) - \mass (T_k)) \leq - \varepsilon\, .
\]
Finally, since the map $\p_k$ is the identity on $\Sigma_k$, again thanks to \eqref{e:Lip1} and to the observation on the supports of $\hat{T}_k - T_k$, it turns out that $\spt ((\p_k)_\sharp \hat{T}_k - T_k) \subset \Sigma_k \cap \bB_R$ for $k$ large enough. We thus have contradicted the minimality of $T_k$.

\medskip

Observe that, if in the argument above we replace $\hat{T}$ with $T$ itself, we easily achieve that, for every fixed $\rho>0$, there is a sequence $\{s_k\} \subset ]\rho, \frac{R+\rho}{2}[$ converging to some $s\in [\rho, \frac{R+\rho}{2}]$, with the property that
\[
\liminf_{k\uparrow \infty} (\mass (T\res \bB_{s_k}) - \mass (T_k \res \bB_{s_k}))\geq 0\, .  
\]
By this and by the semicontinuity of the $p$-mass under flat convergence, we easily conclude that
\[
\lim_{k\to \infty} \|T_k\| (\bB_\rho) = \|T\| (\bB_\rho) \qquad \mbox{for every $\rho < R$.}
\]
The latter implies then that $\|T_k\| \weaks \|T\|$ in the sense of measures in $\bB_R$. Consider now the rectifiable sets $E_k$, $E$ and the Borel functions $\Theta_k: E_k \to \mathbb N\setminus \{0\}$, $\Theta: E\to \mathbb N \setminus \{0\}$ such that 
\[
\|T_k\| = \Theta_k \, \mathcal{H}^m \res E_k\,, \qquad \|T\| = \Theta \, \mathcal{H}^m \res E\, .
\]
Let $T_q E_k$ (resp. $T_q E)$ be the approximate tangent space to $E_k$ (resp. $E$) at $\Ha^m$-a.e. point $q$.
The varifold $\V (T_k)$ is then defined to be $\Theta_k \mathcal{H}^m \res E_k \otimes \delta_{T_q E_k}$. If the varifold limit $V$ is given by  $\Theta' \mathcal{H}^m \res F \otimes \delta_{T_q F}$, we then know that $\|V_k\|\weaks \|V\| =
\Theta' \mathcal{H}^m \res F$. But since $\|V_k\| = \|T_k\|$, we then know that $\mathcal{H}^m ((F\setminus E) \cup (E\setminus F)) = 0$ and that $\Theta' = \Theta$ $\mathcal{H}^m$-almost everywhere. This shows then that $V = \V (T)$.
\end{proof}

\section{Slicing formula $\modp$}

In this section we prove a suitable version of the slicing formula for currents $\modp$, which will be useful in several contexts. We let $\In_m^p(C)$ denote the group of \emph{integral currents $\modp$}, that is of classes $\left[ T \right] \in \Rc_m^p(C)$ such that $\partial^p\left[T\right] \in \Rc_{m-1}^p(C)$.

\begin{lemma}\label{modp_slicing_formula}
Let $\Omega \subset \mathbb{R}^{m+n}$ be a bounded ball, let $\left[ T \right] \in \In_m^{p} (\overline{\Omega})$ be an integral current $\modp$, and let $f: \overline{\Omega} \to \mathbb R$ be a Lipschitz function. If $T \in \Rc_{m}(\overline{\Omega})$ is any rectifiable representative of $\left[ T \right]$ and $Z \in \Rc_{m-1}(\overline{\Omega})$ is any rectifiable representative of $\left[ \partial T \right]$, then the following holds for a.e. $t\in \mathbb R$:
\begin{itemize}
\item[(i)] $\langle T, f,t\rangle = \partial (T \res \{f<t\}) - Z \res \{f < t \}\, \modp$;
\item[(ii)] $\langle T, f, t\rangle$ is a representative $\modp$ if $T$ is a representative $\modp$;
\item[(iii)] if $T$ is a representative $\modp$, and if $\partial T = 0 \, \modp$, then 
\[
\mass (\langle T, f, t\rangle) = \mass^p (\partial (T \res \{f<t\}))\, .
\]
\end{itemize} 
\end{lemma}

Before coming to the proof of Lemma \ref{modp_slicing_formula} we wish to point out two elementary consequences of the theory of currents $\modp$ which are going to be rather useful in the sequel.

\begin{lemma}\label{l:elementary_modp}
If $T$ is an integer rectifiable $m$-dimensional current in $\mathbb R^{m+n}$ and $f: \mathbb R^{m+n} \to \mathbb R^k$ is a Lipschitz map with $k\leq m$, then:
\begin{itemize}
\item[(i)] $T$ is a representative $\modp$ if and only if the density of $T$ is at most $\frac{p}{2}$ $\|T\|$-a.e.
\item[(ii)] If $T$ is a representative $\modp$, then $\langle T, f, t\rangle$ is a representative $\modp$ for a.e. $t\in \mathbb R^k$.
\item[(iii)] If $n=0$ and $\spt (T) \subset K$ for a compact set $K$, then $\mathscr{F}_K^p (T) = \mass^p (T)$. 
\item[(iv)] Let $T$ be as in (iii) and in particular $T = \Theta \a{K}$, where $\Theta$ is integer valued. If we let
\begin{equation}\label{e:norma_modp}
|\Theta (x)|_p := \min \{|\Theta (x) - kp| : k \in \mathbb Z\}\, ,
\end{equation}
then 
\begin{equation}\label{e:calcolo_massa_modp}
\mass^p (T \res E) = \int_E |\Theta (x)|_p\, dx \qquad \mbox{for all Borel $E\subset \mathbb R^{m}$.}
\end{equation}
\end{itemize} 
\end{lemma}

\begin{proof}
(i) is an obvious consequence of Federer's characterization in \cite{Federer69}: an integer rectifiable current $T$ of dimension $m$ is a representative $\modp$ if and only if  $\|T\| (E) \leq \frac{p}{2} \mathcal{H}^m (E)$ for every Borel set $E$. By the coarea formula for rectifiable sets, this property is preserved for a.e. slice and thus (ii) is immediate. Moreover, again by Federer's characterization, if $T$ is as in (iv), and if $k(x) = \arg\min\{\abs{\Theta(x) - kp} \, \colon \, k \in \Z\}$, then setting $\Theta'(x) := \Theta(x) - k(x) \, p$ we have that $T' = \Theta'\, \a{K}$ is a representative $\modp$ of $T$, and thus, since $\abs{\Theta'} = \abs{\Theta}_p$, \eqref{e:calcolo_massa_modp} follows directly from $\mass^p (T\res E) = \|T'\| (E)$. 

As for (iii), since $T$ is a top-dimensional current, $\mathscr{R}_{m+1} (K) = \{0\}$. We thus have
\[
\mathscr{F}^p_K (T) = \inf \left\{\mass (R) : T = R + p P \quad
\mbox{for some $R \in \mathscr{R}_{m}(K)$ and $P\in \mathscr{F}_m (K)$}\right\} \, .
\]
Observe however that, since $K$ is $m$-dimensional, $\mathscr{F}_m (K)$ consists of the integer rectifiable currents with support in $K$. A simple computation gives then 
\[
\mathscr{F}^p_K (T) = \int_K |\Theta (x)|_p \, dx
\]
and we can use (iv) to conclude.
\end{proof}

\begin{proof}[Proof of Lemma \ref{modp_slicing_formula}] (ii) has been addressed already in Lemma \ref{l:elementary_modp}, and (iii) is a simple consequence of Lemma \ref{l:elementary_modp} and of (i) with the choice $Z = 0$. 

\medskip

We now come to the proof of (i).
By \cite[Theorem 3.4]{MS_a}, there exists a sequence $\{ P_{k} \}_{k=1}^{\infty}$ of integral polyhedral chains and currents $R_{k} \in \Rc_{m}(\overline{\Omega})$, $S_{k} \in \Rc_{m+1}(\overline{\Omega})$ and $Q_k \in \In_m (\overline{\Omega})$, with the following properties for every $k\geq 1$:
\begin{align}
& T - P_{k} = R_{k} + \partial S_{k} + p Q_{k}\, ,\label{flat_p_decomposition}\\
& \M^{p}(P_{k}) \leq \M^{p}(T) + \frac{1}{k 2^k}\, , \label{poly_mass_control}\\
&\M^{p}(\partial P_{k} \res \Omega) \leq \M^p(\partial T \res \Omega) + \frac{1}{k 2^k}\, ,\label{poly_mass_control2}\\
&\M(R_{k}) + \M(S_{k}) \leq \frac{2}{k 2^k}.  \label{mass_estimate}
\end{align}
Since $P_{k}$ is an integral current, by the classical slicing theory (cf. for instance \cite[Lemma 28.5(2)]{Simon83}), the following formula holds for a.e. $t \in \R$:
\begin{equation} \label{poly_slicing}
\langle P_{k}, f, t \rangle = \partial \left( P_{k} \res \{ f < t \} \right) - (\partial P_{k}) \res \{ f < t \}.
\end{equation}
The identity \eqref{flat_p_decomposition} implies that $\partial S_{k}$ has locally finite mass, and thus $S_{k}$ is an integral current. In particular, $\partial \langle S_k, f, t\rangle = - \langle \partial S_k, f ,t\rangle$. Furthermore, the slicing formula holds true for $S_{k}$ as well, that is for a.e. $t \in \R$ one has:
\begin{equation} \label{bdry_slicing}
\langle S_{k}, f, t \rangle = \partial \left( S_{k} \res \{f < t \} \right) - (\partial S_{k}) \res \{f < t\}\,.
\end{equation}

Since $Z = \partial T \, \modp$, there exist currents $\tilde{R}_{k} \in \Rc_{m-1}(\overline{\Omega})$, $\tilde{S}_{k} \in \Rc_{m}(\overline{\Omega})$ and $\tilde{Q}_{k} \in \In_{m-1}(\overline{\Omega})$ such that for every $k \geq 1$:
\begin{align}
& Z - \partial T = \tilde{R}_{k} + \partial \tilde{S}_{k} + p \tilde{Q}_{k}\,, \label{bdry_p_decomposition}\\
& \mass(\tilde{R}_{k}) + \mass(\tilde{S}_{k}) \leq \frac{1}{k 2^k}\,. \label{mass_is_small}
\end{align}

Combining \eqref{flat_p_decomposition} and \eqref{bdry_p_decomposition}, we can therefore write:
\begin{equation} \label{bdry_p_decomposition_2}
\begin{split}
Z - \partial P_{k} &= \partial T - \partial P_{k} + Z - \partial T  \\
&= \tilde{R}_{k} + \partial ( R_{k} + \tilde{S}_{k} ) + p ( \partial Q_{k} + \tilde{Q}_{k} )\,. 
\end{split}
\end{equation}

The identity \eqref{bdry_p_decomposition_2} implies that $\partial (R_{k} + \tilde{S}_{k})$ has locally finite mass, and thus in particular $R_{k} + \tilde{S}_{k}$ is an integral current. Hence, for a.e. $t \in \R$ the slicing formula holds true for $R_{k} + \tilde{S}_{k}$, that is:
\begin{equation} \label{bdry_slicing_2}
\langle R_{k} + \tilde{S}_{k}, f, t \rangle = \partial \left( (R_{k} + \tilde{S}_{k}) \res \{f < t\} \right) - \left( \partial (R_{k} + \tilde{S}_{k}) \right) \res \{f < t\}\,.
\end{equation}

From the identities \eqref{flat_p_decomposition} and \eqref{bdry_p_decomposition_2}, and using \eqref{poly_slicing}, \eqref{bdry_slicing}, \eqref{bdry_slicing_2}, and the slicing formula for $Q_{k}$ we easily conclude that the following holds for a.e. $t \in \R$:
\begin{align} 
 &\langle T,f,t \rangle - \partial (T \res \{f<t\}) + Z \res \{f < t \}\nonumber\\
= &\tilde{R}_{k} \res \{f < t \} - \langle \tilde{S}_{k},f,t \rangle + \partial (\tilde{S}_{k} \res \{f<t\}) + p \tilde{Q}_{k} \res \{f< t\} \,.\label{conclusion_1}
\end{align}

Now, $\tilde{Q}_{k} \res \{f < t\}$ is an integral current and thus, setting $K:= \overline{\Omega}$, we can estimate
\[
\F_K^p (\langle T, f, t \rangle - (\partial \left( T \res \{ f < t \} - Z \res \{f < t\}) \right)
\leq \mass(\tilde{R}_{k}) + \mass(\tilde{S}_{k}) + \mass(\langle \tilde{S}_{k},f,t\rangle)\,.
\]
Since $\lim_k \left( \mass (\tilde{R}_{k}) + \mass (\tilde{S}_{k}) \right) = 0$, it remains to show that, for a.e. $t$,
\[
\lim_{k\to \infty}  \mass (\langle \tilde{S}_{k}, f,t\rangle) = 0\, .
\]
In order to see this, fix $\eps > 0$. By \cite[Lemma 28.5(1)]{Simon83}, we have that there is a Borel set $E_k$ with measure $|E_{k}| \leq \frac{\eps}{2^k}$ such that
\begin{equation} \label{coarea_slice}
\M(\langle \tilde{S}_{k}, f, t \rangle) \leq \Lip(f) \frac{2^k}{\eps} \M(\tilde{S}_k) \quad \mbox{for all } t \not\in E_{k}\, .
\end{equation}
In particular, if we set $E:= \bigcup_k E_k$, we have $|E|\leq 2\varepsilon$, and using \eqref{mass_is_small} we see that
\[
\M(\langle \tilde{S}_{k}, f, t \rangle) \leq \varepsilon^{-1} \Lip (f) k^{-1}\qquad \mbox{for all $t\not \in E$}\, .
\]
Hence $\lim_{k\to \infty}  \mass (\langle \tilde{S}_k, f,t\rangle) =0$ for all $t\not \in E$. Since $\varepsilon$ is arbitrary, this concludes the proof. 
\end{proof}

\begin{remark}
We are actually able to give a much shorter proof of Lemma \ref{modp_slicing_formula}(i), provided one can prove that there exists an \emph{integral} current $\tilde{T}$ such that $\tilde{T} = T\; \modp$. Indeed, in this case, since $\tilde{T}$ is integral the classical slicing formula gives
\[
\langle \tilde{T}, f, t \rangle = \partial \left( \tilde{T} \res \{ f < t \} \right) - (\partial \tilde{T}) \res \{f < t\}.
\]
On the other hand, the conditions $\tilde{T} = T\; \modp$ and $\partial \tilde{T} = \partial T = Z \; \modp$ imply that there are rectifiable currents $R$ and $Q$ such that $T = \tilde{T} + p R$ and $Z = \partial \tilde{T} + pQ$, and thus we deduce
\[
\langle T, f, t \rangle = \partial \left( T \res \{ f < t \} \right) - Z \res \{f < t\} + p \left(- \partial \left( R \res \{ f < t \} \right) + \langle R, f, t \rangle + Q \res \{f < t\}\right)\,,
\]
as we wanted.

The existence of an integral representative in any integral class $\modp$ is in fact a very delicate question. If $K$ is any given compact subset of $\R^{m+n}$ then a class $\left[ T \right] \in \In_{m}^p(K)$ does \emph{not} necessarily have a representative in $\In_m(K)$ when $m \geq 2$; see \cite[Proposition 4.10]{MS_a}. Positive answers have been given, instead, when $m=1$ in the class $\In_m(K)$ for any given compact $K$ in \cite[Theorem 4.5]{MS_a}, and in any dimension in the class $\bigcup_{K} \In_m(K)$ in the remarkable work \cite{Young}. 
\end{remark}

\section{Monotonicity formula and tangent cones} \label{s:mono_cones}

From Lemma \ref{l:varifold} and the classical monotonicity formula for stationary varifolds, cf. \cite{Allard72} and \cite{Simon83}, we conclude directly the following corollary.

\begin{corollary}\label{c:monotonicity}
Let $T, \Sigma$ and $\Omega = \bB_R$ be as in Assumption \ref{ass:main}. Then, if $q\in \spt (T) \cap \Omega$, the following monotonicity identity holds for every $0<s<r< R-|q|$:
\begin{align}
&\;r^{-m} \|T\| (\bB_r (q)) - s^{-m} \|T\| (\bB_s (q)) - \int_{\bB_r (q)\setminus \bB_s (q)} \frac{|(x-q)^\perp|^2}{|x-q|^{m+2}}\, d\|T\| (x)\nonumber\\
= &\; \int_s^r \rho^{-m-1} \int_{\bB_\rho (q)} (x-q)^\perp \cdot \vec{H}_T (x)\, d\|T\| (x)\, d\rho\, ,
\label{e:monot_identity}
\end{align}
where $Y^\perp (x)$ denotes the component of the vector $Y (x)$ orthogonal to the tangent plane of $T$ at $x$ (which is oriented by $\vec{T} (x)$).
In particular:
\begin{itemize}
\item[(i)] There is a dimensional constant $C (m)$ such that the map $r\to e^{C\|A_\Sigma\|_0 r} \frac{\|T\| (\bB_r (q))}{\omega_m r^m}$ is monotone increasing.
\item[(ii)] The limit
\[
\Theta_T (q) := \lim_{r\downarrow 0} \frac{\|T\| (\bB_r (q))}{\omega_m r^m}
\] 
exists and is finite at every point $q\in \bB_R$.
\item[(iii)] The map $q\mapsto \Theta_T (q)$ is upper semicontinuous and it is a positive integer at $\mathcal{H}^m$-a.e. $q\in \spt (T)$. In particular 
$\spt (T) \cap \bB_R = \{\Theta_T \geq 1\}$. 
\end{itemize} 
\end{corollary}

Next, we introduce the usual blow-up procedure to analyze tangent cones at $q\in \spt (T)$.

\begin{definition}\label{def:tc}
Fix a point $q\in\spt(T)$ and define
\[
\iota_{q,r}(x):=\frac{x-q}{r}\quad\forall\,r>0\,.
\]
We denote by $T_{q,r}$ the currents
\[
T_{q,r}:=(\iota_{q,r})_\sharp T\quad\forall\,r>0\,.
\]
\end{definition}

Recalling Allard's theory of stationary varifolds, we then know that, for every sequence $r_k\downarrow 0$, a subsequence, not relabeled, of $\V (T_{q,r_k})$ converges locally to a varifold $C$ which is a stationary cone in $T_q \Sigma$ (the tangent space to $\Sigma$ at $q$). Combined with Proposition \ref{p:compactness} we achieve the following corollary.

\begin{corollary}\label{c:tangent_cones}
Let $T, \Sigma$ and $\Omega = \bB_R$ be as in Assumption \ref{ass:main}, let $q\in \spt (T) \cap \Omega$, and let $r_k\downarrow 0$. Then there is a subsequence, not relabeled, and a current $T_0$ with the following properties:
\begin{itemize}
\item[(i)] $T_0\res \bB_\rho \in \Rc_m (T_q \Sigma)$, $\partial T_0 \res \bB_\rho =0$ $\modp$ for every $\rho>0$;
\item[(ii)] $T_0\res \bB_\rho$ is a representative $\modp$ and is area minimizing $\modp$ in $\bB_\rho \cap T_q\Sigma$ for every $\rho>0$;
\item[(iii)] $T_0$ is a cone, namely $(\iota_{0,r})_\sharp T_0 = T_0$ for every $r>0$;
\item[(iv)] For every $\rho>0$ there is $r\geq \rho$ and $K\supset \bB_r$ such that 
\[
\lim_{k\to\infty} \F^p_K (T_{q,r_k} \res \bB_r - T_0\res \bB_r) =0\, .
\]
\item[(v)] If $\spt^p (T_0) = \spt (T_0)$ is contained in an $m$-dimensional plane $\pi_0$, then $T_0 = Q \a{\pi_0}$ for some $Q\in \mathbb Z \cap [-\frac{p}{2}, \frac{p}{2}]$. 
\end{itemize}
\end{corollary}

Before coming to its proof, let us state an important lemma which will be used frequently during the rest of the paper. See \cite[Theorem 7.6]{DPH2} for a proof. 

\begin{lemma}[Constancy Lemma]\label{l:constancy}
Assume $\pi\subset \mathbb R^{m+n}$ is an $m$-dimensional plane and let $\Omega\subset \mathbb R^{m+n}$ be an open set such that $\Omega \cap \pi$ is connected. Assume $T\in \Rc_m (\pi)$ is a current such that $(\partial^p T) \res \Omega =0$. Finally let $\vec{v} = v_1 \wedge \ldots \wedge v_m$ for an orthonormal basis $v_1, \ldots , v_m$ of $\pi$. Then there is a $Q\in \mathbb Z \cap [-\frac{p}{2}, \frac{p}{2}]$ such that 
$T\res \Omega = Q \, \vec{v} \, \mathcal{H}^m \res (\Omega \cap \pi)$ $\modp$. 
\end{lemma}

\begin{proof}[Proof of Corollary \ref{c:tangent_cones}] Note that (v) is an obvious consequence of the constancy lemma and of (i). 
In order to prove the remaining statements, first extract a subsequence such that $V_k = \V (T_{q,r_k})$ converges to a stationary cone $C$ as above. 
Then observe that for every $j\in \mathbb N$, using a classical Fubini argument and Lemma \ref{modp_slicing_formula} we find a radius $\rho (j) \in [j,j+1]$ such that
\begin{align*}	
\liminf_k \mass^p (\partial (T_{q, r_k} \res \bB_{\rho (j)})) &= \liminf_k \mass (\langle T_{q, r_k}, |\cdot|, \rho(j)\rangle)  \\
&\leq \liminf_k \|T_{q,r_k} \| (\bB_{j+1}\setminus \bB_j) = \omega_m \Theta_T (q) ((j+1)^m - j^m)\, .
\end{align*}
Thus we can find a subsequence to which we can apply the compactness Proposition \ref{p:compactness}. By a standard diagonal argument we can
thus find a single subsequence $r_k$ with the following properties:
\begin{itemize}
\item[(a)] For each $j$ there is a current $T^j \in \Rc_m (T_q \Sigma)$ such that 
\[
\lim_{k\to \infty} \F_{\overline{\bB}_{j+1}}^p (T_{q,r_k}\res \bB_{\rho (j)}- T^j) = 0\, .
\]
\item[(b)] Each $T^j$ is a representative $\modp$ and $\V (T^j) = C \res \bB_{\rho (j)}$.
\item[(c)] Each $T^j$ is area minimizing $\modp$ in $\bB_{\rho (j)}$.
\end{itemize}
Notice next that $T^j \res \bB_{\rho (i)} = T^i$ $\modp$ for every $i\leq j$. If we then define the current
\[
T_0 := \sum_{i\in \mathbb N} T^i \res (\bB_{\rho (i)}\setminus \bB_{\rho (i-1)})\, ,
\]
with $\rho(-1) := 0$, then the latter satisfies the conclusions (i), (ii) and (iv). 

In the remaining part of the proof we wish to show (iii), after possibly changing $T_0$ to another representative $\modp$ of the same class. 
\medskip

To this aim, consider that, by standard regularity theory for stationary varifolds, the closed set $R = \spt (C)$ is countably $m$-rectifiable, it is a cone with
vertex at the origin and $\|C\| = \Theta_C (x) \mathcal{H}^m \res R$, where $\Theta_C$ is the density of the varifold $C$. 
By the monotonicity formula and $\V(T)=C$ we have
\[
\Theta_{T_0} (x) =  \Theta_C (x)\, .
\]
If $x$ is a point where the approximate tangent $T_x R$ exists, we then conclude easily that, up to subsequences, we can apply the same
argument above and find that $(T_0)_{x, r_k}$ with $r_k\downarrow 0$ converges locally $\modp$ to a current $S$ satisfying the corresponding conclusions:
\begin{itemize}
\item[(i)'] $S\res \bB_\rho \in \Rc_m (T_q \Sigma)$ and $\partial S \res \bB_\rho =0$ $\modp$ for every $\rho>0$;
\item[(ii)'] $S\res \bB_\rho$ is a representative $\modp$ and is area minimizing $\modp$ in $\bB_\rho \cap T_q\Sigma$ for every $\rho>0$;
\item[(iv)'] For every $\rho>0$ there is $r\geq \rho$ and $K\supset \bB_r$ such that 
\[
\lim_{k\to\infty} \F^p_K ((T_0)_{x,r_k} \res \bB_r - S\res \bB_r) =0\, .
\]
\end{itemize}
However, for $S$ we would additionally know that it is supported in $T_x R$, which is an $m$-dimensional plane. We then could apply the Constancy Lemma and conclude that, if $v_1, \ldots, v_m$ is an orthonormal basis of $T_x R$, then
$\Theta_C (x) \in \mathbb N \cap [1, \frac{p}{2}]$ and, for any $\rho>0$,  
\begin{align*} 
\mbox{either}\; &\; S \res \bB_\rho = \Theta_C (x) v_1\wedge \ldots \wedge v_m \mathcal{H}^m \res T_x R\cap \bB_\rho \qquad \modp\\
\mbox{or}\; &\; S \res \bB_\rho = - \Theta_C (x) v_1\wedge \ldots \wedge v_m \mathcal{H}^m \res T_x R \cap \bB_\rho\qquad \modp\, .
\end{align*}
In particular we conclude that there is a Borel function $\epsilon: \spt (C) = R \to \{-1,1\}$ such that
\begin{equation}\label{e:repr_1}
T_0 = \epsilon\, \Theta_C\,  \vec{v}\, \mathcal{H}^m \res R\, ,
\end{equation}
where $\vec{v} (x)$ is an orienting Borel unit $m$-vector for $T_x R$. Clearly, since $R$ is a cone, we can choose $\vec{v} (x)$ with the additional property that $\vec{v} (x) = \vec{v} (\lambda x)$ for every positive $\lambda$. Also, since the varifold $C$ is a cone, the density $\Theta_C$ is $0$-homogeneous as well. Moreover, at all points $x$ where $\Theta_C (x) = \frac{p}{2}$ we can arbitrarily
set $\epsilon (x) =1$, since this would neither change the class $\modp$, nor the fact that $T_0$ is representative $\modp$. 

\medskip

Fix now a radius $s>0$ such that the conclusions of Lemma \ref{modp_slicing_formula} hold with $T = T_0$, $f = \abs{\cdot}$, and $t=s$, and consider the cone $T':= \langle T_0 , |\cdot|,s \rangle \cone \{0\}$. Observe that $\partial (T'-T_0 \res \bB_s)=0$ $\modp$. We now make the following simple observation: if $Z\in \Rc_m (\mathbb R^{m+n})$ with $\spt(Z)$ compact is such that 
$\partial Z=0$ $\modp$ in $\mathbb R^{m+n}$, then $\partial (Z\cone \{0\}) = Z$ $\modp$. The proof is in fact a simple consequence of the
definition, since $\partial Z=0$ $\modp$ implies the existence of integer rectifiable currents $Q^{(1)}_k$ and $Q^{(2)}_k$
and flat currents $Q_k$ such that
\[
\partial Z = p Q_k + Q^{(1)}_k + \partial Q^{(2)}_k
\]
and 
\[
\mass (Q^{(1)}_k) + \mass (Q^{(2)}_k) \to 0\, .
\]
Using the general formula $\partial (A\cone 0) = A - (\partial A)\cone 0$ we then obtain 
\[
\partial (Z\cone 0) = Z - p Q_k \cone 0 - Q^{(1)}_k \cone \{0\} + \partial (Q^{(2)}_k \cone 0) - Q^{(2)}_k\, , 
\]
which by 
\[
\mass (Q^{(1)}_k \cone \{0\} + Q^{(2)}_k) + \mass (Q^{(2)}_k \cone 0) \to 0
\]
implies that indeed $\partial (Z\cone 0) = Z$ $\modp$. 

We apply the above observation to $Z= T' - T_0\res \bB_s$. In that case we conclude however that the cone 
\[
Z\cone 0 \qquad \mbox{is identically $0$,}
\]
because it is
an $(m+1)$-dimensional rectifiable current supported in the countably $m$-rectifiable set $R$. We thus must necessarily have that
$T' - T_0\res \bB_s = 0\, \modp$. Applying the argument of the previous paragraph, we of course again conclude that
\begin{equation}
T' = \epsilon'\, \Theta_C\, \vec{v}\, \mathcal{H}^m \res R \cap \bB_s\, .
\end{equation}
Consider now, as above, a point $x\in \bB_s$ where the approximate tangent plane to $R$ exists. Then $(T')_{x,r}$ converges, as $r\downarrow 0$, to
$\epsilon' (x)\, \Theta_C (x)\, \vec{v} (x)\, \mathcal{H}^m \res T_x R$, whereas $(T_0)_{x,r}$ converges, as $r\downarrow 0$, to
$\epsilon (x)\, \Theta_C (x)\, \vec{v} (x)\, \mathcal{H}^m \res T_x R$. However the two limits must be congruent $\modp$ and, in case $\Theta_C (x) <\frac{p}{2}$, this necessarily implies $\epsilon (x) = \epsilon' (x)$. 

Fix now $\lambda >0$. 
Since $T'$ is a cone and $s$ is arbitrary, we conclude that for $\mathcal{H}^m$ a.e. $x\in R \cap \{\Theta_C < \frac{p}{2}\}$ we must necessarily have $\epsilon (x) = \epsilon' (x) = \epsilon' (\lambda x) = \epsilon (\lambda x)$. 
On the other hand we already have $\epsilon (x) = \epsilon (\lambda x) =1$ if $\Theta_C (x) = \frac{p}{2}$. Hence we have concluded that $\epsilon (\lambda x) = \epsilon (x)$ for $\mathcal{H}^m$-a.e. $x\in R$. In particular $(\iota_{0,\lambda})_\sharp T_0 = T_0$. The arbitrariness of $\lambda$ implies now the desired conclusion (iii) and completes the proof of the corollary. 
\end{proof}

\section{Strata and blow-up sequence}  \label{s:strata}

\begin{definition}[$Q$-points] \label{Q-points}
For every $Q \in \N \setminus \{ 0 \}$, we will let ${\rm D}_{Q}(T)$ denote the points of density $Q$ of the current $T$, namely
\[
{\rm D}_{Q}(T) := \lbrace q \in \Omega \, \colon \, \Theta_{T}(q) = Q \rbrace\, .
\]
We also set
\[
\reg_{Q}(T) := \reg(T) \cap {\rm D}_{Q}(T) \quad \mbox{and} \quad \sing_{Q}(T) := \sing(T) \cap {\rm D}_{Q}(T).
\]
\end{definition}

Theorem \ref{t:main3} is thus equivalent to

\begin{theorem}\label{t:main5}
Under Assumption \ref{ass:main}, for every $Q<\frac{p}{2}$ the set $\sing_Q (T)$ has Hausdorff dimension at most $m-2$. 
\end{theorem}

Before proceeding, we need to recall the following definition.

\begin{definition} \label{def:symmetric cones and stratification}
An integral $m$-varifold $V$ is called a \emph{$k$-symmetric} cone (where $0 \leq k \leq m$) if it can be written as the product of a $k$-dimensional plane passing through the origin times an $(m-k)$-dimensional cone. The largest plane passing through the origin such that the above holds is called the \emph{spine} of $V$. If $V$ is stationary, then the standard \emph{stratification} of $V$ is 
\begin{equation} \label{e:stratification}
\cS^{0} \subset \cS^{1} \subset \dots \subset \cS^{m},
\end{equation}
where
\begin{equation} \label{e:strata definition}
\cS^{k} := \lbrace q \in \spt(V) \, \colon \, \mbox{no tangent cone to $V$ at $q$ is $(k+1)$-symmetric} \rbrace.
\end{equation}
\end{definition}

As a consequence of Corollary \ref{c:tangent_cones} and of the classical Almgren's stratification theorem, we have now the following

\begin{proposition} \label{cor:densities}
Let $T, \Sigma$ and $\Omega$ be as in Assumption \ref{ass:main} and consider the set
\[
Z := \Omega \cap \spt (T) \setminus \bigcup_{Q\in \mathbb N\setminus \{0\}, Q\leq \frac{p}{2}} {\rm D}_Q (T)\, .
\]
Then $\mathcal{H}^{m-1+\alpha} (Z) =0$ for every $\alpha>0$.
\end{proposition}
\begin{proof}
By Lemma \ref{l:varifold}, the varifold $V = \V(T)$ is stationary in $\Sigma \cap \Omega$, thus we can consider the stratification of $V$ as in \eqref{e:stratification} and \eqref{e:strata definition}. If $q \in \cS^{m} \setminus \cS^{m-1}$ then there is at least one tangent cone to $V$ at $q$ which is supported in a flat plane $\pi_{0}$. Then there is a current $T_0$ as in Corollary \ref{c:tangent_cones}, obtained as a limit $T_{q,r_k}$ for an appropriate $r_k\downarrow 0$, which satisfies $\V (T_0) = V$. Thus by the constancy lemma $\Theta_{T_0} (0) = \Theta_T (q)$ must belong to $[1, \frac{p}{2}]\cap \mathbb N$. This implies that $Z \subset \mathcal{S}^{m-1}$. Our statement then follows immediately from the well known fact that $\dim_{\Ha} \cS^k \leq k$ for every $0 \leq k \leq m$.
\end{proof}

We shall also need the following elementary yet fundamental lemmas. Given $v \in \R^{m+n}$, we will adopt the notation $\tau_v := \iota_{v,1}$, so that $\tau_v(x) := x-v$. 

\begin{lemma} \label{l:splitting}

Assume $T \in \Rc_m(\R^{m+n})$ is an $m$-dimensional integer rectifiable current such that $\partial T = 0 \, \modp$ and the associated varifold $\V(T)$ is a $k$-symmetric cone with spine $\R^k \times \{0\} \subset \R^{m+n}$. Then 
\begin{equation} \label{e:invariance modp}
(\tau_v)_\sharp T = T \, \modp \qquad \mbox{for every $v \in \R^k \times \{0\}$}\,,
\end{equation}
and there exists an $(m-k)$-dimensional cone $T'$ such that
\begin{equation} \label{e:splitting}
T = \llbracket \R^k \rrbracket \times T' \quad \modp\,.
\end{equation}

Furthermore, if $T$ is a representative $\modp$ then so is $T'$; in this case, $\V(T) = \V(\a{\R^k} \times T')$, and $\V(T')$ has trivial spine. Finally, if $T$ is locally area minimizing $\modp$, then so is $T'$.
%
%
%
%Furthermore, if $T$ is a cone, and if $\Theta(T,0) < \frac{p}{2}$, then in fact
%\begin{equation} \label{e:invariance}
%(\tau_v)_\sharp T = T  \qquad \mbox{for every $v \in \R^k \times \{0\}$}\,,
%\end{equation}
%and thus there exists an $(m-k)$-dimensional rectifiable cone $T'$ such that
%\begin{equation} \label{e:splitting}
%T = \llbracket \R^k \rrbracket \times T'\,.
%\end{equation}
%
%Finally, if $T$ is (locally) area minimizing $\modp$ then so is $T'$.
\end{lemma}

\begin{proof}

Write $T = \llbracket M, \vec{\tau}, \theta \rrbracket$, so that $\V(T) = \V(M,|\theta|)$. Since $\V(T)$ is a $k$-symmetric cone with spine $\R^k \times \{0\}$, the locally $\Ha^m$-rectifiable set $M$ is a cone which is invariant with respect to $\R^k \times \{0\}$, in the sense that there exists a locally $\Ha^{m-k}$-rectifiable set $M' \subset \R^{m+n-k}$ such that $M = \R^k \times M'$. Furthermore, $|\theta|$ is a $0$-homogeneous function such that $|\theta|(x+v) = |\theta|(x)$ for every $v \in \R^k \times \{0\}$. By the properties of $M$, modulo changing the sign of $\theta$, we can also assume that the orienting unit $m$-vector field $\vec{\tau}$ is a $0$-homogeneous function such that $\vec{\tau}(x+v) = \vec{\tau}(x)$ for every $v \in \R^k \times \{0\}$.

\smallskip

Now, given two smooth and proper maps $f,g \colon \R^{m+n} \to \R^{m+n}$, and letting $h \colon \left[ 0, 1 \right] \times \R^{m+n} \to \R^{m+n}$ be the linear homotopy from $f$ to $g$, namely the function defined by
\[
h(t,x) := (1-t) \, f(x) + t\, g(x)\,,
\]
the homotopy formula (see \cite[Equation 26.22]{Simon83}) states that
\begin{equation} \label{homotopy formula}
g_\sharp T - f_\sharp T = \partial h_\sharp(\llbracket \left(0,1\right)\rrbracket \times T) + h_\sharp(\llbracket \left(0,1\right)\rrbracket \times \partial T)\,.
\end{equation}

Since $\partial T = 0 \, \modp$, \eqref{homotopy formula} yields
\begin{equation} \label{homotopy modp}
g_\sharp T - f_\sharp T = \partial h_\sharp(\llbracket \left(0,1\right)\rrbracket \times T) \quad \modp\,.
\end{equation}

Now, let $v \in \R^k \times \{0\}$, and apply \eqref{homotopy modp} with
\[
f(x) = x \qquad \mbox{and} \qquad g(x) = \tau_v(x) = x-v\,.
\]
We can compute, for any $\omega \in \mathscr{D}_{1+m}(\R \times \R^{m+n})$:
\[
\begin{split}
h_\sharp(\llbracket \left(0,1\right)\rrbracket \times T)(\omega) :&= (\llbracket \left(0,1\right)\rrbracket \times T)(h^\sharp \omega) \\
&= \int_0^1 dt \int \langle \omega(h(t,x)), [dh_{(t,x)}]_\sharp (e_1 \wedge \vec{T}(x)) \rangle \, d\|T\|(x) \\
&= - \int_0^1 dt \int \langle \omega(h(t,x)), v \wedge \vec{T}(x) \rangle \, d\|T\|(x) = 0\,,
\end{split}
\]
where we have used that $v \in \R^k \times \{0\}$, $\vec{T}(x) \in \Lambda_m({\rm Tan}(M,x))$ at $\|T\|$-a.e. $x$, and $M$ is invariant with respect to $\R^k \times \{0\}$. Using that $\omega$ can be chosen arbitrarily, we conclude \eqref{e:invariance modp} from \eqref{homotopy modp}.

\smallskip

Next, let $\p \colon \R^{m+n} \to \R^{m+n}$ be the orthogonal projection operator onto $\R^k \times \{0\}$. Using standard properties of the slicing of integer rectifiable currents (see e.g. \cite[Theorem 4.3.2(7)]{Federer69}) and \eqref{e:invariance modp}, we can conclude then that 
\begin{equation} \label{e:congruent slices}
(\tau_v)_\sharp \langle T, \p, z+v \rangle = \langle (\tau_v)_\sharp T, \p, z \rangle = \langle T, \p, z \rangle \quad \modp\,,
\end{equation}
for every $z,v \in \R^k \times \{0\}$ such that the slices exist, or, equivalently, that 
\begin{equation} \label{e:congruent slices 2}
\langle T, \p, z \rangle = (\tau_{w-z})_\sharp \langle T, \p, w \rangle \quad \modp
\end{equation}
for every $z,w \in \R^k \times \{0\}$ such that the slices exist. Fix $z$ such that $\langle T, \p, z \rangle$ exists, and let $T' \in \Rc_{m-k}(\R^{m+n-k})$ be such that $\langle T, \p, z \rangle = (\tau_{-z})_\sharp T'$ after identifying $\R^{m+n-k}$ with $\{0\} \times \R^{m+n-k}$. Then, the current $ \tilde T := \llbracket \R^k \rrbracket \times T' $ satisfies
\begin{equation} \label{zero slices modp}
\langle T - \tilde T, \p, z \rangle = 0 \quad \modp \qquad \mbox{for $\Ha^k$-a.e. $z \in \R^k \times \{0\}$}\,.
\end{equation}
Observe that we may write
\begin{equation} \label{struttura rettificabili}
T = \theta\, \vec{\tau} \, \Ha^m \res M\,, \qquad \tilde T = \tilde \theta \, \vec{\tau}\, \Ha^m \res M\,,
\end{equation}
for a $0$-homogeneous function $\tilde \theta$ such that $\tilde \theta (x+v) = \tilde \theta (x)$ for every $v \in \R^k \times \{0\}$. Also notice that, since $M$ is invariant with respect to $\R^k \times \{0\}$ and $\p$ is the orthogonal projection onto $\R^k \times \{0\}$, if we identify $\R^k \times \{0\}$ with $\R^k$ and if we set $\phi := \left.\p\right|_M$, then $J_k\phi(x) > 0$ for $\Ha^m$-a.e. $x \in M$, where $J_k\phi(x)$ is the $k$-dimensional Jacobian of $\phi$, defined by
\[
J_k\phi(x) := \left( \det\big(  d\phi(x) \circ d\phi(x)^T  \big)  \right)^{\sfrac12}\,, \qquad d\phi(x) \colon T_xM \to \R^k
\]
at all points $x \in M$ such that $T_xM$ exists. 

By the considerations above, the standard slicing theory of rectifiable currents (see e.g. \cite[Theorem 4.3.8]{Federer69}) implies that for $\Ha^k$-a.e. $z \in \R^k \times \{0\}$ the set $M_z:=M \cap \p^{-1}(z)$ is $(m-k)$-rectifiable and
\begin{equation} \label{structure of slices}
\langle T, \p, z \rangle = \left\llbracket M_z, \, \zeta,\, \left. \theta \right|_{M_z} \right\rrbracket\,, \qquad \langle \tilde T, \p, z \rangle =  \left\llbracket M_z, \, \zeta,\, \left. \tilde \theta \right|_{M_z} \right\rrbracket
\end{equation}
for a Borel measurable unit $(m-k)$-vector field $\zeta = \zeta_z$ which is uniquely determined by $\vec{\tau}$ and $d\phi$. If $z \in \R^k \times \{0\}$ is such that both \eqref{zero slices modp} and \eqref{structure of slices} hold, then 
\begin{equation} \label{conclusione}
\theta (x) = \tilde\theta(x) \quad \modp \qquad \mbox{at $\Ha^{m-k}$-a.e. $x \in M_z$}\,.
\end{equation}
By Fubini's theorem, the conclusion in \eqref{conclusione} holds at $\Ha^m$-a.e. $x \in M$, so that \eqref{e:splitting} follows from \eqref{struttura rettificabili} and the definition of $\tilde T$. \\

\smallskip

If $T$ is a representative $\modp$, then $\langle T, \p, z \rangle$ is a representative $\modp$ for $\Ha^{k}$-a.e. $z \in \R^k \times \{0\}$, and thus we can choose $z$ such that the corresponding $T'$ is a representative $\modp$. With this choice, $\tilde T$ is a representative $\modp$ as well, and since $\tilde \theta(x) = \theta(x) \; \modp$ for $\Ha^m$-a.e. $x \in M$ we deduce that
\begin{equation} \label{sign changing}
\tilde \theta(x) = \eps(x) \, \theta(x) \qquad \mbox{with $\eps(x) \in \{-1,1\}$, for $\Ha^m$-a.e. $x \in M$}\,,
\end{equation}
where $\eps(x) = 1$ or $\abs{\theta(x)} = \frac{p}{2}$. As a consequence, $\abs{\tilde \theta} = \abs{\theta}$ $\Ha^m \res M$-a.e., which in turn implies that $\V(\tilde T) = \V(T)$. The last conclusion of the lemma is elementary, and the details of the proof are omitted.
 \qedhere

%{\color{red}Since the $m$-vector field orienting $T-\tilde T$ is tangent to $M$ and $M$ is invariant with respect to $\R^k \times \{0\}$, we infer from \eqref{zero slices modp} that $T = \tilde T \, \modp$, thus completing the proof of \eqref{e:splitting}. The validity of the last two statements is obvious by construction.}
%
%
%Now, we suppose that $T$ is a cone, and that $\Theta(T,0) < \frac{p}{2}$. Since $T$ is a representative $\modp$, by well-known properties of cones it is evident that $\Theta(T,x) \in \left[0, \frac{p}{2} \right[$ everywhere. We claim that this implies [can we prove this fact in a clean way at this stage?] that in fact $\partial T = 0$ in the classical sense. We postpone the proof of this fact [also because we don't have it at the moment :-D ]. But then \eqref{homotopy modp} becomes a classical identity of currents, and \eqref{e:invariance} follows from the argument of the previous paragraph.
%
%\smallskip
%
%Finally, once we have \eqref{e:invariance}, the rest of the statement follows from \cite[Lemma 35.5]{Simon83}.

\end{proof}

\begin{lemma}\label{l:symmetric}
Assume $T_0 \in \Rc_{m}(\R^{m+n})$ is an $m$-dimensional locally area minimizing current $\modp$ without boundary $\modp$ which is a cone (in the sense of Corollary \ref{c:tangent_cones} (iii)). Suppose, furthermore, that $\V(T_0)$ is $(m-1)$-symmetric but not $m$-symmetric (namely not flat). Then, $\Theta (T_0,0)\geq \frac{p}{2}$. 
\end{lemma}

\begin{proof}

Let $T_0 = \llbracket M, \vec{\tau}, \theta \rrbracket$, so that $\V(T_0) = \V(M,\abs{\theta})$. Since $\V(T_0)$ is $(m-1)$-symmetric but not $m$-symmetric, by Lemma \ref{l:splitting} $T_0 = \llbracket \pi \rrbracket \times T_0' \; \modp$, where $\pi$ is the $(m-1)$-dimensional spine of $\V(T_0)$, and $T_0'$ is a one-dimensional cone which has no boundary $\modp$ and is locally area minimizing $\modp$. Since $\Theta(T_0',0) = \Theta(T_0,0)$, we can reduce the proof of the lemma to the case when $m=1$.

Thus we can assume that
$T_0 = \sum_i Q_i \a{\ell_i}$, where $\ell_1, \ldots, \ell_N$ are pairwise distinct oriented half lines in $\mathbb R^{1+n}$ with the origin as common endpoint and the $Q_i$'s are integers. Without loss of generality we can assume that $\partial \a{\ell_i} = - \a{0}$. Observe that
\[
\Theta (T_0, 0) = \frac{1}{2} \sum_i |Q_i|
\]
and that $\sum_i Q_i = 0\;\; \modp$ since $T_0$ has no boundary $\modp$. If $\sum_i Q_i =0$, then $T_0$ would be an integral current without boundary, which in turn would have to be area minimizing. But since $T_0$ is by assumption not flat, this is not possible. Thus $\sum_i Q_i = k p$ for some nonzero integer $k$. This clearly implies
\[
\sum_i |Q_i|\geq |k|p\geq p\, ,
\]
which in turn yields $\Theta (T_0, 0) \geq \frac{p}{2}$. 
\end{proof}

We are now ready to state the starting point of our proof of Theorem \ref{t:main2} and Theorem \ref{t:main5}, which will be achieved by contradiction.

\begin{proposition}[Contradiction sequence]\label{p:contradiction_sequence}
Assume Theorem \ref{t:main5} is false. Then there are integers $m, n \geq 1$ and $2\leq Q < \frac{p}{2}$ and reals $\alpha, \eta >0$ with the following property. There are 
\begin{itemize}
\item[(i)] $T, \Sigma$ and $\Omega$ as in Assumption \ref{ass:main} such that $0\in \sing_Q (T)$;
\item[(ii)] a sequence of radii $r_k\downarrow 0$ and an $m$-dimensional plane $\pi_0$ such that $\V (T_{0, r_k})$ converges to $V = Q \mathcal{H}^m \res \pi_0 \otimes \delta_{\pi_0}$;
\item[(iii)] $\lim_{k\to\infty} \mathcal{H}^{m-2+\alpha}_\infty ({\rm D}_Q (T_{0, r_k})\cap \bB_1) \geq \eta$. 
\end{itemize}
If Theorem \ref{t:main2} is false then either there is a sequence as above or, for $Q=\frac{p}{2}$, there is a sequence as above where (iii) is replaced by
\begin{itemize}
\item[(iii)s]  $\lim_{k\to\infty} \mathcal{H}^{m-1+\alpha}_\infty ({\rm D}_Q (T_{0, r_k})\cap \bB_1) \geq \eta$.
\end{itemize}
\end{proposition} 

\begin{proof} Suppose first that Theorem \ref{t:main2} is false. Fix $p\in \mathbb N \setminus \{0,1\}$, and let $m\geq 1$ be the smallest integer for which the assertion of Theorem \ref{t:main2} is false. Observe that $m>1$. Fix thus a $T, \Sigma$ and $\Omega$ satisfying Assumption \ref{ass:main} for which there is an $\alpha >0$ with 
$\mathcal{H}^{m-1+\alpha} (\sing (T)) >0$. Then, by Proposition \ref{cor:densities}, there must be a $Q\in \mathbb N \cap [1, \frac{p}{2}]$ such that
$\mathcal{H}^{m-1+\alpha} (\sing_Q (T)) >0$. 

By \cite[Theorem 3.6]{Simon83}, $\Ha^{m-1+\alpha}$-a.e. point in $\sing_{Q}(T)$ has positive $\Ha^{m-1+\alpha}_{\infty}$-upper density: fix a point $q$ with this property, and assume, without loss of generality, that $q = 0$ and that $(\partial T) \res {\bB}_{1} = 0\; \modp$. Then, there exists a sequence of radii $r_{k}$ such that $r_{k} \downarrow 0$ as $k \to \infty$ and such that
\begin{equation} \label{density_est}
\lim_{k\to \infty} \Ha^{m-1+\alpha}_{\infty}(\sing_{Q}(T_{0,r_k}) \cap {\bB}_{1})= \lim_{k \to \infty} \frac{\Ha^{m-1+\alpha}_{\infty}(\sing_{Q}(T) \cap {\bB}_{r_k})}{r_{k}^{m-1+\alpha}} > 0
\end{equation}
Moreover, we can assume that the sequence of stationary varifolds $\V (T_{0,r_k})$ converges to a stationary cone $C\subset T_0 \Sigma$. Consider
the compact sets $\{\Theta_{T_{0,r_k}} \geq Q\}\cap \bB_1$ and assume, without loss of generality, that they converge in the Hausdorff sense to a compact set $K$. As it is well known, by the monotonicity formula for stationary varifolds we must have $\Theta_C (q) \geq Q$ for every $q\in K$. On the other hand, this implies that every point $q\in K$ belongs to the spine of the cone $C$; see \cite{BW97}. In turn, by the upper semicontinuity of the $\mathcal{H}^{m-1+\alpha}_\infty$ measure with respect to Hausdorff convergence of compact sets, we have
\begin{equation} \label{limit_density}
\Ha^{m-1+\alpha}_{\infty}(K) \geq \limsup_{k \to \infty} \Ha^{m-1+\alpha}_{\infty}({\rm D}_{Q}(T_{0,r_k}) \cap \overline{\bB}_{1}) > 0\, .
\end{equation}
Recall that the spine of the cone $C$ is however a linear subspace of $\mathbb R^{m+n}$, cf. again \cite{BW97}. This implies in turn that
$C$ must be supported in a plane, which completes the proof under the assumption that Theorem \ref{t:main2} is false.

\smallskip

Now, let us suppose Theorem \ref{t:main5} is false. Then, we can find $p,m,n$ and $Q < \frac{p}{2}$, together with $\Omega, \Sigma,T$ as in Assumption \ref{ass:main}, and $\alpha > 0$ such that $\Ha^{m-2+\alpha}(\sing_Q(T)) > 0$. Arguing as above, we can then find a point $q \in \sing_Q(T)$ with positive $\Ha^{m-2+\alpha}_{\infty}$-upper density, and we can suppose, without loss of generality, that $q = 0$. Then, there is a sequence of radii $r_k$ with $r_k \downarrow 0$ as $k \to \infty$ such that:
\begin{itemize}
\item[•] the blow-up sequence $T_{0,r_k}$ converges, in the sense of Corollary \ref{c:tangent_cones} (iv), to a current $T_0 \in \Rc_{m}(T_0\Sigma)$ satisfying properties (i), (ii), and (iii) of Corollary \ref{c:tangent_cones};

\item[•] $\lim_{k \to \infty} \Ha^{m-2+\alpha}_{\infty} (\sing_Q(T_{0,r_k}) \cap \bB_1  ) > 0$;

\item[•] the sequence of varifolds $\V(T_{0,r_k})$ converges to a stationary cone $C$ in $T_0\Sigma$;

\item[•] $C = \V(T_0)$.

\item[•] the spine of $C$ is a linear subspace of $T_0\Sigma$ having dimension at least $m-1$.
\end{itemize}

Now, if the spine of $C$ is $(m-1)$-dimensional, then $C$ is $(m-1)$-symmetric but not flat, hence forcing $\Theta(T_0,0) \geq \frac{p}{2}$ by Lemma \ref{l:symmetric}, which is a contradiction to the fact that $0 \in {\rm D}_Q(T)$ with $Q < \frac{p}{2}$. Thus, $C$ is supported in an $m$-dimensional plane, and the proof is complete.
\end{proof}

\newpage

\part{Approximation with multiple valued graphs} \label{part:Lp}

Following the blueprint of Almgren's partial regularity theory for area minimizing currents, we now wish to show that any area minimizing current modulo $p$ can be efficiently approximated, in a region where it is ``sufficiently flat'', with the graph of a multiple valued function which minimizes a suitably defined Dirichlet energy. Suppose that, in the region of interest, the current is a $Q$-fold cover of a given $m$-plane $\pi$, where $Q \in \left[1, \frac{p}{2} \right]$. The ``classical'' theory of $\Dir$-minimizing $Q$-valued functions as in \cite{DLS_Qvfr} is powerful enough to accomplish the task whenever $Q < \frac{p}{2}$ (which is always the case when $p$ is odd). If $p$ is even and $Q = \frac{p}{2}$, on the other hand, Almgren's $Q$-valued functions are not anymore the appropriate maps, and we will need to work with the class of \emph{special} multiple valued function defined in \cite{DLHMS_linear}.

\section{First Lipschitz approximation}

From now on we denote by $B_r (x,\pi)$ the disk $\bB_r (x) \cap (x+\pi)$, where $\pi$ is some linear $m$-dimensional plane. 
The symbol $\bC_r (x, \pi)$, instead, will always denote the cylinder $B_r (x, \pi) \times \pi^\perp$. If we omit the plane $\pi$ we then assume that $\pi = \pi_0 := \mathbb R^m \times \{0\}$, and the point $x$ will be omitted when it is the origin. 
Let $e_i$ be the unit vectors in the standard basis. We will regard 
$\pi_0$ as an oriented plane and we will denote by $\vec \pi_0$
the $m$-vector $e_1\wedge \ldots \wedge e_m$ orienting it.
We denote by $\p_\pi$ and $\p^\perp_\pi$ the
orthogonal projection operators onto, respectively, $\pi$ and its orthogonal complement $\pi^\perp$. If we omit the subscript we then assume again that $\pi = \pi_0$. 

We will make the following 

\begin{ipotesi}\label{ipotesi_base_app1}
$\Sigma\subset\R^{m+n}$ is a $C^2$
submanifold of dimension $m + \bar n = m + n - l$, which is the graph of an entire
function $\Psi: \R^{m+\bar n}\to \R^l$ satisfying the bounds
\begin{equation}\label{e:Sigma}
\|D \Psi\|_0 \leq c_0 \quad \mbox{and}\quad \bA := \|A_\Sigma\|_0
%\| D^2\bPsi \|_{L^\infty}
\leq c_0\,,
\end{equation}
where $c_0$ is a positive (small) dimensional constant.
$T$ is a representative $\modp$ of dimension $m$ with $\spt (T) \subset \Sigma$ and which, for some open cylinder $\bC_{4r} (x)$ 
(with $r\leq 1$)
and some positive integer $Q \leq \frac{p}{2}$, satisfies
\begin{equation}\label{e:(H)}
\p_\sharp T = Q\a{B_{4r} (x)} \modp \quad\mbox{and}\quad
(\partial T) \res \bC_{4r} (x) = 0\, \modp \, .
\end{equation}
\end{ipotesi}

We next define the following relevant quantities. 

\begin{definition}[Excess measure]\label{d:excess}
For a current $T$ as in Assumption \ref{ipotesi_base_app1} we define the \textit{cylindrical excess} $\bE(T,\bC_{4r} (x))$,
the \textit{excess measure} $\e_{T}$ and its \textit{density} $\bd_{T}$:
\[
\begin{split}
\bE(T, \bC_{4r}(x)) &:= \frac{1}{\omega_{m}(4r)^{m}} \left( \| T \|(\bC_{4r}(x)) - Q |B_{4r}(x)| \right), \\
\e_{T}(A) &:= \| T \|(A \times \R^{n}) -  Q |A| \qquad \text{for every Borel }\;A\subset B_{4r} (x) \\
\bd_{T}(y) &:= \limsup_{s \to 0} \frac{\e_{T}(B_{s}(y))}{\omega_{m} s^{m}} = \limsup_{s \to 0} \bE(T, \bC_{s}(y))\,.
\end{split}
\]
The subscript $_{T}$ will be omitted whenever it is clear from the context.

We define the \emph{height function} of $T$ in the cylinder $\bC_{4r}(x)$ by	
\begin{equation*} 
\bh (T, \bC_{4r} (x), \pi_0) := \sup \{|\p^\perp (q) - \p^\perp (q')|: q,q'\in \spt (T) \cap \bC_{4r} (x)\}.
\end{equation*}
\end{definition}

\begin{remark} \label{rmk:excess}
Note that, since $T$ is a representative $\modp$, we have $\|T\| = \|T\|_p$, where $\|T\|_p$ denotes the Radon measure on $\R^{m+n}$ defined by the mass $\modp$. However, it is false in general that $\|\p_\sharp T\| (A) = Q|A|$, since $\p_\sharp T$ is not necessarily a representative $\modp$. The excess written above can thus be rewritten as $\omega_m^{-1} \, (4r)^{-m}\left(\|T\|_p (\bC_{4r}(x)) - \|\p_\sharp T\|_p (\bC_{4r} (x))\right)$, but not as 
$\omega_m^{-1} \, (4r)^{-m}\left(\|T\| (\bC_{4r} (x)) - \|\p_\sharp T\| (\bC_{4r} (x))\right)$, which is the standard cylindrical excess in the classical regularity theory for area minimizing currents. Of course, since $\| \p_\sharp T\|_p \leq \| \p_\sharp T \|$ as measures, this ``excess $\modp$'' is, in general, \emph{larger} than the classical excess. 

Observe also that, under the assumptions valid in the regularity theory for classical area minimizing currents, where the identities in \eqref{e:(H)} hold in the sense of currents and not only $\modp$, the cylindrical excess can be classically written as
\begin{equation} \label{cylindrical=tilt}
\frac{1}{\omega_m (4r)^m}\left(\|T\| (\bC_{4r} (x)) - \|\p_\sharp T\| (\bC_{4r} (x))\right) = \frac{1}{2 \omega_m (4r)^m} \int_{\bC_{4r}(x)} \abs{\vec{T}-\vec{\pi_0}}^2 \, d\|T\|\,,
\end{equation}
see e.g. \cite[Lemma 9.1.5]{KP}. The quantity appearing on the right-hand side of \eqref{cylindrical=tilt} is the most flexible and natural in view of the forthcoming elliptic estimates. Unfortunately, in our setting not only the identity in \eqref{cylindrical=tilt} is false, but we do not have an integral representation of the excess $\modp$ either. For these reasons, later on we shall introduce a different notion of excess, called the \emph{nonoriented excess} (see \eqref{e:no_excess}), which shares the structural features of the quantity on the right-hand side of \eqref{cylindrical=tilt}. The nonoriented excess and the excess $\modp$ are then shown to be comparable in appropriate smallness regimes in Theorem \ref{thm:strong-alm-unoriented}. Nonetheless, in the context of the Lipschitz approximation we will work with the excess $\modp$, because it is more suitable to the arguments involving slicing which are needed in the BV estimate of Lemma \ref{conto_unidimensionale}.
\end{remark}

\begin{definition}
In general, given a measure $\mu$ on a domain $\Omega \subset \R^m$ we define its noncentered maximal function as
\[
\bmax \mu (y) := \sup \left\{ \frac{\mu (B_s (z))}{\omega_m s^m} : y\in B_s (z) \subset \Omega\right\}\, .
\]
If $f$ is a locally Lebesgue integrable non-negative function, we denote by $\bmax f$ the maximal function of the measure $f \mathscr{L}^m$. 
\end{definition}

The first Lipschitz approximation is given by the following proposition, according to which a representative $\modp$ $T$ as in Assumption \ref{ipotesi_base_app1} can be realized as the graph of a Lipschitz continuous multiple valued function in regions where the maximal function of its excess measure is suitably small. As already motivated, the approximating function needs to be a \emph{special} multi-valued function whenever $p$ is even and $Q = \frac{p}{2}$. Concerning special multi-valued functions, we will adopt the notation introduced in \cite{DLHMS_linear}: in particular, the space of special $Q$-points in $\R^n$ is denoted $\Iqspec$, $\cG_s$ is the metric on it, and $\abs{S} := \cG_s(S, Q \a{0})$ if $S \in \Iqspec$. Given a function $u \colon \Omega \to \Iqspec$ (possibly classical, namely with target $\Iqs$), we will let $\gr (u)$ and $\bG_u$ denote the set-theoretic graph of $u$ and the integer rectifiable current associated with it, respectively; see Section \ref{sec:guide} and \cite[Definition 4.1]{DLHMS_linear}. Also, we will let ${\rm osc}(u)$ denote the quantity
\begin{equation} \label{d:oscillation}
{\rm osc}(u) := \inf_{q \in \R^n} \||u\ominus q|\|_{L^\infty(\Omega)} = \inf_{q\in \R^n} \| \cG_s(u(x), Q \llbracket q \rrbracket)\|_{L^\infty(\Omega)}\,.
\end{equation}

\begin{remark}\label{rmk:oscillo_ma_non_mollo}

The definition given in \eqref{d:oscillation} for the quantity ${\rm osc}(u)$ is the special multi-valued counterpart of the definition provided in \cite{DLS_Lp} for the $\mathcal{A}_Q(\R^n )$-valued case. In \cite{DLS_Center}, on the other hand, the following comparable definition for the oscillation is used: 
\[ {\rm osc}_C(u):=\sup\{\abs{v-w}\, \colon\, x,y \in \Omega,\, v \in \spt(u(x)), w \in \spt(u(y)) \}\,.
\]
More precisely one has \[ \frac12\, {\rm osc}_C(u) \le {\rm osc}(u) \le \sqrt{Q} \,{\rm osc}_C(u)\,. \] To see the first inequality, let $x,y \in \Omega$ and $v \in \spt(u(x)), w \in \spt(u(y))$; then, for any $q \in \R^n$ we have \[\abs{v-w} \le \abs{v-q} + \abs{w-q} \le \abs{u(x)\ominus q} + \abs{u(y)\ominus q}  \le 2 \|\abs{ u \ominus q }\|_{L^\infty(\Omega)}. \] Taking the infimum over all $q \in \R^n$ gives the claimed inequality. For the second inequality, fix any arbitrary $y\in \Omega$ and $q \in \spt(u(y))$. Then, for any $x \in \Omega$ we have 
\[ \abs{u(x)\ominus q} \le \sqrt{Q}\, {\rm osc}_C(u).\]
Taking the supremum over all $x \in \Omega$ and afterwards the infimum in $q \in \spt(u(y))$ gives the desired bound. 

\end{remark}

\begin{proposition}[Lipschitz approximation]\label{p:max}
There exists a constant $C = C (m,n,Q)>0$ with the following properties.
Let $T$ and $\Psi$ be as in Assumption \ref{ipotesi_base_app1} in the cylinder $\bC_{4s} (x)$.
Set $E:=\bE (T,\bC_{4s}(x))$, let $0<\delta<1$ be such that $16^m E < \delta$,
and define 
\[
K := \big\{\bmax\e_T\leq \delta\big\}\cap B_{3s}(x)\, .
\]
Then, there is a Lipschitz map $u$ defined on $B_{3s} (x)$ and taking either values in $\Iqs$, if $Q<\frac{p}{2}$, or
in $\Iqspec$, if $Q=\frac{p}{2}$, for which the following facts hold.
\begin{itemize}
\item[(i)] $\gr (u) \subset \Sigma$;
\item[(ii)] $\Lip (u)\leq C\,\big(\delta^{\sfrac{1}{2}} + \|D\Psi\|_0\big)$ and ${\rm osc}\, (u) \leq C \bh (T, \bC_{4s} (x), \pi_0) + C s \|D \Psi\|_0$.
\item[(iii)] $\bG_u \res (K\times \R^{n})= T\mres (K\times \R^{n})$ $\modp$;
\item[(iv)] For $r_0 := 16 \sqrt[m] {E/\delta} < 1$ we have
\begin{equation}\label{e:max1}
|B_{r}(x)\setminus K|\leq \frac{5^m}{\delta}\,\e_T \Big(\{\bmax\e_T \geq \delta \}\cap B_{r+r_0s}(x)\Big) \quad\forall\;
r\leq 3\,s .
\end{equation}
\end{itemize}
\end{proposition}

We remark that in Proposition \ref{p:max} we are not assuming that $T$ is area minimizing modulo $p$. The proof of the proposition will require a suitable $BV$ estimate for $0$-dimensional slices $\modp$, which is the content of the next section. This Jerrard-Soner type estimate is in fact a delicate point of the present paper, since the approach of \cite{DLS_Lp} (which relies on testing the current with a suitable class of differential $m$-forms) is unavailable in our setting, since Assumption \ref{ipotesi_base_app1} only guarantees $\partial T \res \bC_{4s}(x) = 0\, \modp$ and not $\partial T \res \bC_{4s} (x) = 0$.

\section{A BV estimate for slices modulo $p$}

Recall that $\F_k(C)$ denotes the group of $k$-dimensional integral flat chains supported in a closed set $C$.

\begin{definition}\label{d:X_spaces}
We define the groups
\begin{align*}
X  &:= \left\lbrace Z \in \F_{0}(\R^n) \, \colon \, Z = \partial S \mbox{ for some } S \in \Rc_{1}(\R^n)\right\rbrace\,, \\
\tilde{X}^p &:= \left\lbrace Z \in \F_{0}(\R^n) \, \colon \, Z = \partial S + pP \mbox{ for some } S \in \Rc_{1}(\R^n), P \in \F_{0}(\R^n) \right\rbrace\,.
\end{align*}

On $X$ we define the distance function
\[
\bd_{\Fl}(T_{1}, T_{2}) = \Fl(T_{1} - T_{2}) := \inf\big\lbrace \M(S) \, \colon \, S \in \Rc_{1}(\R^{n}) \mbox{ such that } T_{1} - T_{2} = \partial S \big\rbrace,
\]
whereas on $\tilde{X}^{p}$ we set
\[
\begin{split}
\bd_{\Fl^{p}}(T_{1}, T_{2}) = \Fl^{p}(T_{1} - T_{2}) := \inf\big\lbrace \M(S) \, \colon \, S \in \Rc_{1}(\R^{n}) \mbox{ such that } &T_{1} - T_{2} = \partial S + pP \\ & \mbox{ for some } P \in \F_{0}(\R^{n}) \big\rbrace.
\end{split}
\]
\end{definition}

\begin{remark}
Note that the following properties are satisfied:
\begin{itemize}
\item[$(i)$] both $X$ and $\tilde{X}^p$ are subgroups of $\F_{0}(\R^n)$, with $X \subset \tilde{X}^p$;
\item[$(ii)$] $\tilde{X}^{p} = \left\lbrace T \in \F_{0}(\R^n) \, \colon \, T = S\, \modp \mbox{ for some } S \in X \right\rbrace$, the non-trivial inclusion being a consequence of \cite[Corollary 4.7]{MS_a}. Hence, the quotient groups $X/\modp$ and $\tilde{X}^p / \modp$ coincide and they are characterized by $X/\modp = \tilde{X}^p/\modp = X^p$, where
\[
X^p := \left\lbrace \left[ T \right] \in \F_{0}^{p}(\R^n) \, \colon \, T = \partial S\, \modp \mbox{ for some } S \in \Rc_{1}(\R^n) \right\rbrace\,;
\]
\item[$(iii)$] for $T \in X$ (resp. $T \in \tilde{X}^p$), one has $\Fl(T) \geq \F(T)$ ( resp. $\Fl^p(T) \geq \F^p(T)$);
\item[$(iv)$] $\left( X, \bd_{\Fl} \right)$ is a complete metric space; the pseudo-metric $\bd_{\Fl^p}$ induces a complete metric space structure on the quotient $X^p$, which we still denote $\bd_{\Fl^p}$.
\end{itemize}
\end{remark}

In the rest of the section we will use the theory of $BV$ maps defined over Euclidean domains and taking values in metric spaces, as established in Ambrosio's foundational paper \cite{Amb}.  

\begin{lemma} \label{conto_unidimensionale}
Assume $T$ is a one-dimensional integer rectifiable current satisfying Assumption \ref{ipotesi_base_app1} in $\bC_{4}$ (that is, set $m = 1$, $x = 0$ and $r = 1$ in Assumption \ref{ipotesi_base_app1}), and let $T_{t}$ be the slice $\langle T, \p, t \rangle \in \Rc_{0}(\R^{1+n})$ for a.e. $t \in B_{4} = ] -4, 4 [$. Then, the map $\Phi \colon t \in J := ] -4,4 [ \mapsto \left[ \p^{\perp}_{\sharp} T_{t} \right]$ is in $BV (J, X^{p})$, and moreover
\begin{equation} \label{eq:conto_unidimensionale}
|D\Phi|(I)^{2} \leq 2 \e_{T}(I) \| T \|(I \times \R^{n}) \quad \mbox{ for every Borel set } I \subset J. 
\end{equation}
\end{lemma}

\begin{proof}
Let us first observe that since $(\partial T) \res \bC_4 = 0\, \modp$ then by Lemma \ref{modp_slicing_formula} for a.e. $t \in J$ we have
\begin{equation}\label{eq:slicing_modp}
T_{t} = \partial \left( T \res \{ \p < t \} \right) \modp\,,
\end{equation}
and thus $\Phi(t) = \left[ \partial \p^{\perp}_{\sharp} \left( T \res \{ \p < t \} \right) \right] \in X^p$.
Fix now $t_{0} \in J$ such that \eqref{eq:slicing_modp} holds. Again by Lemma \ref{modp_slicing_formula}, for a.e. $t \in ] t_{0}, 4 [$ we have
$\Phi(t) - \Phi(t_{0}) = \left[ \partial \p^{\perp}_{\sharp}\left( T \res (\left( t_0, t \right) \times \R^n) \right) \right]$. So
\begin{equation} \label{est_flatp_dist}
\Fl^{p}(\Phi(t) - \Phi(t_0)) \leq \M(\p^{\perp}_{\sharp}\left(T \res (\left( t_0, t \right) \times \R^n)\right)).
\end{equation}
Arguing analogously for the $t \in \left( -4, t_0 \right)$ and integrating allows to conclude
\begin{equation}
\int_{-4}^{4} \bd_{\Fl^p}(\Phi(t) , \Phi(t_0)) \, dt \leq C \M(T \res \bC_4)\,,
\end{equation}
which shows that $\Phi \in L^1(J, X^p)$. \\

\smallskip

Next, we pass to the proof of \eqref{eq:conto_unidimensionale}. Without loss of generality, assume $I = \left( a, b \right)$ to be an interval with $a$ and $b$ Lebesgue points for $\Phi$. It is a consequence of \cite[Theorem 4.5.9]{Federer69} (see also \cite[Section 8.1]{DPH}) that $|D\Phi|(I)$ equals the classical \emph{essential variation} $\essvar(\Phi)$ given by
\begin{equation} \label{def:essvar}
\begin{split}
\essvar(\Phi) := \sup\bigg\lbrace \sum_{i=1}^{N} \bd_{\Fl^{p}}(\Phi(t_{i}), \Phi(t_{i-1})) \, \colon \, &a \leq t_{0} < t_{1} < \dots t_{N} \leq b \\ \mbox{ with } & t_{0},\dots,t_{N} \mbox{ Lebesgue points for } \Phi \bigg\rbrace. 
\end{split}
\end{equation}

Let $t_{0}, \dots, t_{N}$ be as in \eqref{def:essvar}, and let $e$ denote the constant unit $1$-vector orienting $\R \times \{0\} \subset \R^{1+n}$. Then, one has
\[
\begin{split}
\sum_{i=1}^{N} d_{\Fl^{p}}(\Phi(t_{i}), \Phi(t_{i-1})) &= \sum_{i=1}^{N} \Fl^{p}(\p^{\perp}_{\sharp}T_{t_i} - \p^{\perp}_{\sharp}T_{t_{i-1}}) 
\leq \sum_{i=1}^{N} \M(\p^{\perp}_{\sharp} (T \res (\left( t_{i-1}, t_{i} \right) \times \R^n))) \\
&\leq \int_{I \times \R^{n}} |\vec{T} - \langle \vec{T}, e \rangle e| \, d\|T\| 
= \int_{I \times \R^n} \sqrt{1 - \langle \vec{T}, e \rangle^{2}} \, d\| T \| \\
&\leq \sqrt{2} \int_{I \times \R^{n}} \sqrt{1 - \langle \vec{T}, e \rangle} \, d\| T \| \\
&\leq \sqrt{2} \left( \| T \|(I \times \R^n) - \| \p_{\sharp}T \|(I \times \R^{n})  \right)^{\frac{1}{2}} (\| T \|(I \times \R^{n}))^{\frac{1}{2}} \\
&\leq \sqrt{2} (\e_{T}(I))^{\frac{1}{2}} (\|T\|(I \times \R^{n}))^{\frac{1}{2}},
\end{split}
\]
where the first inequality has been deduced analogously to \eqref{est_flatp_dist}, and the last one follows from $\|\p_\sharp T\|_p \leq \|\p_\sharp T\|$ as measures.
This shows \eqref{eq:conto_unidimensionale} and concludes the proof.
\end{proof}

\section{Comparison between distances}

Another delicate point in the proof of Proposition \ref{p:max} is that Lemma \ref{conto_unidimensionale} is not powerful enough to guarantee the Lipschitz continuity of the approximating map $u$. To that aim, we shall need to combine the Jerrard-Soner type estimate \eqref{eq:conto_unidimensionale} with the result of Proposition \ref{SORBILLO_GRATIS_X2} below. \\

\smallskip

Let $Q$ and $p$ be positive integers with $Q \leq \frac{p}{2}$, and fix any $A,B \in \A_{Q}(\R^n)$. Observe that $A,B \in \mathscr{F}_0(\R^n)$. Furthermore, the flat chain $A-B$ is an element of the subgroup $X$ of Definition \ref{d:X_spaces}, so that we can compute $\Fl(A-B)$. Next, let us consider the flat chain $A + B$. In the case when $Q = \frac{p}{2}$, we claim that $A + B \in \tilde X^p$, so that we can compute $\Fl^p(A + B)$. Indeed, fix any $z \in \R^n$, and let $h_{z} \colon \left( 0,1 \right) \times \R^n \to \R^n$ be the function defined by
\[
h_z(t,x) := z + t(x-z).
\]
Then, the cone over $A+B$ with vertex $z$, that is the $1$-dimensional integral current $R$ given by
\[
R := z \cone (A+B) := (h_z)_{\sharp}(\llbracket \left( 0,1 \right) \rrbracket \times (A+B))
\]
satisfies
\[
\partial R = A + B - 2Q\, \llbracket z \rrbracket = A + B - p \, \llbracket z \rrbracket\,,
\]
which proves our claim. Furthermore, the above argument also shows that 
\begin{equation} \label{motivation for the extension}
\Fl^p(A+B) \leq \mass (R) = \Fl(A - Q\, \a{z}) + \Fl(B - Q\, \a{z})\,.
\end{equation}
Having this in mind, we extend the norm $\Fl$ to $A + B$ by setting
\begin{equation} \label{reduced flat norm definition}
\Fl(A + B) := 
\inf_{z \in \R^n} \lbrace \Fl(A - Q\llbracket z \rrbracket) + \Fl(B - Q \llbracket z \rrbracket) \rbrace  \qquad \mbox{when $Q = \frac{p}{2}$}\,,
\end{equation}
so that \eqref{motivation for the extension} implies that
\begin{equation} \label{basic inequality special}
\Fl^p(A+B) \leq \Fl(A+B) \qquad \mbox{for every $A,B \in \A_Q(\R^n)$ when $Q = \frac{p}{2}$}\,.
\end{equation}

The following result holds.
\begin{proposition} \label{SORBILLO_GRATIS_X2}
Let $p$ and $Q$ be positive integers with $Q \leq \frac{p}{2}$. Let $A := \sum_{i=1}^{Q} \llbracket A_{i} \rrbracket$ and $B := \sum_{i=1}^{Q} \llbracket B_{i} \rrbracket$ in $\A_{Q}(\R^{n})$, and let $\sigma \in \{-1,1\}$. If 
\begin{itemize}
\item[(a)] either $\sigma = 1$,
\item[(b)] or $\sigma = -1$ and $Q = \frac{p}{2}$,
\end{itemize}
then
\begin{equation} \label{margherita}
\Fl^{p}(A - \sigma B) = \Fl(A - \sigma B)\,.
\end{equation}
\end{proposition}

The proof of Proposition \ref{SORBILLO_GRATIS_X2} hinges upon a simple combinatorial argument. However, in order not to divert attention away from the proof of Proposition \ref{p:max}, we postpone it to Appendix \ref{appendix}.

\section{Proof of Proposition \ref{p:max}}

Since the statement is scaling and translation invariant, there is no loss of generality in assuming $x = 0$ and $s = 1$. Consider the slices $T_{x} := \langle T, \p, x \rangle \in \Rc_{0}(\R^{m+n})$ for a.e. $x \in \R^{m} \times \{0\}$ and use \cite[Theorem 4.3.2(2)]{Federer69} and \cite[Corollary 2.23]{AFP} to conclude that
\begin{equation} \label{e:mass_slice_below}
\M(T_{x}) \leq \lim_{r \to 0} \frac{\|T\|(\bC_{r}(x))}{\omega_{m} r^{m}} \leq \bmax\e_{T}(x) + Q \quad \mbox{ for a.e. }x.
\end{equation}
Now, since $\bmax\e_{T}(x) \leq \delta < 1$ for every $x \in K$, we conclude that $\M(T_{x}) < Q + 1$ for a.e. $x \in K$. On the other hand, setting ${\bf M}(x) := \M(T_x)$ for $x \in B_4$ we have the simple inequality
\begin{equation} \label{e:trivial_inequalities_slice}
{\bf M}\, \mathscr{L}^m \res B_{4} \geq \|\p_\sharp T\| \geq
\|\p_\sharp T\|_p = Q \mathscr{L}^m  \res B_{4}\,, 
\end{equation}
so that we deduce
\begin{equation} \label{e:mass_slice_above}
\M(T_x) = {\bf M}(x) \geq Q \qquad \mbox{for a.e. $x \in B_4$}\,.
\end{equation}
From \eqref{e:mass_slice_below} and \eqref{e:mass_slice_above} we infer then that $\M(T_{x}) = Q$ for a.e. $x \in K$. Hence, there are $Q$ functions $g_{i} \colon K \to \R^{n}$ such that $\p^{\perp}_{\sharp}T_{x} = \sum_{i=1}^{Q} \sigma_{i}(x) \a{g_{i}(x)}$ for a.e. $x \in K$, with $\sigma_{i}(x) \in \{-1,1\}$. In fact, since $\| \p_\sharp T \| \geq Q\, \mathscr{L}^m \res B_4$, the values of $\sigma_{i}(x)$, for fixed $x$, are independent of $i$, and thus $\p^{\perp}_{\sharp}T_{x} = \sigma(x) \sum_{i=1}^{Q} \llbracket g_{i}(x) \rrbracket$. Furthermore, since $\p_{\sharp} T= Q \llbracket B_{4} \rrbracket\; \modp$, it has to be $\sigma(x) Q \equiv Q\; \modp$ as integers. We therefore have to distinguish between two cases:
\begin{itemize}
\item[(A)] \textit{$Q < \frac{p}{2}$}. In this case, the condition $\sigma(x) Q \equiv Q\; \modp$ is satisfied if and only if $\sigma(x) = 1$. Hence, the functions $g_{i}$ allow to define a measurable map $g \colon K \to \A_{Q}(\R^n)$ by setting 
\[
g(x) := \sum_{i=1}^{Q} \llbracket g_{i}(x) \rrbracket\, .
\]
\item[(B)] \textit{$Q = \frac{p}{2}$}. In this case, any measurable choice of $\sigma \colon K \to \{-1,1\}$ would satisfy the condition $\sigma (x) Q \equiv Q\; \modp$. On the other hand
\[
g (x) := \left( \sum_{i=1}^{Q} \llbracket g_{i}(x) \rrbracket , \sigma(x) \right)
\]
defines a measurable function $g \colon K \to \Iqspec$.
\end{itemize}

\subsection{Lipschitz estimate} Fix $j \in \{1,\dots,m\}$, and let $\hat{\p}_{j} \colon \R^{m+n} \to \R^{m-1}$ be the orthogonal projection onto the $(m-1)$-plane given by ${\rm span}(e_{1},\dots,e_{j-1},e_{j+1},\dots,e_{m})$. For almost every $z \in \R^{m-1}$, consider the one-dimensional slice $T^{j}_{z} := \langle T, \hat{\p}_{j}, z \rangle$, and observe that 
\[
\int_{\R^{m-1}} \M(T^{j}_{z}) \, dz \leq \M(T).
\] 
Observe that $T^{j}_z$ satisfies Assumption \ref{ipotesi_base_app1} with $m=1$ for a.e. $z$. 
Let now $\p_{j}$ be the orthogonal projection $\p_{j} \colon \R^{m+n} \to {\rm span}(e_{j})$, and for almost every $t \in \R$ let $\left( T^{j}_{z}\right)_{t} := \langle T^{j}_{z}, \p_{j}, t \rangle$. By Lemma \ref{conto_unidimensionale}, the map $\Phi^{j}_{z} \colon t \mapsto \p^{\perp}_{\sharp}\left( T^{j}_{z} \right)_{t}$ is $BV(\R, X^{p})$, and moreover
\begin{equation}
|D\Phi^{j}_{z}|(I)^{2} \leq 2 \e_{T^{j}_{z}}(I) \| T^{j}_{z} \|(I \times \R^{n}) \quad \mbox{ for every Borel set } I \subset B_{4} \cap {\rm span}(e_{j}).
\end{equation}
Now, observe that
\[
\Phi^{j}_{z}(t) = \p^{\perp}_{\sharp}\left( T^{j}_{z} \right)_{t} = \p^{\perp}_{\sharp} \left\langle \langle T, \hat{\p}_{j}, z \rangle, \p_{j}, t \right\rangle = (-1)^{m-j} \p^{\perp}_{\sharp} \langle T, \p, x(j,z,t) \rangle = (-1)^{m-j} \p^{\perp}_{\sharp} T_{x(j,z,t)},
\]
where $x(j,z,t) := (z_{1},\dots,z_{j-1},t,z_{j+1},\dots,z_{m}) \in \R^{m}$. By \cite[formula (79)]{DPH}, we can therefore conclude that the map $\Phi \colon x \in \R^{m} \mapsto \p^{\perp}_{\sharp}T_{x}$ is in $BV(\R^{m}, X^{p})$. Furthermore, if for every Borel set $A \subset B_{4}$, for every $j \in \{1,\dots,m\}$ and for every $z = (z_{1},\dots,z_{j-1},z_{j+1},\dots,z_{m}) \in \R^{m-1}$ we denote $A^{j}_{z} := \left\lbrace t \in \R \, \colon \, (z_{1},\dots,z_{j-1},t,z_{j+1},\dots,z_{m}) \in A \right\rbrace$, we have 
\begin{equation}
\begin{split}
|D\Phi|(A) &\leq \sum_{j=1}^{m} \int_{\R^{m-1}} |D\Phi^{j}_{z}|(A^{j}_{z}) \, dz \\
&\overset{\eqref{eq:conto_unidimensionale}}{\leq} \sqrt{2} \sum_{j=1}^{m} \int_{\R^{m-1}} \left( \e_{T^{j}_{z}}(A^{j}_{z}) \right)^{\frac{1}{2}} \left( \| T^{j}_{z} \|(A^{j}_{z} \times \R^{n}) \right)^{\frac{1}{2}} \, dz \\
&\leq \sqrt{2} \sum_{j=1}^{m} \left( \int_{\R^{m-1}} \e_{T^{j}_{z}}(A^{j}_{z}) \, dz \right)^{\frac{1}{2}} \left( \int_{\R^{m-1}} \| T^{j}_{z} \|(A^{j}_{z} \times \R^{n}) \, dz \right)^{\frac{1}{2}} \\
&\leq \sqrt{2} m \left( \e_{T}(A) \right)^{\frac{1}{2}} \left( \| T \|(A \times \R^{n}) \right)^{\frac{1}{2}}.
\end{split}
\end{equation} 

Thus, from the definition of excess measure modulo $p$ we deduce
\[
|D\Phi|(B_{r}(y))^{2} \leq 2m^{2} \e_{T}(B_{r}(y))\left( Q|B_{r}(y)| + \e_{T}(B_{r}(y)) \right),
\]
for any $B_{r}(y) \subset B_{4}$. Hence, if we define the maximal function
\[
\bmax|D\Phi|(x) := \sup_{x \in B_{r}(y) \subset B_{4}} \frac{|D\Phi|(B_{r}(y))}{|B_{r}(y)|},
\]
we can conclude that
\[
(\bmax|D\Phi|(x))^{2} \leq 2m^{2} \left( Q \bmax\e_{T} (x) + (\bmax\e_{T} (x))^{2} \right)\leq C\delta \quad \mbox{for every } x \in K.
\]

By \cite[Lemma 7.3]{AK00}, one immediately obtains
\[
\Fl^{p}(\Phi(x) - \Phi(y)) \leq C \delta^{\sfrac{1}{2}} |x - y| \quad \mbox{for every } x,y \in K \mbox{ Lebesgue point of } \Phi.
\]
On the other hand, for a.e. $x \in K$ we can regard $\Phi(x) = g(x) \in \A_{Q}(\R^n)$ if $Q < \frac{p}{2}$ or $\Phi(x) = \sigma(x) g_{0}(x)$ with $\sigma(x) \in \{ -1, 1 \}$ and $g_{0}(x) \in \A_{Q}(\R^n)$ if $Q = \frac{p}{2}$. In any case, Proposition \ref{SORBILLO_GRATIS_X2} implies that in fact
\[
\Fl(\Phi(x) - \Phi(y)) \leq C \delta^{\sfrac{1}{2}} |x - y| \quad \mbox{for every } x,y \in K \mbox{ Lebesgue point of } \Phi.
\]
Now, first consider the case $Q < \frac{p}{2}$. Writing $\Phi(\cdot) = g(\cdot)$, we observe that
\[
\Fl(g(x) - g(y)) = \min_{\sigma \in \mathcal{P}_{Q}} \sum_{i=1}^{Q} |g_{i}(x) - g_{\sigma(i)}(y)| \geq \min_{\sigma \in \mathcal{P}_{Q}} \left( \sum_{i=1}^{Q} |g_{i}(x) - g_{\sigma(i)}(y)|^{2} \right)^{\sfrac{1}{2}} = \G(g(x), g(y)),
\]
where $\mathcal{P}_{Q}$ denotes the group of permutations of $\{1,\dots,Q\}$. 

If $Q = \frac{p}{2}$, instead, we have $\Phi(\cdot) = \sigma(\cdot) g_{0}(\cdot)$. If $\sigma(x) = \sigma(y)$, then the same computation produces
\[
\Fl(\sigma(x) g_0(x) - \sigma(y) g_0(y)) \geq \G(g_0(x), g_0(y)) = \G_{s}(g(x),g(y)).
\] 
If, on the other hand, $\sigma(x) \neq \sigma(y)$, and to fix the ideas say that $\sigma(x) = 1$ and $\sigma(y) = -1$, then
\[
\begin{split}
\Fl(g_0(x) + g_0(y)) :&= \inf_{z \in \R^n} \left\lbrace \Fl(g_0(x) - Q \llbracket z \rrbracket) + \Fl(g_0(y) - Q \rrbracket z \rrbracket) \right\rbrace \\
&\geq \inf_{z \in \R^n} \left\lbrace \G(g_0(x), Q \llbracket z \rrbracket) + \G(g_0(y), Q \rrbracket z \rrbracket) \right\rbrace \\
&\geq \inf_{z \in \R^n} \left( \G(g_0(x), Q\llbracket z \rrbracket)^{2} + \G(g_0(y), Q\llbracket z \rrbracket)^{2} \right)^{\sfrac{1}{2}}\, .
\end{split}
\]
Now observe that 
\begin{align*}
 &\G(g_0(x), Q\llbracket z \rrbracket)^{2} + \G(g_0(y), Q\llbracket z \rrbracket)^{2}\\
 =\;&
|g_0 (x) \ominus \etab \circ g_0 (x)|^2 + |g_0 (y) \ominus \etab \circ g_0 (y)|^2 + Q |\etab \circ g_0 (x) - z|^2 
+ Q |\etab \circ g_0 (y) - z|^2\, . 
\end{align*}
Thus
\begin{align*}
& \inf_{z \in \R^n} \left( \G(g_0(x), Q\llbracket z \rrbracket)^{2} + \G(g_0(y), Q\llbracket z \rrbracket)^{2}\right)\\
=\; & |g_0 (x) \ominus \etab \circ g_0 (x)|^2 + |g_0 (y) \ominus \etab \circ g_0 (y)|^2 + \frac{Q}{2} |\etab \circ g_0 (x) -
\etab \circ g_0 (y)|^2\, .\\
\geq\; & \frac{1}{2} \G_s (g_0 (x), g_0 (y))^2\, .
\end{align*}
This shows that $g \in \Lip(K, \A_{Q}(\R^{n}))$ (resp. $g \in \Lip(K, \Iqspec$) with $\Lip(g) \leq C \delta^{\sfrac{1}{2}}$.

\subsection{Conclusion} Next, in case $Q< \frac{p}{2}$, write 
\[
g (x) = \sum_i \llbracket(h_i (x), \Psi (x, h_i (x)))\rrbracket.
\]
Obviously, $x\mapsto h (x) := \sum_i \llbracket h_i (x)\rrbracket \in \A_{Q}(\R^{\bar n})$ is a Lipschitz map on $K$ with Lipschitz constant $\leq C\, \delta^{\sfrac{1}{2}}$. Recalling \cite[Theorem 1.7]{DLS_Qvfr}, we can extend it to a map $\bar{h}\in \Lip (B_3, \A_{Q}(\R^{\bar n}))$ satisfying $\Lip(\bar h)\leq C\,\delta^{\sfrac{1}{2}}$ (for a possibly larger $C$) and
${\rm osc}\, (\bar{h})\leq C {\rm osc}\, (h)$.
Finally, set
\[
u(x) := \sum_i \llbracket (\bar{h}_i (x), \Psi(x, \bar{h}_i(x))) \rrbracket.
\] 
The same computations of \cite[Section 3.2]{DLS_Lp} then show the Lipschitz and the oscillation bound in Claim (ii) of the Proposition.

For $Q=\frac{p}{2}$ we argue analogously, using this time the Extension Corollary \cite[Corollary 5.3]{DLHMS_linear}
in place of \cite[Theorem 1.7]{DLS_Qvfr}. 

Note that the points (i) and (iii) of the proposition are obvious by construction. 
Next observe that, since $\bmax \e_T$ is lower semicontinuous, $K$ is obviously closed. Let $U := \{\bmax \e_T > \delta\}$ be its complement. Fix $r\leq 3$ and
for every
point $x\in U\cap B_r$ consider a ball $B^x$ of radius $r(x)$ which contains $x$ and satisfies $ \e_T (B^x) > \delta \omega_m r(x)^m$. Since
$ \e_T (B^x) \leq E$ we obviously have
\[
r (x) < \sqrt[m]{\frac{E}{\omega_m \delta}} < r_0 < 1\,.
\]
Now, by the definition of the maximal function it follows clearly that $B^x \subset U \cap B_{r+r_0}$. In turn, by the $5r$ covering theorem we
can select countably many pairwise disjoint $B^{x_i}$ such that the corresponding concentric balls $\hat{B}^i$ with radii $5 r (x_i)$ cover $U\cap B_r$.  
Then we get
\[ \abs{U \cap B_r} \le 5^m\sum_i \omega_m r(x_i)^m \leq \frac{5^m}{\delta} \sum_i  \e_T (B^{x_i}) \leq \frac{5^m}{\delta} \e_T (U \cap B_{r+r_0})\, . 
\]
This shows claim (iv) of the proposition and completes the proof. 

\section{First harmonic approximation}

\begin{remark}[Good system of coordinates]\label{rmk:good_coord}
Let $T$ be as in Assumption \ref{ipotesi_base_app1} in the cylinder $\bC_{4r}(x)$. If the excess $E = \bE(T, \bC_{4r}(x))$ is smaller than a geometric constant, then without loss of generality we can assume that the function $\Psi \colon \R^{m+\bar{n}} \to \R^l$ parametrizing the manifold $\Sigma$ satisfies $\Psi(0) = 0$, $\|D\Psi\|_0 \leq C (E^{\sfrac{1}{2}} + r \bA)$ and $\|D^2\Psi\|_0 \leq C\bA$. This can be shown using a small variation of the argument outlined in \cite[Remark 2.5]{DLS_Lp}. First of all, as anticipated in Remark \ref{rmk:excess}, we introduce a suitable notion of nonoriented excess. Given the plane $\pi_0$ we consider
the $m$-vector $\vec{\pi}_0$ of mass $1$ which gives the standard orientation to it. We then let
\begin{equation}\label{e:no_excess_0}
|\vec{T} (y) - \pi_0|_{no} := \min \{|\vec{T} (y) - \vec{\pi}_0|, |\vec{T} (y) + \vec{\pi}_0|\}\,,
\end{equation}
where $\abs{\,\cdot\,}$ is the norm associated to the standard inner product on the space $\Lambda_m(\R^{m+n})$ of $m$-vectors in $\R^{m+n}$, and define
\begin{equation}\label{e:no_excess}
\bE^{no} (T, \bC_{4r} (x)) = \frac{1}{2 \omega_m (4r)^m} \int_{\bC_{4r} (x)} |\vec{T} (y) - \pi_0|_{no}^2\, d\|T\| (y)\, .
\end{equation}
Consider next the orthogonal projection $\p: \mathbb R^{m+n}\to \pi_0$ and the corresponding slices $\langle T, \p,y \rangle$ with
$y\in B_{4r} (x)$. For a.e. $y$, such a slice is an integral $0$-dimensional current and we let ${\bf M}(y)\in \mathbb N$ be its mass. Once again (cf. \eqref{e:trivial_inequalities_slice}), we observe that under the Assumption \ref{ipotesi_base_app1} we have 
\[
{\bf M}\, \mathscr{L}^m \res B_{4r} (x) \geq \|\p_\sharp T\| \geq
\|\p_\sharp T\|_p = Q \mathscr{L}^m  \res B_{4r}(x)\, .
\] 
Thus, an elementary computation gives 
\[
\begin{split}
\bE^{no} (T, \bC_{4r} (x)) &=  \frac{1}{\omega_m (4r)^m} \left( \|T\|(\bC_{4r}(x)) - \int_{B_{4r} (x)} {\bf M} (y)\, dy \right)\\
& \leq \frac{1}{\omega_m (4r)^m} \left( \|T\|(\bC_{4r}(x)) - \|\p_{\sharp}T\|(\bC_{4r}(x)) \right) \\
&\leq \frac{1}{\omega_m (4r)^m} \left( \|T\|(\bC_{4r}(x)) - \|\p_{\sharp}T\|_p(\bC_{4r}(x)) \right) \\ 
&= \bE(T,\bC_{4r}(x)) = E\, .
\end{split}
\]
At this point we find clearly a point $q\in \spt (T)\cap \bC_{4r} (x)$ such that
\[
\min \{|\vec{T} (q) - \vec{\pi}_0|,  |\vec{T} (q) - (-\vec{\pi}_0)|\} \leq C E^{\sfrac{1}{2}}
\] 
and we can proceed with the very same argument of \cite[Remark 3.5]{DLS_Lp}. 
\end{remark}

\begin{definition}[$E^\beta$-Lipschitz approximation]
Let $\beta \in \left(0,\frac{1}{2m} \right)$, let $T$ be as in Proposition \ref{p:max} such that $32E^{\frac{1-2\beta}{m}} < 1$. If the coordinates are fixed as in Remark \ref{rmk:good_coord}, then the Lipschitz approximation of $T$ provided by Proposition \ref{p:max} corresponding to the choice $\delta = E^{2\beta}$ will be called the $E^\beta$-Lipschitz approximation of $T$ in $\bC_{3s}(x)$.
\end{definition}

In the following theorem, we show that the minimality assumption on the current $T$ and the smallness of the excess imply that the $E^\beta$-Lipschitz approximation of $T$ in $\bC_{3s}(x)$ is close to a Dirichlet minimizer $h$, and we quantify the distance between $u$ and $h$ in terms of the excess.

\begin{theorem}\label{t:harm_1}
For every $\eta_*>0$ and every $\beta\in (0, \frac{1}{2m})$ there exist constants $\varepsilon_*>0$ and $C>0$ with the following property.
Let $T$ and $\Psi$ be as in Assumption \ref{ipotesi_base_app1} in the cylinder $\bC_{4s}(x)$, and assume that $T$ is area minimizing $\modp$ in there. Let $u$ be the $E^\beta$-Lipschitz approximation of $T$ in $B_{3s} (x)$, and let $K$ be the set satisfying all the properties of Proposition \ref{p:max} for $\delta=E^{2\beta}$. If $E\le \epsilon_*$ and 
$s\bA\le \epsilon_* E^{\frac 12}$, then 
\begin{equation}\label{e:eta-star-condition}
\e_T(B_{5s/2}\setminus K)\le \eta_*E  s^m\,,
\end{equation}
and 
\begin{equation}\label{e:small-energy-condition}
\Dir(u, B_{2s}(x)\setminus K) \le C\eta_* E s^m\,. 
\end{equation}
Moreover, there exists a map $h$ defined on $B_{3s} (x)$ and taking either values in $\Iqs$, if $Q<\frac{p}{2}$, or
in $\Iqspec$, if $Q=\frac{p}{2}$, for which the following facts hold:
\begin{itemize}
\item[(i)] $h(x)=(\bar{h}(x), \Psi(x,\bar{h}(x))$ with $\bar{h}$ Dirichlet minimizing;
\item[(ii)]
\begin{align}
s^{-2}&\int_{B_{2s}(x)}\!\!\G_s(u,h)^2+\int_{B_{2s}(x)}\!\!\left(|Du|-|Dh|\right)^2\le\; \eta_*E s^m\,\label{e:harm-app1}\\
&\int_{B_{2s}(x)} \abs{D(\etab\circ u) - D(\etab \circ h)}^2 \le\;  \eta_* E s^m\,.\label{e:harm-app2}
\end{align}
\end{itemize}
\end{theorem}

\begin{remark}\label{rmk:Dir estimates in the good region}
There exists a dimensional constant $c$ such that, if $E \leq c$ and $s\bA \leq E^{\sfrac12}$, then the $E^\beta$-Lipschitz approximation $u$ of $T$ in $\bC_{3s}(x)$ satisfies:
\begin{align} 
& \Lip(u) \leq C\, E^{\beta} \,, \label{e:Ebeta Lipschitz}\\
& \Dir(u, B_{3s}(x)) \leq C \, E \, s^m\,. \label{e:Ebeta Dir}
\end{align} 

Equation \eqref{e:Ebeta Lipschitz} follows from property ${\rm (ii)}$ of the Lipschitz approximation in Proposition \ref{p:max}, the choice of $\delta = E^{2\beta}$, and the scaling of $\bA$. The estimate in \eqref{e:Ebeta Dir}, instead, is a consequence of the Taylor expansion of the mass of multiple valued graphs deduced in \cite[Corollary 13.2]{DLHMS_linear}. Indeed, the remainder term in equation \cite[Equation (13.5)]{DLHMS_linear} can be estimated by
\[  
\int_{B_{3s}(x)} \sum_{i} \bar{R}_4(Du_i) \leq C \, \int_{B_{3s}(x)} |Du|^4 \leq C \, E^{2\beta} \, \Dir(u, B_{3s}(x)) < \frac{1}{4} \, \Dir(u, B_{3s}(x))
\]
for suitably small $E$. Hence, \cite[Equation (13.5)]{DLHMS_linear} yields
\[
\begin{split}
\frac{1}{4} \Dir(u, B_{3s}(x)) &\leq \|\bG_u\|(\bC_{3s}(x)) - Q\, \omega_m (3s)^m  \\
&\leq (\| T\|(\bC_{3s}(x)) - Q\, \omega_m (3s)^m   ) + \|\bG_u\|((B_{3s}(x) \setminus K) \times \R^n) \\
&\leq \omega_m \, E \, (3s)^m + C \, E^{2\beta} \, |B_{3s}(x) \setminus K| \leq C\, E \, s^m
\end{split}
\]
by property ${\rm (iv)}$ in Proposition \ref{p:max}.
\end{remark}

\begin{proof}
Let us first observe that \eqref{e:eta-star-condition} implies \eqref{e:small-energy-condition}: indeed, the estimate \eqref{e:max1} implies:
\[ \Dir(u,B_{2s}(x) \setminus K) \le \Lip(u)^2 \abs{B_{2s}(x) \setminus K} \le C\, \e_T(B_{\frac 52 s}(x) \setminus K). \]

Then, note that we can embed $\Iqs$ naturally and isometrically into $\Iqspec$ using the map $T \in \Iqs \mapsto ( T, 1 )$. Hence, without loss of generality we may assume that $u$ takes values in $\Iqspec$. Furthermore, each Lipschitz approximation is of the form $u(x)= (\bar{u}(x), \Psi(x,\bar{u}))$ with $\bar{u}$ taking values in $\mathscr{A}_Q (\R^{\bar{n}})$.

Finally, since the statement is scale invariant we may assume $x=0$ and $s=1$. 

We will now show the following.\\

Given any sequence of currents $T_k$ supported in manifolds $\Sigma_k=\gr ({\Psi_k})$ and corresponding Lipschitz approximations $u_k$ satisfying all the assumptions in $B_3$ with
\[ E_k \to 0 \quad \text{ and } \quad \bA_k=o(E_k^{\frac12}) \qquad \text{ as } k \to \infty ,\]
then the following conclusions hold:
\begin{itemize}
\item[(i)] \begin{align*}
 \e_{T_k}(B_{\frac52}\setminus K_k)&= o(E_k)
 %\Dir(u_k, B_2\setminus K_k) &= o(E_k) 
 \end{align*}
 \item[(ii)]
 One of the following holds true:
 either there is a single Dirichlet minimzing map $\bar{h} \in W^{1,2}(B_{\frac 5 2}, \mathscr{A}_Q(\R^{\bar{n}}))$ such that
 \begin{align*} 
 \int_{B_s} \G_s( E_{k}^{-\frac12} \bar{u}_k, \bar{h} )^2 + \left( E_{k}^{-\frac{1}{2}} \abs{D\bar{u}_k} - \abs{D\bar{h}} \right)^2 = o(1) \qquad \text{ for all } s < \frac52; 	
 \end{align*}
 \\
 or there are Dirichlet minimizing maps $h_j \in W^{1,2} (B_{\frac 52}, \mathcal{A}_{Q_j}(\R^{\bar{n}})$ with $j=1, \dotsc, J$, $\sum_j Q_j = Q$, and sequences $\{y_{j,k}\}_{k \in \N} \in \R^{\bar{n}}$ such that if we consider the sequence of maps in $W^{1,2}(B_{\frac 52}, \mathscr{A}_Q(\R^{\bar{n}}))$ given by
 \[ \bar{h}_k:= \left(\sum_j \a{ y_{j,k} \oplus h_j }, \sigma \right) \]
 with $\sigma \in \{-1,1\}$ fixed we have
 \begin{align*} 
 \int_{B_s} \G_s( E_{k}^{-\frac12} \bar{u}_k , \bar{h}_k)^2 + \left( E_k^{-\frac12}\abs{D\bar{u}_k} - \abs{D\bar{h}_k} \right)^2 = o(1) \qquad \text{ for all } s < \frac52.	
 \end{align*}
\end{itemize}
For sufficiently large $k$ the conclusion of the Theorem therefore holds, since we can replace in point (ii)
 $\bar{u}_k$ by $u_k$ and $\bar{h}_k$ by $h_k=(\bar{h}_k,E_k^{-\frac12}\Psi_k(\cdot,E_k^\frac12\bar{h}_k))$.  This can be seen as follows. Recall that by remark \ref{rmk:good_coord}, we have $\norm{D\Psi_k}_0+ \norm{D^2\Psi_k}_0= O(E_k^{\frac 12})$. As a first step, we may replace in (ii) $(E_k^{-\frac12}\abs{D\bar{u}_k}-\abs{D\bar{h}})^2$ by $\abs{E_k^{-1}\abs{D\bar{u}_k}^2-\abs{D\bar{h}}^2}$. Indeed, for any sequence of non-negative measurable functions $a_k,b_k$ we have
\[ \int \abs{a_k-b_k}^2 \le \int \abs{ a_k^2 - b_k^2} = \int \abs{a_k+b_k}\,\abs{a_k-b_k} \le 2\, \left( \int \abs{b_k}^2 \right)^\frac12 \left( \int \abs{a_k-b_k}^2 \right)^\frac12 +  \int \abs{a_k-b_k}^2\,;\] hence $\norm{a_k-b_k}_2=o(1)$ if and only if $\norm{(a_k)^2-(b_k)^2}_1=o(1)$.  Thus it remains to show that  $E_k^{-1}\int_{B_s}\Abs{ \abs{D\Psi_k(\cdot, \bar{u}_k)}^2 - \abs{D\Psi_k(\cdot, E_k^{\frac12} \bar{h}_k)}^2}$ is $o(1)$. We compute explicitly: 
\begin{align*}
&\sum_{i=1}^Q E_k^{-1} \int_{B_s}\Abs{ \abs{D\Psi_k(\cdot, \bar{u}_{k}^i)}^2 - \abs{D\Psi_k(\cdot, E_k^{\frac12} \bar{h}_{k}^i)}^2}\\&=\sum_{i=1}^Q E_k^{-1} \int_{B_s}\Abs{ \abs{D_x\Psi_k(\cdot, \bar{u}_{k}^i)+ D_y\Psi(x,\bar{u}_{k}^i) \, D\bar{u}_k^i}^2 - \abs{D_x\Psi_k(\cdot, E_k^\frac12 \bar{h}_{k}^i)+ D_y\Psi(x,E_k^\frac12 \bar{h}_{k}^i)\, E^{\frac12}\,D\bar{h}_{k}^i }^2} \\
&\le \sum_{i=1}^Q \int_{B_s} E_k^{-1} \Abs{ \abs{D_x\Psi_k(\cdot, \bar{u}_{k}^i))}^2 - \abs{D_x\Psi_k(\cdot, E_k^\frac12 \bar{h}_{k}^i)}^2}\\&\quad + E_k^{-\frac12} C^1_k(x)\left(E_k^{-\frac12}\abs{D\bar{u}_{k}^i}+ \int_{B_s} E_k^{-1} \abs{D\bar{u}_{k}^i}^2 \right) + E_k^{-\frac12} C^2_k(x)\left(\abs{D\bar{h}_{k}^i}+\int_{B_s} E_k^{\frac12}\abs{D\bar{h}_k^i}^2\right)\, ,
\end{align*}
where the measurable functions $C^{j}_k(x)$, $j=1,2$, consist of a product of two first derivatives of $\Psi_k$, and hence $\norm{C^{j}_k}_0 =O( E_k )$). Since $E_k^{-1} \Dir(\bar{u}_k, B_{\frac52}), \Dir(\bar{h}_k, B_{\frac52})$ are uniformly bounded by \eqref{e:Ebeta Dir}, the last two integrals are $o (1)$.

The remaining term can be estimated by 
\begin{align*}
&\int_{B_s} \sum_{i=1}^Q E_k^{-1} \Abs{ \abs{D_x\Psi_k(\cdot, \bar{u}_{k}^i))}^2 - \abs{D_x\Psi_k(\cdot, E_k^\frac12 \bar{h}_{k}^i)}^2}\\
& \le \int_{B_s} \sum_{i=1}^Q E_k^{-\frac12}\, \Abs{ D_x\Psi_k(\cdot, \bar{u}_{k}^i)+ D_x\Psi_k(\cdot, E_k^\frac12 \bar{h}_{k}^i)} \, E_k^{-\frac12}\,\Abs{ D_x\Psi_k(\cdot, \bar{u}_{k}^i)- D_x\Psi_k(\cdot, E_k^\frac12 \bar{h}_{k}^i)} \\
& \le C\, \int_{B_s} E_k^{-\frac12} \norm{D\Psi_k}_0 \norm{D^2\Psi_k}_0 \G_s(E_k^{-\frac12} \bar{u}_k, \bar{h}_k) = o(1)\,.
\end{align*}

\subsection{Construction of the maps $\bar{h}$ or $h_j$} Let $\iso$ be the isometry defined in \cite[Proposition 2.6]{DLHMS_linear}, and define $\left( \bar{v}_k, \bar{w}_k, \bfeta \circ \bar{u}_k \right) = \iso \circ \bar{u}_k$. 
As in \cite[Definition 2.7]{DLHMS_linear}, we set
\begin{align*}
 B_+^k &:= \{ x \in B_{\frac 52} \colon |\bar{v}_k| = \abs{\bar{u}_k^+ \ominus \bfeta \circ \bar{u}_k} >0 \} \quad \text{ and }\\
 B_-^k &:= \{ x \in B_{\frac 52} \colon |\bar w_k| = \abs{\bar{u}_k^- \ominus \bfeta \circ \bar{u}_k} >0 \}\, .
\end{align*}
We distinguish if the limit 
\[ \limsup_{k \to \infty } \min\{ \abs{B^k_+}, \abs{B^k_-} \}=:b \]
satisfies $b>0$ or $b=0$.\\

\emph{Case $b>0$ :} After translating the currents $T_k$ vertically we may assume without loss of generality that $\fint_{B_\frac52} \etab \circ \bar{u}_k = 0 $ for all $k$. Since both $\bar{v}_k$ and $\bar{w}_k$ vanish on sets of measure at least $b>0$, we claim that there exists a constant $C= C_b$ such that 
\begin{equation} \label{e:bound on the L2 norm}
\int_{B_{\frac{5}{2}}} \abs{\bar{u}_k}^2 \leq C_b \, E_k\,.
\end{equation}
Indeed, observe that the classical Poincar\'e inequality gives
\[
\begin{split}
\int_{B_{\frac{5}{2}}} \abs{\bar{u}_k}^2 &= \int_{B_{\frac{5}{2}}} \abs{\bar{u}_k \ominus \bfeta \circ \bar{u}_k}^2 + Q \int_{B_{\frac{5}{2}}} \abs{\bfeta \circ \bar{u}_k}^2 \\
&= \int_{B_{\frac{5}{2}}} \abs{\bar{v}_k}^2 + \int_{B_{\frac{5}{2}}} \abs{\bar{w}_k}^2 + Q \int_{B_{\frac{5}{2}}} \abs{\bfeta \circ \bar{u}_k}^2 \\
&\leq C_b \int_{B_{\frac{5}{2}}} |D |\bar{v}_k||^2 + C_b \int_{B_{\frac{5}{2}}} |D|\bar w_k||^2 + C_b \int_{B_{\frac{5}{2}}} | D \etab\circ \bar u_k|^2
\leq C_{b} \, \Dir(\bar{u}_k, B_{\frac{5}{2}})\,,
\end{split}
\]
which implies \eqref{e:bound on the L2 norm} again by \eqref{e:Ebeta Dir}.

Modulo passing to an appropriate subsequence, we therefore have that 
\[ E_k^{-\frac 12} \bar{u}_k \to \bar{h} \]
weakly in $W^{1,2}(B_{\frac 52} , \mathscr{A}_{Q}(\R^{\bar{n}}))$.

\emph{Case $b=0$ :} We assume that $\abs{B_-^k} \to 0$, the other case being equivalent. 
Consider the map $\bar{u}^+_k$ in $W^{1,2}(B_{\frac 52} , \Iq(\R^{\bar{n}}))$. When needed, we may identify $\bar{u}_k^+$ with $(\bar{u}_k^+, 1 )$ taking values in $\mathscr{A}_Q(\R^{\bar{n}})$. We note that
\begin{align*}
	\Dir(\bar{u}_k^+, B_{\frac 52}) &\le C\, \Dir(\bar{u}_k, B_{\frac 52}) \le C E_k\,;\\
	\int_{B_{\frac52}} \G_s( \bar{u}_k, \bar{u}_k^+)^2 &= \int_{B_-^k} \abs{\bar{u}_k^- \ominus \bfeta \circ \bar{u}_k}^2 \le \abs{B_-^k}^{1- \frac{2}{2^*}} \left( \int_{B_-^k} \abs{\bar{u}_k^- \ominus \bfeta \circ \bar{u}_k}^{2^*} \right)^{\frac{2}{2^*}} \le C \abs{B_-^k}^{1- \frac{2}{2^*}} E_k=o(E_k).
\end{align*}
We used in the last line Poincar\'e's inequality for $\bar{u}_k^-$ that is vanishing on a set of uniformly positive measure. 
Now we can apply the concentration compactness lemma, \cite[Proposition 4.3]{DLS_Lp}, to the sequence $E_k^{-\frac12} \bar{u}_k^+$ and deduce the existence of translating sheets
\[ \bar{h}_k = \sum_{j} \a{ y_{j,k} \oplus h_j } \]
with maps $h_j \in W^{1,2}(B_{\frac{5}{2}}, \mathcal{A}_{Q_j}(\R^{\bar{n}}))$ and points $y_{j,k} \in \R^{\bar{n}}$ such that the following properties are satisfied:
\begin{align}
&\norm{ \G_s(E_k^{-\frac12} \bar{u}_k^+, \bar{h}_k) }_2 \to 0 \label{e:new sequence for $b=0$}\\
&\liminf_{k \to \infty} \left(\int_{B_{\frac52} \cap K_k} E_k^{-1}\abs{D\bar{u}_k^+}^2 - \int_{B_{\frac 52}} \abs{D\bar{h}_k}^2 \right)\ge 0  \label{e:lower semicontinuity of the Dirichlet for translating sheet}\\
&\limsup_{k \to \infty} \int_{B_{\frac52}} \left( E_k^{-\frac 12}\abs{D\bar{u}_k^+} - \abs{D\bar{h}_k}\right)^2 \le \limsup_{k \to \infty} \left( E_k^{-1} \Dir(\bar{u}_k^+, B_{\frac52}) - \Dir(\bar{h}_k, B_{\frac52})\right)\,. \label{e: dirichlet convergence equal norm convergence}
\end{align}

\subsection{Lipschitz approximation of the competitors to $\bar{h}$ and $h_j$}

We fix a radius $s<\frac{5}{2}$.\\

To be able to interpolate later between $\bar{h}$ ($\bar{h}_k$) and $\bar{u}_k$ and similarly between the currents $T_k$ and $\bG_{u_k}$, by using a Fubini type argument we may fix $s< t< \frac{5}{2}$ such that for some $C>0$ depending on $\frac{5}{2} -s$ we have
\begin{align}
	\limsup_{k \to \infty} \int_{\partial B_t} \frac{ \G_s( E_k^{-\frac12} \bar{u}_k, \bar{h})^2}{\norm{\G_s( E_k^{-\frac12} \bar{u}_k, \bar{h})}_2^2} + E_k^{-1} \abs{D\bar{u}_k}^2 + \abs{D\bar{h}}^2 &\le C \qquad \text{ in case } b>0\,,\label{e:slice for $b>0$}\\
	\limsup_{k \to \infty} \int_{\partial B_t} \frac{ \G_s( E_k^{-\frac12} \bar{u}_k, \bar{h}_k)^2}{\norm{\G_s( E_k^{-\frac12} \bar{u}_k, \bar{h}_k)}_2^2} + E_k^{-1} \abs{D\bar{u}_k}^2 + \abs{D\bar{h}_k}^2 &\le C \qquad \text{ in case } b=0\,,\label{e:slice for $b=0$}\\
	 \mass^p( \langle T_k - \boldsymbol{G}_{u_k}, f, t \rangle ) \le C \mass^p( (T_k - \boldsymbol{G}_{u_k})\res\bC_3) &\le C E_k^{ 1-2\beta}\,, \label{e:slice of the current}
\end{align}
where, in \eqref{e:slice of the current}, $f$ is the function defined by $f(y,z) := \abs{y}$ for $(y,z) \in \pi_0 \times \pi_0^\perp$. Also, in \eqref{e:slice for $b=0$} we identified as before $\bar{h}_k$ with the map $(\bar{h}_k, 1)$ taking values in $\mathscr{A}_{Q}(\R^{\bar{n}})$ and used \eqref{e:new sequence for $b=0$}; in \eqref{e:slice of the current} we used the conclusions of Proposition \ref{p:max} as well as the Taylor expansion in \cite[Equation (13.5)]{DLHMS_linear}.\\

Now let us fix an arbitrary $\epsilon>0$.\\

\emph{Case $b>0$:}
Given any competitor $\bar{c} \in W^{1,2}(B_{\frac 52}, \mathscr{A}_{Q}(\R^{\bar{n}}))$ to $\bar{h}$ that agrees with $\bar{h}$ outside of $B_s$, we may apply the Lipschitz approximation Lemma for special multi-valued maps \cite[Lemma 5.5]{DLHMS_linear} to $\bar{h}$ and $\bar{c}$ in order to obtain Lipschitz continuous maps $\bar{h}^{\varepsilon}$ and $\bar{c}^{\varepsilon}$ for which the inequalities \cite[Equations (5.20) \& (5.21)]{DLHMS_linear} 
hold true with $\varepsilon^2$ in place of  $\epsilon$.\\

\emph{Case $b=0$:}
We apply the same procedure as in the case of $b>0$. Given competitors $c_j \in W^{1,2}(B_{\frac 52}, \mathscr{A}_{Q_j}(\R^{\bar{n}}))$ to $h_j$ that agree with $h_j$ outside of $B_s$ we may apply the Lipschitz approximation lemma to each $h_j$ and $c_j$ in order to obtain Lipschitz continuous maps $h_j^{\varepsilon}$ and $c_j^\varepsilon$ such that the inequalities  \cite[Equations (5.20) \& (5.21)]{DLHMS_linear} hold true with  $\varepsilon^2$ in place of  $\epsilon$. Furthermore we define 
\begin{align*}
\bar{h}^\epsilon_k:= \sum_j \a{ y_{j,k} \oplus h^\epsilon_j } \\
\bar{c}^\epsilon_k:= \sum_j \a{ y_{j,k} \oplus c^\epsilon_j }
\end{align*}

\subsection{Interpolating functions}
The argument below does not distinguish between the cases $b>0, b=0$. To handle them simultaneously, we just consider the trivial sequence $\bar{h}_k = \bar{h}$ in the case when $b>0$. \\

For each $k$ we fix now an interpolating map $\varphi_k \in W^{1,2}(B_t\setminus B_{(1-\epsilon)t}, \mathscr{A}_{Q}(\R^{\bar{n}}))$ by means of Luckhaus' Lemma \cite[Lemma 5.4]{DLHMS_linear} such that 
\begin{align*}
&\varphi_k(x) = E_k^{-\frac12} \bar{u}_k(x) \text{ and } \varphi_k((1-\epsilon)x) = \bar{h}^\epsilon_k(x) \text{ for all } x \in \partial B_t\\
&\int_{B_t\setminus B_{(1-\epsilon)t}} \abs{D\varphi_k}^2 \le C \epsilon \left(\int_{\partial B_t} E_k^{-1} \abs{D\bar{u}_k}^2 + \abs{D\bar{h}^\epsilon_k}^2\right) + \frac{C}{\epsilon} \int_{\partial B_t} \G_s(E_k^{-\frac12}\bar{u}_k, \bar{h}_k^\epsilon)^2 
\end{align*}
Observe that by our choice of the Lipschitz approximation $\bar{h}_k^\epsilon$ we have
\begin{equation}\label{e:small dirichlet in the interpolation function}
	\int_{B_t\setminus B_{(1-\epsilon)t}} \abs{D\varphi_k}^2 \le C \epsilon \qquad \mbox{ for large $k$ (depending on $\varepsilon$)}\,.
\end{equation}

Moreover, observe that, by construction, $\limsup_{k \to \infty} \Lip(\bar{h}_k^\eps) \leq C^*_\eps$, where $C^*_\eps$ is a constant depending on $\eps$ but independent of $k$. Also, again for large values of $k$ (depending on our fixed $\eps$):
\begin{align*}
\norm{\cG_s(E_k^{-\sfrac12} \bar{u}_k,  \bar{h}_k^\eps)}_{L^\infty(\partial B_t)} & \leq C\, \norm{\cG_s(E_k^{-\sfrac12} \bar{u}_k,  \bar{h}_k^\eps)}_{L^2(\partial B_t)} + C\, \Lip(E_k^{-\sfrac12} \bar{u}_k) + C\, \Lip(\bar{h}_k^\eps) \\
&\leq C\, \eps + C\, E_k^{\beta - \sfrac12} + C^*_\eps\,.
\end{align*}
Hence, from \cite[Equation (5.19)]{DLHMS_linear} we conclude that
\begin{equation} \label{interpolation Lipschitz estimate}
\Lip(\varphi_k) \leq C_\eps\, E_k^{\beta - \sfrac12} + C_\eps \leq C_\eps \, E_k^{\beta - \sfrac12}\,,
\end{equation}
where the last inequality is a consequence of the fact that $E_k^{\beta - \sfrac12} \to \infty$ as $k \uparrow \infty$.

In particular we can define competitors to $E_k^{-\frac 12}\bar{u}_k$ on $B_t$ by
\[\hat{c}_k(x):=\begin{cases}
\varphi_k(x) &\text{ for } (1-\epsilon) t \le \abs{x} \le t \\
\bar{c}_k^\epsilon(\frac{x}{1-\epsilon}) &\text{ for } \abs{x} \le (1-\epsilon) t
\end{cases}\]
We observe that by our construction we have
\begin{equation}\label{e:lim inf of Dirichlet 1}
	\liminf_{k \to \infty} E_k^{-1}\Dir(\bar{u}_k, B_t \cap K_k) - \Dir(\hat{c}_k, B_t) \ge \left(\sum_j \Dir(h_j, B_t) - \Dir(c_j,B_t)\right) - C \epsilon. 
\end{equation}
We have used \eqref{e:lower semicontinuity of the Dirichlet for translating sheet}, the closeness of the Dirichlet energies of $c_j$ and $c_j^\epsilon$ and \eqref{e:small dirichlet in the interpolation function}.
As we have seen in the calculations below point (ii) above, we can use the fact that $\norm{D\Psi_k}_0+ \norm{D^2\Psi_k}_0 = O(E_k^\frac12)$ to pass to $u_k$ and $w_k = (E_k^{\frac{1}{2}}\hat{c}_k, \Psi_k(\cdot, E_k^{\frac{1}{2}} \hat{c}_k))$ still satisfying
\begin{equation}\label{e:lim inf of Dirichlet 2}
	\liminf_{k \to \infty} E_k^{-1}\left(\Dir(u_k, B_t \cap K_k) - \Dir(w_k, B_t)\right) \ge \left(\sum_j \Dir(h_j, B_t) - \Dir(c_j,B_t)\right) - C \epsilon. 
\end{equation}\\

\subsection{Interpolating Currents}
By our choice of $t$, \eqref{e:slice of the current}, and the fact that the boundary operator commutes with slicing we have
\[ \partial^p \langle T_k - \boldsymbol{G}_{u_k}, f, t \rangle =0. \]
Using \cite[$(4.2.10)^\nu$]{Federer69}, we can fix an isoperimetric filling $S_k$, which can be assumed to be representative $\modp$, such that 
\[
\partial S_k = \langle T_k - \boldsymbol{G}_{u_k}, f, t \rangle \, \modp
\] 
and 
\[\mass(S_k) = \mass^p (S_k) \le C \mass^p( \langle T_k - \boldsymbol{G}_{u_k}, f, t \rangle)^{\frac{m}{m-1}} \le C\, E_k^{\frac{m (1-2\beta)}{m-1}} = o(E_k)\]
by the choice of $\beta$.
\\

\subsection{Dirichlet minimality}
We can now finally define a competitor to $T_k$ by
\[ Z_k:= T_k \res (\bC_4 \setminus \bC_t) + S_k + \boldsymbol{G}_{w_k}. \]
Observe that, by the hypotheses on $T_k$, Lemma \ref{modp_slicing_formula}, and the choice of $S_k$, we have
\[
\partial^p Z_k = - \left[ \langle T_k, f,t \rangle \right] + \left[ \langle T_k - {\bf G}_{u_k}, f,t \rangle \right] + \left[ \langle {\bf G}_{u_k}, f,t \rangle \right] = 0\,. 
\]
Let us observe that by construction, and using once again the Taylor expansion of the mass of a special multi-valued graph \cite[Equation (13.5)]{DLHMS_linear}, we compute:
\begin{align*}
\e_{T_k}(B_t) - \frac12 \Dir(u_k, B_t \cap K_k) &= \e_{T_k}(B_t \setminus K_k) + o(E_k)\,,\\
\e_{Z_k}(B_t) - \frac 12 \Dir(w_k, B_t) & \leq \mass(S_k) + \e_{\boldsymbol{G}_{w_k}}(B_t)	- \frac 12 \Dir(w_k, B_t) \leq o(E_k) \,,
\end{align*}
where in the last equality we have used that $\Dir(w_k, B_t) = O(E_k)$ whereas $\Lip(w_k) \leq C_\eps \, E_k^\beta$, so that
\[
\be_{\bG_{w_k}} - \frac{1}{2} \Dir(w_k, B_t) = \int_{B_t} \sum_{i} \bar R_4(Dw_{k}^{i}) \leq C \, E_{k}^{1+2\beta} = o(E_k) \quad \mbox{as $k \uparrow \infty$}\,.
\]
By minimality of $T_k$ in $\bC_3$ we then have
\begin{align*} 0 &\ge \mass(T_k \res \bC_3) - \mass(Z_k \res \bC_3)\\
&= \e_{T_k}(B_t) - \e_{Z_k}(B_t)\\
&\ge \frac12 \left(\Dir(u_k, B_t \cap K_k) -  \Dir(w_k, B_t)\right) + \e_{T_k}(B_t \setminus K_k) - o(E_k)\,.
\end{align*}
Hence dividing by $E_k$ and taking the $\limsup$ as $k \to \infty$ we deduce by \eqref{e:lim inf of Dirichlet 2}
\[ 0 \ge \frac12 \left(\sum_j \Dir(h_j, B_t) - \Dir(c_j,B_t)\right) + \limsup_{k \to \infty} E_k^{-1} \e_{T_k}(B_t \setminus K_k). \]
Since $\epsilon$ is arbitrary:
\begin{itemize}
\item[(i)] Choosing $c_j=h_j$, we see that $\limsup_{k \to \infty} E_k^{-1} \e_{T_k}(B_t \setminus K_k) =0$;\\
\item[(ii)] By the arbitrariness of $c_j$ we conclude the Dirichlet minimality of $h_j$. Afterwards by \eqref{e:lower semicontinuity of the Dirichlet for translating sheet} we deduce that $\limsup_{k \to \infty} E^{-1}_k \Dir(u_k, B_t\cap K_t) - \Dir(h_k, B_t) = 0$. In combination with \eqref{e: dirichlet convergence equal norm convergence} we obtain the second part of (ii), thus completing the proof.
\end{itemize}

\end{proof}

\section{Improved excess estimate and higher integrability}

So far, Proposition \ref{p:max} and Theorem \ref{t:harm_1} have shown that if $T$ is as in Assumption \ref{ipotesi_base_app1} then there is a Lipschitz continuous multiple valued function (possibly special, in case $p$ is an even integer and $Q = \frac{p}{2}$) whose graph coincides with the current in a region where the excess measure is suitably small in a uniform sense; furthermore, if $T$ is also area minimizing $\modp$ then such an approximating Lipschitz multiple valued function is almost Dirichlet minimizing, and both the Dirichlet energy of the approximating function and the excess of the original current in the ``bad region'' decay faster than the excess. The goal of this section is to exploit the closeness of the Lipschitz approximation to a $\Dir$-minimizer in order to deduce extra information concerning the behavior of the excess measure of $T$. We begin observing that the classical result on the higher integrability of the gradient of a harmonic function extends not only to classical multiple valued functions, as it is shown in \cite[Theorem 6.1]{DLS_Lp}, but also to special multiple valued functions. 

\begin{theorem}\label{t:higher integrability of the gradient}
There exists $p>2$ such that for every $\Omega' \Subset \Omega	\subset \R^m$ open domains, there is a constant $C>0$ such that 
\[ \norm{Du}_{L^p(\Omega')} \le C \norm{Du}_{L^2(\Omega)} \text{ for every $\Dir$-minimizing } u \in W^{1,2}(\Omega, \Iqspec).\]
\end{theorem}
\begin{proof}
The proof is the very same presented in \cite[Theorem 6.1]{DLS_Lp}: one only has to replace the Almgren embedding $\boldsymbol{\xi}$ for $\Iqs$ used in there with the new version of the Almgren embedding $\boldsymbol{\zeta}$ for $\Iqspec$ introduced in \cite[Theorem 5.1]{DLHMS_linear}.
\end{proof}

As a direct corollary of the first harmonic approximation and the higher integrability of the gradient we obtain the following result.
\begin{corollary}\label{cor:weak improved excess estimate}
For every $\eta>0$ there exist an $\epsilon>0$ and a constant $C>0$ with the property that, if $T$ satisfies Assumption \ref{ipotesi_base_app1} and is area minimzing $\modp$ in the cylinder $\bC_{4s}(x)$ with $E \le \epsilon$ then for every $A\subset B_s$ with $\abs{A\cap B_s} \le \epsilon \abs{B_s}$ we have
\begin{equation}\label{e:weak estimate 1} \e_T(A) \le \left(\eta E + C \bA^2s^2\right)s^m.
\end{equation}
\end{corollary}

\begin{proof}
By scaling and translating we may assume without loss of generality that $x=0$ and $s=1$. We fix $\beta= \frac{1}{4m}$ and $\eta_* >0$ to be determined below. Now let $\epsilon_*=\epsilon_*(\beta, \eta_*)$ taken from Theorem \ref{t:harm_1}.  
We distinguish the following two cases: either $\bA \le \epsilon_* E^{\frac12}$ or $\bA > \epsilon_* E^{\frac12}$. In the latter case the inequality holds trivially with $C= \epsilon_*^{-2}$ because
\[ \e_T(A) \le E \le \epsilon_*^{-2} \bA^2. \]
In the first case, we can apply the first harmonic approximation, Theorem \ref{t:harm_1}. Now let $h(x)=(\bar{h}(x), \Psi(x,\bar{h}(x)))$, with $\bar{h}$ Dirichlet minimizing, the associated map as in (i). 
By \eqref{e:eta-star-condition} we directly conclude that
\begin{equation}\label{e:A without K} \e_T(A\setminus K) \le \eta_* E\,,	
\end{equation}
where $K$ is, as usual, the ``good set'' for the $E^\beta$-Lipschitz approximation of $T$ in $\bC_3$ as in Proposition \ref{p:max}. In order to estimate the $\e_T$ measure of the portion of $A$ inside $K$, we observe that
\begin{align*}\Abs{\e_{T}(A\cap K)  - \frac12 \int_{A\cap K} \abs{Dh}^2} &= \Abs{\e_{\boldsymbol{G}_{u}}(A \cap K)-  \frac12 \int_{A\cap K} \abs{Dh}^2}\\
& \le \Abs{\e_{\boldsymbol{G}_{u}}(A \cap K)- \frac12 \int_{A\cap K} \abs{Du}^2}+ \frac12 \Abs{ \int_{A\cap K} \abs{Du}^2 - \abs{Dh}^2}\\	&=: I + II
\end{align*}
The first addendum can be bounded by the Taylor expansion of mass by
\[ I \le C\, \Lip(u)^2 \, \int_{A\cap K} \abs{Du}^2 \le C E^{1+2\beta}; \]
the second can be estimated using \eqref{e:harm-app1}
and $\abs{Du}^2 - \abs{Dh}^2 = (\abs{Du} + \abs{Dh}) ( \abs{Du}- \abs{Dh})$ by
\[ II \le C \left(\int_{A\cap K} \abs{Du}^2 + \abs{Dh}^2 \right)^{\frac12}\left(\int_{A\cap K} (\abs{Du} - \abs{Dh})^2 \right)^{\frac12} \le C \eta_*^{\frac12}E. \]
Recall that $\bA \le \epsilon_* E^{\frac12}$ implies that $\norm{D\Psi}\le C\, E^{\frac12}$. Hence we have
\begin{align*} \int_{A \cap K} \abs{Dh}^2 &= \int_{A \cap K} \abs{D\bar{h}}^2 + \abs{ D_x\Psi(x, \bar{h}) + D_y\Psi(x,\bar{h})D\bar{h}}^2\\
&\le (1 + C\, E) \int_{A \cap K} \abs{D\bar{h}}^2 + C \, E \abs{A \cap K}
\end{align*} 
Using the higher integrability for Dirichlet minimizers we can estimate further
\begin{align*} \int_{A\cap K} \abs{D\bar{h}}^2 &\le \abs{A \cap K}^{1-\frac{2}{p}} \left( \int_{A\cap K} \abs{D\bar{h}}^p \right)^{\frac 2p} \\&\le C \abs{A \cap K}^{1-\frac2p} \int_{B_2} \abs{D\bar{h}}^2 \le C \abs{A\cap K}^{1-\frac2p} E.\end{align*}
Collecting all the estimates we get in conclusion
\begin{align*}
\e_T(A) & \le \e_T(A\setminus K) + \Abs{\e_T(A \cap K) - \frac12 \int_{A\cap K} \abs{Dh}^2} + 	\frac12 \int_{A\cap K} \abs{Dh}^2 \\
&\le \left(\eta_* + C E^{2\beta} + C \eta_*^\frac12  + C \abs{A \cap K}^{1-\frac2p} \right) \,E.
\end{align*}
Hence, the estimate in \eqref{e:weak estimate 1} follows also in this case after suitably choosing $\eps$ and $\eta_*$ depending on $\eta$.
\end{proof}

For the following proof, we introduce the centered maximal function for a general radon measure $\mu$ on $\R^m$ by setting
\[ \bmax_c\mu(x):= \sup_{s\ge 0 } \frac{ \mu(B_s(x))}{\omega_m s^m}\]
Observe that one has the straightforward comparison between the centered and non-centered maximal functions
\[ \bmax_c\mu(x) \le \bmax \mu(x) \le 2^m\, \bmax_c \mu(x). \]
Although the two quantities are therefore comparable, we decided to work for this proof with the centered version since in our opinion the geometric idea becomes more easily accessible.
Furthermore we note that since the map $x \mapsto \frac{\mu(B_{s}(x))}{\omega_m\, s^m}$ is lower semicontinuous, $x \mapsto \bmax_c\mu(x)$ is lower semicontinuous as it is the supremum of a family of lower semicontinuous functions.

\begin{theorem}\label{thm:gradient Lp estimate}
	There exist constants $0<q<1$, $C, \epsilon>0$ with the following property. If $T$ is area minimzing $\modp$ in the cylinder $\bC_{4}$ and
satisfies Assumption \ref{ipotesi_base_app1} with $E \le \epsilon$ then
\begin{equation}\label{eq:gradient Lp estimate}
	\int_{B_2}  (\min\{\bmax_c \e, 1\})^{q} \, d\e \le C e^{C\bA^2} E^{1+q}.
\end{equation}
\end{theorem}

In particular this implies the following estimate 
\[ \int_{B_2\cap \{ \bmax_c \e \le 1 \}}  (\bmax_c \e)^{q} \, d\e \le C e^{C\bA^2} E^{1+q}\]

\begin{remark} \label{rmk:old_vs_new_grad_Lp_est}
Observe that the excess measure $\e$ can be decomposed as
\[
\e = \bd \, \mathscr{L}^m + \e_{sing}\,,
\]
where $\mathscr{L}^m$ denotes the Lebesgue measure in $\R^m$, $\e_{sing} \perp \mathscr{L}^m$ and $\bd$ is the excess density as in Definition \ref{d:excess}. Since $\bd(x) \leq \bmax_c \e(x)$ for every $x \in B_2$, we have 
\[
\int_{B_2}  (\min\{\bmax_c \e, 1\})^q \, d\e \geq \int_{B_2}  (\min\{\bd, 1\})^q \, d\e \geq \int_{B_2}  (\min\{\bd, 1\})^q \, \bd \, dx\,,
\] 
so that formula \eqref{eq:gradient Lp estimate} in particular implies the following higher integrability of the excess density:
\begin{equation} \label{eq:old_grad_Lp_est}
\int_{\{\bd \leq 1\} \cap B_2} \bd^{1+q} \, dx \leq C e^{C \bA^2} E^{1+q} \leq C E^{1+q}\,.
\end{equation}

\end{remark}

\begin{proof}
Let us first observe that given any measure $\mu$ on $\R^m$ we have that, for any fixed $r > 0$ and $t > 0$, if 
\[ \frac{\mu(B_s(x))}{\omega_m s^m} \le \left(\frac34\right)^m t \quad \forall 	s \ge 4r \]
then for some constant $C$ depending on $m$ we have
\begin{equation}\label{e:center maximal 4Br} \abs{ B_r(x) \cap \{ y \colon \bmax_c\mu(y) > t \} } \le \frac{C}{t} \mu\left( B_{4r}(x) \cap \left\{ y \colon \bmax_c\mu(y) > \frac{t}{2} \right\} \right). \end{equation}

This can be seen as follows: we first note that for $y\in B_r(x)$ we have
\[ \frac{\mu(B_s(y))}{\omega_m s^m} \le \begin{cases} \left( \frac{4r}{s}\right)^m \frac{\mu(B_{4r}(x))}{\omega_m (4r)^m} & \text{ if } s+\abs{x-y} \le 4r \\
 	\left( \frac{s+\abs{x-y}}{s}\right)^m \frac{\mu(B_{{s+ \abs{x-y}}}(x))}{\omega_m (s+\abs{x-y})^m} & \text{ if } s+\abs{x-y} \ge 4r.
 \end{cases}
 \]
Hence, we deduce that if $s \geq 3r$ then $\frac{\mu(B_s(y))}{\omega_m s^m} \le t $: in other words, if $\frac{\mu(B_s(y))}{\omega_m s^m} > t$ then we must have $B_s(y) \subset B_{4r}(x)$. This implies that
\[ B_r(x) \cap \{ y \colon \bmax_c\mu(y) \ge t \} = B_r(x) \cap \{ y \colon \bmax_c\mu\res_{B_{4r}(x)} (y) \ge t \}\,, \]
so that \eqref{e:center maximal 4Br} follows by a variation of the classical maximal function estimate applied to $\mu\res_{B_{4r}(x)}$. \footnote{The variation in use here can be deduced in a straightforward fashion from the classical estimate for the whole space: apply the classical estimate (see e.g. \cite[Theorem 2.19 (2)]{Mattila95}) to the measure $\tilde{\mu} := \mu\res\{ \bmax_c\mu > \frac t2\}$ and note that since $\mu \le \tilde{\mu} + \frac t2  \, \mathscr{L}^m$ we have $\{\bmax_c\mu > t\}\subset \{\bmax_c\tilde{\mu} > \frac t2\}$.}

Furthermore we recall that by classical differentiation theory of radon measures \footnote{Note for each $ y \in B_r(x)\cap\{ \bmax_c\mu\le t \}$ one has $\liminf_{r\downarrow 0} \frac{\mu(B_r(y))}{\abs{B_r(y)}} \le t$, hence \eqref{eq:classical maximal measure estimate2} follows for instance from \cite[Lemma 2.13 (1)]{Mattila95}.} one has as well 
\begin{equation}\label{eq:classical maximal measure estimate2} \mu\left( B_r(x) \cap \{ y \colon \bmax_c \mu(y) \le t \}\right) \le t \abs{B_r(x) \cap \{ y \colon \bmax_c \mu(y) \le t \}}.
\end{equation}

In what follows, for the sake of simplicity, we will work with the measure $\e = \e\res{B_4}$, which is defined on the whole $\R^m$.

\smallskip

{\bf Step 1:} For every  $\eta>0$ there exist positive constants $\lambda, \epsilon, C$ with the property that if \begin{equation}\label{eq:e=t/beta} r:= \sup\left\{ s \colon \frac{\e(B_{s}(x))}{\omega_m s^m} \ge  \frac{t}{\lambda}\right\} \quad \text{ and } \quad \frac{t}{\lambda} \le \epsilon	
\end{equation}
then 
\begin{align}\label{eq:first decay estimate for the excess measure}
	& \e\left( B_r(x) \cap \left\{ y \colon \bmax_c\e(y) > t \right\} \right) \\& \nonumber \le \left(2\,\omega_m^{-1}\,\eta + C\bA^2 \left(\frac{2\lambda}{t}\, \e\left(B_{r}(x) \cap \left\{ y \colon \bmax_c\e(y) > \frac{t}{2\lambda} \right\} \right)\right)^{\frac{m+2}{m}} \right)   \e\left(B_{r}(x) \cap \left\{ y \colon \bmax_c\e(y) > \frac{t}{2\lambda} \right\} \right) 
\end{align}
{\bf Proof of Step 1:}
Let $\eta>0$ be given, and let $\epsilon>0$ be given by Corollary \ref{cor:weak improved excess estimate} in correspondence with this choice of $\eta$. Also fix $\lambda>\left({\frac43}\right)^m$.
By the definition of $r$ and the continuity of measures along increasing and decreasing sequence of sets, we easily see that 
\begin{equation}\label{eq:e=t/beta 2} \frac{\be(\overline{B_r(x)})}{\omega_m r^m}=\frac{\be(B_r(x))}{\omega_m r^m}= \frac{t}{\lambda}> \frac{\be(B_s(x))}{\omega_m s^m} \text{ for all } s > r. \end{equation}
%In particular the right inequality holds for $s=4r$.
Thus we can apply \eqref{e:center maximal 4Br} with $\mu=\e$, thus deducing that
\[ \abs{ B_r(x) \cap \{ y \colon \bmax_c\e(y) > t \} } \le \frac{C}{t} \e\left( B_{4r}(x) \cap \left\{ y \colon \bmax_c\be(y) > \frac{t}{2} \right\} \right) \le \frac{C}{\lambda} \omega_m (4r)^m. \]
Since $\frac{t}{\lambda} \le \epsilon$, if we choose  $\lambda \ge \frac{4^m C}{\epsilon}$ then we can apply Corollary \ref{cor:weak improved excess estimate}, which, together with \eqref{eq:e=t/beta 2}, yields
\begin{equation}\label{eq:1step application of the weak improved estimate} \e\left( B_r(x) \cap \{ y \colon \bmax_c\e(y) > t \} \right) \le \omega_m^{-1} 4^{-m}\eta \,\e(B_{4r}(x)) + C \bA^2 r^{ m+2} \le \omega_m^{-1} \eta \,\e(B_{r}(x)) + C \bA^2 r^{ m+2}.	
\end{equation}
 
Using \eqref{eq:classical maximal measure estimate2} and \eqref{eq:e=t/beta 2}, namely the identity $\frac{t}{\lambda} \omega_m r^m = \e(B_{r}(x))$ we have
\[ \e\left( B_{r}(x) \cap \left\{ y \colon \bmax_c\e(y) \le \frac{t}{2\lambda} \right\} \right) \leq \frac{t}{2\lambda} \abs{B_r(x)} \leq  \frac{1}{2} \e(B_{r}(x)).  \]
This implies that 
\[ \omega_m r^m = \frac{\lambda}{t} \, \e(B_{r}(x)) \le \frac{2\lambda}{t} \,\e\left( B_{r}(x) \cap \left\{ y \colon \bmax_c\e(y) > \frac{t}{2\lambda} \right\} \right). \]
Using this estimate in \eqref{eq:1step application of the weak improved estimate} we deduce \eqref{eq:first decay estimate for the excess measure}.

\smallskip

{\bf Step 2:} For every $\eta>0$ there exist positive constants $\lambda, \epsilon, C$ such that if \[4^{2m} E\le \frac{t}{\lambda}   \le \epsilon \quad \mbox{and} \quad r \le 3\] then, setting $\bar{r} := r+ 4\left(\frac{\lambda E}{t}\right)^{\frac1m}$, we have 
\begin{equation}\label{eq:second decay estimate for the excess measure} 
\e\left( B_r \cap \{ y \colon \bmax_c\e(y) > t \} \right)\\
\le c_B \, \left(\eta + C\bA^2 \left(\frac{2\lambda E}{t}\right)^{\frac{m+2}{m}} \right)\,  \e\left(B_{\bar r} \cap \left\{ y \colon \bmax_c\e(y) > \frac{t}{2\lambda} \right\} \right),
\end{equation}
where $c_B$ denotes the Besicovitch constant in $\R^m$. \\

{\bf Proof of Step 2:}
For each $x \in B_r \cap \{ y \colon \bmax_c\e(y) > t \}$ we let
\[ r_x:= \sup\left\{ s \colon \frac{\e(B_{s}(x))}{\omega_m s^m} \geq  \frac{t}{\lambda}\right\}. \]
We must have $0 < r_x\le \frac14$, since $\bmax_c\e(x) > t \geq t/\lambda$, and since for each $x \in B_3$ we have 
\[ \frac{\e(B_s(x))}{\omega_m s^m} \le 4^{2m} E \leq \frac{t}{\lambda} \quad\forall s \ge \frac14\,. \]
We apply the Besicovitch covering theorem to the family
\[ \mathcal{B}:= \{ \overline{B_{r_x}(x)} \colon x \in B_r \cap \{ y \colon \bmax_c\e(y) > t \} \} \]
and obtain sub-collections $\mathcal{B}_1, \dotsc, \mathcal{B}_{c_B}$ of balls such that each subfamily is pairwise disjoint and 
\[B_r \cap \{ y \colon \bmax_c\e(y) > t \}\subset \bigcup_{j=1}^{c_B} \bigcup_{B_{r_x}(x)\in \mathcal{B}_j} \overline{B_{r_x}(x)}\,. \]
 
Since for each of these balls we have $\omega_m\, {r_x}^m = \frac{\lambda}{t}\, \e(B_{r_x}(x)) \le \frac{\lambda}{t} E \omega_m 4^m$,
we deduce $B_{r_x}(x) \subset B_{\bar r}$. Hence the result follows from 
\[ \e\left( B_r \cap \{ y \colon \bmax_c\e(y) > t \} \right)\le \sum_{i=1}^{c_B} \sum_{B_{r_x}(x) \in \mathcal{B}_i} \e\left( B_{r_x}(x) \cap \{ y \colon \bmax_c\e(y) > t \} \right)\,,\]
where we used that by Step 1 $\be(\partial B_{r_x}(x)=0$ for each of these balls, and then applying \eqref{eq:first decay estimate for the excess measure} of Step 1 to each. %for each ball $B_s(y)$.

\smallskip

{\bf Step 3:} For every $\eta >0$ there are constants $C, \lambda, \epsilon$ such that for every $k \ge 2$ with  \[(2\lambda)^k E \le \epsilon \quad \mbox{and} \quad r\le \frac52 \] we have 
\begin{equation}\label{eq:third decay estimate for the excess measure} \e\left( B_r \cap \{ y \colon \bmax_c\e(y) > (2\lambda)^k E \} \right)\le (c_B\eta)^k e^{ C \bA^2 } \,  \e\left(B_{r+ \frac12} \cap \left\{ y \colon \bmax_c\e(y) > 2\lambda E \right\} \right)\end{equation}
{\bf Proof of Step 3:}
 This is obtained by iterating Step 2. More precisely, for each $2\le l \le k$ we set \[t_l:= (2\lambda)^l E\,, \qquad 
\begin{cases}
r_k := r &\,, \\ 
  r_{l-1}:=r_l + 4 \left(\frac{\lambda E}{t_l}\right)^{\frac1m} = r_l + \frac{4\lambda^{\frac1m}}{(2\lambda)^{\frac{l}{m}}} & \mbox{for $2 \leq l \leq k-1$} \,. \end{cases}\] Using $f(r,t):=\e(B_r \cap \{ y \colon \bmax_c\e(y)>t\})$ and $c_\bA:=\frac{C\bA^2}{c_B\eta}$ we may write \eqref{eq:second decay estimate for the excess measure} as
\[ f(r_l, t_l) \le c_B \eta \left( 1 + c_{\bA}\left(\frac{2\lambda E}{t_l}\right)^\frac{m+2}{m}\right) f(r_{l-1}, t_{l-1}) = c_B \eta \left( 1 + c_{\bA}\left(\frac{1}{2\lambda}\right)^{\frac{m+2}{m} (l-1)}\right) f(r_{l-1}, t_{l-1}).\]
Now \eqref{eq:third decay estimate for the excess measure} is a consequence of the following estimates ($\lambda$ is sufficient large)
\begin{align*}
&r_1 = r_k + 4\lambda^{\frac1m} \sum_{l=2}^k (2\lambda)^{-\frac{l}{m}} \le  r + 4 \lambda^{\frac{1}{m}} \sum_{l=2}^{\infty} (2\lambda)^{-\frac{l}{m}} \leq   r+ \frac12\\
&\prod_{l=2}^k c_B\eta \left( 1 + c_{\bA}\left(\frac{1}{2\lambda}\right)^{\frac{m+2}{m} (l-1)}\right) \le (c_B\eta)^k e^{C c_\bA}\,.
\end{align*}
In particular, the first estimates ensures that we may apply step 2 for each pair $(t_l, r_l)$. 

\smallskip

{\bf Conclusion:}
First we fix $\eta>0$ sufficiently small, so that $c_B \eta <1$, and afterwards $q>0$ such that $ a:=(2\lambda)^q \, c_B \eta < 1$. 
Now we observe that with $(2\lambda)^{k_0}E \le \epsilon \lambda < (2\lambda)^{k_0+1}E$ %$k_0 = \lfloor \frac{\ln(\frac{\epsilon}{E})}{\ln(2\lambda)}\rfloor$ we have 
we have 
\begin{align*}
	&\int_{B_2 \cap \{(2\lambda)^2 E < \bmax_c\e \}} (\min\{\bmax_c\e,\lambda\epsilon\})^q \, d\e\\
	& \le \sum_{k=2}^{k_0} \int_{ B_2\cap \{(2\lambda)^k E < \bmax_c\e \le (2\lambda)^{k+1} E \}} (\bmax_c\e)^q \, d\e + \int_{ B_2\cap \{(2\lambda)^{k_0 + 1} E < \bmax_c\e \}} \left((2\lambda)^{k_0+1}E\right)^q \, d\e\\ 
	&\le (4\lambda)^q E^q\sum_{k=2}^{k_0+1} (2\lambda)^{qk} \, \e(B_2 \cap \{ \bmax_c\e > (2\lambda)^k E \}) \le (4\lambda)^q E^q e^{C\bA^2} E (\omega_m 4^m) \, \sum_{k=2}^{k_0+1} a^k \\
	&\le C e^{C\bA^2} E^{1+q}.
\end{align*}
Combining this with
\[ \int_{B_2 \cap \{ \bmax_c\e \le (2\lambda)^2 E\}} (\bmax_c\e)^q \, d\e \le (2\lambda)^{2q} E^q \e(B_2 \cap \{ \bmax_c\e \le 2\lambda E\} ) \le C E^{1+q} \]
proves the result, modulo choosing a smaller value for $\epsilon$. 
\end{proof}

\section{Almgren's strong approximation theorem}

We can finally state and prove the main Lipschitz approximation result for area minimizing currents $\modp$, which contains improved estimates with respect to Proposition \ref{p:max}.

\begin{theorem}[Almgren's strong approximation] \label{thm:almgren_strong_approx}

There exist constants $\eps, \gamma, C > 0$ (depending on $m,\bar{n},n,Q$) with the following property. Let $T$ be as in Assumption \ref{ipotesi_base_app1} in the cylinder $\bC_{4r} (x)$, and assume it is area minimizing $\modp$. Also assume that $E = \bE(T, \bC_{4r}(x)) < \eps$. Then, there are $u \colon B_{r}(x) \to \Iqs$ if $Q < \frac{p}{2}$, or $u \colon B_{r}(x) \to \Iqspec$ if $Q = \frac{p}{2}$, and a closed set $K \subset B_{r}(x)$ such that: 
\begin{align}\label{e:inclusion}
&\gr (u) \subset \Sigma\,, \\ \label{e:lip_osc_est}
&\Lip (u)\leq C(E + \bA^2 r^2)^\gamma\quad \mbox{ and }\quad {\rm osc}\, (u) \leq C \bh (T, \bC_{4r} (x), \pi_0) + C r (E^{\sfrac{1}{2}} + r\bA)\,,\\\label{e:graph_current}
&\bG_u \res (K\times \R^{n})= T\mres (K\times \R^{n}) \,\modp\,,\\ \label{e:volume_estimate}
&|B_{r}(x)\setminus K|\leq \|T\| ((B_r (x)\setminus K)\times \R^n) \leq C (E + r^2\bA^2)^{1+\gamma} r^m\,,\\ \label{e:Taylor_expansion}
& \Abs{\|T\|(\bC_{\sigma r}(x)) - Q \omega_m (\sigma r)^m - \frac{1}{2} \int_{B_{\sigma r}} \abs{Du}^2} \leq (E + r^2\bA^2)^{1+\gamma} r^m \quad \forall\, 0<\sigma<1\,. 
\end{align}
\end{theorem}

The key improvement with respect to the conclusions of Proposition \ref{p:max} lies in the superlinear power of the excess in \eqref{e:volume_estimate} and \eqref{e:Taylor_expansion}. In turn, this gain is a consequence of the following improved excess estimate, analogous to \cite[Theorem 7.1]{DLS_Lp}.

\begin{theorem}[Almgren's strong excess estimate] \label{thm:almgren_strong_excess}
There exist constants $\eps_*,\gamma_*,C > 0$ (depending on $m,\bar{n},n,Q$) with the following property. Assume $T$ satisfies Assumption \ref{ipotesi_base_app1} and is area minimizing $\modp$ in $\bC_4$. If $E := \bE(T, \bC_4) < \eps_*$, then
\begin{equation} \label{e:almgren_estimate}
\e_{T}(A) \leq C (E^{\gamma_*} + \abs{A}^{\gamma_*}) (E + \bA^2) \quad \mbox{for every Borel $A \subset B_{\sfrac{9}{8}}$}\,.
\end{equation}

\end{theorem}

Let us assume for the moment the validity of Theorem \ref{thm:almgren_strong_excess}, and let us then show how Theorem \ref{thm:almgren_strong_approx} follows.

\begin{proof}[Proof of Theorem \ref{thm:almgren_strong_approx}]

As usual, since the statement is scale-invariant, we may assume $x = 0$ and $r = 1$. Choose $\beta < \min\left\lbrace \frac{1}{2m}, \frac{\gamma_*}{2(1+\gamma_*)} \right\rbrace$, where $\gamma^*$ is given by Theorem \ref{thm:almgren_strong_excess}. Let $u$ be the $E^{\beta}$-Lipschitz approximation of $T$, so that \eqref{e:inclusion} and \eqref{e:graph_current} are an immediate consequence of Proposition \ref{p:max}. Also the estimates in \eqref{e:lip_osc_est} follow in a straightforward fashion if we choose $\gamma \le \beta$ and we recall that $\| D\Psi \|_{0} \leq C (E^{\sfrac{1}{2}} + \bA)$. Now we come to the proof of the volume estimate \eqref{e:volume_estimate}. Set $A := \left\lbrace \bmax \e_T > E^{2\beta} \right\rbrace \cap B_{\sfrac{9}{8}}$. By \eqref{e:max1}, we have that $\abs{A} \leq C E^{1-2\beta}$. In order to improve the estimate, we use Almgren's strong excess estimate: indeed, equation \eqref{e:almgren_estimate} implies that
\begin{equation} \label{gain}
\e_T(A) \leq C E^{\gamma_*} (1 + E^{-2\beta\gamma_*}) (E + \bA^2)\,,
\end{equation}
so that when we plug \eqref{gain} back into \eqref{e:max1} we have
\[
\abs{B_1 \setminus K} \leq C E^{-2\beta} \e_T(A) \leq C E^{\gamma_* - 2\beta (1+\gamma_*)} (1+E^{2\beta\gamma_*}) (E + \bA^2) \leq C E^{\gamma_* - 2\beta(1+\gamma_*)} (E + \bA^2)\,,
\]
and the inequality
\[
\abs{B_1 \setminus K} \leq C (E+ \bA^2)^{1+\gamma}
\] 
follows with $\min\{ \gamma_* - 2\beta(1+\gamma_*), \beta \}> 0$ because of our choice of $\beta$. 
\eqref{e:volume_estimate} is then a simple consequence of
\[
\|T\| ((B_1\setminus K)\times \R^n) \leq \e_T (B_1\setminus K) + Q |B_1\setminus K|\, .
\]
Finally, we take any $0 < \sigma < 1$ and we estimate:
\[
\begin{split}
&\Abs{\|T\|(\bC_{\sigma}(x)) - Q \omega_m \sigma^m - \frac{1}{2} \int_{B_{\sigma}} \abs{Du}^2}\\
&\leq \e_{T}(B_\sigma \setminus K) + \e_{{\bf G}_u}(B_\sigma \setminus K) + \Abs{\e_{{\bf G}_u}(B_{\sigma}) - \frac{1}{2} \int_{B_{\sigma}} \abs{Du}^2} \\
&\overset{\eqref{gain}}{\leq} C (E + \bA^2)^{1+ \gamma} + C \abs{B_\sigma \setminus K} + C \Lip(u)^2 \int_{B_\sigma} \abs{Du}^2 \\ &\leq C (E + \bA^2)^{1+\gamma}\,.\qedhere
\end{split}
\]
\end{proof}

We turn now to the proof of Theorem \ref{thm:almgren_strong_excess}. We will use in an essential way the minimality $\modp$ of $T$, and in order to do that we need to construct a suitable competitor. In this process, a key role will be played by the following result, analogous to \cite[Proposition 7.3]{DLS_Lp}

\begin{proposition} \label{competitor_construction}
Let $\beta \in \left( 0, \frac{1}{2m} \right)$, and assume that $T$ satisfies Assumption \ref{ipotesi_base_app1} and is area minimizing $\modp$ in $\bC_4$. Let $u$ be its $E^{\beta}$-Lipschitz approximation. Then, there exist constants $\eps, \gamma, C > 0$ and a subset of radii $B \subset \left[\sfrac{9}{8},2\right]$ with measure $\abs{B} > \sfrac{1}{2}$ with the following property. If $\bE(T, \bC_4) < \eps$, then for every $\sigma \in B$ there exists a $Q$-valued map $g \in \Lip(B_{\sigma}, \Iqs)$ if $Q < \frac{p}{2}$ or $g \in \Lip(B_{\sigma}, \Iqspec)$ if $Q = \frac{p}{2}$ such that
\begin{equation}
\left. g\right|_{\partial B_\sigma} = \left. u \right|_{\partial B_\sigma}, \quad \Lip(g) \leq C (E + r^2 \bA^2)^{\beta}, \quad \spt(g(x)) \subset \Sigma \,\, \forall \, x \in B_\sigma,
\end{equation}
and
\begin{equation} \label{e:Dirichlet_replacement_est}
\int_{B_\sigma} \abs{Dg}^2 \leq \int_{B_\sigma \cap K} \abs{Du}^2 + C (E + \bA^2)^{1+\gamma}\,.
\end{equation}

\end{proposition}

\begin{proof}

The proof is obtained by a ``regularization by convolution'' procedure, analogous to that of \cite[Proposition 7.3]{DLS_Lp}, modulo using the embedding $\zetab$ of \cite[Theorem 5.1]{DLHMS_linear} in place of $\xib$.
%
%The proof is analogous to \cite[Proposition 7.3]{DLS_Lp}, modulo the following modifications. We have that $\abs{Df}^2 \leq C \bmax_c \e$ on $K$. Furthermore, $\bmax_c \e \leq C \bmax \e \leq C E^{2\beta}$, so that $\bmax_c\e \leq 1$ on $K$ if $E = \bE(T, \bC_4)$ is sufficiently small. Hence, by Theorem \ref{thm:gradient Lp estimate} we deduce with $q_1 = 2p_1 := 2(1+q)$, where $q$ is the exponent in Theorem \ref{thm:gradient Lp estimate}:
%\begin{equation}
%\| \abs{Df} \|_{L^{q_1}(K \cap B_2)}^{2} \leq C e^{C \bA^2 E^{\sfrac{2}{m}}} E {\color{red} \leq C E} \,.
%\end{equation}
%
%All the rest stays the same, modulo using the embedding $\zetab$ in place of $\xib$. One defines $f = (f_1, \Psi(\cdot, f_1))$, set $f' := \zetab \circ f_1$, $g'$ as in their paper, and then $g_1 := \zetab^{-1} \circ g'$ and $g = (g_1, \Psi(\cdot, g_1))$. The estimates on the Lipschitz constants stay the same, as well as \textit{Step 1} and \textit{Step 2}. Coming to \textit{Step 3}, we use formula \eqref{e:est_rho_star} to deduce the equivalent of \cite[Equation (7.11)]{DLS_Lp}. In the remaining part of the computation, the only term which is changing is the coefficient in front of $\left( \frac{\eps}{\delta^{\bar{n}Q +1}} \right)^{\frac{2(p_1-1)}{p_1}}$, which  is
\end{proof}

\begin{proof}[Proof of Theorem \ref{thm:almgren_strong_excess}]
Choose $\beta := \frac{1}{4m}$, and let $B \subset \left[ \sfrac{9}{8}, 2 \right]$ be the set of radii provided by Proposition \ref{competitor_construction}. By a standard Fubini type argument analogous to what has been used in deriving \eqref{e:slice of the current}  and the isoperimetric inequality $\modp$, we deduce that there exists $s \in B$ and an integer rectifiable current $R$ which is representative $\modp$ such that
\[
\partial R = \langle T - {\bf G}_u, \varphi, s \rangle \,\modp \quad \mbox{and} \quad \mass(R) \leq C E^{\frac{2m-1}{2m-2}}\,,
\]
where $u$ is the $E^\beta$-Lipschitz approximation of $T$ and $\varphi(x) = \abs{x}$. Now, let $g$ be the Lipschitz map given in Proposition \ref{competitor_construction} corresponding with the choice $\sigma = s$. Since $\left. g \right|_{\partial B_s} = \left. u \right|_{\partial B_s}$, it also holds $\langle {\bf G}_u - {\bf G}_g, \varphi, s \rangle = 0 \,\modp$. Furthermore, since $(\partial {\bf G}_g) \res \bC_s = 0 \, \modp$, and since $g$ takes values in $\Sigma$, the current ${\bf G}_g \res \bC_s + R$ is a competitor for $T$ in $\bC_s$, and thus, using \cite[Equation (4.1)]{DLHMS_linear}, the minimality of $T$ yields for some $\gamma > 0$:
\begin{equation} \label{exc_outK_est1}
\begin{split}
\|T\|({\bC_s}) &\leq \|{\bf G}_g \res \bC_s + R \|(\bC_s) \leq Q \abs{B_s} + \frac{1}{2} \int_{B_s} \abs{Dg}^2 + C E^{1+\gamma} \\ &\overset{\eqref{e:Dirichlet_replacement_est}}{\leq} Q \abs{B_s} + \frac{1}{2} \int_{B_{s} \cap K} \abs{Du}^2 + C E^{\gamma}(E + \bA^2)\,.
\end{split}
\end{equation}

On the other hand, again by \cite[Equation (4.1)]{DLHMS_linear} we also have:
\begin{equation} \label{exc_outK_est2}
\begin{split}
\|T\|(\bC_s) &= \|T\|((B_s \setminus K) \times \R^n) + \|{\bf G}_u\|((B_s \cap K) \times \R^n) \\
&\geq \|T\|((B_s \setminus K) \times \R^n) + Q \abs{B_s \cap K} + \frac{1}{2} \int_{B_s \cap K} \abs{Du}^2 - CE^{1+\gamma}\,.  
\end{split}
\end{equation}

Combining \eqref{exc_outK_est1} and \eqref{exc_outK_est2} we conclude that $\e_T(B_s \setminus K) \leq C E^{\gamma}(E + \bA^2)$. Now, we are able to prove the estimate \eqref{e:almgren_estimate}. Let $A \subset B_{\sfrac{9}{8}}$ be any Borel set. We get:
\begin{align} 
\e_T(A) = \e_T(A \cap K) + \e_T(A \setminus K) & \leq \frac{1}{2} \int_{A \cap K} \abs{Du}^2 + C E^{1+\gamma} + \e_T(B_s \setminus K)\nonumber\\
& \leq \frac{1}{2} \int_{A \cap K} \abs{Du}^2 + C E^{\gamma}(E + \bA^2)\,.\label{stima_finale_1}
\end{align}
On the other hand, observe that %$\abs{Du}^2 \leq C \bd_T \leq CE^{2\beta}$ on $K$, and therefore $\bd_T \leq 1$ on $K$ if $E$ is suitably small.
$\abs{Du}(x)^2 \leq C \bmax_c\e(x) \leq CE^{2\beta}$ on $K$, and therefore $\bmax_c\e(x) \leq 1$ on $K$ if $E$ is suitably small. Let $q>0$ be the exponent given by Theorem \ref{thm:gradient Lp estimate}, we deduce from \eqref{eq:gradient Lp estimate} that
%Then, if we set $q_{1} := 2(1+q)$, where $q > 0$ is the exponent given by Theorem \ref{thm:gradient Lp estimate}, we deduce from \eqref{eq:old_grad_Lp_est} that
\[
\int_{A \cap K} \abs{Du}^{2(1+q)} \leq C E^{1+q}\,,
\]  
and thus the H\"older inequality produces
\begin{equation}\label{grad_higher_integr}
\int_{A \cap K} \abs{Du}^2 \leq \left( \int_{A \cap K} \abs{Du}^{2(1+q)} \right)^{\frac{2}{1+q}} \abs{A \cap K}^{\frac{q}{1+q}} \leq C E \abs{A \cap K}^{\frac{q}{1+q}}\,. 
\end{equation}

Plugging \eqref{grad_higher_integr} into \eqref{stima_finale_1}, we finally conclude \eqref{e:almgren_estimate}, by possibly choosing a smaller $\gamma > 0$.

\end{proof}

As a corollary of Theorem \ref{thm:almgren_strong_approx} and of Theorem \ref{t:harm_1}, we obtain the following result.

\begin{theorem} \label{thm:final_harm_approx}
Let $\gamma$ be the constant of Theorem \ref{thm:almgren_strong_approx}. Then, for every $\bar{\eta} > 0$ there is a constant $\bar{\eps} > 0$ with the following property. Assume $T$ as in Assumption \ref{ipotesi_base_app1} is area minimizing $\modp$ in $\bC_{4r}(x)$, $E = \bE(T,\bC_{4r}(x)) < \bar{\eps}$ and $r\bA \leq \bar{\eps} E^{\sfrac{1}{2}}$. If $u$ is the map in Theorem \ref{thm:almgren_strong_approx} and we fix good Cartesian coordinates, then there exists a $\Dir$-minimizing $\bar{h} \colon B_{r}(x) \to \mathcal{A}_{Q}(\R^{\bar{n}})$ if $Q < \frac{p}{2}$ or $\bar{h} \colon B_{r}(x) \to \mathscr{A}_{Q}(\R^{\bar{n}})$ if $Q = \frac{p}{2}$ such that $h := (\bar{h}, \Psi(\cdot, \bar{h}))$ satisfies
\begin{equation} \label{e:final_harm_approx}
r^{-2} \int_{B_{r}(x)} \G(u, h)^2 + \int_{B_{r}(x)} (\abs{Du} - \abs{Dh})^{2} + \int_{B_{r}(x)} \abs{D(\bfeta \circ u) - D(\bfeta \circ h)}^2 \leq \bar{\eta} E r^m\,.
\end{equation} 
\end{theorem}

\section{Strong approximation with the nonoriented excess}

In this section we show that it is possible to draw the same conclusions of the previous section replacing the cylindrical excess $\bE (T, \bC_{4r} (x))$ with the nonoriented
$\bE^{no} (T, \bC_{4r} (x))$ defined in \eqref{e:no_excess}. This will be vital, because in the remaining part of the paper we will in fact use mostly the nonoriented excess, which is structurally more suited to the arguments needed in the construction of the center manifold. As discussed in Remark \ref{rmk:excess}, in the classical regularity theory for integral currents the cylindrical excess already possesses the required structural features; see \cite[Remark 2.5]{DLS_Lp}.

\begin{theorem}\label{thm:strong-alm-unoriented}
There exist constants $\eps, \gamma, C > 0$ (depending on $m,\bar{n},n,Q$) with the following property. Let $T$ be as in Assumption \ref{ipotesi_base_app1} in the cylinder $\bC_{4r} (x)$, and assume it is area minimizing $\modp$. Also assume that $E = \bE(T, \bC_{4r}(x)) < \frac{1}{2}$ and that $E^{no} := \bE^{no} (T, \bC_{4r} (x)) \leq \varepsilon$. Then 
\begin{equation}\label{e:o-to-no}
\bE (T, \bC_{2r} (x)) \leq C \bE^{no} (T, \bC_{4r} (x)) + C \bA^2 r^2\, .
\end{equation} 
and in particular all the conclusions of Theorem \ref{thm:almgren_strong_approx} (and of Theorem \ref{thm:final_harm_approx}, provided
$r^2 \bA^2 \leq \bar \varepsilon^2 E \leq \bar\varepsilon^3$ for a suitable $\bar \varepsilon (\bar \eta)>0$)  hold in $B_r (x)$ with estimates where $E^{no}$ replaces $E$.
\end{theorem}

Before coming to the proof we state a simple variant of Theorem \ref{thm:almgren_strong_approx}, where the estimates are inferred in a radius which is just slightly smaller than the starting one. 

\begin{proposition}\label{p:raggio_quasi_uguale}
There are a constant $C \geq 1$ and a $\bar\varepsilon >0$ with the following property. Let $\gamma$ be as in Theorem \ref{thm:almgren_strong_approx}. 
Fix a cylinder $\bC_{4r} (x)$ and a current $T$ which satisfies all the assumptions of Theorem \ref{thm:almgren_strong_approx} with the stronger bound $E := \bE (T, \bC_{4r} (x)) \leq \bar \varepsilon$. Choose $\omega$ such that $(1-\omega m) (1+ \gamma) = 1 + \frac{\gamma}{2}$ and set $\rho = r (1- C (E+ r^2\bA^2)^\omega)$. Then there are a map $u \colon B_{4\rho }(x) \to \Iqs$ if $Q < \frac{p}{2}$, or $u \colon B_{4 \rho}(x) \to \Iqspec$ if $Q = \frac{p}{2}$, and a closed set $K \subset B_{4\rho}(x)$ such that: 
\begin{align}\label{e:inclusion_bis}
&\gr (u) \subset \Sigma\,, \\ \label{e:lip_osc_est_bis}
&\Lip (u)\leq C(E + r^2 \bA^2)^{\gamma/2}\quad \mbox{ and }\quad {\rm osc}\, (u) \leq C \bh (T, \bC_{4r} (x), \pi_0) + C r (E^{\sfrac{1}{2}} + r\bA)\,,\\\label{e:graph_current_bis}
&\bG_u \res (K\times \R^{n})= T\mres (K\times \R^{n}) \,\modp\,,\\ \label{e:volume_estimate_bis}
&|B_{4\rho}(x)\setminus K|\leq \|T\| ((B_{4\rho} (x)\setminus K)\times \R^n) \leq C (E + r^2\bA^2)^{1+\gamma/2} r^m\,,\\ \label{e:Taylor_expansion_bis}
& \Abs{\|T\|(\bC_{4 \sigma \rho}(x)) - Q \omega_m (4 \sigma \rho)^m - \frac{1}{2} \int_{B_{4 \sigma \rho}(x)} \abs{Du}^2} \leq (E + r^2\bA^2)^{1+\gamma/2} r^m \quad \forall\, 0<\sigma<1\, .
\end{align}
\end{proposition}

\begin{proof} For every point $y\in B_{4r (1-(E+ r^2\bA^2)^\omega)} (x)$ and a corresponding cylinder $\bC^y := \bC_{4r(E+r^2\bA^2)^\omega} (y)$, note that
\[
\bE (T, \bC^y) = \frac{\be_T(B_{4r(E+ r^2\bA^2)^\omega} (y) )}{\omega_m\, (4r)^m\, (E+ r^2\bA^2)^{m\omega}} \leq (E+ r^2\bA^2)^{-m\omega}\, \bE(T,\bC_{4r}(x)) \leq E^{1-m\omega}\,.
\]

Thus, by choosing $\bar\varepsilon$ suitably small compared to $\varepsilon$ in Theorem \ref{thm:almgren_strong_approx} we fall under its assumptions. In particular, we find a function $u^y$ defined on the ball $B^y:= B_{r(E+r^2\bA^2)^\omega} (y)$ taking values into either $\Iqs$ or $\Iqspec$ (depending on whether $Q<\frac{p}{2}$ or $Q=\frac{p}{2}$) and a set $K^y$ for which the following conclusions hold: 
\begin{align}\label{e:inclusion_ter}
&\gr (u^y) \subset \Sigma\,, \\ \label{e:lip_osc_est_ter}
&\Lip (u^y)\leq C(E + \bA^2 r^2)^{(1-m\omega) \gamma}\,,\\\label{e:graph_current_ter}
&\bG_{u^y} \res (K^y\times \R^{n})= T\mres (K^y\times \R^{n}) \quad\modp\,,\\ \label{e:volume_estimate_ter}
&|B^y\setminus K^y|\leq \|T\| (B^y\setminus K^y)\times \R^n) \leq C (E + r^2\bA^2)^{(1-m\omega)(1+\gamma)} |B^y|\, .
\end{align}
We now consider the regular lattice $(r(E+ r^2\bA^2)^{\omega})/(\sqrt{m}) \mathbb Z^m$ and for each element $y$ of the lattice contained in $B_{4r (1-(E+ r^2\bA^2)^\omega)} (x)$ we consider the corresponding ball $B^y$. Accordingly, we get a collection $\mathcal{B}$ of balls satisfying the following properties:
\begin{itemize}
\item[(o1)] $\mathcal{B}$ covers $B_{4\rho} (x)$;
\item[(o2)] The cardinality of $\mathcal{B}$ is bounded by $C (E+ r^2\bA^2)^{-m\omega}$ for a geometric constant $C = C (m)$;
\item[(o3)] Each element of $\mathcal{B}$ intersects at most $N$ elements of $\mathcal{B}$ for a geometric constant $N = N(m)$;
\item[(o4)] Every pair $z,w\in B_{4\rho} (x)$ with $|z-w|\leq c(m)  r (E+ r^2\bA^2)^\omega$ is contained in a single ball $B^i$, where $c(m)$ is a positive geometric constant;
\item[(o5)] For each pair $z,w\in B_{4\rho} (x)$ with $\ell := |z-w|\geq c (m) r (E+ r^2\bA^2)^\omega$ there is a chain of balls $B^1, \ldots , B^{\bar{N}}\in \mathcal{B}$ such that
\begin{itemize}
\item[(c1)] $\bar N \leq C\, \ell\,  r^{-1} (E + r^2 \bA^2)^{-\omega}$ for $C = C(m)$;
\item[(c2)] $z\in B^1$ and $w\in B^{\bar N}$;
\item[(c3)] $|B^i \cap B^{i+1}| \geq \bar c(m)\, r^m (E+ r^2\bA^2)^{m\omega}$ for every $i=1,\ldots,\bar N -1$ for a geometric constant $\bar c(m)>0$.
\end{itemize}
\end{itemize}
We now consider for each $B^i = B^{y_i}$ the corresponding sets $\tilde{K}^i := K^{y_i}$ and functions $u^i:= u^{y_i}$. We next define the sets

\[
K^i := \tilde{K}^i \setminus \bigcup_{j\,:\, B^j \cap B^i \neq\emptyset} (B^j \setminus \tilde{K}^j)\, . 
\] 

We then set $K:= \bigcup_i K^i$ and observe that, by (o2), (o3) and \eqref{e:volume_estimate_ter}, we must have
\begin{align}\label{eq:estimate on the bad set}
|B_{4\rho}(x)\setminus K| &\leq \|T\| ((B_{4\rho}(x)                                                                                                                                                                                                                                                                                                                                                                                                                                                                                                                                                                                                                                                                                                                                                                                                                                                                                                                                                                                                                                                                                                                                                                                                                                                                                                                                                                                                                                                                                                                                                                                                                                                                                                                                                                                                                                                                                                                                                                                                                                                                                                                                                                                                                                                                                                                                                                                                                                                                                                                                                                                                                                                                                                                                                                                                                                                                                                                                                                                                                                                                                                                                                                                                                                        \setminus K)\times \R^n) \leq \sum_i \|T\| ((B^i\setminus K^i)\times \R^n)\nonumber\\
 &\leq C \rho^m (E+r^2\bA^2)^{(1-m\omega)(1+\gamma)} = C \rho^m (E+r^2\bA^2)^{1+\gamma/2}\, .
\end{align}
Next, we find a globally defined function $g$ on $K$ by setting $\left.g\right|_{K^i} := \left.u^i\right|_{K^i}$. This function certainly enjoys the estimate $\Lip (\left.g\right|_{K^i}) \leq C (E+r^2\bA^2)^{(1-m\omega) \gamma} \leq C (E+r^2\bA^2)^{\gamma/2}$ on each $K^i$. So, taken two points $z,w\in K$ with $|z-w|\leq c(m)  r  (E+ r^2\bA^2)^{\omega}$ we get, by (o4), the estimate 
\[
\mathcal{G} (g(z), g(w)) \leq C (E+r^2\bA^2)^{\gamma/2} |z-w| \qquad \left(\mbox{resp. } \mathcal{G}_s (g(z), g(w)) \leq C (E+r^2\bA^2)^{\gamma/2} |z-w|\right)\, .
\] 
If $\ell := |z-w| \geq c(m) r (E+ r^2\bA^2)^{\omega}$, we use the chain of balls $B^i$ of (o5) and remark that, thanks to the estimate on $|B^i\setminus K^i|$, we can guarantee the existence of intermediate points $y_i \in K^i\cap K^{i+1}$ towards the estimate
\[
\mathcal{G} (g(z), g(w)) \leq C (E+r^2\bA^2)^{\gamma/2} |z-w| \qquad \left(\mbox{resp. } \mathcal{G}_s (g(z), g(w)) \leq C (E+r^2\bA^2)^{\gamma/2} |z-w|\right)\, .
\] 
This proves that $g$ has the global Lipschitz bound $C (E+r^2\bA^2)^{\gamma/2}$ on $K$. Furthermore, since the graph $\bG_g$ is $\modp$ equivalent to the current $T$ in the cylinder $K \times \R^n$, we have ${\rm osc}(g) \leq C \, \bh(T, \bC_{4r}(x), \pi_0)$, see Remark \ref{rmk:oscillo_ma_non_mollo}. Now we can proceed as in Proposition \ref{p:max} or Theorem \ref{thm:almgren_strong_approx}. More precisely, we write $g = \sum_{i} \llbracket \left( h, \Psi(\cdot, h) \right) \rrbracket$, with $h \colon K \to \A_Q(\R^{\bar n})$ if $Q < \frac{p}{2}$ or $h \colon K \to \mathscr{A}_Q(\R^{\bar n})$ if $Q = \frac{p}{2}$. The map $h$ satisfies $\Lip(h) \leq C (E + r^2 \bA^2)^{\gamma/2}$ and ${\rm osc}(h) \leq C\, \bh(T, \bC_{4r}(x), \pi_0)$. Hence, taking advantage of \cite[Theorem 1.7]{DLS_Qvfr} if $Q < \frac{p}{2}$ or \cite[Corollary 5.3]{DLHMS_linear} when $Q=\frac{p}{2}$, we can extend $h$ to a map $\bar h \colon B_{4\rho}(x) \to \mathcal{A}_Q(\R^n)$ (resp. $\bar h \colon B_{4\rho}(x) \to \Iqspec$) which again satisfies $\Lip(\bar h) \leq C (E + r^2\bA^2)^{\gamma/2}$ and ${\rm osc}(\bar h) \leq C \, \bh(T, \bC_{4r}(x), \pi_0)$. Finally, we set $u := \sum_{i} \llbracket  \bar h, \Psi(\cdot, \bar h) \rrbracket$, thus achieving
\[
\Lip(u) \leq C \, [(E + r^2\bA^2)^{\frac{\gamma}{2}} + \| D\Psi\|_0]\,, \qquad {\rm osc}(u) \leq C \, \bh(T,\bC_{4r}(x),\pi_0) + C\,r\, \|D\Psi\|_0\,.
\]
The estimate in \eqref{e:lip_osc_est_bis} is then a consequence of the choice of coordinates discussed in Remark \ref{rmk:good_coord}.

Finally, the estimate \eqref{e:Taylor_expansion_bis} is a consequence of the other ones, following the argument already given for \eqref{e:Taylor_expansion}. Since \eqref{e:inclusion_bis} and \eqref{e:graph_current_bis} are obvious by construction, this completes the proof.
\end{proof}

\begin{proof}[Proof of Theorem \ref{thm:strong-alm-unoriented}]
First of all we observe that it is enough to prove \eqref{e:o-to-no}. Indeed, if $\varepsilon$ is sufficiently small, from \eqref{e:o-to-no} we conclude that we can apply Theorem \ref{thm:almgren_strong_approx} to any cylinder $\bC_{4 (r/4)} (y)$ with $y\in B_r (x)$. Since $B_r (x)$ can be covered with a finite number $C(m)$ of balls $B_{r/4} (y_i)$ with centers $y_i\in B_r (x)$, the existence of a suitable Lipschitz approximation over $B_r (x)$ follows easily. Theorem \ref{thm:final_harm_approx} can then be concluded by arguing as done for Theorem \ref{t:harm_1}. 

In order to show \eqref{e:o-to-no} we start observing that, by scaling and translating, we can assume $x=0$ and $r=1$. We then argue in several
steps.

\medskip

{\bf Step 1.} First of all we claim that, for every $\delta >0$ there is $\varepsilon$ sufficiently small such that $\bE (T, \bC_3) < \delta$. Otherwise, by contradiction, there would be a sequence $\{T_k\}_{k=1}^{\infty}$ of area minimizing currents $\modp$ satisfying the hypotheses in Assumption \ref{ipotesi_base_app1} in $\bC_4$ together with $\bE(T_k, \bC_4) < \frac12$ for which $\bE^{no} (T_k, \bC_{4}) \to 0$ and $\mass^p (T_k \res \bC_3) \geq (Q+ \delta) \omega_m 3^m$. In particular, because of the uniform bound on the excess, we can assume that $T_k$ converge, up to subsequences, to a $T$ which is an area minimizing current $\modp$ and satisfies Assumption \ref{ipotesi_base_app1}. By convergence of the $\mass^p$ in the interior, we also know that
\begin{equation}\label{e:troppa_massa}
\mass^p (T \res \bC_3) \geq (Q+ \delta) \omega_m 3^m\, .
\end{equation}
On the other hand, since we can assume by Proposition \ref{p:compactness} that $\V(T_k \mres \bC_4) \to \V(T \mres \bC_4)$ as varifolds, and since the nonoriented excess is continuous in the varifold convergence, we must have $\bE^{no} (T, \bC_4) =0$. Moreover, since $T$ is a representative $\modp$ we must have $\|T\| (\bC_4) \leq \omega_m (Q+\frac{1}{2})4^m$ by the hypothesis that $\bE(T_k, \bC_4) < \frac{1}{2}$ for every $k$. The first condition implies that $T$ is supported in a finite number of planes parallel to 
$\pi_0$. By the constancy Lemma \ref{l:constancy} we can assume that $T$ is a sum of integer multiples of $m$-dimensional disks of radius $4$ parallel to $B_{4} (0, \pi_0)$. We thus have that the sum of the moduli of such integers must be at most $Q$. This contradicts \eqref{e:troppa_massa}. 

\medskip

{\bf Step 2}. First of all, if $E:= \bE (T, \bC_3) \leq \bA^2$, then there is nothing to prove. Hence, without loss of generality assume that
\[
E \geq \bA^2\, . 
\]
Now apply Proposition \ref{p:raggio_quasi_uguale} to obtain a Lipschitz map $u: B_{3 - C E^\omega} \to \Iqs$ if $Q < \frac{p}{2}$ and $u:B_{3 - C E^\omega} \to \Iqspec$ if $Q = \frac{p}{2}$, and a closed set $K \subset B_{3 - C E^\omega}(x)$ such that:
\begin{align}
&\Lip (u)\leq C E^{\sfrac{\gamma}{2}},\\ \label{e:graph_current_2}
&\bG_u \res (K\times \R^{n})= T\mres (K\times \R^{n}) \,\modp\,,\\ \label{e:volume_estimate_2}
&|B_{3- C E^\omega}\setminus K|\leq C E^{1+\sfrac{\gamma}{2}} \,,\\ \label{e:Taylor_expansion2}
& \left|\|T\|(\bC_{3- C E^\omega}) - Q \omega_m (3-C E^\omega)^m - \frac{1}{2} \int_{B_{3 - C E^\omega}} \abs{Du}^2\right| \leq C E^{1+\sfrac{\gamma}{2}}\,. 
\end{align}
Now we set $r_1 := 3 -C E^\omega$, $E_1 := \bE (T, \bC_{r_1})$ and we consider the following three alternatives:
\begin{itemize}
\item[(a)] $E_1 \leq \bA^2$;
\item[(b)] $E_1 \geq \max \{\frac{E}{2}, \bA^2\}$;
\item[(c)] $\frac{E}{2} \geq E_1 \geq \bA^2$.
\end{itemize}
In the first case, assuming $\varepsilon$ sufficiently small, since $\bC_2 \subset \bC_{r_1}$, we have concluded our desired estimate \eqref{e:o-to-no}.
In the second case observe first that from the estimates above we easily conclude
\[
\|T\| (\bC_{r_1} \setminus (K\times \mathbb R^n)) \leq C E^{1+\sfrac{\gamma}{2}} \leq C E_1^{1+\sfrac{\gamma}{2}}\, .
\]
Consider now that, using $T\mres K\times \mathbb R^n = \bG_u \res K\times \mathbb R^n$ and standard computations,
we have
\[
\|T\| (K\times \mathbb R^n) - Q|K| =  \frac{1}{2} \int_{K\times \mathbb R^n} |\vec{T} (y) - \pi_0|_{no}^2 \, d\|T\|
\]
We thus can combine these two estimates and claim
\begin{equation}
E_1 = \bE (T, \bC_{r_1}) \leq C E_1^{1+\sfrac{\gamma}{2}}
+ \bE^{no} (T, \bC_{r_1}) \leq \frac{E_1}{2} + C \bE^{no} (T, \bC_4)\, .
\end{equation}
In particular we easily get
\[
\bE (T, \bC_2) \leq C \bE (T, \bC_{r_1}) \leq C \bE^{no} (T, \bC_4)\, ,
\]
and again we have proved \eqref{e:o-to-no}. 

Finally, if we are in case (c) we iterate the step above and get a Lipschitz approximation in the cylinder $\bC_{r_2}$ where $r_2 = 3 - C E^\omega - C E_1^\omega$ and the new excess is $E_2 := \bE (T, \bC_{r_2})$. We keep iterating this procedure which we stop at a certain radius 
\[
r_k = 3 - C \sum_{i=0}^k E_i^\omega\, ,
\]
if either $E_k \leq \bA^2$ or $E_k \geq \frac{E_{k-1}}{2}$. Observe that as long as the procedure does not end we have the recursive property
$E_i \leq \frac{E_{i-1}}{2}$. We can thus estimate 
\[
r_k \geq 3 - C E^\omega \sum_{i=0}^\infty 2^{-\omega\,i} \geq 3 - C \bar C (\omega) E^\omega\, .
\]
Since $\omega$ is a fixed exponent, provided $\delta >E$ is sufficiently small (which from the first step can be achieved by choosing $\varepsilon$ sufficiently small), we have $r_k \geq 2$. Thus, if the procedure stops we have proved \eqref{e:o-to-no}.
If the procedure does not stop, since $E_k\to 0$ we conclude easily that:
\begin{itemize}
\item[(i)] $\bA =0$;
\item[(ii)] If we set $r_\infty:= \lim_{k\to \infty} r_k$, then $2\leq r_\infty$ and $\bE (T, \bC_{r_\infty}) =0$. 
\end{itemize}
This implies that $\|T\| (\bC_{r_\infty}) = Q \omega_m r_\infty^m$. Given that $\p_\sharp T\res \bC_{r_\infty} = Q \a{B_{r_\infty} (0, \pi_0)} \, \modp$, this is
only possible if the current $T$ in $\bC_{r_\infty}$ consists of a finite number of disks parallel to $B_{r_\infty} (0, \pi_0)$ counted with integer multiplicities $\theta_i$ so that $\sum_i |\theta_i|=Q$. In particular, since $2 \leq r_\infty$, obviously $\bE (T, \bC_2) =0 \leq \bE^{no} (T, \bC_4)$,
which shows the validity of \eqref{e:o-to-no} even in this case. 
\end{proof}

\newpage

\part{Center manifold and approximation on its normal bundle} \label{part:Center}
This part of the paper deals with the construction of the center manifold. As it is the case with the proof of the partial regularity result for area minimizing currents in codimension higher than one, one might now attempt a proof of Theorems \ref{t:main2} and \ref{t:main5} carrying on the following program:
\begin{itemize}
\item[(1)] Apply Almgren's strong approximation Theorem \ref{thm:almgren_strong_approx} to construct a sequence of Lipschitz maps $u_k$ approximating $T_{0,r_k}$: here, $r_k$ is the contradiction sequence of radii appearing in Proposition \ref{p:contradiction_sequence}, and the maps $u_k$ take values in $\mathcal{A}_Q(\pi_0^\perp)$ or in $\mathscr{A}_Q(\pi_0^\perp)$ depending on whether $Q < \frac{p}{2}$ or $Q = \frac{p}{2}$, respectively;

\item[(2)] Apply Theorem \ref{thm:final_harm_approx} to show that, after suitable normalization, a subsequence of the $u_k$ converges to a multiple valued map $u_\infty$ minimizing the Dirichlet energy (as in \cite{DLS_Qvfr} if $Q < \frac{p}{2}$ or as in \cite{DLHMS_linear} if $Q = \frac{p}{2}$);

\item[(3)] Use (iii) (resp. (iii)s) in Proposition \ref{p:contradiction_sequence} to infer that $u_\infty$ has a singular set of positive $\Ha^{m-2+\alpha}$ measure (resp. of positive $\Ha^{m-1+\alpha}$ measure), thus contradicting the linear theory in \cite{DLS_Qvfr} if $Q < \frac{p}{2}$ or in \cite{DLHMS_linear} if $Q = \frac{p}{2}$, respectively.

\end{itemize} 

The obstacle towards the success of this program is making point (3) work, namely, showing that the ``large'' singular set of the currents \emph{persists} in the limit as the approximating functions $u_k$ converge to $u_\infty$. As it was just stated, this is false: at this stage, nothing forces $u_\infty$ to actually exhibit any singularities. The center manifold construction is needed precisely to address this issue: when we approximate the current from the center manifold, we ``subtract the regular part'' of the $\Dir$-minimizer in the limit, which in turn allows us to close the contradiction argument.\\

\medskip

In the first section of this part we will outline the arguments and present the statements of the main results. The subsequent sections will then be devoted to the proofs.

\section{Outline and main results}

\subsection{Preliminaries for the construction of the center manifold}

\begin{notazioni}[Distance and nonoriented distance between $m$-planes] Throughout this part, $\pi_0$ continues to denote the plane $\R^m \times \{0\}$, with the standard orientation given by $\vec \pi_0 = e_1 \wedge \ldots \wedge e_{m}$. Given a $k$-dimensional plane $\pi$ in $\R^{m+n}$, we will in fact always identify $\pi$ with a simple unit $k$-vector $\vec \pi = v_1 \wedge \ldots \wedge v_k$ orienting it (thereby making a distinction when the same plane is given opposite orientations). By a slight abuse of notation, given two $k$-planes $\pi_1$ and $\pi_2$, we will sometimes write $\abs{\pi_1 - \pi_2}$ in place of $\abs{\vec \pi_1 - \vec \pi_2}$, where the norm is induced by the standard inner product in $\Lambda_k(\R^{m+n})$. Furthermore, for a given integer rectifiable current $T$, we recall the definition of $\abs{\vec T(y) - \pi_0}_{no}$ from \eqref{e:no_excess_0}. More in general, if $\pi_1$ and $\pi_2$ are two $k$-planes, we can define $\abs{\pi_1 - \pi_2}_{no}$ by 
\[
\abs{\pi_1 - \pi_2}_{no} := \min\left\lbrace \abs{\vec \pi_1 - \vec\pi_2}, \, \abs{\vec\pi_1 + \vec\pi_2} \right\rbrace\,.
\]
It is understood that $\abs{\pi_1 - \pi_2}_{no}$ does not depend on the choice of the orientations $\vec\pi_1$ and $\vec\pi_2$.
\end{notazioni}

\begin{definition}[Excess and height]\label{d:excess_and_height}
Given an integer rectifiable $m$-dimensional current $T$ which is a representative $\modp$ in $\mathbb R^{m+n}$ with finite mass and compact support and an $m$-plane $\pi$, we define the {\em nonoriented excess} of $T$ in the ball $\bB_r (x)$ with respect to the plane $\pi$ as
\begin{align}
\bE^{no} (T,\B_r (x),\pi) &:= \left(2\omega_m\,r^m\right)^{-1}\int_{\B_r (x)} |\vec T - \pi|_{no}^2 \, d\|T\|\,.
\end{align}
The {\it height function} in a set $A \subset \R^{m+n}$ with respect to $\pi$ is
\[
\bh(T,A,\pi) := \sup_{x,y\,\in\,\spt(T)\,\cap\, A} |\p_{\pi^\perp}(x)-\p_{\pi^\perp}(y)|\, .
\]
\end{definition}

\begin{definition}[Optimal planes]\label{d:optimal_planes}
We say that an $m$-dimensional plane $\pi$ {\em optimizes the nonoriented excess} of $T$ in a ball $\B_r (x)$ if
\begin{equation}\label{e:optimal_pi}
\bE^{no} (T,\B_r (x)):=\min_\tau \bE^{no} (T, \B_r (x), \tau) = \bE^{no} (T,\B_r (x),\pi)
\end{equation} 
and if, in addition:
\begin{equation}\label{e:optimal_pi_2}
\mbox{among all other $\pi'$ s.t. \eqref{e:optimal_pi} holds, $|\pi-\pi_0|$ is minimal.}
\end{equation}
Observe that in general the plane optimizing the nonoriented excess is not necessarily unique and $\bh (T, \bB_r (x), \pi)$ might
depend on the optimizer $\pi$. Since for notational purposes it is convenient to define
a unique ``height'' function $\bh (T, \B_r (x))$, we call a plane $\pi$ as in \eqref{e:optimal_pi} and \eqref{e:optimal_pi_2} {\em optimal} if in addition
\begin{equation}\label{e:optimal_pi_3}
\bh(T,\B_r(x)) := \min \big\{\bh(T,\B_r (x),\tau): \tau \mbox{ satisfies \eqref{e:optimal_pi} and
\eqref{e:optimal_pi_2}}\big\} = \bh(T,\B_r(x),\pi)\,,
\end{equation}
i.e. $\pi$ optimizes the height among all planes that optimize the nonoriented excess. However \eqref{e:optimal_pi_3} does not play
any further role apart from simplifying the presentation. 
\end{definition}

\begin{remark}\label{r:why_optimal_2}
Observe that there are two differences with \cite[Definition 1.2]{DLS_Center}: first of all here we consider the nonoriented excess; secondly we have the additional requirement \eqref{e:optimal_pi_2}. In fact the point of \eqref{e:optimal_pi_2} is to ensure that the planes $\pi$ ``optimizing the nonoriented excess'' always satisfy $|\pi - \pi_0| = |\pi - \pi_0|_{no}$.
\end{remark}

We are now ready to formulate the main assumptions of the statements in this section.

\begin{ipotesi}\label{ipotesi}
$\eps_0\in ]0,1]$ is a fixed constant and
$\Sigma \subset \bB_{7\sqrt{m}} \subset \R^{m+n}$ is a $C^{3,\eps_0}$ $(m+\bar{n})$-dimensional submanifold with no boundary in $\bB_{7\sqrt{m}}$. We moreover assume that, for each $q\in \Sigma$,  $\Sigma$ is the graph of a
$C^{3, \eps_0}$ map $\Psi_q: T_q\Sigma\cap \bB_{7\sqrt{m}} \to T_q\Sigma^\perp$. We denote by $\mathbf{c} (\Sigma)$ the
number $\sup_{q\in \Sigma} \|D\Psi_q\|_{C^{2, \eps_0}}$. 
$T^0$ is an $m$-dimensional integer rectifiable current of $\mathbb R^{m+n}$ which is a representative $\modp$ and with support in $\Sigma\cap \bar\bB_{6\sqrt{m}}$. $T^0$ is area-minimizing $\modp$ in $\Sigma$ and moreover
\begin{gather}
\Theta (T^0,0) = Q\quad \mbox{and}\quad \partial T^0 \res \B_{6\sqrt{m}} = 0\quad \modp,\label{e:basic}\\
\quad \|T^0\| (\B_{6\sqrt{m} \rho}) \leq \big(\omega_m Q (6\sqrt{m})^m + \eps_2^2\big)\,\rho^m
\quad \forall \rho\leq 1,\label{e:basic2}\\
\bE^{no}\left(T^0,\B_{6\sqrt{m}}\right)=\bE^{no}\left(T^0,\B_{6\sqrt{m}},\pi_0\right),\label{e:pi0_ottimale}\\
\bmo := \max \left\{\mathbf{c} (\Sigma)^2, \bE^{no}\left(T^0,\B_{6\sqrt{m}}\right)\right\} \leq \eps_2^2 \leq 1\, .\label{e:small ex}
\end{gather}
Here, $Q$ is a positive integer with $2\leq Q \leq \lfloor \frac{p}{2}\rfloor$, and $\eps_2$ is a positive number whose choice will be specified in each subsequent statement.
\end{ipotesi}
Constants depending only upon $m,n,\bar{n}$ and $Q$ will be called geometric and usually denoted by $C_0$.

\begin{remark}\label{r:sigma_A}
Note that
\eqref{e:small ex} implies $\bA := \|A_\Sigma\|_{C^0 (\Sigma)}\leq C_0 \bmo^{\sfrac{1}{2}}$,
where $A_\Sigma$ denotes, as usual, the second fundamental form of $\Sigma$ and $C_0$ is a geometric constant. 
Observe further that for $q\in \Sigma$ the oscillation of $\Psi_q$ is controlled in $T_q \Sigma \cap \bB_{6\sqrt{m}}$ by $C_0 \bmo^{\sfrac{1}{2}}$. 
\end{remark}

In what follows we set $l:= n - \bar{n}$. To avoid discussing domains of definitions it is convenient to extend $\Sigma$ so
that it is an entire graph over all $T_q \Sigma$. Moreover
we will often need to parametrize $\Sigma$ as the graph of a map $\Psi: \mathbb R^{m+\bar n}\to \mathbb R^l$. However
we do not assume that $\mathbb R^{m+\bar n}\times \{0\}$ is tangent to $\Sigma$ at any $q$ and thus we need the
following lemma.

\begin{lemma}\label{l:tecnico3}
There are positive constants $C_0 (m,\bar{n}, n)$ and $c_0 (m, \bar{n}, n)$ such that,
provided $\eps_2 < c_0$, the following holds. If $\Sigma$ is as in Assumption \ref{ipotesi}, then
we can (modify it outside $\bB_{6\sqrt{m}}$ and) extend it to a complete submanifold of $\mathbb R^{m+n}$ which, for every $q\in \Sigma$, is the graph of
a global $C^{3,\eps_0}$ map $\Psi_q : T_q \Sigma \to T_q \Sigma^\perp$ with $\|D \Psi_q\|_{C^{2,\eps_0}}\leq C_0 \bmo^{\sfrac{1}{2}}$.
$T^0$ is still area-minimizing $\modp$ in the extended manifold and in addition
we can apply a global affine isometry which leaves $\mathbb \R^m \times \{0\}$ fixed and maps $\Sigma$ onto
$\Sigma'$ so that
\begin{equation}\label{e:small_tilt}
|\R^{m+\bar{n}}\times \{0\} - T_0 \Sigma'|\leq C_0 \bmo^{\sfrac{1}{2}}\, 
\end{equation}
and $\Sigma'$ is the graph of a $C^{3, \eps_0}$ map $\Psi: \mathbb{R}^{m+\bar{n}} \to \mathbb R^l$ with
$\Psi (0)=0$ and
$\|D\Psi\|_{C^{2, \eps_0}} \leq C_0 \bmo^{\sfrac{1}{2}}$.
\end{lemma}

From now on we assume w.l.o.g. that $\Sigma' = \Sigma$.
The next lemma is a standard consequence of the theory of area-minimizing currents (we include the proofs of Lemma \ref{l:tecnico3} and
Lemma \ref{l:tecnico1} in Section \ref{ss:height} for the reader's convenience).

\begin{lemma}\label{l:tecnico1}
There are positive constants $C_0 (m,n, \bar{n},Q)$ and $c_0 (m,n,\bar{n}, Q)$ with the following property.
If $T^0$ is as in Assumption \ref{ipotesi}, $\eps_2 < c_0$ and $T:= T^0 \res \bB_{23\sqrt{m}/4}$, then:
\begin{align}
\partial T \res \bC_{11\sqrt{m}/2} (0, \pi_0)&= 0 \quad \modp\label{e:geo semplice 0} \\
(\p_{\pi_0})_\sharp T\res \bC_{11 \sqrt{m}/2} (0, \pi_0) &= Q \a{B_{11\sqrt{m}/2} (0, \pi_0)}\quad \modp \label{e:geo semplice 1}\\
\mbox{and}\quad\bh (T, \bC_{5\sqrt{m}} (0, \pi_0)) &\leq C_0 \bmo^{\sfrac{1}{2m}}\, .\label{e:pre_height}
\end{align}
In particular, for each $x\in B_{11\sqrt{m}/2} (0, \pi_0)$ there is  a point
$q\in \spt (T)$ with $\p_{\pi_0} (q)=x$.
\end{lemma}

\subsection{Construction of the center manifold}

From now we will always work with the current $T$ of Lemma \ref{l:tecnico1}.
We specify next some notation which will be recurrent in the paper when dealing with cubes of $\pi_0$.
For each $j\in \N$, $\sC^j$ denotes the family of closed cubes $L$ of $\pi_0$ of the form 
\begin{equation}\label{e:cube_def}
[a_1, a_1+2\ell] \times\ldots  \times [a_m, a_m+ 2\ell] \times \{0\}\subset \pi_0\, ,
\end{equation}
where $2\,\ell = 2^{1-j} =: 2\,\ell (L)$ is the side-length of the cube, 
$a_i\in 2^{1-j}\Z$ $\forall i$ and we require in
addition $-4 \leq a_i \leq a_i+2\ell \leq 4$. 
To avoid cumbersome notation, we will usually drop the factor $\{0\}$ in \eqref{e:cube_def} and treat each cube, its subsets and its points as subsets and elements of $\mathbb R^m$. Thus, for the {\em center $x_L$ of $L$} we will use the notation $x_L=(a_1+\ell, \ldots, a_m+\ell)$, although the precise one is $(a_1+\ell, \ldots, a_m+\ell, 0, \ldots , 0)$.
Next we set $\sC := \bigcup_{j\in \N} \sC^j$. 
If $H$ and $L$ are two cubes in $\sC$ with $H\subset L$, then we call $L$ an {\em ancestor} of $H$ and $H$ a {\em descendant} of $L$. When in addition $\ell (L) = 2\ell (H)$, $H$ is {\em a son} of $L$ and $L$ {\em the father} of $H$.

\begin{definition}\label{e:whitney} A Whitney decomposition of $[-4,4]^m\subset \pi_0$ consists of a closed set $\bGam\subset [-4,4]^m$ and a family $\mathscr{W}\subset \sC$ satisfying the following properties:
\begin{itemize}
\item[(w1)] $\bGam \cup \bigcup_{L\in \mathscr{W}} L = [-4,4]^m$ and $\bGam$ does not intersect any element of $\mathscr{W}$;
\item[(w2)] the interiors of any pair of distinct cubes $L_1, L_2\in \mathscr{W}$ are disjoint;
\item[(w3)] if $L_1, L_2\in \mathscr{W}$ have nonempty intersection, then $\frac{1}{2}\ell (L_1) \leq \ell (L_2) \leq 2\, \ell (L_1)$.
\end{itemize}
\end{definition}

Observe that (w1) - (w3) imply 
\begin{equation}\label{e:separazione}
{\rm sep}\, (\bGam, L) := \inf \{ |x-y|: x\in L, y\in \bGam\} \geq 2\ell (L)  \quad\mbox{for every $L\in \mathscr{W}$.}
\end{equation}
However, we do {\em not} require any inequality of the form 
${\rm sep}\, (\bGam, L) \leq C \ell (L)$, although this would be customary for what is commonly 
called a Whitney decomposition in the literature.

The algorithm for the construction of the center manifold involves several parameters which depend in
a complicated way upon several quantities and estimates. We introduce these parameters and specify some relations among them
in the following

\begin{ipotesi}\label{parametri}
$C_e,C_h,\beta_2,\delta_2, M_0$ are positive real numbers and $N_0$ is a natural number for which we assume
always
%\footnote{{ Q: A number of times later we encounter $\beta_2/4$; why not replace these with $\delta_2$?}
%{ A: because it is just an accident that those constants equal $\delta_2$.  In particular if we choose $\delta_2$ smaller those statements would still hold with $\beta_2/4$.}}
\begin{gather}
\beta_2 = 4\,\delta_2 = \min \left\{\frac{1}{2m}, \frac{\gamma_1}{100}\right\}, \quad
\mbox{where $\gamma_1$ is the exponent in the estimates of Theorem \ref{thm:almgren_strong_approx},}\label{e:delta+beta}\\
M_0 \geq C_0 (m,n,\bar{n},Q) \geq 4\,  \quad
\mbox{and}\quad \sqrt{m} M_0 2^{7-N_0} \leq 1\, . \label{e:N0}
\end{gather}
\end{ipotesi}

As we can see, $\beta_2$ and $\delta_2$ are fixed. The other parameters are not fixed but are subject to further restrictions in the various statements, respecting the following ``hierarchy''. As already mentioned, ``geometric constants'' are assumed to depend only upon $m, n, \bar{n}$ and $Q$. The dependence of other constants upon the various parameters $p_i$ will be highlighted using the notation $C = C (p_1, p_2, \ldots)$.

\begin{ipotesi}[Hierarchy of the parameters]\label{i:parametri}
In all the coming statements:
\begin{itemize}
\item[(a)] $M_0$ is larger than a geometric constant (cf. \eqref{e:N0}) or larger than a costant $C (\delta_2)$, see Proposition~\ref{p:splitting};
\item[(b)] $N_0$ is larger than $C (\beta_2, \delta_2, M_0)$ (see for instance \eqref{e:N0} and Proposition \ref{p:compara});
\item[(c)] $C_e$ is larger than $C(\beta_2, \delta_2, M_0, N_0)$ (see the statements
of Proposition \ref{p:whitney}, Theorem \ref{t:cm} and Proposition \ref{p:splitting});
\item[(d)] $C_h$ is larger than $C(\beta_2, \delta_2, M_0, N_0, C_e)$ (see Propositions~\ref{p:whitney} and \ref{p:separ});
\item[(e)] $\eps_2$ is smaller than $c(\beta_2, \delta_2, M_0, N_0, C_e, C_h)$ (which will always be positive).
\end{itemize}
\end{ipotesi}

The functions $C$ and $c$ will vary in the various statements: 
the hierarchy above guarantees however that there is a choice of the parameters for which {\em all} the restrictions required in the statements of the next propositions are simultaneously satisfied. To simplify our exposition, for smallness conditions on $\eps_2$ as in (e) we will use the sentence ``$\eps_2$ is sufficiently small''.

Thanks to Lemma \ref{l:tecnico1}, for every $L\in \sC$,  we may choose $y_L\in \pi_0^\perp$ so that $p_L := (x_L, y_L)\in \spt (T)$ (recall that $x_L$ is the center of $L$). $y_L$ is in general not unique and we fix an arbitrary choice.
A more correct notation for $p_L$ would be $x_L + y_L$. This would however become rather cumbersome later, when we deal with various decompositions of the
ambient space in triples of orthogonal planes. We thus abuse the notation slightly in using $(x,y)$ instead of $x+y$ and, consistently, $\pi_0\times \pi_0^\perp$ instead of $\pi_0 \oplus \pi_0^\perp$.

\begin{definition}[Refining procedure]\label{d:refining_procedure}
For $L\in \sC$ we set $r_L:= M_0 \sqrt{m} \,\ell (L)$ and 
$\B_L := \bB_{64 r_L} (p_L)$. We next define the families of cubes $\sS\subset\sC$ and $\sW = \sW_e \cup \sW_h \cup \sW_n \subset \sC$ with the convention that
$\sS^j = \sS\cap \sC^j, \sW^j = \sW\cap \sC^j$ and $\sW^j_{\square} = \sW_\square \cap \sC^j$ for $\square = h,n, e$. We define $\sW^i = \sS^i = \emptyset $ for $i < N_0$. We proceed with $j\geq N_0$ inductively: if { no ancestor of $L\in \sC^j$ is in $\sW$}, then 
\begin{itemize}
\item[(EX)] $L\in \sW^j_e$ if $\bE^{no} (T, \B_L) > C_e \bmo\, \ell (L)^{2-2\delta_2}$;
\item[(HT)] $L\in \sW_h^j$ if $L\not \in \mathscr{W}_e^j$ and $\bh (T, \B_L) > C_h \bmo^{\sfrac{1}{2m}} \ell (L)^{1+\beta_2}$;
\item[(NN)] $L\in \sW_n^j$ if $L\not\in \sW_e^j\cup \sW_h^j$ but it intersects an element of $\sW^{j-1}$;
\end{itemize}
if none of the above occurs, then $L\in \sS^j$.
We finally set
\begin{equation}\label{e:bGamma}
\bGam:= [-4,4]^m \setminus \bigcup_{L\in \sW} L = \bigcap_{j\geq N_0} \bigcup_{L\in \sS^j} L.
\end{equation}
\end{definition}
Observe that, if $j>N_0$ and $L\in \sS^j\cup \sW^j$, then necessarily its father belongs to $\sS^{j-1}$.

\begin{proposition}[Whitney decomposition]\label{p:whitney}\label{P:WHITNEY}
Let Assumptions \ref{ipotesi} and \ref{parametri} hold and let $\eps_2$ be sufficiently small.
Then $(\bGam, \mathscr{W})$ is a Whitney decomposition of $[-4,4]^m \subset \pi_0$.
Moreover, for any choice of $M_0$ and $N_0$, there is $C^\star := C^\star (M_0, N_0)$ such that,
if $C_e \geq C^\star$ and $C_h \geq C^\star C_e$, then 
\begin{equation}\label{e:prima_parte}
\sW^{j} = \emptyset \qquad \mbox{for all $j\leq N_0+6$.}
\end{equation}
Finally, the following 
estimates hold with $C = C(\beta_2, \delta_2, M_0, N_0, C_e, C_h)$:
\begin{gather}
\bE^{no} (T, \B_J) \leq C_e \bmo^{}\, \ell (J)^{2-2\delta_2} \quad \text{and}\quad
\bh (T, \B_J) \leq C_h \bmo^{\sfrac{1}{2m}} \ell (J)^{1+\beta_2}
\quad \forall J\in \sS\,, \label{e:ex+ht_ancestors}\\
 \bE^{no} (T, \B_L) \leq C\, \bmo^{}\, \ell (L)^{2-2\delta_2}\quad \text{and}\quad
\bh (T, \B_L) \leq C\, \bmo^{\sfrac{1}{2m}} \ell (L)^{1+\beta_2}
\quad \forall L\in \sW\, . \label{e:ex+ht_whitney}
\end{gather}
\end{proposition}

We will prove Proposition \ref{p:whitney} in Section \ref{s:tilting_whitney}. Next, we fix two important functions $\vartheta,\varrho: \R^m \to \R$.

\begin{ipotesi}\label{mollificatore}
$\varrho\in C^\infty_c (B_1)$ is radial, $\int \varrho =1$ and $\int |x|^2 \varrho (x)\, dx = 0$. For $\lambda>0$ $\varrho_\lambda$ 
denotes, as usual, $x\mapsto \lambda^{-m} \varrho (\frac{x}{\lambda})$.
$\vartheta\in C^\infty_c \big([-\frac{17}{16}, \frac{17}{16}]^m, [0,1]\big)$ is identically $1$ on $[-1,1]^m$.
\end{ipotesi}

$\varrho$ will be used as convolution kernel for smoothing maps $z$ defined on $m$-dimensional planes $\pi$ of $\mathbb R^{m+n}$. In particular, having fixed an isometry $A$ of $\R^m$ onto $\pi$, the smoothing will be given by $[(z \circ A) * \varrho_\lambda] \circ A^{-1}$. Observe that since $\varrho$ is radial, our map does not depend on the choice of the isometry
and we will therefore use the shorthand notation $z*\varrho_\lambda$.

\begin{definition}[$\pi$-approximations]\label{d:pi-approximations}
Let $L\in \sS\cup \sW$ and $\pi$ be an $m$-dimensional plane. If $T\res\bC_{32 r_L} (p_L, \pi)$ fulfills 
the assumptions of Theorem \ref{thm:strong-alm-unoriented} in the cylinder $\bC_{32 r_L} (p_L, \pi)$, then the resulting map $u$ given by the theorem, which is defined on $B_{8r_L} (p_L, \pi)$ and takes values either in $\Iq (\pi^\perp)$ (if $Q < \frac{p}{2}$) or in $\mathscr{A}_Q (\pi^\perp)$ (if $Q=  \frac{p}{2}$) is called a {\em $\pi$-approximation of $T$ in $\bC_{8 r_L} (p_L, \pi)$}. 
The map $\hat{h}:B_{7r_L} (p_L, \pi) \to \pi^\perp$ given
by $\hat{h}:= (\etab\circ u)* \varrho_{\ell (L)}$ will be called the {\em smoothed average of the $\pi$-approximation}. 
\end{definition}

\begin{definition}[Reference plane $\pi_L$] \label{d:ref_plane}
For each $L\in \sS\cup \sW$ we let $\hat\pi_L$ be an optimal plane in $\bB_L$ (cf. Definition \ref{d:optimal_planes}) and choose an $m$-plane $\pi_L\subset T_{p_L} \Sigma$ which minimizes $|\hat\pi_L-\pi_L|$.
\end{definition}

The following lemma, which will be proved in Section \ref{s:tilting_whitney}, deals with graphs of multivalued functions $f$ in several systems of coordinates. 

\begin{lemma}\label{l:tecnico2}
Let the assumptions of Proposition \ref{p:whitney} hold and assume $C_e \geq C^\star$ and $C_h \geq C^\star C_e$ (where $C^\star$ is the
constant of Proposition \ref{p:whitney}). For any choice of the other parameters,
if $\eps_2$ is sufficiently small, then $T\res \bC_{32 r_L} (p_L, \pi_L)$ satisfies the assumptions of Theorem \ref{thm:strong-alm-unoriented}
for any $L\in \sW\cup \sS$. 
Moreover, if $f_L$ is a $\pi_L$-approximation, denote by $\hat{h}_L$ its smoothed average and by $\bar{h}_L$ the map $\p_{T_{p_L}\Sigma} (\hat{h}_L)$,
which takes values in the plane $\varkappa_L := T_{p_L} \Sigma \cap \pi_L^\perp$, i.e. the orthogonal complement of $\pi_L$ in $T_{p_L} \Sigma$.
If we let $h_L$ be the map $x \in B_{7 r_L}(p_L, \pi_L)\mapsto h_L (x):= (\bar{h}_L (x), \Psi_{p_L} (x, \bar{h}_L (x)))\in \varkappa_L \times T_{p_L} \Sigma^\perp$, 
then there is a smooth map $g_L: B_{4r_L} (p_L, \pi_0)\to \pi_0^\perp$ such that 
$\bG_{g_L} = \bG_{h_L}\res \bC_{4r_L} (p_L, \pi_0)$.
\end{lemma}

For the sake of simplicity, in the future we will sometimes regard $g_L$ as a map $g_L \colon B_{4r_L}(x_L,\pi_0) \to \pi_0^\perp$ rather than as a map $g_L \colon B_{4r_L}(p_L,\pi_0) \to \pi_0^\perp$. In particular, we will sometimes consider $g_L(x)$ with $x \in B_{4r_L}(x_L,\pi_0)$ even though the correct writing is the more cumbersome $g_L((x,y_L))$.

\begin{definition}[Interpolating functions]\label{d:glued}
The maps $h_L$ and $g_L$ in Lemma \ref{l:tecnico2} will be called, respectively, the
{\em tilted $L$-interpolating function} and the {\em $L$-interpolating function}.
For each $j$ let $\sP^j := \sS^j \cup \bigcup_{i=N_0}^j \sW^i$ and
for $L\in \sP^j$ define $\vartheta_L (y):= \vartheta (\frac{y-x_L}{\ell (L)})$. Set
\begin{equation}
\hat\varphi_j := \frac{\sum_{L\in \sP^j} \vartheta_L\, g_L}{\sum_{L\in \sP^j} \vartheta_L} \qquad \mbox{on $]-4,4[^m$},
\end{equation}
let $\bar{\varphi}_j (y)$ be the first $\bar{n}$ components of $\hat{\varphi}_j (y)$ and define
$\varphi_j (y) := \big(\bar{\varphi}_j (y), \Psi (y, \bar{\varphi}_j (y))\big)$, where $\Psi$ is the map of Lemma \ref{l:tecnico3}.
$\varphi_j$ will be called the {\em glued interpolation} at the step $j$.
\end{definition}

\begin{theorem}[Existence of the center manifold]\label{t:cm}
Assume that the hypotheses of the Lemma \ref{l:tecnico2} hold and
let $\kappa := \min \{\eps_0/2, \beta_2/4\}$. For any choice of the other parameters,
if $\eps_2$ is sufficiently small, then
\begin{itemize}
\item[(i)] $\|D\varphi_j\|_{C^{2, \kappa}} \leq C \bmo^{\sfrac{1}{2}}$ and $\|\varphi_j\|_{C^0}
\leq C \bmo^{\sfrac{1}{2m}}$, with $C = C(\beta_2, \delta_2, M_0, N_0, C_e, C_h)$.
\item[(ii)] if $L\in \sW^i$ and $H$ is a cube concentric to $L$ with $\ell (H)=\frac{9}{8} \ell (L)$, then $\varphi_j = \varphi_k$ on $H$ for any $j,k\geq i+2$. 
\item[(iii)] $\varphi_j$ converges in $C^3$ to a map $\phii$ and $\cM:= \gr (\phii|_{]-4,4[^m}
)$ is a $C^{3,\kappa}$ submanifold of $\Sigma$.
\end{itemize}
\end{theorem}

\begin{definition}[Whitney regions]\label{d:cm}
The manifold $\cM$ in Theorem \ref{t:cm} is called
{\em a center manifold of $T$ relative to $\pi_0$}, and 
$(\bGam, \sW)$ the {\em Whitney decomposition associated to $\cM$}. 
Setting $\Phii(y) := (y,\phii(y))$, we call
$\Phii (\bGam)$ the {\em contact set}.
Moreover, to each $L\in \sW$ we associate a {\em Whitney region} $\cL$ on $\cM$ as follows:
\begin{itemize}
\item[(WR)] $\cL := \Phii (H\cap [-\frac{7}{2},\frac{7}{2}]^m)$, where $H$ is the cube concentric to $L$ with $\ell (H) = \frac{17}{16} \ell (L)$.
\end{itemize}
\end{definition}

We will present a proof of Theorem \ref{t:cm} in Section \ref{s:cm_construction}

\subsection{The $\mathcal{M}$-normal approximation and related estimates}

In what follows we assume that the conclusions of Theorem \ref{t:cm} apply and denote by $\cM$ the corresponding
center manifold. For any Borel set $\cV\subset \cM$ we will denote 
by $|\cV|$ its $\cH^m$-measure and will write $\int_\cV f$ for the integral of $f$
with respect to $\cH^m \mres \cV$. 
$\cB_r (q)$ denotes the geodesic open balls in $\cM$.

\begin{ipotesi}\label{intorno_proiezione}
We fix the following notation and assumptions.
\begin{itemize}
\item[(U)] $\bU :=\big\{x\in \R^{m+n} : \exists !\, y = \p (x) \in \cM \mbox{ with $|x- y| <1$ and
$(x-y)\perp \cM$}\big\}$.
%\footnote{{ Q: Why not let $\bU$ be the set of points $x\in \mathbb R^{m+n}$ such that {\em either} $x\in \cM$ {\em or} there exists $y\in \cM$ such that $|x-y|<1$ and $\bB_{|x-y| }(x) \cap \cM = \{y\}$.} { A: our definition includes automatically $\cM$ in $\bU$ without distinguishing two cases, provided $\eps_2$ is sufficiently small, see (R). Besides your definition picks more points when $x$ is close to the boundary of $\cM$.}}
\item[(P)] $\p : \bU \to \cM$ is the map defined by (U).
\item[(R)] For any choice of the other parameters, we assume $\eps_2$ to be so small that
$\p$ extends to $C^{2, \kappa}(\bar\bU)$ and
$\p^{-1} (y) = y + \overline{B_1 (0, (T_y \cM)^\perp)}$ for every $y\in \cM$.
\item[(L)] We denote by $\partial_l \bU := \p^{-1} (\partial \cM)$ 
the {\em lateral boundary} of $\bU$.
\end{itemize}
\end{ipotesi}

The following is then a corollary of Theorem \ref{t:cm} and the construction algorithm; see Section \ref{s:normal_approx} for the proof.

\begin{corollary}\label{c:cover}
Under the hypotheses of Theorem \ref{t:cm} and of Assumption \ref{intorno_proiezione}
we have:
\begin{itemize}
\item[(i)] $\spt^p (\partial (T\res \bU)) \subset \partial_l \bU$, 
$\spt (T\res [-\frac{7}{2}, \frac{7}{2}]^m \times \R^n) \subset \bU$, 
and $\p_\sharp (T\res \bU) = Q \a{\cM} \; \modp$;
\item[(ii)] $\spt (\langle T, \p, \Phii (q)\rangle) \subset 
\big\{y\, : |\Phii (q)-y|\leq C \bmo^{\sfrac{1}{2m}} 
\ell (L)^{1+\beta_2}\big\}$ for every $q\in L\in \sW$, where\\
$C= C(\beta_2, \delta_2, M_0, N_0,  C_e, C_h)$;
\item[(iii)]  $\langle T, \p, q\rangle = Q \a{q}$ for every $q\in \Phii (\bGam)$.
\end{itemize}
\end{corollary}

The next main goal is to couple the center manifold of Theorem \ref{t:cm} with a good approximating map defined on it.

\begin{definition}[$\cM$-normal approximation]\label{d:app}
An {\em $\cM$-normal approximation} of $T$ is given by a pair $(\cK, F)$ with the following properties. $\cK\subset \cM$ is closed and contains $\Phii \big(\bGam\cap [-\frac{7}{2}, \frac{7}{2}]^m\big)$. Moreover:
\begin{itemize}
\item[(a)] If $Q =  \frac{p}{2}$, $F$ is a Lipschitz map which takes values in $\mathscr{A}_Q(\R^{m+n})$ and satisfies the requirements of \cite[Assumption 11.1]{DLHMS_linear}. 
\item[(b)] If $Q <  \frac{p}{2}$, $F$ is a Lipschitz map which takes values in $\Iq (\R^{m+n})$ and has the special form 
$F (x) = \sum_i \a{x+N_i (x)}$.
\end{itemize}
In both cases we require that
\begin{itemize}
\item[(A1)] $\spt (\bT_F) \subset \Sigma$;
\item[(A2)] $\bT_F \res \p^{-1} (\cK) = T \res \p^{-1} (\cK)\; \modp$,
\end{itemize}
where $\bT_F$ is the integer rectifiable current induced by $F$; see \cite[Definition 11.2]{DLHMS_linear}. The map $N$ (for the case $Q = \frac{p}{2}$ see \cite[Assumption 11.1]{DLHMS_linear}) is {\em the normal part} of $F$.
\end{definition}

In the definition above it is not required that the map $F$ approximates efficiently the current
outside the set $\Phii \big(\bGam\cap [-\frac{7}{2}, \frac{7}{2}]^m\big)$. However, all the maps constructed
will approximate $T$ with a high degree of accuracy
in each Whitney region: such estimates are detailed
in the next theorem, the proof of which will be tackled in Section \ref{s:normal_approx}.

\begin{theorem}[Local estimates for the $\cM$-normal approximation]\label{t:approx}
Let $\gamma_2 := \frac{\gamma}{4}$, with $\gamma$ the constant
of Theorem \ref{thm:almgren_strong_approx}.
Under the hypotheses of Theorem \ref{t:cm} and Assumption~\ref{intorno_proiezione},
if $\eps_2$ is suitably small (depending upon all other parameters), then
there is an $\cM$-normal approximation $(\cK, F)$ such that
the following estimates hold on every Whitney region $\cL$ associated to
a cube $L\in \sW$, with constants $C = C(\beta_2, \delta_2, M_0, N_0, C_e, C_h)$:
\begin{gather}
\Lip (N|
_\cL) \leq C \bmo^{\gamma_2} \ell (L)^{\gamma_2} \quad\mbox{and}\quad  \|N|
_\cL\|_{C^0}\leq C \bmo^{\sfrac{1}{2m}} \ell (L)^{1+\beta_2},\label{e:Lip_regional}\\
|\cL\setminus \cK| + \|\bT_F - T\|_p (\p^{-1} (\cL)) \leq C \bmo^{1+\gamma_2} \ell (L)^{m+2+\gamma_2},\label{e:err_regional}\\
\int_{\cL} |DN|^2 \leq C \bmo \,\ell (L)^{m+2-2\delta_2}\, .\label{e:Dir_regional}
\end{gather}
Moreover, for any $a>0$ and any Borel $\cV\subset \cL$, we have (for $C=C(\beta_2, \delta_2, M_0, N_0, C_e, C_h)$)
\begin{equation}\label{e:av_region}
\int_\cV |\etab\circ N| \leq 
{ C \bmo \left(\ell (L)^{m+3+\sfrac{\beta_2}{3}} + a\,\ell (L)^{2+\sfrac{\gamma_2}{2}}|\cV|\right)}  + \frac{C}{a} 
\int_\cV \cG_{\square} \big(N, Q \a{\etab\circ N}\big)^{2+\gamma_2}\, ,
\end{equation} 
where $\square =s$ in case $p=2Q$, and it is empty otherwise.
\end{theorem}

From \eqref{e:Lip_regional} - \eqref{e:Dir_regional} it is not difficult to infer analogous ``global versions'' of the estimates.

\begin{corollary}[Global estimates]\label{c:globali} Let $\cM'$ be
the domain $\Phii \big([-\frac{7}{2}, \frac{7}{2}]^m\big)$ and $N$ the map of Theorem \ref{t:approx}. Then,  (again with $C = C(\beta_2, \delta_2, M_0, N_0, C_e, C_h)$)
\begin{gather}
\Lip (N|_{\cM'}) \leq C \bmo^{\gamma_2} \quad\mbox{and}\quad \|N|_{\cM'}\|_{C^0}
\leq C \bmo^{\sfrac{1}{2m}},\label{e:global_Lip}\\ 
|\cM'\setminus \cK| + \|\bT_F - T\|_p (\p^{-1} (\cM')) \leq C \bmo^{1+\gamma_2},\label{e:global_masserr}\\
\int_{\cM'} |DN|^2 \leq C \bmo\, .\label{e:global_Dir}
\end{gather}
\end{corollary}

\subsection{Separation and domains of influence of large excess cubes}

We now analyze more in detail the consequences of the various stopping conditions for the cubes in $\sW$. 
We first deal with $L\in \sW_h$.

\begin{proposition}[Separation]\label{p:separ}
There is a constant $C^\sharp (M_0) > 0$ with the following property.
Assume the hypotheses of Theorem \ref{t:approx} and in addition
$C_h^{2m} \geq C^\sharp C_e$. 
If $\eps_2$ is sufficiently small, then the following conclusions hold for every $L\in \sW_h$:
\begin{itemize}
\item[(S1)] $\Theta (T, q) \leq Q - \frac{1}{2}$ for every $q\in \B_{16 r_L} (p_L)$;
\item[(S2)] $L\cap H= \emptyset$ for every $H\in \sW_n$
with $\ell (H) \leq \frac{1}{2} \ell (L)$;
\item[(S3)] $\cG_{\square} \big(N (x), Q \a{\etab \circ N (x)}\big) \geq \frac{1}{4} C_h \bmo^{\sfrac{1}{2m}}
\ell (L)^{1+\beta_2}$  for every $x\in \Phii (B_{2 \sqrt{m} \ell (L)} (x_L, \pi_0))$, where $\square =s$ if $p=2Q$ or $\square = \;$ otherwise.
\end{itemize}
\end{proposition}

A simple corollary of the previous proposition is the following.

\begin{corollary}\label{c:domains}
Given any $H\in \sW_n$ there is a chain $L =L_0, L_1, \ldots, L_j = H$ such that:
\begin{itemize}
\item[(a)] $L_0\in \sW_e$ and $L_i\in \sW_n$ for all $i>0$; 
\item[(b)] $L_i\cap L_{i-1}\neq\emptyset$ and $\ell (L_i) = \frac{1}{2} \ell (L_{i-1})$ for all $i>0$.
\end{itemize}
In particular,  $H\subset B_{3\sqrt{m}\ell (L)} (x_L, \pi_0)$.
\end{corollary}

We use this last corollary to partition $\sW_n$.

\begin{definition}[Domains of influence]\label{d:domains}
We first fix an ordering of the cubes in $\sW_e$ as $\{J_i\}_{i\in \mathbb N}$ so that their sidelengths do not increase. Then $H\in \sW_n$
belongs to $\sW_n (J_0)$ (the domain of influence of $J_0$) if there is a chain as in Corollary \ref{c:domains} with $L_0 = J_0$.
Inductively, $\sW_n (J_r)$ is the set of cubes $H\in \sW_n \setminus \cup_{i<r} \sW_n (J_i)$ for which there is
a chain as in Corollary \ref{c:domains} with $L_0 = J_r$.
\end{definition}

\subsection{Splitting before tilting}

The following proposition contains a ``typical'' splitting-before-tilting phenomenon: the key assumption of the 
theorem (i.e. $L\in \sW_e$) is that the excess does not decay at some given scale (``tilting'') and the main conclusion \eqref{e:split_2} implies a certain amount of separation between the sheets of the current (``splitting''); see Section \ref{s:splitting} for the proof.

\begin{proposition}(Splitting I)\label{p:splitting}
There are functions $C_1 (\delta_2), C_2 (M_0, \delta_2)$ such that, if $M_0 \geq C_1 ( \delta_2)$, $C_e \geq C_2 (M_0, \delta_2)$, if
the hypotheses of Theorem~\ref{t:approx} hold and if $\eps_2$ is chosen sufficiently small,
then the following holds. If $L\in \sW_e$, $q\in \pi_0$ with $\dist (L, q) \leq 4\sqrt{m} \,\ell (L)$ and $\Omega = \Phii (B_{\ell (L)/4} (q, \pi_0))$, then (with $C, C_3 = C(\beta_2, \delta_2, M_0, N_0, C_e, C_h)$):
\begin{align}
&C_e \bmo \ell(L)^{m+2-2\delta_2} \leq \ell (L)^m \bE^{no} (T, \B_L) \leq C \int_\Omega |DN|^2\, ,\label{e:split_1}\\
&\int_{\cL} |DN|^2 \leq C \ell (L)^m \bE^{no} (T, \B_L) \leq C_3 \ell (L)^{-2} \int_\Omega |N|^2\, . \label{e:split_2}
\end{align}
\end{proposition}

\subsection{Persistence of multiplicity $Q$ points}

We next state two important properties triggered by the existence of $q\in \spt (T)$ with $\Theta (T,q)=Q$,
both related to the splitting before tilting. Their proofs will be discussed in Section \ref{s:persistence}.

\begin{proposition}(Splitting II)\label{p:splitting_II}
Let the hypotheses of Theorem~\ref{t:cm} hold and assume $\eps_2$ is sufficiently small. For any
$\alpha, \bar{\alpha}, \hat\alpha >0$, there is $\eps_3 = \eps_3 (\alpha, \bar{\alpha}, \hat\alpha, \beta_2, \delta_2, M_0, N_0, C_e, C_h) >0$ as follows. 

When $Q<\frac{p}{2}$, if for some $s\leq 1$
\begin{equation}\label{e:supL}
\sup \big\{\ell (L): L\in \sW, L\cap B_{3s} (0, \pi_0) \neq \emptyset\big\} \leq s\, ,
\end{equation}
\begin{equation}\label{e:many_Q_points}
\cH^{m-2+\alpha}_\infty \big(\{\Theta (T, \cdot) = Q\}\cap \B_{s}\big) \geq \bar{\alpha} s^{m-2+\alpha},
\end{equation}
and $\min \big\{s, \bmo\big\}\leq\eps_3$, then,
\[
\sup \big\{ \ell (L): L\in \sW_e \mbox{ and } L\cap B_{19 s/16} (0, \pi_0)\neq \emptyset\big\} 
\leq \hat{\alpha} s\, .
\]

When $Q=\frac{p}{2}$, the same conclusion can be reached if \eqref{e:many_Q_points} is replaced by
\begin{equation}\label{e:many_Q_points_mannaggia}
\cH^{m-1+\alpha}_\infty \big(\{\Theta (T, \cdot) = Q\}\cap \B_{s}\big) \geq \bar{\alpha} s^{m-1+\alpha}\, .
\end{equation}
\end{proposition}

\begin{proposition}(Persistence of $Q$-points)\label{p:persistence}
Assume the hypotheses of Proposition \ref{p:splitting} hold.
For every $\eta_2>0$ there are $\bar{s}, \bar{\ell} > 0$, depending upon $\eta_2, \beta_2, \delta_2, M_0, N_0, C_e$ and $C_h$, such that,
if $\eps_2$ is sufficiently small, then the following property holds. 
If $L\in \sW_e$, $\ell (L)\leq \bar\ell$, $\Theta (T, q) = Q$ and
$\dist (\p_{\pi_0} (\p (q)), L) \leq 4 \sqrt{m} \,\ell (L)$, then
\begin{equation}\label{e:persistence}
\mint_{\cB_{\bar{s} \ell (L)} (\p (q))} \cG_{\square} \big(N, Q \a{\etab\circ N}\big)^2 \leq \frac{\eta_2}{\ell(L)^{m-2}}
\int_{\cB_{\ell (L)} (\p (q))} |DN|^2\, ,
\end{equation}
where $\square =s$ if $p=2Q$ or $\square = \;$ otherwise.
\end{proposition}

\subsection{Comparison between center manifolds}

We list here a final key consequence of the splitting before tilting phenomenon. $\iota_{0,r}$ denotes the map $z\mapsto \frac{z}{r}$.

\begin{proposition}[Comparing center manifolds]\label{p:compara}
There is a geometric constant $C_0$ and a function $\bar{c}_s (\beta_2, \delta_2, M_0, N_0, C_e, C_h) >0$ with
the following property. Assume the hypotheses of Proposition \ref{p:splitting}, $N_0 \geq C_0$, $c_s := \frac{1}{64\sqrt{m}}$
and $\eps_2$ is sufficiently small. If for some $r\in ]0,1[$:
\begin{itemize}
\item[(a)] $\ell (L) \leq  c_s \rho$ for every $\rho> r$ and every
$L\in \sW$ with $L\cap B_\rho (0, \pi_0)\neq \emptyset$;
\item[(b)] $\bE^{no} (T, \B_{6\sqrt{m} \rho}) < \eps_2$ for every $\rho>r$;
\item[(c)] there is $L\in \sW$ such that $\ell (L) \geq c_s r $ and $L\cap \bar B_r (0, \pi_0)\neq\emptyset$;
\end{itemize}
then
\begin{itemize}
\item[(i)] the current $T':= (\iota_{0,r})_\sharp T \res \B_{6\sqrt{m}}$ and the submanifold 
$\Sigma':= \iota_{0,r} (\Sigma)\cap \bB_{7\sqrt{m}}$ satisfy the assumptions
of Theorem \ref{t:approx} for
some plane $\pi$ in place of $\pi_0$;
\item[(ii)] for the center manifold $\cM'$ of $T'$ relative to $\pi$ and the 
$\cM'$-normal approximation $N'$ as in Theorem \ref{t:approx}, we have
\begin{equation}\label{e:restart}
\int_{\cM'\cap \bB_2} |N'|^2 \geq \bar{c}_s \max
\big\{\bE^{no} (T', \bB_{6\sqrt{m}}), \mathbf{c} (\Sigma')^2\big\}\, . 
\end{equation}
\end{itemize}
\end{proposition}

\section{Height bound and first technical lemmas}\label{ss:height}

We can now discuss the proofs of the main results outlined in the previous section. We begin with a $\modp$ version of the sheeting lemma appearing in \cite[Theorem A.1]{DLS_Center}. 

\begin{theorem}\label{t:height_bound} Let $p$, $Q$, $m$, $\bar{n}$ and $n$ be positive integers, with $Q \leq \frac{p}{2}$. Then there are
$\eps (Q,m,p,\bar{n}, n)>0$, $\omega (Q, m, p, \bar{n}, n)>0$, and $C_0 (Q, m, \bar{n}, n)$ with the following property. For $r>0$ and $\bC = \bC_{r} (x_0) = \bC_{r}(x_0, \pi_0)$ assume:
\begin{itemize}
\item[(h1)] $\Sigma$ and $T$ are as in Assumption \ref{ass:main};
\item[(h2)] $\partial T\res \bC = 0 \; \modp$, $(\p_{\pi_0})_\sharp T
\res \bC = Q\a{B_r (\p_{\pi_0} (x_0), \pi_0)} \; \modp$,
and ${E:= \bE (T, \bC) < \eps}$.
%, where $\pi_0 = \R^m\times \{0\}$ and $\bC_r (x_0) = \bC_r (x_0, \pi_0)$.
\end{itemize}
Then there are $k\in \N\setminus\{0\}$, points $\{y_1, \ldots, y_k\}\subset \R^{n}$ and integers $Q_1, \ldots, Q_k$ such that:
\begin{itemize}
\item[(i)] having set $\sigma:= C_0 (E + \bA^2)^{\frac{1}{2m}}$ and $\rho:= r (1- 2(E+\bA^2)^\omega)$, the open sets 
\[
\bS_i := \R^m \times (y_i +\, ]-r \sigma, r \sigma[^n)
\]
are pairwise disjoint and 
\[
\spt (T)\cap \bC_\rho (x_0) \subset \bigcup_i \bS_i\, ;
\]
\item[(ii)] $(\p_{\pi_0})_\sharp [T \res (\bC_\rho (x_0) \cap \bS_i)] =Q_i  \a{B_\rho (\p_{\pi_0} (x_0), \pi_0)}$ $\modp$ $\forall i\in \{1, \ldots , k\}$, with $Q_i \in \mathbb Z$. When $Q<\frac{p}{2}$ all $Q_i$ must be positive, whereas for $Q=\frac{p}{2}$ either they are all positive or they are all negative; in any case, $\sum_i \abs{Q_i} = Q$;
\item[(iii)] for every $q\in \spt (T)\cap \bC_\rho (x_0)$ we have $\Theta (T, q) < \max_i  |Q_i| + \frac{1}{2}$.
\end{itemize}
If we keep the same assumptions with $E$ replaced by $E^{no} := \bE^{no} (T, \bC)$, the conclusions hold if we set $\rho := r (1-\eta - 2(E+\bA^2)^\omega)$, where $\eta>0$ is any fixed constant (in turn $\varepsilon$ will depend also on $\eta$).  
\end{theorem}

\begin{remark}
The proof that we are going to present is substantially different from the one in \cite[Theorem A.1]{DLS_Center}, and it could be easily adapted to the case of area minimizing integral currents as well. The statement above is sufficient for our purposes; nonetheless, the proof is actually going to give us more. In particular, in dimension $m \geq 3$ the result holds with a better estimate on the bandwidth of the various $\bS_i$, namely with $\sigma = C_0\, (E+\bA^2)^{\frac1m}$ in place of $\sigma = C_0\, (E+\bA^2)^{\frac{1}{2m}}$. In dimension $m=2$, the proof below also produces the height bound with the optimal estimate featuring $\sigma = {\rm O}(E^{\sfrac12})$, but only in the cylinder $C_{\frac{r}{2}}(x_0)$.
\end{remark}

\begin{proof} 
In the rest of the proof we denote by $\p$ the orthogonal projection onto $\pi_0 = \mathbb R^m \times \{0\}$. The last part of the statement, where $E$ is replaced with $E^{no}$ follows from Theorem \ref{thm:strong-alm-unoriented}. Moreover, we assume $x_0=0$ and $r=1$ after appropriate translation and rescaling. We also observe, as in the proof of \cite[Theorem A.1]{DLS_Center} that (iii) is a corollary of the interior monotonicity formula
(the only ingredients of the argument in there are the stationarity of the varifold induced by $T_i := T \res (\bC_{\rho} \cap \bS_i)$ and the inequality $\mass (T_i) \leq \omega_m \, \rho^m (\abs{Q_i} + E)$).

We therefore focus on (i) and (ii) and since the case $Q<\frac{p}{2}$ is entirely analogous, for the sake of simplicity we assume $Q=\frac{p}{2}$. We first prove (i). We start by considering an approximation as in Proposition \ref{p:raggio_quasi_uguale}. We thus find an exponent $\omega>0$ (which depends only on $Q,m$ and $n$), a Lipschitz map $u: B_{ 1- (E+\bA^2)^\omega} \to \Iqspec$ and a $K\subset B_{1-(E +\bA^2)^\omega}$ with the following properties:
\begin{itemize}
\item[(i)] $\Lip (u) \leq C\, (E+\bA^2)^\omega$;
\item[(ii)] $\bG_u \res K\times \mathbb R^n = T\res K\times \mathbb R^n \;  \modp$;
\item[(iii)] $\|T\| ((B_{1-(E +\bA^2)^\omega}\setminus K)\times \mathbb R^n)\leq C (E+\bA^2)^{1+\omega}$.
\end{itemize}
We consider first the case $m >2$. Recall the Poincar\'e inequality and find a point $T_0 \in \Iqspec$ such that
\begin{equation}\label{e:Jonas_trick}
\left(\int_{B_{1- (E+\bA^2)^\omega}} \cG_s (T_0, u (x))^{2^*}\, dx\right)^{1/2^*} \leq C \|Du\|_{L^2 (B_{1-(E+\bA^2)^\omega})} \leq C (E+\bA^2)^{\frac{1}{2}}\, ,
\end{equation}
where $2^*= \frac{2m}{m-2}$. 
Define next the set $K^*:= \{x\in B_{1-(E+\bA^2)^\omega}: \cG_s (u(x), T_0)\leq \bar{C} (E+\bA^2)^{\frac{1}{m}}\}$, where $\bar C$ is a constant which will be later chosen sufficiently large. Using \eqref{e:Jonas_trick} and Chebyshev's inequality, we easily conclude
\begin{equation}\label{e:Jonas_trick2}
|B_{1-(E+\bA^2)^\omega}\setminus K^*|\, \bar C^{\frac{2m}{m-2}}\, (E+\bA^2)^{\frac{2}{m-2}} \leq C\, (E +\bA^2)^{\frac{m}{m-2}}\, .
\end{equation}
In particular, for any fixed $\bar\eta$, if $\bar C$ is chosen large enough, we reach the estimate
\begin{equation}\label{e:Jonas_trick3}
|B_{1-(E+\bA^2)^\omega}\setminus K^*| \leq \bar \eta (E+\bA^2)\, .
\end{equation}
Consider now the set $\bar{K} := K\cap K^*$ and observe that, by choosing $\varepsilon$ sufficiently small, we reach
\begin{equation}\label{e:Jonas_trick4}
\|T\| ((B_{1-(E+\bA^2)^\omega}\setminus \bar K)\times \mathbb R^n) \leq  2\bar \eta (E+\bA^2)\, .
\end{equation}
To fix ideas assume now that $T_0 = (\sum_{j=1}^J k_j \a{p_j}, 1)$, where the $p_j$'s are pairwise distinct and all $k_j$ are positive. Let $\spt (T_0) = \{p_1, \ldots p_J\}$. From (ii) and the definition of $\bar K$, it follows easily that $\dist (\spt (T_0), \p^\perp (\spt (\langle T, \p, x)) \leq \bar C (E+\bA^2)^{\frac{1}{m}}$ for $x \in \bar K$. Define thus the sets
$\bU :=  \bigcup_j \{(x,y): |y- p_j|\leq \bar C  (E+\bA^2)^{\frac{1}{m}}\}$ and $\bU' := \bigcup_j  \{(x,y): |y- p_j|\leq 
(\bar C+1)  (E+\bA^2)^{\frac{1}{m}}\}$, then 
\begin{equation}\label{e:Jonas_trick5}
\|T\| (\bC_{1- (E+\bA^2)^\omega} \setminus \bU)
\leq \|T\| ((B_{1-(E+\bA^2)^\omega}\setminus \bar K)\times \mathbb R^n) 
\leq  2\bar \eta (E+\bA^2)\, .
\end{equation}
If $q\in \bC_{1- 2(E+\bA^2)^\omega}\setminus \bU'$, then
$\bB_{(E+\bA^2)^\frac{1}{m}} (q) \subset \bC_{1- (E+\bA^2)^\omega}  \setminus \bU$ (we are imposing here $\omega \leq \frac{1}{m}$), and by the monotonicity formula $\|T\| (\bB_{(E+\bA^2)^\frac{1}{m}} (q)) \geq c_0 (E+\bA^2)$, where $c_0$ is a geometric constant. This is however incompatible with \eqref{e:Jonas_trick5} as soon as $2\bar\eta$ is chosen smaller than $c_0$, thus showing that $\spt(T) \cap \bC_{1-2\,(E + \bA^2)^{\omega}} \subset \bU'$. We can now subdivide $\bU'$ in a finite number of disjoint stripes $\bS_i$ of width $\tilde{C} (E+\bA^2)^{\frac{1}{m}}$, where $\tilde{C}$ is larger than $\bar C$ by a factor which depends only on $Q$. This shows therefore the claim (i) of the theorem when $m>2$.

The case $m=2$ is slightly more subtle. Observe first that $|Du|^2\leq \min\{{\mathbf{m}}_c \mathbf{e},1\}$ and hence we can use the same argument as in the proof of Theorem \ref{thm:gradient Lp estimate} to achieve
\begin{equation} \label{e:Jonas_trick6}
\int_K |Du|^{2 (1+q)} \leq C E^{1+q-\omega}\, .
\end{equation}
The subtlety is in losing at most $(E+\bA^2)^\omega$ in the radius of the ball; as usual, the price to pay is a slightly worse estimate, cf. \eqref{e:Jonas_trick6} with \eqref{eq:gradient Lp estimate}. Since $|B_{1-(E+\bA^2)^\omega}\setminus K|\leq E^{1+\omega}$, if we choose $q$ small enough we easily reach the estimate 
\[
\|Du\|_{L^{2+2q} (B_{1-(E+\bA^2)^\omega})}\leq C E^{\frac{1}{4}}\, .
\]
In particular, if we set in this case $K^*:= \{x\in B_{1-(E+\bA^2)^\omega}: \cG_s (u(x), T_0)\leq \bar{C} (E+\bA^2)^{\frac{1}{4}}\}$ then from Morrey's embedding follows that $K^* = B_{1-(E+\bA^2)^\omega}$, provided $\bar C$ is chosen large enough. \eqref{e:Jonas_trick3} is thus trivially true and the rest of the argument remains unchanged. 

We now come to claim (ii). By the constancy theorem, it is easy to see that 
\[
\p_\sharp (T\res \bC_{1- 2(E+\bA^2)^\omega}\cap \bS_i) = Q_i \a{B_{1-2(E+\bA^2)^\omega}} \qquad \modp\, ,
\]
for some integer $Q_i \in \{-(Q-1), \ldots , -1, 0, 1, \ldots , Q\}$. However, recall that for $x\in \bar{K}$: 
\begin{itemize}
\item the support $S$ of the current $Z_i (x) := \langle T, \p, x\rangle \res \bC_{1- 2(E+\bA^2)^\omega}\cap \bS_i$ consists of at most $Q$ points;
\item either all points in $S$ have positive integer multiplicity, or they all have negative integer multiplicity;
\item $\mass (Z_i (x)) \leq Q$.
\end{itemize}
We thus conclude that each $Q_i$ is nonzero and that $|Q_i|= \mass (Z_i (x))$. Now, since $\mass (\langle T, \p, x\rangle) = Q$, we must have $\sum |Q_i| = Q$. On the other hand
\[
\sum_i \p_\sharp (T\res \bC_{1- 2(E+\bA^2)^\omega}\cap \bS_i)
= p_\sharp (T\res \bC_{1-2 (E+\bA^2)}) = Q\a{B_{1-2(E+\bA^2)^\omega}}\qquad \modp\, .
\]
Hence $\sum_i Q_i = Q\; \modp$. Hence we conclude that either all $Q_i$'s are positive or they are all negative. 
\end{proof}

Before coming to the proofs of the Lemmas \ref{l:tecnico3} and \eqref{l:tecnico2}, we wish to make the following elementary remark, which will be used throughout the rest of the paper: 

\begin{proposition}\label{p:or_nonor}
There are dimensional constants $\varepsilon (m, n)>0$ and $C(m,  n)>0$ with the following property.
Consider an oriented $m$-dimensional plane $\pi\subset \mathbb R^{m+n}$ and an oriented 
$(m+d)$-dimensional plane $\Pi\subset \mathbb R^{m+n}$, where $d\in \{0, \ldots , n\}$. Let $\pi'\subset \Pi$ be an oriented $m$-dimensional  plane  for which $|\pi - \pi'| = \min_{\tau\subset \Pi} |\pi - \tau|$, and assume$|\pi - \pi'| < \varepsilon$. Then 
\[
|\pi - \p_{\Pi} (\pi)|_{no} = |\pi - \p_{\Pi} (\pi)| \leq C |\pi - \pi'|\, .
\]
In particular: 
\begin{itemize}
\item[(Eq)] if $\alpha$ and $\beta$ are $m$-dimensional oriented planes of $\mathbb R^{m+n}$ for which $|\alpha - \beta|$ is smaller than a positive geometric constant, then $|\alpha - \beta|_{no} = |\alpha - \beta|$.
\end{itemize}
\end{proposition}

The proposition is a simple geometric observation, and its proof is left to the reader. 

\begin{proof}[Proof of Lemma \ref{l:tecnico3}]
The argument given in \cite[Section 4]{DLS_Center} of \cite[Lemma 1.5]{DLS_Center} for the existence of the global extension of $\Sigma$ and the minimality of $T^0$ in the extended manifold works in our case as well, with straightforward modifications. 

We now come to the proof of \eqref{e:small_tilt}, which again follows that given in \cite[Section 4]{DLS_Center} of \cite[Lemma 1.5]{DLS_Center}, but needs some extra care. First of all, by Assumption \ref{ipotesi} and Remark \ref{r:sigma_A}, $\bA \leq C_0 \bmo^{\sfrac{1}{2}} \leq C_0$. Then,
by the monotonicity formula, $\|T^0\| (\bB_1) \geq c_0>0$ and so there is $q\in \spt (T^0)\cap \bB_1$ such that
\[
|\vec{T^0} (q) - \pi_0|_{no}^2  \leq C_0 \frac{\bE^{no} (T^0, \bB_1, \pi_0)}{\|T^0\| (\bB_1)} \leq C_0 \bmo\, .
\]
Now, both $\vec{T^0} (q)$ and $-\vec{T^0} (q)$ orient a plane contained in $T_q \Sigma$. We can thus apply Proposition \ref{p:or_nonor} provided $\bmo$ is sufficiently small. From it we conclude that $\p_{T_q \Sigma} (\pi_0)$ is an $m$-dimensional plane with
$|\p_{T_q \Sigma} (\pi_0) - \pi_0| \leq C_0 \bmo^{\sfrac{1}{2}}$. From this inequality we then conclude following literally the final arguments of  \cite[Proof of Lemma 1.5]{DLS_Center}.
\end{proof}

\begin{proof}[Proof of Lemma \ref{l:tecnico1}]
We follow the proof of \cite[Lemma 1.6]{DLS_Center} given in \cite[Section 4]{DLS_Center}. First of all we notice that, once \eqref{e:geo semplice 0} and \eqref{e:geo semplice 1} are established, \eqref{e:pre_height} follows from Theorem \ref{t:height_bound}, since we clearly have that $\bE^{no} (T, \bC_{11\sqrt{m}/2}, \pi_0) \leq C \bE^{no} (T^0, \bB_{6\sqrt{m}}, \pi_0)$.  Moreover, recall that there is a set of full measure $A\subset B_{5\sqrt{m}}$ such that $\langle T, \p_{\pi_0}, x\rangle$ is an integer rectifiable current for every $x\in A$. For any such $x$ we have thus $\langle T, \p_{\pi_0}, x\rangle = \sum_i^{J(x)} k_i (x)\a{p_i}$ where $p_1, \ldots, p_{J (x)}$ is a finite collection of points and each $k_i (x)$ is an integer. In particular we must have $\sum_i k_i (x) = Q\;\modp$ and since $1\leq Q \leq \frac{p}{2}$, at least one $k_i (x)$ must be nonzero, which means in turn that $\spt (\langle T, \p_{\pi_0}, x\rangle) \neq \emptyset$. Hence we conclude that $\spt (T) \cap \p_{\pi_0}
 ^{-1} (x) \neq \emptyset$ for every $x\in A$, and by the density of $A$ we conclude that  $\spt (T) \cap \p_{\pi_0}^{-1} (x) \neq \emptyset$ for every $x\in B_{5\sqrt{m}}$. 

We next come to \eqref{e:geo semplice 0} and \eqref{e:geo semplice 1}. As in the proof of \cite[Lemma 1.6]{DLS_Center}, we argue by contradiction and assume that one among \eqref{e:geo semplice 0} and \eqref{e:geo semplice 1} fails for a sequence $T^0_k$ of currents which satisfy Assumption \ref{ipotesi} with $\eps_2 = \varepsilon_2 (k)\downarrow 0$. The compactness property given by Proposition \ref{p:compactness} ensures the existence of a subsequence, not relabeled, converging to a current $T^0_\infty$ in the $\mathscr{F}_K^p$ norm for every compact $K\subset\bB_{6\sqrt{m}}$: in fact Proposition \ref{p:compactness} ensures also that $T_\infty^0$ is area minimizing $\modp$ in a suitable $(m+\bar n)$-dimensional plane (the limit of the Riemannian manifolds $\Sigma_k$) and that the varifolds induced by $T^0_k$ converge to the varifold induced by $T^0_\infty$. In particular, 
$\partial T_\infty^0 =0$ $\modp$ in $\bB_{6\sqrt{m}}$ and the tangent plane to $T_\infty^0$ is parallel to $\pi_0$ $\|T_\infty^0\|$-almost everywhere. 

Observe that by upper semicontinuity of the density, \eqref{e:basic} implies that $0$ is a point of density $Q$ for $T_\infty^0$. On the other hand \eqref{e:basic2} implies that $\|T_\infty^0\| (B_{6\sqrt{m}}) \leq Q \omega_m (6\sqrt{m})^m$. Hence, by the monotonicity formula, $T_\infty^0$ must be a cone. Observe that if $q\in \spt (T_\infty^0)$ is a point where the approximate tangent space $\pi_q$ exists, since $T_\infty^0$ is a cone, we must have that $q\in \pi_q$. Thus $q\in \pi_0$ for $\|T_\infty^0\|$-a.e. $q$, which in turn implies that $\spt (T_\infty^0) \subset \pi_0$. In conclusion $T_\infty^0 = Q \llbracket B_{6\sqrt{m}}\rrbracket \;  \modp$, and moreover the varifold convergence holds in the whole $\mathbb R^{m+n}$.  

Again by the monotonicity formula, $\spt (T^0_k)$ is converging locally in the sense of Hausdorff to $\spt (T_\infty^0)$. In particular if we set $T_k := T^0_k\res \bB_{23\sqrt{m}/4}$, for $k$ large $T_k$ will have no boundary $\modp$ in $\bC_{11\sqrt{m}/2}$. Hence it must be \eqref{e:geo semplice 1} which fails for an infinite number of $k$'s. On the other hand we certainly have $(\p_{\pi_0})_\sharp T_k \res \bC_{11\sqrt{m}/2} = Q_k \llbracket B_{11\sqrt{m}/2}\rrbracket$ $\modp$. Notice that by the varifold convergence we have $\|T^0_k\| (\bC_{11\sqrt{m}/2}\setminus \bB_{23\sqrt{m}/4}) \to 0$ as $k\to \infty$. In particular the limit of the currents $(\p_{\pi_0})_\sharp T_k \res \bC_{11\sqrt{m}/2}$ is the same as the limit of the currents $(\p_{\pi_0})_\sharp T^0_k \res \bC_{11\sqrt{m}/2}$ and thus it must be $Q_k = Q$ $\modp$ for $k$ large enough. 
\end{proof}

\section{Tilting of planes and proof of Proposition \ref{p:whitney}} \label{s:tilting_whitney}

Following \cite{DLS_Center}, the first important technical step in the proof of the existence of the center manifold is to gain a control on the tilting of the optimal planes as the cubes get refined. The following proposition corresponds to \cite[Proposition 4.1]{DLS_Center}. 

\begin{proposition}[Tilting of optimal planes]\label{p:tilting opt}
Assume that the hypotheses of Assumptions \ref{ipotesi} and \ref{parametri} hold, that
$C_e \geq C^\star$ and $C_h \geq C^\star C_e$, where $C^\star (M_0, N_0)$ is the constant of the previous section. 
If $\eps_2$ is sufficiently small, then 
\begin{itemize}
\item[(i)] $\bB_H\subset\bB_L \subset \bB_{5\sqrt{m}}$ for all $H, L\in \sW\cup \sS$ with $H\subset L$.
\end{itemize}
Moreover, if $H, L \in \sW\cup\sS$ and either $H\subset L$ or $H\cap L \neq \emptyset$ and $\frac{\ell (L)}{2} \leq \ell (H) \leq \ell (L)$, then the following holds, for $\bar{C} = \bar{C} (\beta_2, \delta_2, M_0, N_0, C_e)$ and $C = C(\beta_2, \delta_2, M_0, N_0, C_e, C_h)$:
\begin{itemize}
\item[(ii)]$|\hat\pi_H - \pi_H| \leq \bar{C} \bmo^{\sfrac{1}{2}} \ell (H)^{1-\delta_2}$;
\item[(iii)] $|\pi_H-\pi_L| \leq \bar{C} \bmo^{\sfrac{1}{2}} \ell (L)^{1-\delta_2}$;
\item[(iv)] $|\pi_H-\pi_0| \leq  \bar C \bmo^{\sfrac{1}{2}}$;
\item[(v)] $\bh (T, \bC_{36 r_H} (p_H, \pi_0)) \leq C \bmo^{\sfrac{1}{2m}} \ell (H)$ and $\spt (T) \cap \bC_{36 r_H} (p_H, \pi_0) \subset \bB_H$; 
\item[(vi)]  For $\pi= \pi_H, \hat{\pi}_H$, $\bh (T, \bC_{36r_L} (p_L, \pi))\leq C \bmo^{\sfrac{1}{2m}} \ell (L)^{1+\beta_2}$
and $\spt (T) \cap \bC_{36 r_L} (p_L, \pi)\subset \bB_L$.
\end{itemize}
In particular, the conclusions of Proposition~\ref{p:whitney} hold.
\end{proposition}
\begin{proof}
First of all we observe that, if we replace (ii), (iii) and (iv) with
\begin{itemize}
\item[(ii)no] $|\hat\pi_H - \pi_H|_{no} \leq \bar{C} \bmo^{\sfrac{1}{2}} \ell (H)^{1-\delta_2}$,
\item[(iii)no] $|\pi_H-\pi_L|_{no} \leq \bar{C} \bmo^{\sfrac{1}{2}} \ell (L)^{1-\delta_2}$, and
\item[(iv)no] $|\pi_H-\pi_0|_{no} \leq  \bar C \bmo^{\sfrac{1}{2}}$,
\end{itemize} 
then the arguments given in the \cite[Proof of Proposition 4.1]{DLS_Center} can be followed literally with minor adjustments. Indeed those arguments depend only on:
\begin{itemize}
\item the monotonicity formula;
\item the triangle inequality $|\alpha -\gamma| \leq |\alpha - \beta| + |\beta - \gamma|$;
\item the elementary geometric observation that, for every subset $E$ and every pair of $m$-planes $\alpha$ and $\beta$, we have the inequality
\[
\bh (T, E, \alpha) \leq \bh (T, E, \beta) + C {\rm diam}\, (E) |\alpha -\beta|\, ,
\]
where $C$ is a geometric constant.
\end{itemize}
However, it can be easily verified that all such properties remain true if we replace $|\cdot|$ with $|\cdot|_{no}$. 

We next come to (ii), (iii) and (iv). First observe that both $ \pi_H$ and the (oriented) $m$-plane with the same support and opposite orientation belong to $T_{p_H} \Sigma$. For this reason, the definition of $\pi_H$ implies that $\abs{\pi_H - \hat \pi_H}_{no} = \abs{\pi_H - \hat \pi_H}$, thus allowing us to infer (ii) from (ii)no.

%But then we conclude that
%\[
%\min \{|\hat\pi_H - \pi'|: \pi' \subset T_{p_H} \Sigma\} \leq \bar C \bmo^{\sfrac{1}{2}}\, .
%\]
%Hence, if $\varepsilon_2$ is sufficiently small, we can apply Proposition \ref{p:or_nonor} and conclude that 
%\[
%{\color{red}|\hat \pi_H - \pi_H| = | \hat\pi_H - \p_{T_{p_H} \Sigma} (\hat\pi_H)| = | \hat\pi_H - \p_{T_{p_H} \Sigma} (\hat\pi_H)|_{no} = |\hat \pi_H - \pi_H|_{no}}\, ,
%\]
%and thus infer (ii) from (ii)no. 

Next, recall that we have $|\hat\pi_H - \pi_0| = |\hat\pi_H - \pi_0|_{no}$, cf. Remark \ref{r:why_optimal_2}. Hence (iv)no implies (iv). Now, combining (iv) for two planes $H$ and $L$ as in statement (iii) of the proposition, we conclude that $|\pi_H - \pi_L|\leq |\pi_H-\pi_0| + |\pi_L - \pi_0| \leq C \bmo^{\sfrac{1}{2}}$. Hence, again assuming that $\varepsilon_2$ is sufficiently small, we can apply Proposition \ref{p:or_nonor}, in particular conclusion (Eq): $|\pi_H - \pi_L| = |\pi_H - \pi_L|_{no}$. Thus (iii) is a consequence of (iii)no. 
\end{proof}

\begin{remark} 
Notice that, even though our arguments use the nonoriented excess as control parameter, the estimates of Proposition \ref{p:tilting opt} on the tilt of optimal planes which are needed for the construction of the center manifold involve a measure of the classical distance between \emph{oriented} planes. As seen in the proof, such estimates continue to be valid in our setting thanks to our choice of optimal planes made in \eqref{e:optimal_pi_2} and to the observation made in Proposition \ref{p:or_nonor}.
\end{remark}

Arguing as in \cite[Section 4.3]{DLS_Center} we get the following existence theorem with very minor modifications
(the only adjustment that needs to be taken into consideration is that the identities \cite[(4.9)]{DLS_Center}, \cite[(4.10)]{DLS_Center} and the subsequent analogous ones must be replaced with the same equalities $\modp$):

\begin{proposition}[Existence of interpolating functions]\label{p:gira_e_rigira}
Assume the conclusions of the Proposition \ref{p:tilting opt} apply.
The following facts are true provided $\eps_2$ is sufficiently small. 
Let $H, L\in \sW\cup \sS$ be such that either $H\subset L$ or $H\cap L \neq \emptyset$ and
$\frac{\ell (L)}{2} \leq \ell (H) \leq \ell (L)$. Then,
\begin{itemize}
\item[(i)] for $\pi= \pi_H, \hat{\pi}_H$, $(\p_{\pi})_\sharp T\res \bC_{32r_L} (p_L, \pi) = Q \a{B_{32r_L} (p_L, \pi))} \; \modp$ and $T$ satisfies the assumptions of \ref{thm:strong-alm-unoriented} in the cylinder $\bC_{32 r_L} (p_L, \pi)$;
\item[(ii)] Let $f_{HL}$ be the $\pi_H$-approximation of $T$ in $\bC_{8 r_L} (p_L, \pi_H)$ and $h_{HL} := (\etab\circ f_{HL})*\varrho_{\ell (L)}$ be its smoothed average. Set $\varkappa_H := \pi_H^\perp \cap T_{p_H} \Sigma$ and consider the maps 
\begin{equation*}
\begin{array}{lll}
x\quad \mapsto\quad \bar{h} (x)  &:=  \p_{T_{p_H}\Sigma} (h)&\in \varkappa_H\\ 
x\quad \mapsto\quad h_{HL} (x) &:= (\bar{h} (x), \Psi_{p_H} (x, \bar{h} (x)))&\in \varkappa_H \times (T_{p_H} (\Sigma))^\perp\, .
\end{array}
\end{equation*} 
Then there is a smooth
$g_{HL} :  B_{4r_L} (p_L, \pi_0)\to \pi_0^\perp$ s.t. $\bG_{g_{HL}} = \bG_{h_{HL}}\res \bC_{4r_L} (p_L, \pi_0)$.
\end{itemize}
\end{proposition}

\begin{definition}\label{d:mappe_h_HL}%\label{d:interpolations 2}
$h_{HL}$ and $g_{HL}$ will be called, respectively, {\em tilted $(H,L)$-interpolating function} and {\em $(H,L)$-interpolating function}.
\end{definition}

Observe that the tilted $(L,L)$-interpolating function and the $(L,L)$-interpolating function correspond to the tilted $L$-interpolating function and to the $L$-interpolating function 
of Definition~\ref{d:glued}. Obviously, Lemma \ref{l:tecnico2} is just a particular case of Proposition \ref{p:gira_e_rigira}. As in Definition \ref{d:glued}, we will set $h_L := h_{LL}$ and $g_L := g_{LL}$.

\section{The key construction estimates} \label{s:cm_construction}

Having at disposal the Existence Proposition \ref{p:gira_e_rigira} we can now come to the main estimates on the building blocks of the center manifold, which in fact correspond precisely to \cite[Proposition 4.4]{DLS_Center} and are
thus restated here only for the reader's convenience. 

\begin{proposition}[Construction estimates]\label{p:stime_chiave}
 Assume the conclusions of Propositions \ref{p:tilting opt} and \ref{p:gira_e_rigira} apply 
and set $\kappa = \min \{\beta_2/4, \eps_0/2\}$. Then,
the following holds for any pair of cubes $H, L\in \sP^j$ (cf. Definition \ref{d:glued}), where
$C = C (\beta_2, \delta_2, M_0, N_0, C_e, C_h)$: 
\begin{itemize}
\item[(i)] $\|g_H\|_{C^0 (B)}\leq C\, \bmo^{\sfrac{1}{2m}}$ and
$\|Dg_H\|_{C^{2, \kappa} (B)} \leq C \bmo^{\sfrac{1}{2}}$, for $B = B_{4r_H} (x_H, \pi_0)$;
\item[(ii)] if $H\cap L\neq \emptyset$,
then $\|g_H-g_L\|_{C^i (B_{r_L} (x_L,\pi_0))} \leq C \bmo^{\sfrac{1}{2}} \ell (H)^{3+\kappa-i}$ 
for every $i\in \{0, \ldots, 3\}$;
\item[(iii)] $|D^3 g_H (x_H) - D^3 g_L (x_L)| \leq C \bmo^{\sfrac{1}{2}} |x_H-x_L|^\kappa$;
\item[(iv)] $\|g_H-y_H\|_{C^0} \leq C \bmo^{\sfrac{1}{2m}} \ell (H)$ and 
$|\pi_H - T_{(x, g_H (x))} \bG_{g_H}| \leq C \bmo^{\sfrac{1}{2}} \ell (H)^{1-\delta_2}$
$\forall x\in H$;
\item[(v)] if $L'$ is the cube concentric to $L\in \sW^j$ with $\ell (L')=\frac{9}{8} \ell (L)$, 
then
\[
\|\varphi_i - g_L\|_{L^1 (L')} \leq C\, \bmo\, \ell (L)^{m+3+\beta_2/3} \quad \text{for all }\; i\geq j\,.
\]
\end{itemize}
\end{proposition}

The proof of Theorem \ref{t:cm} assuming the validity of Proposition \ref{p:stime_chiave} is given in \cite[Section 4.4, Proof of Theorem 1.17]{DLS_Center}. As for the proof of Proposition \ref{p:stime_chiave}, we discuss briefly why the arguments given in \cite[Section 5]{DLS_Center} apply in our case as well. First of all, the key tool in the proof, namely \cite[Proposition 5.2]{DLS_Center}, is valid under our assumptions for the following reason. The proof given in \cite[Section 5.1]{DLS_Center} is based on the following facts:
\begin{itemize}
\item The first variation of $T$ vanishes, and this allows to estimate the first variation of $\bG_f = \bG_{f_{HL}}$ as in
\cite[Eq. (5.4)]{DLS_Center};
\item The estimates claimed in \cite[Eqs. (5.5)--(5.9)]{DLS_Center} are valid because of Theorem 
\ref{thm:strong-alm-unoriented} and the Taylor expansion of \cite[Corollary 13.2]{DLHMS_linear}. 
\item Using the decomposition $\delta \bG_f = \delta (\bG_{f^+} \res B_+) + \delta (\bG_{f^-}\res B_-) +
Q \delta (\bG_{\etab\circ f}\res B_0)$ we can show the validity of \cite[Eq. (5.11)]{DLS_Center}.  
\end{itemize}
The three ingredients above are then used to show the first estimate of \cite[Proposition 5.2]{DLS_Center}, namely \cite[Eq. (5.1)]{DLS_Center}. The derivation of the remaining part of \cite[Proposition 5.2]{DLS_Center} is then a pure PDE argument based only on \cite[Eq. (5.1)]{DLS_Center}. 

In \cite[Section 5.2]{DLS_Center} the \cite[Proposition 4.4]{DLS_Center} is used to derive \cite[Lemma 5.3]{DLS_Center}, which in fact includes the conclusions (i) and (ii) of Proposition \ref{p:stime_chiave}. This derivation does not depend anymore on the underlying current and thus the proof given in \cite[Section 5.2]{DLS_Center} works literally in our case as well. The remaining part of Proposition \ref{p:stime_chiave} is derived from \cite[Lemma 5.5]{DLS_Center}. The latter is based solely on the estimates on the Lipschitz approximation (which are provided by Theorem \ref{thm:strong-alm-unoriented}) and on \cite[Lemma 5.5]{DLS_Center}, whose role is taken, in our setting, by \cite[Lemma 16.1]{DLHMS_linear}. 

\section{Existence and estimates on the $\mathcal{M}$-normal approximation} \label{s:normal_approx}

Corollary \ref{c:cover} can be proved following the argument of \cite[Section 6.1]{DLS_Center}. The only adjustement needed is in the argument for claim (iii). Following the one of \cite[Section 6.1]{DLS_Center} we conclude that at every $q\in \mathbf{\Phi} (\mathbf{\Gamma})$, if we denote by $\pi$ the oriented tangent plane to $\mathcal{M}$ at $q$, then the current $Q\a{\pi}$ is the unique tangent $\modp$ of $T$ at $q$, in the sense of Corollary \ref{c:tangent_cones}. We then can use Proposition \ref{p:compactness} to conclude that $\Theta (T,q) = Q$. 

For Theorem \ref{t:approx} we can repeat the arguments of \cite[Section 6.2]{DLS_Center} in order to prove the existence of the $\mathcal{M}$-normal approximation  and the validity of \eqref{e:Lip_regional} and \eqref{e:err_regional}. As for \eqref{e:av_region} we can repeat the arguments of \cite[Section 6.3]{DLS_Center}, whereas in order to get \eqref{e:Dir_regional} we make the following adjustments to the first part of \cite[Section 6.3]{DLS_Center}. The paragraphs leading to \cite[Eq. (6.11)]{DLS_Center} are obviously valid in our setting. However \cite[Eq. (6.11)]{DLS_Center} must be replaced with the following analogous estimate
\begin{align}
\int_{\p^{-1} (\cL)} |\vec{\bT}_F (x) -& \vec{\cM} (\p (x))|_{no}^2 d\|\bT_F\| (x)\nonumber\\ 
\leq\; &\int_{\p^{-1} (\cL)} |\vec{T} (x) - \vec{\cM} (\p (x))|_{no}^2 d\|T\| (x) + C \bmo^{1+\gamma_2}
\ell (L)^{m+2+\gamma_2}\nonumber\\
\leq\; &\int_{\p^{-1} (\cL)} |\vec{T} (x) - \vec{\pi}_L|_{no}^2 d \|T\| (x) + C \bmo \ell (L)^{m+2-2\delta_2}\label{e:eccesso_storto}
\end{align}
From this one we proceed as in the rest of \cite[Section 6.3]{DLS_Center} using the Taylor expansion of 
\cite[Proposition 13.3]{DLHMS_linear} in place of \cite[Proposition 3.4]{DLS_Currents}. 

\section{Separation and splitting before tilting} \label{s:splitting}

The arguments for Proposition \ref{p:separ} and Corollary \ref{c:domains} can be taken from \cite[Section 7.1]{DLS_Center}, modulo using Theorem \ref{t:height_bound} in place of \cite[Theorem A.1]{DLS_Center}. 

We next come to the proof of Proposition \ref{p:splitting}. A first important ingredient is the Unique continuation property of \cite[Lemma 7.1]{DLS_Center}, which we will now prove it is valid for $\Iqspec$ minimizers as well. 

\begin{lemma}[Unique continuation for $Dir$-minimizers]\label{l:UC}
For every $\eta \in (0,1)$ and $c>0$, there exists $\gamma>0$ with the following property.
If $w: \R^m\supset B_{2\,r} \to \Iqspec$ is Dir-minimizing,
${\rm Dir}\, (w, B_r)\geq c$ and  ${\rm Dir}\, (w, B_{2r}) =1$,
then
\[
{\rm Dir}\, (w, B_s (q)) \geq \gamma \quad \text{for every 
$B_s(q)\subset B_{2r}$ with $s \geq \eta\,r$}.
\]
\end{lemma}

\begin{proof} We follow partially the argument of \cite[Section 7.2]{DLS_Center} for \cite[Lemma 7.1]{DLS_Center}. In particular, the second part of the argument, which reduces the statement to the following claim, can be applied with no alterations:
\begin{itemize}
\item[(UC)] if $\Omega$ is a connected open set and $w\in W^{1,2} (\Omega, \Iqspec)$ 
is Dir-minimizing in any every bounded
$\Omega'\subset\subset \Omega$, then either $w$ is 
constant or $\int_J |Dw|^2 >0$ for every nontrivial open $J\subset \Omega$. 
\end{itemize}
However, the proof given in \cite[Section 7.1]{DLS_Center} of (UC) when $w\in W^{1,2} (\Omega, \Iqs)$ cannot be repeated in our case, since it uses heavily the fact that the singular sets of $\Iqs$-valued Dir-minimizers
cannot disconnect the domain, a property which is not enjoyed by $\Iqspec$-valued Dir-minimizers. We thus have to modify the proof somewhat, although the tools used are essentially the same. 

Assume by contradiction that there are a connected open set $\Omega\subset \mathbb R^m$, a map $w\in W^{1,2}_{loc} (\Omega, \Iqspec)$ and a nontrivial open subset $J\subset \Omega$ such that
\begin{itemize}
\item[(a)] $w$ is Dir-minimizing on every open $\Omega'\subset\subset \Omega$;
\item[(b)] $w$ is not constant, and thus $\int_{\Omega'} |Dw|^2 >0$ for some $\Omega'\subset\subset \Omega$;
\item[(c)] $\int_J |Dw|^2 =0$.
\end{itemize}
Observe first that, from the classical unique continuation of harmonic functions, either $\etab\circ w$ is constant, or it has positive Dirichlet energy on any nontrivial open subset of $\Omega$. Since however the Dirichlet energy of $\etab \circ w$ is controlled from above by that of $w$, (c) excludes the second posssibility. Thus $\eta\circ w$ is constant and hence, without loss of generality, we can assume $\etab\circ w \equiv 0$.

Next assume, without loss of generality, that $J$ is connected. Clearly, $w$ is constantly equal to some $P\in \Iqspec$ on $J$. Since, without loss of generality, we could ``flip the signs of the Dirac masses'' which constitute the values of $u$, we can always assume that $P = (\sum_i \a{P_i}, 1)$. We then distinguish two cases.

\medskip

{\em First Case.} The diameter of $\spt (P)$ is positive, namely $|P_i-P_j|>0$ for some $i\neq j$. In this case consider the interior $U$ of the set $\{w = P\}$. We want to argue that $U=\Omega$, which contradicts (b). Since $\Omega$ is open and connected, it suffices to show that $\partial U \cap \Omega = \emptyset$. In order to show this, consider a point $x\in \partial U$. If $x\in \Omega$, using the continuity of the map $w$, we know that in a sufficiently small ball $B_\rho (x)$ there is an $\Iqs$-valued map $z$ such that $w (y) = (z(y), 1)$ for all $y\in B_\rho (x)$. As such, $z$ must be a Dir-minimizer to which we can apply \cite[Section 7.2]{DLS_Center}: since $\int_{J'} |Dz|^2 =0$ for some nontrivial open $J'\subset B_\rho (x)$, we must have that $z$ is constant on $B_\rho (x)$. But then we would have $B_\rho (x) \subset U$, thus contradicting the assumption that $x \in \partial U$. 
 
\medskip

{\em Second Case.} The remaining possibility is that $P = Q \a{\etab\circ w (x)}= Q\a{0}$ (which equals both  $(Q\a{0}, 1)$ and $(Q\a{0}, -1)$, since the latter points are identified in $\Iqspec$). Define therefore 
\[
K := \{w = Q\a{0}\}\, ,
\]
and (since $K\supset J$) observe that $|K|>0$. Consider now the set $\tilde{K}$ of points $x \in \R^m$ such that
\begin{equation}\label{e:a_sequence}
0 < \lim_{k \to \infty} \frac{|K\cap B_{r_k} (x)|}{\omega_m r_k^m} <1 \qquad \mbox{for some $r_k\downarrow 0^+$}\,,
\end{equation}
and notice that $\tilde K \subset K$ since $w$ is continuous. The set $\tilde{K}$ is necessarily nonempty. If it were empty, we could first apply the classical characterization of Federer of sets of finite perimeter, cf. \cite[Theorem 4.5.11]{Federer69}, to infer that $K$ is a set of finite perimeter, and subsequently we could then apply the classical structure theorem of De Giorgi to conclude that, since the reduced boundary of $K$ would be empty, $D \mathbf{1}_K =0$. The latter would imply that $\mathbf{1}_K$ is constant on the connected set $\Omega$, namely that $\Omega\setminus K$ has zero Lebesgue measure, which in turn would contradict (b). 

Fix a point $x\in \tilde{K}$. Clearly it must be $\int_{B_\rho (x)} |Dw|^2 >0$ for every $\rho>0$, otherwise $w$ would be constant in a neighborhood of $x$ and thus $x$ would be an interior point of $K$. Denoting $I_{x,w}(\cdot)$ the \emph{frequency function} of $w$ at $x$ as in \cite[Definition 9.1]{DLHMS_linear}, from \cite[Theorem 9.2]{DLHMS_linear} we must then have
\[
\infty> I_0 := \lim_{r\downarrow 0} I_{x,w} (r) >0\, .
\]
Define then the maps $y \mapsto w_r(y)$, whose positive and negative parts are given by
\[
w^{\pm}_r(y) := \sum_{i} \a{r^{-I_0}\, w_{i}^{\pm}(r\,y + x)}\,,
\]
 and observe that a subsequence of $\{w_{r_k}\}_{k\in \mathbb N}$, not relabeled, is converging to a nontrivial $w_0\in W^{1,2}_{loc} (\mathbb R^m, \Iqspec)$ which minimizes the Dirichlet energy on every $\Omega'\subset\subset \mathbb R^m$ and is $I_0$-homogeneous. 

Next define the sets $K_{r_k} := r_k^{-1} (K - x)$, where the maps $w_{r_k}$ vanish identically, and observe that, by \eqref{e:a_sequence}, 
$\liminf_k |K_{r_k} \cap \overline{B}_1|>0$. Since the sets $K_{r_k} \cap \overline{B}_1$ are compact we can, without loss of generality, assume that they convergence in the sense of Hausdorff to some set $K_0$. The limiting map $w_0$ vanishes on such set because the $w_{r_k}$ are converging locally uniformly to $w_0$. On the other hand it is elementary to see that the Lebesgue measure is upper semicontinuous under Hausdorff convergence and we thus conclude $|K_0|>0$.

We can now repeat the procedure above on some point $y\neq 0$ where the Lebesgue density of $K_0$ does not exist or it is neither zero nor one. We find thus a corresponding tangent function $w_1$ that has all the properties of $w_0$, namely
\begin{itemize} 
\item it is nontrivial, 
\item it vanishes identically on a set of positive measure,
\item it is $I_1$-homogeneous for some positive constant $I_1$,
\item and it minimizes the Dirichlet energy on any bounded open set.
\end{itemize} 
In addition $w_1$ is invariant under translations along the direction $\frac{y}{|y|}$. Assuming, after rotations, that such vector is $e_m = (0, 0, \ldots, 0, 1)$, the function $w_1$ depends therefore only on the variables $x_1, \ldots , x_{m-1}$ and can thus be treated as a function defined over $\mathbb R^{m-1}$. Iterating $m-2$ more times such procedure we achieve finally a function $w_{m-1}: \mathbb R \to \Iqspec$ with the following properties:
\begin{itemize}
\item[(A)] $w_{m-1}$ is identically $Q\a{0}$ on some set of positive measure;
\item[(B)] $\int_1^{-1} |Dw_{m-1}|^2 >0$;
\item[(C)] $w_{m-1}$ is Dir-minimizing on $]a, b[$ for every $0<a<b<\infty$;
\item[(D)] $w_{m-1}$ is $\alpha$-homogeneous for some positive $\alpha>0$;
\item[(E)] $\etab\circ w_{m-1} \equiv 0$. 
\end{itemize}
Because of (A) and (D), $w_{m-1}$ must be identically equal to $Q\a{0}$ on at least one of two half-lines $]-\infty, 0]$ and $[0, \infty[$. Without loss of generality we can assume this happens on the $]-\infty, 0[$. Let now 
$w_{m-1} (1) = (\sum_i \a{c_i}, \epsilon)$, where $\epsilon \in \{-1,1\}$. By (D) we then have
\[
w_{m-1} (x) = \left(\sum_i \llbracket c_i x^\alpha\rrbracket, \epsilon\right) \qquad \forall x\geq 0\, .
\] 
Observe that, because of (B), at least one of the $c_i$'s is nonzero. Therefore
$\epsilon$ cannot be equal to $1$, otherwise $w_{m-1}$ would give an $\Iqs$-valued Dir-minimizer on the real line with a singularity, which is not possible. However, since $(Q\a{0}, 1) = (Q\a{0}, -1)$, if $\varepsilon$ equals $-1$ we reach precisely the same contradiction. This completes the proof. 
\end{proof}

We keep following the strategy of \cite[Section 7.2]{DLS_Center} towards a proof of Proposition \ref{p:splitting}. First of all, we introduce some useful notation. 

\begin{definition}\label{d:centered}
Let $w: E\to \Iqspec$, let $E_+, E_-$ and $E_0$ be the canonical decomposition of $E$ induced by $w$ and let $w^+, w^-$ and $\etab\circ w$ the corresponding maps, as in \cite[Definition 2.7]{DLHMS_linear}. For any $f: E \to\R^n$ we denote by $w\oplus f$ (resp. $w\ominus f$) the $\Iqspec$-valued map which 
\begin{itemize}
\item on $E_+$ coincides with $(w^+\oplus f,1)$ (resp. $(w^+\ominus f,1)$), 
\item on $E_-$ coincides with $(w^-\oplus f,-1)$ (resp. $(w^-\ominus f, -1)$), 
\item and on $E_0$ coincides with $Q\a{\etab\circ w+f}$ (resp. $Q\a{\etab\circ w-f}$.  
\end{itemize}
Moreover we use the shorthand notation $\bar{w}$ for $w\ominus \etab\circ w$. 
\end{definition}

We next show that if the energy of an $\Iqspec$-valued Dir-minimizer $w$ does not decay appropriately, then the map must ``split'', in other words $\bar{w}$ cannot be too small compared to $\etab\circ w$. 
As in \cite[Section 7.2]{DLS_Center}, we fix $\lambda>0$ such that 
\begin{equation}\label{e:lambda}
(1+\lambda)^{(m+2)}< 2^{\delta_2}\, ,
\end{equation}
and we claim the following analog of \cite[Proposition 7.2]{DLS_Center}.

\begin{proposition}[Decay estimate for ${\rm Dir}$-minimizers]\label{p:harmonic split}
For every $\eta>0$, there is $\gamma >0 $ with the following property.
Let $w: \R^m \supset B_{2r} \to \Iqspec$ be Dir-minimizing 
in every $\Omega'\subset\subset B_{2r}$ such that
\begin{equation}\label{e.no decay}
\int_{B_{(1+\lambda) r}} \cG_s \big(Dw, Q \a{D (\etab \circ w) (0)}\big)^2 \geq 2^{\delta_2-m-2}\, {\rm Dir}\, (w, B_{2r})\, .
\end{equation}
Then, if we let $\bar{w}$ be as in Definition \ref{d:centered}, the following holds:
\begin{equation}\label{e.harm split 2}
\gamma \, {\rm Dir}\, (w, B_{(1+\lambda) r}) \leq {\rm Dir}\, (\bar{w}, B_{(1+\lambda) r}) 
\leq \frac{1}{\gamma \, r^2} \int_{B_{s}(q)} |\bar{w}|^2 \quad\forall \;B_s(q) \subset B_{2\,r} \;\text{with }\, s\geq \eta \,r\, .
\end{equation}
\end{proposition}

The proof of \cite[Proposition 7.2]{DLS_Center} can be literally followed for our case using the Unique continuation Lemma \ref{l:UC} in combination with the next simple algebraic computation (which is the counterpart of \cite[Lemma 7.3]{DLS_Center}). 

\begin{lemma}\label{l:medie}
Let $B\subset \R^m$ be a ball centered at $0$, $w\in W^{1,2} (B , \Iqspec)$ ${\rm Dir}$-minimizing and $\bar{w}$ as in Definition \ref{d:centered}
We then have
\begin{align}
Q \int_B |D (\etab\circ w) - D (\etab \circ w) (0)|^2
&=\int_B \cG_s (Dw, Q \a{ D (\etab \circ w) (0)})^2 - {\rm Dir}\, (\bar w, B)\, .\label{e:mean+centered}
\end{align}
\end{lemma}

The detail of the necessary modifications to the argument in \cite[Proof of Proposition 7.2]{DLS_Center} towards proving Proposition \ref{p:harmonic split} are left to the reader; we will instead show how to prove the lemma above. 

\begin{proof} Let $u := \etab \circ w$ and observe that it is harmonic. Thus, using the mean value property of harmonic functions and a straightforward computation we get
\begin{equation}\label{e:media1}
Q \int_B |Du - Du (0)|^2 = Q \int_B |Du|^2 - Q |B| |Du (0)|^2\, .
\end{equation}
On the other hand, using again the mean value property of harmonic functions, it is easy to see that
\[
\int_B \cG_s (Dw, Q \a{Du (0)})^2 = \sum_{\epsilon=+,-} \int_{B^\epsilon} \cG (Dw^\epsilon, Q \a{Du (0)})^2 
+ Q \int_{B_0} |Du - Du (0)|^2 
\]
and
\[
\int_{B^\epsilon} \cG (Dw^\epsilon, Q \a{Du (0)})^2 = \int_{B^\epsilon} (|Dw^\epsilon|^2 - 2 Q Du : Du (0) + Q |Du (0)|^2)\, .  
\]
In particular, we get
\[
\int_B \cG_s (Dw, Q \a{Du (0)})^2 =\int_B |Dw|^2 + Q |B| |Du (0)|^2 - 2Q Du (0) : \int_B Du
\]
and again by the mean value property we conclude
\begin{equation}\label{e:media2}
\int_B \cG_s (Dw, Q \a{Du (0)})^2 = \int_B |Dw|^2  - Q |B| |Du (0)|^2\, .
\end{equation}
Combining \eqref{e:media1} and \eqref{e:media2} we thus get
\begin{align}
&\int_B \cG_s (Dw, Q \a{ D (\etab \circ w) (0)})^2 - Q \int_B |D (\etab\circ w) - D (\etab \circ w) (0)|^2\nonumber\\
=\; &\int_B \cG_s (Dw, Q \a{Du (0)})^2 - Q \int_B |Du - Du (0)|^2= \int_B |Dw|^2 - Q \int_B |Du|^2\nonumber\\
=\; & \int_B |Dw|^2 - Q \int_B |D(\etab\circ w)|^2 \, .\label{e:media3}
\end{align}
Next, a simple algebraic computations shows 
\begin{align}
 \int_B |Dw|^2 &=  \sum_{\epsilon=+,-} \int_{B^\epsilon} |Dw^\epsilon|^2 + Q \int_{B_0} |D (\etab\circ w)|^2\nonumber\\
 &= \sum_{\epsilon=+,-} \left(\int_{B^\epsilon} |D \bar w^\epsilon|^2 + Q |D(\etab\circ w^\epsilon)|^2\right)
 + Q \int_{B_0} |D (\etab\circ w)|^2\nonumber\\
 &= \int_B |D\bar w|^2 + Q \int_B |D(\etab\circ w)|^2\label{e:algebra}
\end{align}
Clearly, \eqref{e:media3} and \eqref{e:algebra} give \eqref{e:mean+centered} and conclude the proof. 
\end{proof}

\begin{proof}[Proof of Proposition \ref{p:splitting}]
Having at hand the analogs of the tools used in \cite[Section 7.3]{DLS_Center}, we can following the argument given there for \cite[Proposition 3.4]{DLS_Center}. In the first step of the proof (namely \cite[Step 1, p. 548]{DLS_Center}) we use \cite[Corollary 13.2]{DLHMS_linear}
in place of \cite[Corollary 3.3]{DLS_Currents}, we use Theorem \ref{thm:strong-alm-unoriented} in place of \cite[Theorem 2.4]{DLS_Lp} and we replace $\bE$ with $\bE^{no}$ in the various formulas. We also replace $\mathcal{G}$ with $\mathcal{G}_s$ in case $p=2Q$. We then follow \cite[Step 2, p. 550]{DLS_Center}, where we use Lemma \ref{l:UC} and Proposition \ref{p:harmonic split} in place of \cite[Lemma 7.1 \& Proposition 7.2]{DLS_Center} in case $p=2Q$. In the final \cite[Step 3, p. 551]{DLS_Center} we use the reparametrization Theorem \cite[Theorem 15.1]{DLHMS_linear} in place of the corresponding \cite[Theorem 5.1]{DLS_Currents} and measure the distance between $m$-planes using $|\cdot|_{no}$ in place of $|\cdot|$.
\end{proof}

\section{Persistence of multiplicity $Q$ points} \label{s:persistence}

The proofs of Proposition \ref{p:splitting_II} and Proposition \ref{p:persistence} can be easily adapted to our case from \cite[Proofs of Proposition 3.5 \& Proposition 3.6]{DLS_Center} once we prove the following analog of \cite[Theorem 2.7]{DLS_Lp}:

\begin{theorem}[Persistence of $Q$-points]\label{t:persistence}
For every $\hat{\delta}, C^\star>0$, there is $\bar{s}\in ]0, \frac{1}{2}[$ such that, for every $s<\bar{s}$, there exists $\hat{\eps} (s,C^*,\hat{\delta})>0$ with the following property. If $T$ is as in Theorem \ref{thm:strong-alm-unoriented}, $E^{no} := \bE^{no} (T,\bC_{4\,r} (x)) < \hat{\eps}$,
$r^2 \bA^2 \leq C^\star E^{no}$ and $\Theta (T, (p,q)) = Q$ at some $(p,q)\in \bC_{r/2} (x)$, then the approximation
$f$ of Theorem \ref{thm:almgren_strong_approx} satisfies
\begin{equation}\label{e:persistence2}
\int_{B_{sr} (p)} \cG_\square (f, Q \a{\etab\circ f})^2 \leq \hat{\delta} s^m r^{2+m} E^{no}\, ,
\end{equation}
where $\square =s$ if $p=2Q$ or $\square = \;$ otherwise. 
\end{theorem}

In order to show Theorem \ref{t:persistence} we can follow literally \cite[Section 9]{DLS_Lp}. Indeed the proof in \cite[Section 9]{DLS_Lp} relies on the H\"older estimates for Dir minimizers (which are valid in the $\Iqspec$ case by \cite[Theorem 8.1]{DLHMS_linear}), the estimates on the Lipschitz approximation (given by Theorem \ref{thm:strong-alm-unoriented} and the classical monotonicity formula in the slightly improved version of \cite[Lemma A.1]{DLS_Lp}. Although the latter is stated for stationary integral currents in a Riemannian manifold, it is easy to see that the proof is in fact valid for stationary varifolds and as such can be applied to $\modp$ area-minimizing currents. We formulate the precise theorem here for the reader's convenience.

\begin{lemma}\label{l:monot}
There is a constant $C$ depending only on $m$, $n$ and $\bar{n}$ with the following property. If $\Sigma\subset \R^{m+n}$ is a $C^2$ $(m+\bar{n})$-dimensional submanifold with $\|A_\Sigma\|_\infty \leq \bA$, $U$ is an open set in $\mathbb R^{m+n}$ and $V$ an $m$-dimensional integral varifold supported in $\Sigma$ which is stationary in $\Sigma\cap U$, then for every $\xi\in \Sigma\cap U$ the function $\rho\mapsto \exp (C \bA^2 \rho^2) \rho^{-m} \|V\| (\bB_\rho (\xi))$ is monotone on the interval $]0, \bar{\rho}[$, where $\bar{\rho} := \min \{\dist (x, \partial U), (C \bA)^{-1}\}$.
\end{lemma}

\begin{remark}
The proof of Theorem \ref{t:persistence} can also be given following the alternative argument of Spolaor in \cite{Spolaor}, which uses the Hardt-Simon inequality and the classical version by Allard of Moser's iteration for subharmonic functions on varifolds.  While Spolaor's argument is more flexible and indeed works for integral currents which are not minimizing but sufficiently close to minimizing ones in a suitably quantified way, we prefer to 
adhere to the strategy of \cite{DLS_Lp} because it is more homogeneous to our notation and terminology.
\end{remark}

\section{Proof of Proposition \ref{p:compara}} \label{s:compara}

The proof follows the one of \cite[Proposition 3.7]{DLS_Center} given in \cite[Section 9]{DLS_Center} with minor modifications. The necessary tools used there, namely the splitting before tilting Propositions, the height bound and 
the reparametrization theorem are all available from the previous sections.

\newpage

\part{Blow-up and final argument} \label{part:Blowup}

\section{Intervals of flattening} 

Our argument for Theorem \ref{t:main2} is by contradiction, and we start therefore fixing a current $T$, a submanifold $\Sigma$, an open set $\Omega$, an integer $2\leq Q \leq \frac{p}{2}$, positive reals $\alpha$ and $\eta$ and a sequence $r_k\downarrow 0$ of radii as in Proposition \ref{p:contradiction_sequence}. In this section we proceed as in \cite[Section 2]{DLS_Blowup} and define appropriate intervals of flattening $]s_j, t_j]$, which are intervals over which we will construct appropriate center manifolds.
These intervals, which will be ordered so that $t_{j+1}\leq s_j$ will satisfy several properties, among which we anticipate the following fundamental one: aside from finitely many exceptions, each radius $r_k$ belongs to one of the intervals. In particular, if they are finitely many, then $0$ is the left endpoint of the last one, whereas if they are infinitely many, then $t_j\downarrow 0$. The definition of these intervals is taken literally from \cite[Section 2.1]{DLS_Blowup}, the only difference being that we take advantage of Theorem \ref{t:cm} in place of \cite[Theorem 1.17]{DLS_Center}. However we repeat the details for the reader's convenience. 
 
Without loss of generality we assume that $\bB_{6\sqrt{m}} (0) \subset \Omega$, and we fix a small parameter $\varepsilon_3\in ]0, \varepsilon_2[$, where $\varepsilon_2$ is the constant appearing in \eqref{e:small ex} of Assumption \ref{ipotesi}. Then, we take advantage of Proposition \ref{p:contradiction_sequence} and of a simple rescaling argument to assume further that:
\begin{align} \label{e:density_and_boundary}
T_{0}\Sigma = \R^{m+\bar n} \times \{0\}\,,& \quad \Theta(T,0) = Q\,, \quad \partial T \res \bB_{6\sqrt{m}}(0) = 0\; \modp\,,\\ \label{e:mass control}
\|T\|(\bB_{6\sqrt{m} \rho}(0)) &\leq \left( Q \, (6\sqrt{m})^{m} + \eps_3^2 \right) \, \rho^m \qquad \mbox{for all $\rho \leq 1$}\,,\\ \label{e:appiattimento}
{\bf c}(\Sigma \cap \bB_{7\sqrt{m}}(0)) &\leq \eps_3\,.
\end{align}

We next define
\begin{equation}
\mathcal{R} := \left\{r\in ]0,1] : \bE^{no} (T, \bB_{6\sqrt{m} r} (0)) \leq \varepsilon_3^2\right\}\, ,
\end{equation}
Observe that
$\{0\}\cup\mathcal{R}$ is a closed set and that, since $\bE^{no} (T, \bB_{6\sqrt{m} r_k}) \to 0$ as $k\uparrow \infty$, $r_k\in \mathcal{R}$ for $k$ large enough. 

The intervals of flattening will form a covering of $\mathcal{R}$. We first define
$t_0$ as the maximum of $\mathcal{R}$. We then define inductively $s_0, \ldots, t_j, s_j$ in the following way.

Let us first assume that we have defined $t_j$ and we wish to define $s_j$ (in particular this part is applied also with $j=0$ to define $s_0$). We first consider the rescaled current $T_j := ((\iota_{0,t_{j}})_\sharp T)\res \bB_{6\sqrt{m}}$, $\Sigma_j := \iota_{0, t_j} (\Sigma) \cap \bB_{7\sqrt{m}}$; moreover, consider
for each $j$ an orthonormal system of coordinates so that, if we denote by $\pi_0$ the $m$-plane $\mathbb R^m\times \{0\}$, then $\bE^{no} (T_j, \bB_{6\sqrt{m}}, \pi_0) = \bE^{no} (T_j,\B_{6\sqrt{m}})$ (alternatively we can keep the system of coordinates fixed
and rotate the currents $T_j$).

\begin{definition}\label{d:define_center_manifolds}
We let $\cM_j$ be the corresponding center manifold constructed
in Theorem \ref{t:cm} applied to $T_j$ and $\Sigma_j $ with respect to the $m$-plane $\pi_0$.
The manifold $\cM_j$ is then the graph of a map $\phii_j: \pi_0 \supset [-4,4]^m \to \pi_0^\perp$,
and we set $\Phii_j (x) := (x, \phii_j (x)) \in \pi_0\times \pi_0^\perp$. We then let $\sW^{(j)}$ be the Whitney decomposition of $[-4,4]^m \subset \pi_0$ as in Definition \ref{e:whitney}, applied to $T_j$. We denote by $\p_j$ the orthogonal projection on the center manifold $\cM_j$, which, given the $C^{3, \kappa}$ estimate on $\phii_j$, is well defined in a ``slab'' $\bU_j$ of thickness $1$ as defined in point (U) of Assumption \ref{intorno_proiezione}.
\end{definition}

Next we distinguish two cases:
\begin{itemize}
\item[(Go)] For every $L\in \sW^{(j)}$, 
\begin{equation}\label{e:go}
\ell (L) < c_s \dist (0, L)\, ,
\end{equation} 
where $c_s := \frac{1}{64\, \sqrt{m}}$, see Proposition \ref{p:compara}. In this case we set $s_j =0$. Observe that in this case the origin is included in the set $\bGam_j$ defined in \eqref{e:bGamma}.
\item[(Stop)] Assuming that (Go) fails, we fix an $L$ with maximal diameter among those cubes of $\sW^{(j)}$ which violate the inequality 
\eqref{e:go}. We then set 
\begin{equation}\label{e:def_s_j}
s_j := t_j \frac{\ell (L)}{c_s}\, .
\end{equation}
\end{itemize}
Observe that, in both cases, for every $\rho > \bar{r} := s_j/t_j$ we have
\begin{equation}\label{e:cubi_non_rompono}
\ell (L) < c_s \rho \qquad \mbox{for all } L\in \sW^{(j)} \mbox{ with $L\cap B_\rho (0, \pi_0) \neq \emptyset$.}
\end{equation}

We next come to the definition of $t_{j+1}$ once we know $s_j$. If $s_j =0$, then we stop the procedure and we end up with finitely many intervals of flattening. Otherwise we let $t_{j+1}$ be the maximum of $\mathcal{R}\cap ]0, s_j]$. Note that, since the vanishing sequence $\{r_k\}$ belongs to $\mathcal{R}$ except for finitely many elements, clearly the latter set is nonempty and thus $t_{j+1}$ is a positive number. Observe also that, by \eqref{e:prima_parte} of Proposition \ref{p:whitney} and using that $2^{-N_0} < c_s$ by \eqref{e:N0}, we have $\ell (L) \leq 2^{-6-N_0} \leq \frac{c_s}{64}$. Thus, $\frac{s_j}{t_j} < 2^{-5}$. This ensures that, in case (Go) never holds (i.e. the intervals of flattening are infinitely many), $t_j\downarrow 0$. 

\begin{definition}\label{d:interval_of_flattening}
We denote by $\mathcal{F}$ the (finite or countable) family of intervals of flattening as defined above. 
\end{definition}

The following proposition is the analog of \cite[Proposition 2.2]{DLS_Blowup} and, since the proof is a minor modification of the one given in \cite[Section 2.2]{DLS_Blowup} we omit it. Using the notation of Definition \ref{d:refining_procedure} we introduce the subfamilies $\sW^{(j)}_e, \sW^{(j)}_h$ and $\sW^{(j)}_n$. Recall also that, given two sets $A$ and $B$, we have defined their
{\em separation} as the number ${\rm sep} (A,B) := \inf \{|x-y|:x\in A, y\in B\}$.

\begin{proposition}\label{p:flattening}
Assuming $\eps_3$ sufficiently small, then the following holds:
\begin{itemize}
\item[(i)] $s_j < \frac{t_j}{2^5}$ and the family $\mathcal{F}$ is either countable
and $t_j\downarrow 0$, or finite and $I_j = ]0, t_j]$ for the largest $j$;
\item[(ii)] the union of the intervals of $\cF$ cover $\cR$,
and for $k$ large enough the radii $r_k$ in Proposition \ref{p:contradiction_sequence} belong to $\cR$;
\item[(iii)] if $r \in ]\frac{s_j}{t_j},3[$ and $J\in \mathscr{W}^{(j)}_n$ intersects
$B:= \p_{\pi_0} (\cB_r (q_j))$, with $q_j := \Phii_j(0)$,
then $J$ is in the
domain of influence $\mathscr{W}_n^{(j)} (H)$ (see Definition \ref{d:domains})
of a cube $H\in \mathscr{W}^{(j)}_e$ with
\[
\ell (H)\leq 3 \, c_s\, r \quad \text{and}\quad 
\max\left\{{\rm sep}\, (H, B),  {\rm sep}\, (H, J)\right\}
\leq 3\sqrt{m}\, \ell (H) \leq \frac{3 r}{16};
\]
\item[(iv)] $\bE^{no} (T_j, \B_r )\leq C_0 \eps_3^2 \, r^{2-2\delta_2}$ for
every $r\in]\frac{s_j}{t_j},3[$.
\item[(v)] $\sup \{ \dist (x,\cM_j): x\in \spt(T_j) \cap \p^{-1}_j(\cB_r(q_j))\} \leq C_0\, (\bmo^j)^\frac{1}{2m} r^{1+\beta_2}$ for
every $r\in]\frac{s_j}{t_j},3[$, where
$\bmo^j := \max\{\mathbf{c}(\Sigma_j)^2 , \bE^{no} (T_j, \bB_{6\sqrt{m}})\}$. 
\end{itemize}
\end{proposition}

\section{Frequency functions and its variations}

As in \cite[Section 3]{DLS_Blowup} we introduce the following Lipschitz (piecewise linear) weight 
\begin{equation*}%\label{e:def_phi}
\phi (r) :=
\begin{cases}
1 & \text{for }\, r\in [0,\textstyle{\frac{1}{2}}],\\
2-2r & \text{for }\, r\in \,\, ]\textstyle{\frac{1}{2}},1],\\
0 & \text{for }\, r\in \,\, ]1,+\infty[.
\end{cases}
\end{equation*}
For every interval of flattening $I_j = ]s_j, t_j] \in \cF$, we
let $N_j$ be the normal approximation of $T_j$
on the center manifold $\cM_j$ of Thereom \ref{t:approx}. As in \cite[Section 3]{DLS_Blowup} we introduce
the corresponding frequency functions and state the main analytical estimate, which allows us to
exclude infinite order of contact of the normal approximations with the center manifolds $\cM_j$. 

\begin{definition}[Frequency functions]\label{d:frequency}
For every $r\in ]0,3]$ we define: 
\[
\bD_j (r) := \int_{\cM^j} \phi\left(\frac{d_j(q)}{r}
\right)\,|D N_j|^2(q)\, dq\quad\mbox{and}\quad
\bH_j (r) := - \int_{\cM^j} \phi'\left(\frac{d_j (q)}{r}\right)\,\frac{|N_j|^2(q)}{d(q)}\, dq\, ,
\]
where $d_j (q)$ is the geodesic distance on $\cM_j$ between $q$ and $\Phii_j (0)$, and $dq$ is short for $d\Ha^m(q)$.
If $\bH_j (r) > 0$, we define the {\em frequency function}
$\bI_j (r) := \frac{r\,\bD_j (r)}{\bH_j (r)}$. 
\end{definition}

\begin{theorem}[Main frequency estimate]\label{t:frequency}
If $\eps_3$ is sufficiently small, then
there exists a geometric constant $C_0$ such that, for every
$[a,b]\subset [\frac{s_j}{t_j}, 3]$ with $\bH_j \vert_{[a,b]} >0$, we have
\begin{equation}\label{e:frequency}
\bI_j (a) \leq C_0 (1 + \bI_j (b)).
\end{equation}
\end{theorem}

To simplify the notation, in this section we drop the index $j$ and 
omit the measure $\cH^m$ in the integrals over regions of
$\cM$.
The proof exploits four identities collected in Proposition \ref{p:variation}, which is the analog of \cite[Proposition 3.5]{DLS_Blowup}
and whose proof will be discussed in the next sections.
Following \cite[Section 3]{DLS_Blowup} we introduce further auxiliary functions in order to express derivatives and estimates on the 
functions $\bD$, $\bH$ and $\bI$. We also remind the reader that in principle we must distinguish two situations: 
\begin{itemize}
\item If $Q< \frac{p}{2}$, then the normal approximations are $\Iq (\mathbb R^{m+n})$-valued maps and thus all the quantities considered here coincide literally with the ones defined in \cite[Section 3]{DLS_Blowup};
\item If $Q=\frac{p}{2}$, then the normal approximations take values in $\mathscr{A}_Q (\R^{m+n})$; in this case we use the notational conventions of \cite[Subsection 7.1]{DLHMS_linear} and thus, although at the formal level the definitions of the various objects are identical, the notation is underlying the fact that all integrals involved in the computations must be split into three domains to be reduced to integrals of expressions involving the $\mathcal{A}_Q (\R^{m+n})$-valued maps $N^+, N^-$ and $Q \a{\etab\circ N}$. 
\end{itemize}

\begin{definition}\label{d:funzioni_ausiliarie}
We let $\de_{\hat r}$ denote the derivative with respect to arclength along geodesics starting at $\Phii(0)$. We set
\begin{align}
&\qquad\qquad\qquad\bE (r) := - \int_\cM \phi'\left(\textstyle{\frac{d(q)}{r}}\right)\,\sum_{i=1}^Q \langle
N_i(q), \de_{\hat r} N_i (q)\rangle\, dq\,  ,\\
&\bG (r) := - \int_{\cM} \phi'\left(\textstyle{\frac{d(q)}{r}}\right)\,d(q) \left|\de_{\hat r} N (q)\right|^2\, dq
\quad\mbox{and}\quad
\bSigma (r) :=\int_\cM \phi\left(\textstyle{\frac{d(q)}{r}}\right)\, |N|^2(q)\, dq\, .
\end{align}
\end{definition}

As in \cite[Section 3]{DLS_Blowup} we observe that the estimate
\begin{align}\label{e:rough}
\bD(r) \leq { \int_{\cB_r (\Phii (0))} |DN|^2 \leq }\; C_0\, \bmo\, r^{m+2-2\delta_2} \quad \mbox{for every}\quad r\in\left]\textstyle{\frac{s}{t}},3\right[.
\end{align}
is a consequence of the inequality \eqref{e:Dir_regional} in Theorem \ref{t:approx}. Consider indeed that \eqref{e:cubi_non_rompono} bounds the side of each Whitney region $\mathcal{L}$ intersecting $\cB_r (\Phii (0))$ and that, on the contact region $\mathcal{K}$ the map $N$ vanishes identically: it suffices therefore to sum the estimates \eqref{e:Dir_regional} over the aforementioned Whitney regions $\mathcal{L}$. 

We are now ready to state the key four identities, cf. \cite[Proposition 3.5]{DLS_Blowup}:

\begin{proposition}[First variation estimates]\label{p:variation}
For every $\gamma_3$ sufficiently small there is a constant $C = C (\gamma_3)>0$ such that, if
$\eps_3$ is sufficiently small, $[a,b]\subset [\frac{s}{t}, 3]$ and $\bI \geq 1$ on $[a,b]$, then the following
inequalities hold for a.e. $r\in [a,b]$:
\begin{gather}
\left|\bH' (r) - \textstyle{\frac{m-1}{r}}\, \bH (r) - \textstyle{\frac{2}{r}}\,\bE(r)\right|\leq  C \bH (r), \label{e:H'bis}\\
\left|\bD (r)  - r^{-1} \bE (r)\right| \leq C \bD (r)^{1+\gamma_3} + C \eps_3^2 \,\bSigma (r),\label{e:out}\\
\left| \bD'(r) - \textstyle{\frac{m-2}{r}}\, \bD(r) - \textstyle{\frac{2}{r^2}}\,\bG (r)\right|\leq
C \bD (r) + C \bD (r)^{\gamma_3} \bD' (r) + C r^{-1}\bD(r)^{1+\gamma_3},\label{e:in}\\
\bSigma (r) +r\,\bSigma'(r) \leq C  \, r^2\, \bD (r)\, \leq C r^{2+m} \eps_3^{2}.\label{e:Sigma1}
% |\Sigma_\phi'(r)| \leq C  \, r\, D_\phi(r),\label{e:Sigma2}\\
\end{gather}
\end{proposition}

Theorem \ref{t:frequency} follows from the latter four estimates and from \eqref{e:rough} through the computations given in \cite[Section 3]{DLS_Blowup}. The proofs of the estimates \eqref{e:H'bis} and \eqref{e:Sigma1} given in \cite[Section 3]{DLS_Blowup} are valid in our case as well, since they do not exploit the connection between the approximation and the currents, but they are in fact valid for any map $N$ satisfying $\bI\geq 1$. We therefore focus on \eqref{e:out} and \eqref{e:in} which are instead obtained from first variation arguments applied to the area minimizing current $T_j$. In our case the current is area minimizing $\modp$, however a close inspection of the proofs in \cite{DLS_Blowup} shows that the computations in there can be transferred to our case because the varifold induced by $T_j$ is stationary (and the required estimates relating the varifold induced by the graph of $N_j$ in the normal bundle of $\cM_j$ and the current $T_j$ have been proved in the previous 
 section). 

In the rest of the section we omit the subscript $j$ from $T_j, \Sigma_j, \cM_j$ and $N_j$. 

\subsection{First variations} We recall the vector field used in \cite{DLS_Blowup}. We will consider:
\begin{itemize}
\item the {\em outer variations}, where $X (q)= X_o (q) := \phi \left(\frac{d(\p(q))}{r}\right) \, (q - \p(q) )$.
\item the {\em inner variations}, where $X (q) = X_i (q):= Y(\p(q))$ with
\[
Y(q) := \frac{d(q)}{r}\,\phi\left(\frac{d(q)}{r}\right)\, \frac{\de}{\de \hat r} \quad
\forall \; q \in \cM\, .
\]
\end{itemize}
Note that $X_i$ is the infinitesimal generator of
a one parameter family of bilipschitz homeomorphisms $\Phi_\eps$ defined as
$\Phi_\eps (q):= \Psi_\eps (\p (q)) + q - \p (q)$, where 
$\Psi_\eps$ is the one-parameter family of bilipschitz homeomorphisms of $\cM$ generated by $Y$. 

Consider now the map $F(q) := \sum_i \a{q+ N_i (q)}$ and the current $\bT_F$
associated to its image: in particular we are using the conventions of \cite{DLS_Currents} in the case $Q<\frac{p}{2}$ (i.e. when $N$ takes values in
$\mathcal{A}_Q (\R^{m+n})$) and the conventions introduced in \cite[Definition 11.2]{DLHMS_linear} in the case $Q=\frac{p}{2}$ (i.e. when $N$ takes values in $\mathscr{A}_Q (\R^{m+n})$). As in \cite[Section 3.3]{DLS_Blowup} we observe that, although the vector fields $X = X_o$ and $X= X_i$ are not compactly supported, it is easy to see that
$\delta T(X) = \delta T (X^T) + \delta T (X^\perp) = \delta T (X^\perp)$,
where $X=X^T+ X^\perp$ is the decomposition of $X$ in the tangent and normal components
to $T\Sigma$.

Then, we have 
\begin{align}\label{e:Err4-5}
&|\delta \bT_F (X)| \leq |\delta \bT_F (X) - \delta T (X)| + |\delta T(X^\perp)|\nonumber\\
\leq& \underbrace{\int_{\spt (T)\setminus \im (F)}  \left|\dv_{\vec T} X\right|\, d\|T\|
+ \int_{\im (F)\setminus \spt (T)} \left|\dv_{\vec \bT_F} X\right|\, d\|\bT_F\|}_{{\rm Err}_4}
+ \underbrace{\left|\int \dv_{\vec T} X^\perp\, d\|T\| \right|}_{{\rm Err}_5}.
\end{align}
In order to simplify the notation we set $\varphi_r (x) := \phi \left(\frac{d(x)}{r}\right)$. 
Next, we apply \cite[Theorem 4.2]{DLS_Currents} in the case $Q<\frac{p}{2}$ (this corresponds exactly to what done in \cite[Section 3.3]{DLS_Blowup} and \cite[Theorem 14.2]{DLHMS_linear} when $Q=\frac{p}{2}$ to conclude
\begin{align}\label{e:ov graph}
\delta \bT_F (X_o) =& \int_\cM\Big( \varphi_r\,|D N|^2 + \sum_{i=1}^Q N_i\otimes \nabla \varphi_r : D N_i \Big) + \sum_{j=1}^3 \textup{Err}^o_j,
\end{align}
where the errors ${\rm Err}^o_j$ correspond to the terms ${\rm Err_j}$ of \cite[Theorem 4.2]{DLS_Currents} in case $Q<\frac{p}{2}$ and to the analogous terms in \cite[Theorem 14.2]{DLHMS_linear} when $Q=\frac{p}{2}$. This implies
\begin{gather}
{\rm Err}_1^o = - Q \int_\cM \varphi_r \langle H_\cM, \etab\circ N\rangle,\label{e:outer_resto_1}\\
|{\rm Err}_2^o| \leq C_0 \int_\cM |\varphi_r| |A|^2|N|^2,\label{e:outer_resto_2_bis}\\
|{\rm Err}_3^o| \leq C_0 \int_\cM \big(|N| |A| + |DN|^2 \big) \big( |\varphi_r| |DN|^2  + |D\varphi_r| |DN| |N| \big)\label{e:outer_resto_3_bis}\,,
\end{gather}
where $H_\cM$ is the mean curvature vector of $\cM$. In particular we conclude
\begin{equation}\label{e:ov_con_errori}
\left| \bD (r) - r^{-1} \bE (r)\right| \leq \sum_{j=1}^5 \left|\textup{Err}^o_j\right|\, ,
\end{equation}
where ${\rm Err}_4^o$ and ${\rm Err}_5^o$ denote the terms
${\rm Err}_4$ and ${\rm Err}_5$ of \eqref{e:Err4-5} when $X=X_o$.

\medskip

We follow the same arguments with $X=X_i$, applying this time \cite[Theorem 4.3]{DLS_Currents} for $Q<\frac{p}{2}$ and \cite[Theorem 14.3]{DLHMS_linear} for $Q=\frac{p}{2}$. In particular using the formulas \cite[(3.29)\&(3.30)]{DLS_Blowup} for ${\rm div}_\cM Y$ and $D_{\cM} Y$ we conclude
 \begin{equation}\label{e:iv_con_errori}
\left| \bD' (r) - (m-2)r^{-1} \bD (r) - 2 r^{-2} \bG (r)\right|
\leq C_0 \bD (r) + \sum_{j=1}^5 \left|{\rm Err}^i_j\right|\, ,
\end{equation}
where
\begin{gather}
{\rm Err}_1^i = - Q \int_{ \cM}\big( \langle H_\cM, \etab \circ N\rangle\, {\rm div}_{\cM} Y + \langle D_Y H_\cM, \etab\circ N\rangle\big)\, ,\label{e:inner_resto_1_bis}\\
|{\rm Err}_2^i| \leq C_0 \int_\cM |A|^2 \left(|DY| |N|^2  +|Y| |N|\, |DN|\right), \label{e:inner_resto_2_bis}\\
|{\rm Err}_3^i|\leq C_0 \int_\cM \Big( |Y| |A| |DN|^2 \big(|N| + |DN|\big) + |DY| \big(|A|\,|N|^2 |DN| + |DN|^4\big)\Big)\label{e:inner_resto_3_bis}\, ,
\end{gather}
and where ${\rm Err}_4^i$ and ${\rm Err}_5^i$ denote the terms
${\rm Err}_4$ and ${\rm Err}_5$ of \eqref{e:Err4-5} when $X=X_i$.

\subsection{Error estimates} We next proceed as in \cite[Section 4]{DLS_Blowup}. First of all, since the structure and estimates on the size of the cubes of the Whitney decomposition are exactly the same, we can define the regions of \cite[Section 4.1]{DLS_Blowup} and deduce the same conclusions of \cite[Lemma 4.4]{DLS_Blowup}. Next,  since our estimates in Theorem \ref{t:approx} have the same structure of \cite[Theorem 2.4]{DLS_Center}, we conclude the validity of all the estimates in \cite[Section 4.2]{DLS_Blowup}.  In turn we can repeat all the arguments in \cite[Section 4.3]{DLS_Blowup} to conclude the same estimates for the terms of type ${\rm Err}_1^o, {\rm Err}_1^i, {\rm Err}_2^o, {\rm Err}_2^i, {\rm Err}_3^o, {\rm Err}_3^i, {\rm Err}_4^o, {\rm Err}_4^i, {\rm Err}_5^o$. Some more care is needed to handle the term ${\rm Err}_5^i$. First of all we split the latter error into the terms $I_1$ and $I_2$ of \cite[Page 596]{DLS_Blowup}. The term $I_1$ is estimated in the same way. Fo
 r the term $I_2$ we can use the same argument when $Q<\frac{p}{2}$ and hence $F$ is $\mathcal{A}_Q$-valued. However, we need a small modification in the case $Q=\frac{p}{2}$, when $F$ is $\mathscr{A}_Q$-valued. 

As in \cite[Page 597]{DLS_Blowup} we start by introducing an orthonormal frame $\nu_1, \ldots, \nu_l$
for $T_q\Sigma^\perp$ of class $C^{2,\eps_0}$ (cf.~\cite[Appendix A]{DLS_Currents}) and set 
\[
h^j_q (\vec \lambda) := - \sum_{k=1}^m \langle D_{v_k}\nu_j (q),  v_k \rangle
\]
whenever $v_1\wedge\ldots\wedge v_m = \vec \lambda$ is an $m$-vector of $T_q \Sigma$, with 
$v_1, \ldots, v_m$ orthonormal. 

Next, we recall the canonical decomposition of $\cM$ into $\cM_+$, $\cM_-$ and $\cM_0$ induced by $F$ (see Section \ref{sec:guide}) and correspondingly, we decompose the image of $F$ into
\begin{align}
\im_0 (F) &:= \im (F) \cap \p^{-1} (\cM_0)\\
\im_+ (F) &:= \im (F) \cap \p^{-1} (\cM_+)\\
\im_- (F) &:= \im (F) \cap \p^{-1} (\cM_-)\, .
\end{align}
If $q\in \im (F)$, as in \cite[Page 597]{DLS_Blowup} we set
\[
h^j_{\p(q)} := h^j_{\p(q)} (\vec \cM (\p (q)))
\quad \text{and}\quad
h_{\p(q)} = \sum_{j=1}^l h^j_{\p(q)}\, \nu_j(\p(q)).
\]
If $q\in \im_0 (F)\cup \im_+ (F)$, as in \cite[Page 597]{DLS_Blowup} we set
\[
h^j_q := h^j_q (\vec\bT_F (q)) \quad \text{and}\quad
h_q = \sum_{j=1}^l h^j_q\,  \nu_j(q)\, .
\]
We proceed however differently for $q\in \im_- (F)$: in this case we set
\[
h^j_q := h^j_q (- \vec\bT_F (q))\quad \text{and}\quad
h_q = \sum_{j=1}^l h^j_q\,  \nu_j(q)\, .
\]
Observe that, since for $q\in \im_- (F)$ we have $- \vec\bT_F (q) = \vec\bT_{F^-} (q)$, in practice we can follow precisely the same computations of
\cite[Page 597]{DLS_Blowup} in each of the regions $\im_0 (F), \im_+ (F)$ and $\im_- (F)$, to conclude
\begin{align}\label{e:pezzo lineare}
& \langle X_i(q), h_q \rangle  = {} \langle X_i(q), (h_q - h_{\p(q)}) \rangle = \sum_j \langle X_i(\p(q)), D\nu_j(\p(q))\cdot \mathbf{ex}^{-1}_{\p (q)} (q)\rangle h^j_{\p(q)}\notag\\
& \qquad\qquad\qquad\qquad\qquad\qquad\qquad\qquad + \sum_j  \langle \nu_j (q), X_i (q)\rangle \big(h^j_q - h^j_{\p (q)}\big)+
O\left(|q-\p(q)|^2\right)\notag\\
={}&\sum_j \langle X_i(\p(q)), D\nu_j(\p(q))\cdot \mathbf{ex}^{-1}_{\p (q)} (q)\rangle h^j_{\p(q)}\nonumber\\
&\qquad\qquad+ O\left(|\vec \bT_F (q) - \vec \cM (\p (q))|_{no} |q-\p(q)|+ |q-\p(q)|^2\right),
\end{align}
Observe that the only difference with \cite[(4.17)]{DLS_Blowup} is that $|\vec \bT_F (q) - \vec \cM (\p (q))|_{no}$ replaces 
$|\vec \bT_F (q) - \vec \cM (\p (q))|$ in the last line of the above estimate. Next, for $q\in \spt (\bT_F)$, we can bound $|q - \p (q)| \leq |N (q)|$ and $|\vec \bT_F (q) - \vec \cM (\p (q))|_{no} \leq C |DN (\p (q))|$. We therefore conclude the estimate
\begin{equation*}
\langle X_i(q), h_q \rangle = \sum_j \langle X_i(\p(q)), D\nu_j(\p(q))\cdot \mathbf{ex}^{-1}_{\p (q)} (q)\rangle h^j_{\p(q)} + O \big(|N|^2 (\p (q)) +
|DN|^2 (\p (q))\big)\, .
\end{equation*}
Combining the latter inequality with \cite[Theorem 13.1]{DLHMS_linear} we can bound
\begin{align*}
I_2^i & = \left|\int \langle X_i, h_q \rangle d\|\bT_F\| \right|
= \left| \sum_{i=1}^Q\int_\cM \langle Y, h_{F_i} \rangle \bJ F_i \right|
\allowdisplaybreaks\\
&\stackrel{\eqref{e:pezzo lineare}}{\leq} \left|\int_\cM \sum_{j=1}^l \sum_{i=1}^Q \langle Y (x), D\nu_j  (x) \cdot \mathbf{ex}^{-1}_{x} (F_i (x))\rangle h_{x}^j d\mathcal{H}^m (x)\right| + C\,\int_\cM \varphi_r (|N|^2 + |DN|^2)
\end{align*}
We can now proceed as in \cite[Page 598]{DLS_Blowup} to conclude the same estimate for $I_2$.

\section{Boundedness of the frequency function and reverse Sobolev}

We next show the counterpart of \cite[Theorem 5.1]{DLS_Blowup}. 

\begin{theorem}[Boundedness of the frequency functions]\label{t:boundedness}
Let $T$ be as in Proposition \ref{p:contradiction_sequence}. 
If the intervals of flattening are $j_0 < \infty$, then there is $\rho>0$ such that
\begin{equation}\label{e:finita1}
\bH_{j_0} >0 \mbox{ on $]0, \rho[$} \quad \mbox{and} \quad \limsup_{r\to 0} \bI_{j_0} (r)< \infty\, .
\end{equation}
If the intervals of flattening are infinitely many, then there is a number $j_0\in \mathbb N$ { and a geometric constant $j_1\in \mathbb N$} such that
\begin{equation}\label{e:finita2}
\bH_j>0 \mbox{ on $]\frac{s_j}{t_j}, { 2^{-j_1} 3}[$ for all $j\geq j_0$}\, , \qquad 
\sup_{j\geq j_0} \sup_{r\in ]\frac{s_j}{t_j}, { 2^{-j_1} 3}[} \bI_j (r) <\infty\, ,
\end{equation}
{
\begin{equation}\label{e:finita3}
\sup \left\{\min \left\{\bI_j (r), \frac{r^2 \int_{\cB_r} |DN_j|^2}{\int_{\cB_r} |N_j|^2}\right\}\colon j \geq j_0 \mbox{ and } \max \left\{\frac{s_j}{t_j},\frac{3}{2^{j_1}}\right\} \leq r < 3\right\} < \infty\, 
\end{equation}
(in the latter inequality we understand $\bI_j (r) = \infty$ when $\bH_j (r)=0$)}. 
\end{theorem}

\begin{proof}
In the first case we can appeal to the same argument as in \cite[Page 599]{DLS_Blowup}. In the second case we also proceed as in \cite[Page 599]{DLS_Blowup} and partition the extrema $t_j$ of the intervals of flattening into two subsets: the class (A) formed by those $t_j$ such that $t_j = s_{j-1}$ and the complementary class (B). As in \cite[Page 599]{DLS_Blowup} we can assume that $j$ is large enough. In the first case we proceed as in \cite[Page 599]{DLS_Blowup} where we substitute \cite[Proposition 3.7]{DLS_Center} with Proposition \ref{p:compara}. In case (B) by construction there is $\eta_j \in ]0,1[$ such that
$\bE^{no} ((\iota_{0, t_j})_\sharp T, \B_{6\sqrt{m}(1+\eta_j)}) > \eps_3^2$. Up to extraction of a subsequence, we can assume that $T_j = (\iota_{0, t_j})_\sharp T$ converges to a cone $S$: the convergence is strong enough to conclude that the excess
of the cone is the limit of the excesses of the sequence. Moreover (since $S$ is a cone), the excess $\bE^{no} (S, \bB_r)$ is independent of $r$. We then conclude 
\[
\eps_3^2 \leq  \liminf_{j\to\infty, j\in (B)} \bE^{no} (T_j, \B_3)\, .
\]
We then argue as in \cite[Page 601]{DLS_Blowup} using Lemma \ref{l:flatness} below in place of \cite[Lemma 5.2]{DLS_Blowup}.
\end{proof}

\begin{lemma}\label{l:flatness}
Assume the intervals of flattening are infinitely many and
$r_j \in ]\frac{s_j}{t_j},3[$ is a subsequence
(not relabeled) with $\lim_j \|N_j\|_{L^2(\cB_{r_j} \setminus \cB_{r_j/2})} = 0$. If $\eps_3$ is sufficiently small,
then, $\bE^{no} (T_{j},\B_{r_j}) \to 0$. 
\end{lemma}

\begin{proof} The argument is a modification of that of \cite[Lemma 5.2]{DLS_Blowup}, which we include for the reader's convenience.
First of all note that, by Proposition \ref{p:flattening}, $\bE^{no} (T_j, \bB_{r_j}) \to 0$ if $r_j\to 0$. Hence, passing to a subsequence, we can assume the existence of a $c>0$ such that
\begin{equation}\label{e:nontrivial}
r_j \geq c \quad \text{and} \quad \bE^{no} (T_{j}, \B_{6\sqrt{m}})\geq c.
\end{equation}
After the extraction of a further subsequence, we can assume the existence of $r$ such that
\begin{equation}\label{e:(A1)_senza_scale}
\int_{\cB_{r} \setminus \cB_{\frac{3 r}{4}}} |N_{j}|^2 \to 0,
\end{equation}
and the existence of a $\modp$ area-minimizing cone $S$ such that $(\iota_{0, t_j})_\sharp T\to S$. Recall that $S$ is a representative $\modp$. By \eqref{e:nontrivial}, the cone $S$ cannot be an integer multiple of an $m$-dimensional plane. 

We argue as in \cite[Pages 601-602]{DLS_Blowup} and conclude that, if $\cM$ is the limit of a subsequence (not relabeled) of the $\cM_j$, then 
there are two radii $0<s<t$ such that $\spt (S) \cap \bB_t (0) \setminus \bB_s (0) \subset \cM$. In particular, by the Constancy Theorem $\modp$ we conclude that $S\res  \bB_t (0) \setminus \bB_s (0) = Q_0 \a{\cM\cap \bB_t (0) \setminus \bB_s (0)} \; \modp$ for an integer $Q_0$ with $\abs{Q_0} \leq \frac{p}{2}$. Since $S$ is a cone and a representative $\modp$ we can in fact infer that $S \res \bB_t (0) = Q_0 \a{0} \cone \a{\cM \cap \partial \bB_t (0)} \; \modp$ (in fact it can be easily inferred from the argument in \cite[Pages 601-602]{DLS_Blowup} that $Q_0 =Q$, although this is not needed in our argument). Since $\a{0}\cone \a{\cM \cap \partial \bB_t (0)}$ induces a stationary varifold and $\cM$ is the graph of a function with small $C^{3,\varepsilon_0}$ norm, we can applied Allard's Theorem to conclude that in fact $\a{0}\cone \a{\cM \cap \partial \bB_t (0)}$ is smooth. This implies that the latter is in fact $\a{\pi\cap \bB_t (0)}$ for some $m$-dimensional plane $\pi$, contradicting the fact that $S$ is not a flat cone. 
\end{proof}

Finally, Theorem \ref{t:boundedness} can be used as in \cite[Section 5]{DLS_Blowup} to show \cite[Corollary 5.3]{DLS_Blowup}, which we restate in our context for the reader's convenience. 

\begin{corollary}[Reverse Sobolev]\label{c:rev_Sob}
Let $T$ be as in Proposition \ref{p:contradiction_sequence}.
Then, there exists a constant $C>0$ which {\em depends on $T$ but not on $j$} such that, for every $j$
and for every $r \in ]\frac{s_j}{t_j}, 1 ]$, there is $\sigma\in ]\frac{3}{2}r, 3r]$
such that
\begin{equation}\label{e:rev_Sob}
\int_{\cB_{\sigma}(\Phii_j(0))} |D{N}_j|^2 \leq \frac{C}{r^2} \int_{\cB_{\sigma} (\Phii_j(0))} |{N}_j|^2\, .  
\end{equation}
\end{corollary}

\section{Final contradiction argument}

In this section we complete the proof of Theorem \ref{t:main} showing that, by Proposition \ref{p:contradiction_sequence}, under the assumption that the theorem is false, we get a contradiction. In particular fix $T, \Sigma, \Omega$ and $r_k$ as in Proposition \ref{p:contradiction_sequence}. We have already remarked that for each $k$ there is an interval of flattening $I_{j(k)} = ]s_{j(k)}, t_{j(k)}]$ containing $r_k$. We proceed as in \cite[Section 6]{DLS_Blowup} and introduce the following new objects:
\begin{itemize}
\item We first apply Corollary \ref{c:rev_Sob} to $r= \frac{r_k}{t_{j(k)}}$ and set $\bar s_k := t_{j(k)} \sigma_k$, so that $\frac{\bar s_k}{t_{j(k)}} \in \big]\frac32 \frac{r_k}{t_{j(k)}}, 3\frac{r_k}{t_{j(k)}}[$.
\item We set $\bar r_k := \frac{2\bar s_k}{3 t_{j(k)}}$.
\item We rescale our geometric objects, namely
\begin{itemize}
\item[(U1)] The currents $\bar T_k$, the manifolds $\bar \Sigma_k$ and the center manifolds $\bar\cM_k$ are given respectively by
\begin{align}
\bar T_k &= (\iota_{0,\bar r_k})_\sharp T_{j(k)} = ((\iota_{0,\bar r_k t_{j(k)}})_\sharp T) \res \bB_{6\sqrt{m}/\bar{r}_k}\\
\bar \Sigma_k &= \iota_{0, \bar r_k} (\Sigma_{j(k)}) = \iota_{0, \bar r_k t_{j(k)}} (\Sigma)\\
\bar \cM_k &= \iota_{0, \bar r_k} (\cM_{j(k)})\, .
\end{align}
\item[(U2)] In order to define the rescaled maps $\bar N_k$ on $\bar \cM_k$ we need to distinguish two cases. When $Q<\frac{p}{2}$, the map $\bar N_k$ takes values in $\mathcal{A}_Q (\R^{m+n})$ and is defined by
\[
\bar N_k (q) = \sum_{i=1}^Q \a{r^{-1} (N_{j(k)})_i (r q)}\, .
\]
In the case $Q=\frac{p}{2}$, the map $\bar N_k$ takes values in $\mathscr{A}_Q (\R^{m+n})$ and is defined analogously. The reader might either use the decomposition of $\cM_{j(k)}$ into $(\cM_{j(k)})_+, (\cM_{j(k)})_-$ and $(\cM_{j(k)})_0$ or, using the original notation in \cite[Definition 2.2]{DLHMS_linear},
\[
\bar N_k (q) = \left(\sum_{i=1}^Q \a{r^{-1} (N_{j(k)})_i (r q)}, \epsilon (rq)\right),
\]
where
\[
N_{j(k)} (\tilde{q}) = \left(\sum_{i=1}^Q \a{(N_{j(k)})_i (\tilde{q})}, \epsilon (\tilde{q})\right)
\]
and $\eps(\cdot) \in \{-1,1\}$.
\end{itemize}
\end{itemize}
Without loss of generality we can assume that $T_{0} \Sigma =\R^{m+\bar{n}}\times \{0\}$, thus
the ambient manifolds $\bar{\Sigma}_k$ converge to 
$\R^{m+\bar{n}}\times \{0\}$ locally in $C^{3,\eps_0}$. Observe in addition that $\frac12 < \frac{r_k}{\bar r_k t_{j(k)}} < 1$ and hence it follows from 
Proposition~\ref{p:contradiction_sequence}(ii) that
\[
\bE^{no} (\bar T_k, \bB_{\frac12}) \leq C\bE^{no}(T, \bB_{r_k}) \to 0.
\]
Indeed Proposition~\ref{p:contradiction_sequence}(ii) implies that $\bar T_k$ converge to $Q \a{\pi_0}$ both in the sense of varifolds and in the sense of currents $\modp$. Finally, we recall that, by Proposition~\ref{p:contradiction_sequence}(iii)\&(iii)s,
\begin{align}
\cH^{m-2+\alpha}_\infty (\rD_Q (\bar{T}_k) \cap \bB_1)
&\geq C_0 r_k^{- (m-2+\alpha)}\cH^{m-2+\alpha}_\infty (\rD_Q ({T}) \cap \bB_{r_k})\geq \eta>0\, \qquad \mbox{when $Q<\frac{p}{2}$}\label{e:sing grande}\\
\cH^{m-1+\alpha}_\infty (\rD_Q (\bar{T}_k) \cap \bB_1)
&\geq C_0 r_k^{- (m-1+\alpha)}\cH^{m-1+\alpha}_\infty (\rD_Q ({T}) \cap \bB_{r_k})\geq \eta>0\, \qquad \mbox{when $Q=\frac{p}{2}$}\label{e:sing grande_2}
\end{align}
where $\alpha$ is a positive number and $C_0$ a geometric constant. 

As in \cite[Section 6]{DLS_Blowup} we claim the counterpart of \cite[Lemma 6.1]{DLS_Blowup}, namely Lemma \ref{l:vanishing}, which implies
that $\bar\cM_k$ converge locally to
the flat $m$-plane $\pi_0$. We also introduce the exponential maps $\mathbf{ex}_k: B_3\subset\R^m\simeq T_{\bar{q}_k} \bar\cM_k\to \bar\cM_k$ at $\bar{q}_k = \Phii_{j(k)} (0)/\bar{r}_k$ ({ here and in what follows we assume, w.l.o.g.,
to have applied a suitable rotation to each $\bar{T}_k$ so that the tangent plane $T_{\bar{q}_k} \bar{\mathcal{M}}_k$ coincides with
$\mathbb R^m\times \{0\}$}). We are finally ready to define the blow-up maps $N^b_k : B_{3}\subset\R^m \to \Iq (\R^{m+n})$, when $Q<\frac{p}{2}$ and $N^b_k : B_{3}\subset\R^m \to \mathscr{A}_Q (\R^{m+n})$, when $Q=\frac{p}{2}$:
\begin{equation}\label{e:sospirata_successione}
N^b_k (x) := \bh_k^{-1} \bar{N}_k (\mathbf{ex}_k (x))\, ,
\end{equation}
where $\bh_k:=\|\bar N_k\|_{L^2(\cB_{\frac32})}$. 

\begin{lemma}[Vanishing lemma]\label{l:vanishing}
Let $\bar{T}_k, \bar{r}_k, \bar\cM_k$ and $\bar\Sigma_k$ be as above. We then have:
\begin{itemize}
\item[(i)] $\min \{\bmo^{j(k)}, \bar r_k\}\to 0$; 
\item[(ii)] the rescaled center manifolds $\bar{\cM}_k$ converge
(up to subsequences) to
$\pi_0 = \R^m\times \{0\}$ in $C^{3,\kappa/2} (\bB_4)$ and
the maps $\mathbf{ex}_k$ converge in $C^{2, \kappa/2}$ to the identity map
${\rm id}: B_3 \to B_3$;
\item[(iii)] there exists a constant $C>0$, depending only on $T$, such that, for every $k$,
\begin{equation}\label{e:rev_Sob2}
\frac{1}{\bh_k^2} \int_{\cB_{\frac32}} |D\bar N_k|^2 \leq C \int_{B_{\frac32}} |D{N}^b_k|^2 \leq C. 
\end{equation}
\end{itemize}
\end{lemma}

\begin{proof} The argument for (i) can be taken from \cite[Proof of Lemma 6.1]{DLS_Blowup}. As for part (ii) the argument given in \cite[Section 6]{DLS_Blowup} for the convergence of the center manifolds can be shortened considerably observing that it is a direct consequence of Proposition \ref{p:flattening}(v) and the convergence of the currents $\bar{T}_k$. The $C^{2, \kappa/2}$ convergence of the exponential maps follow then immediately from \cite[Proposition A.4]{DLS_Blowup}. 
Finally, (iii) is an obvious consequence of Corollary~\ref{c:rev_Sob}.
\end{proof}

Having defined the blow-up maps, the final contradiction comes from the following statements.

\begin{theorem}[Final blow-up]\label{t:sospirato_blowup}
Up to subsequences, the maps $N^b_k$ converge strongly in $L^2(B_\frac32)$ to:
\begin{itemize} 
\item a function $N^b_\infty: B_{\frac32} \to \Iq (\{0\}\times \R^{\bar{n}}\times \{0\})$ when $Q <\frac{p}{2}$;
\item a function $N^b_\infty: B_{\frac32} \to \mathscr{A}_Q (\{0\}\times \R^{\bar{n}}\times \{0\})$ when $Q =\frac{p}{2}$.
\end{itemize}
Such limit is $\Dir$-minimizing in $B_t$ for every $t\in ]\frac{5}{4}, \frac{3}{2}[$ and satisfies
$\|N^b_\infty\|_{L^2(B_\frac32)}=1$ and $\etab\circ N^b_\infty \equiv 0$.
\end{theorem}

\begin{theorem}[Large singular set]\label{t:large}
Let $N^b_\infty$ be the map of Theorem \ref{t:sospirato_blowup} and define
\[
\Upsilon := \left\{x\in \bar B_{1} : N^b_\infty(x) = Q\a{0}\right\}\, .
\]
Then 
\begin{align}
&\cH^{m-2+\alpha}_\infty (\Upsilon) \geq \frac{\eta}{2} \qquad \mbox{if $Q<\frac{p}{2}$},\\
&\cH^{m-1+\alpha}_\infty (\Upsilon) \geq \frac{\eta}{2} \qquad \mbox{if $Q=\frac{p}{2}$},
\end{align} 
where $\alpha$ and $\eta$ are the positive constants in \eqref{e:sing grande}, resp. \eqref{e:sing grande_2}.
\end{theorem}

The two theorems would contradict \cite[Theorem 0.11]{DLS_Qvfr} in case $Q<\frac{p}{2}$ since, arguing as in \cite[Section 6]{DLS_Blowup} we easily conclude that $\Upsilon$ is a subset of the singularities of $N^b_\infty$. In the case $Q=\frac{p}{2}$ we infer instead from \cite[Proposition 10.3]{DLHMS_linear} that $N_b^\infty = Q \a{\bfeta \circ N_b^\infty}$ on the whole $B_{3/2}$, which in turn would imply $N_b^\infty = Q \a{0}$. This however contradicts 
$\|N_b^\infty\|_{L^2 (B_{3/2})} =1$. 

\subsection{Proof of Theorem \ref{t:sospirato_blowup}} Without loss of generality we may assume that $\bar{q}_k := \bar r_k^{-1} \Phii_{j(k)} (0)$ coincide all with the origin. We then define a new map $\bar{F}_k$ on the geodesic ball $\cB_{3/2}\subset \bar \cM_k$ distinguishing, as usual, the two cases $Q<\frac{p}{2}$ and $Q=\frac{p}{2}$. In the first case we follow the definition of \cite[Section 7.1]{DLS_Blowup}, namely we set
\[
\bar{F}_k (x) := \sum_i \a{x+ (\bar{N}_k)_i (x)}\, .
\] 
In the case $Q=\frac{p}{2}$ the map $\bar F_k$ takes values in $\mathscr{A}_Q (\R^{m+n})$ and it is induced by $\bar N_k$ in the sense explained at point (N) of \cite[Assumption 11.1]{DLHMS_linear}. 
The argument given in \cite[Section 7.1]{DLS_Blowup} works in our case as well and implies the following estimates (where $\gamma$ is some positive exponent independent of $k$)
\begin{gather}
\Lip (\bar{N}_k|_{\cB_{\sfrac{3}{2}}}) \leq C \bh_k^\gamma\quad\text{and}\quad \|\bar{N}_k\|_{C^0 (\cB_{\sfrac32})} \leq C (\bmo^{j(k)} \bar{r}_k)^\gamma,\label{e:Lip_riscalato}\\
\mass^{p}((\bT_{\bar{F}_k} - \bar{T}_k) \res (\p_k^{-1} (\cB_\frac32)) \leq C \bh_k^{2+2\gamma},\label{e:errori_massa_1000}\\
\int_{\cB_\frac32} |\etab\circ \bar{N}_k| \leq C \bh_k^2\label{e:controllo_media}\, 
\end{gather}
From these estimates we conclude the strong $L^2$ converge of (a subsequence of) $N_k^b$ to a map $N_\infty^b$ on $B_{3/2}$ taking values, respectively, on $\Iq (\{0\}\times \R^{\bar{n}}\times \{0\})$ (when $Q<\frac{p}{2}$) and $\mathscr{A}_Q (\{0\}\times \R^{\bar{n}}\times \{0\})$ (when $Q=\frac{p}{2}$). Moreover it is obvious that $\|N_\infty^b\|_{L^2 (B_{3/2})} =1$ and that $\etab\circ N_\infty^b \equiv 0$. Therefore we are only left with proving that $N_\infty^b$ is $\Dir$-minimizing. 

Proceeding as in the \cite[Section 7]{DLS_Blowup} we assume, without loss of generality, that the Dirichlet energy of $N_\infty^b$ is nontrivial (otherwise there is nothing to prove). Thus we can assume that
that there exists $c_0>0$ such that
\begin{equation}\label{e:reverse_control}
c_0 \bh_k^2 \leq \int_{\cB_\frac32} |D\bar{N}_k|^2\, .
\end{equation}
We proceed as in \cite[Section 7.2 \& Section 7.3]{DLS_Blowup}: if there is a radius $t \in \left]\frac54,\frac32\right[$ and a 
function $f$ on $B_\frac32$ (taking values in $\Iq(\R^{\bar n})$ when $Q<\frac{p}{2}$, or in $\mathscr{A}_Q (\R^{\bar n})$ when $Q=\frac{p}{2}$) such that
\[
f\vert_{B_\frac32\setminus B_t} = N^b_\infty\vert_{B_\frac32 \setminus B_t} \quad\text{and}\quad
\Dir (f, B_t) \leq \Dir (N^b_\infty, B_t) -2\,\delta,
\]
for some $\delta>0$, we then produce competitors $\tilde{N}_k$ for the maps $\bar N_k$ satisfying
\begin{gather*}
\tilde{N}_k \equiv \bar{N}_k \quad \mbox{in $\cB_\frac32\setminus\cB_t$},\quad
\Lip (\tilde{N}_k) \leq C \bh_k^\gamma, \quad |\tilde{N}_k| \leq C (\bmo^k \bar{r}_k)^\gamma,\\
\int_{\cB_\frac32} |\etab\circ \tilde{N}_k| \leq C \bh_k^2
\quad\text{and}\quad
\int_{\cB_\frac32} |D\tilde{N}_k|^2 \leq \int_{\cB_\frac32} |D\bar{N}_k|^2 - \delta \bh_k^2.
\end{gather*}
Indeed the construction of the maps in \cite[Section 7.2 \& Section 7.3]{DLS_Blowup} uses the left composition of $\mathcal{A}_Q$-valued maps with classical maps in the sense of \cite[Section 1.3.1]{DLS_Qvfr}, which in the $\mathscr{A}_Q$-valued case is substituted by the left composition as defined in 
\cite[Subsection 7.3]{DLHMS_linear}. 

Consider next the map $\tilde{F}_k$ given by $\tilde{F}_k (x) = \sum_i \a{x+(\tilde{N}_k)_i (x)}$ in the case $Q<\frac{p}{2}$ and by the corresponding
$\left(\sum_i  \a{x+(\tilde{N}_k)_i (x)}, \varepsilon (x)\right)$ in the case $Q=\frac{p}{2}$. The current $\bT_{\tilde{F}_k}$ coincides with 
$\bT_{\bar{F}_k}$ on $\p_k^{-1}( \cB_\frac32\setminus \cB_t)$. Define the function $\varphi_k(q) = \dist_{\bar{\cM}_k} (0, \p_k (q))$
and consider for each $s\in \left]t,\frac32\right[$ the slices 
$\langle \bT_{\tilde{F}_k} - \bar{T}_k, \varphi_k, s\rangle$.
By \eqref{e:errori_massa_1000} we have
\[
\int_t^\frac32 \mass^{p} (\langle \bT_{\tilde{F}_k} - \bar{T}_k, \varphi_k, s\rangle) \leq C \bh_k^{2+\gamma}\, .
\]
Thus we can find for each $k$ a radius $\sigma_k\in \left]t, \frac32\right[$ on which 
$\mass^{p} (\langle \bT_{\tilde{F}_k} - \bar{T}_k, \varphi_k, \sigma_k\rangle) \leq C \bh_k^{2+\gamma}$. Recall from Lemma \ref{modp_slicing_formula}(i), 
$\partial  \langle \bT_{\tilde{F}_k} - \bar{T}_k, \varphi_k, \sigma_k\rangle = 0$ $\modp$. By the isoperimetric
inequality $\modp$ (see \cite[$(4.2.10)^\nu$]{Federer69}) there is an integer rectifiable current $S_k$, which can be assumed to be representative $\modp$, such that
\[
\partial S_k = 
\langle \bT_{\tilde{F}_k} - \bar{T}_k, \varphi_k, \sigma_k\rangle \;\;\modp\, ,
\quad \mass (S_k) = \mass^p (S_k) \leq C \bh_k^{(2+\gamma)m/(m-1)}
\quad\text{and}\quad\spt (S_k) \subset \bar{\Sigma}_k.
\]
Our competitor current is, then, given by
\[
Z_k := \bar{T}_k\res (\p_k^{-1} (\bar\cM_k\setminus \cB_{\sigma_k})) + S_k + \bT_{\tilde{F}_k}\res
(\p_k^{-1} (\cB_{\sigma_k})).
\]
The computations given in \cite[Section 7.4]{DLS_Blowup} would then imply that the $p$-mass of $Z_k$ is strictly smaller than the mass of $\bar{T}_k$ for $k$ large enough, even though $\bar{T}_k - Z_k$ is a cycle $\modp$ supported in the ambient manifold $\bar\Sigma_k$, which is a contradiction to $\bar{T}_k$ being a mass minimizing current $\modp$ in $\bar\Sigma_k$. \qedhere

\subsection{Proof of Theorem \ref{t:large}} We argue by contradiction and assume that:
\begin{align}
&\cH^{m-2+\alpha}_\infty (\Upsilon) < \frac{\eta}{2} \qquad \mbox{if $Q<\frac{p}{2}$}\\
&\cH^{m-1+\alpha}_\infty (\Upsilon) < \frac{\eta}{2} \qquad \mbox{if $Q=\frac{p}{2}$.}
\end{align} 
Since $\Upsilon$ is compact, we cover $\Upsilon$ with finitely many balls $\{\bB_{\sigma_i} (x_i)\}$ in such a way that
\begin{align}
&\sum_i \omega_{m-2+\alpha} (4\sigma_i)^{m-2+\alpha} \leq 
\frac{\eta}{2}\qquad\mbox{if $Q<\frac{p}{2}$},\\
& \sum_i \omega_{m-1+\alpha} (4\sigma_i)^{m-1+\alpha} \leq 
\frac{\eta}{2}\qquad\mbox{if $Q=\frac{p}{2}$}
\end{align}
 Choose a $\bar{\sigma}>0$ so that
the $5\bar\sigma$-neighborhood of $\Upsilon$ is covered by $\{\bB_{\sigma_i} (x_i)\}$.
Denote by $\Lambda_{k}$ the set of multiplicity $Q$ points of $\bar T_k$ far away from the singular set
$\Upsilon$:
\[
\Lambda_{k} := \{ q \in {\rm D}_Q(\bar T_k) \cap \bB_1 : \dist(q, \Upsilon)>4 \bar{\sigma} \}.
\]
Clearly, 
\begin{align}
&\cH^{m-2+\alpha}_\infty (\Lambda_k) \geq \frac{\eta}{2}\qquad\mbox{when $Q<\frac{p}{2}$},\label{e:Lambda_Rocco}\\
&\cH^{m-1+\alpha}_\infty (\Lambda_k) \geq \frac{\eta}{2}\qquad\mbox{when $Q=\frac{p}{2}$}.\label{e:Lambda_Rocco_2}
\end{align} 
As in \cite[Section 6.2]{DLS_Blowup} we denote by $\mathbf{V}$ the neighborhood of $\Upsilon$ of size $2\bar{\sigma}$.
Agruing as in \cite[Section 6.2, Step 1]{DLS_Blowup} we conclude the existence of a positive constant $\vartheta$ such that, for every fixed parameter $\sigma < \bar{\sigma}$, there is a $k_0 (\sigma)$ 
such that the following estimate holds for every $k\geq k_0 (\sigma)$. In the case $Q<\frac{p}{2}$ we have
\begin{equation}\label{e:dall'alto}
\mint_{\cB_{2\sigma} (x)} \cG (\bar{N}_k, Q \a{\etab\circ \bar{N}_k})^2 \geq \vartheta \bh_k^2
\qquad \mbox{$\forall\;x\in \Xi_k:= \p_{\bar\cM_k} (\Lambda_k)$,}
\end{equation}
whereas in the case $Q=\frac{p}{2}$ we have
\begin{equation}\label{e:dall'alto_stronzo}
\mint_{\cB_{2\sigma} (x)} \cG_s (\bar{N}_k, Q \a{\etab\circ \bar{N}_k})^2 \geq \vartheta \bh_k^2
\qquad \mbox{$\forall\;x\in \Xi_k:= \p_{\bar\cM_k} (\Lambda_k)$.}
\end{equation}
Indeed the argument in \cite[Section 6.2]{DLS_Blowup} uses only the H\"older continuity of the $\Dir$-minimizing map $N_\infty^b$ (which
is a consequence of \cite[Theorem 2.9]{DLS_Qvfr} for $Q<\frac{p}{2}$ and a consequence of \cite[Theorem 8.1]{DLHMS_linear} when $Q=\frac{p}{2}$) and the strong convergence proved in Theorem \ref{t:sospirato_blowup}. 

Next, following \cite[Section 6.2, Step 2]{DLS_Blowup}, for every $q\in\Lambda_k$ we define $\bar{z}_k (q) = \p_{\pi_k} (q)$ (where $\pi_k$ is the reference plane for the center manifold
related to $T_{j(k)}$) and 
\[
\bar{x}_k (q) := (\bar{z}_k (q),  \bar{r}_k^{-1} \phii_{j(k)} (\bar{r}_k \bar z_k (q)))\, .
\] 
%$\varrho_q \in ]\bar{s} t(q), 2\sigma]$
Observe that $\bar{x}_k (q)\in \bar\cM_k$. We next claim the existence of 
a suitably chosen geometric constant $1>c_0>0$ (in particular, independent of $\sigma$) such that, when $k$ is large enough,
for each $q\in \Lambda_k$ there is a radius $\varrho_q \leq 2\sigma$ with the following properties:
\begin{gather}
\frac{c_0\, \vartheta}{\sigma^{\alpha}}  \bh_k^2\leq \frac1{\varrho_q^{m-2+\alpha}}
\int_{\cB_{\varrho_q} (\bar{x}_k (q))} |D\bar{N}_k|^2,
\label{e:dal_basso_finale}\\
\cB_{\varrho_q} (\bar{x}_k (q)) \subset \bB_{4\varrho_q} (q)\label{e:inside}\, .
\end{gather}
The argument given in \cite[Section 6.2, Step 2]{DLS_Blowup} can be routinously modified in our case. In particular we define the points $q_k :=\bar{r}_k  q$, $z_k := \bar{r}_k\bar{z}_k (q)$ and
$x_k = \bar{r}_k \bar{x}_k (q) =(z_k, \phii_{j(k)} (z_k))$ and discuss the three different possibilities depending on whether $z_k$ belongs to a cube $L\in \sW^{j(k)}$ or to the contact set $\Gamma_{j(k)}$.

The first case, $z_k \in L\in \sW_h^{j(k)}$ can be excluded with the same argument given in \cite[Section 6.2, Step 2]{DLS_Blowup}, where we replace \cite[Proposition 3.1]{DLS_Center} with Proposition \ref{p:separ}, because $q_k$ is a multiplicity $Q$ point for the current $T_{j(k)}$.

Following the argument in \cite[Section 6.2, Step 2]{DLS_Blowup}, when $z_k \in \sW_n^{j(k)}\cup \sW_e^{j(k)}$ we find a $t (q)\leq \sigma$ with the property that
\begin{equation}\label{e:sotto_dal_cm_2}
\mint_{\cB_{\bar{s} t (q)} (\bar{x}_k (q))} \cG_\square (\bar{N}_k, Q \a{\etab\circ \bar{N}_k})^2 \leq 
\frac{\vartheta}{4 \omega_m t (q)^{m-2}} \int_{\cB_{t (q)} (\bar{x}_k (q))} |D\bar{N}_k|^2\, 
\end{equation}
(where $\square = s$ for $Q=\frac{p}{2}$ and $\square = \phantom{s}$ for $Q<\frac{p}{2}$)
and 
\begin{equation}\label{e:vicinanza_2}
|q- \bar{x}_k (q)| < \bar s\, t (q).
\end{equation}
In the argument \cite[Section 6.2]{DLS_Blowup} we take care of substituing 
\cite[Proposition 3.5]{DLS_Center}, \cite[Lemma 6.1]{DLS_Blowup} and \cite[Proposition 3.6]{DLS_Blowup} respectively with
Proposition \ref{p:splitting_II}, Lemma \ref{l:vanishing} and Proposition \ref{p:persistence}.

In the case $z_k\in \Gamma_{j (k)}$ we find a $t (q)<\sigma$ such that
\begin{equation}\label{e:ancora_un_caso}
 \mint_{\cB_{\bar{s} t (q)} (\bar{x}_k (q))} \cG_\square (\bar{N}_k, Q \a{\etab\circ \bar{N}_k})^2
\leq \frac{\vartheta}{4} \bh_k^2\, , 
\end{equation}
whereas we observe that \eqref{e:vicinanza_2} holds trivially because the left hand side vanishes.

By \eqref{e:vicinanza_2}, for any $\varrho_q\in ]s\bar{t} (q), 2\sigma]$ the inclusion \eqref{e:inside} holds.
The argument is then closed by showing that \eqref{e:dal_basso_finale} must hold for at least one $\varrho_q \in ]\bar{s} t(q), 2\sigma]$.
The rest of the argument in \cite[Section 6.2, Step 2]{DLS_Blowup} uses the Poincar\'e inequality in the $\mathcal{A}_Q$-valued setting to show that, under the assumption that \eqref{e:dal_basso_finale} fails for every $\varrho \in ]\bar{s} t(q), 2\sigma]$, \eqref{e:ancora_un_caso} and \eqref{e:sotto_dal_cm_2} would be incompatible with \eqref{e:dall'alto}. This argument then settles the proof of the existence of $\varrho_q$ satsifying
\eqref{e:dal_basso_finale}-\eqref{e:inside} when $Q<\frac{p}{2}$. Since the analogous Poincar\'e inequality can be easily seen to hold in the 
$\mathscr{A}_Q$-valued case, we easily conclude that the same argument applies when $Q=\frac{p}{2}$ exploiting the case $\square=s$ for 
\eqref{e:sotto_dal_cm_2} and \eqref{e:ancora_un_caso} against \eqref{e:dall'alto_stronzo}.

From \eqref{e:dal_basso_finale}-\eqref{e:inside} we can use the covering argument of \cite[Step 3]{DLS_Blowup} to conclude that the inequality
\eqref{e:Lambda_Rocco} and \eqref{e:Lambda_Rocco_2} would force a large Dirichlet energy of $\bar{N}_k$ on $\mathcal{B}_{3/2}$, in particular
\begin{align}
&\frac{\eta}{2} \leq \frac{C_0}{c_0}\frac{\sigma^{\alpha}}{\vartheta\bh_k^2} \int_{\cB_{\frac32}} |D\bar{N}_k|^2\,\qquad\mbox{for $Q<\frac{p}{2}$}\,,\\
&\frac{\eta}{2} \leq \frac{C_0}{c_0}\frac{\sigma^{1+\alpha}}{\vartheta\bh_k^2} \int_{\cB_{\frac32}} |D\bar{N}_k|^2\, \qquad\mbox{for $Q=\frac{p}{2}$}\,,
\end{align}
where $C_0, c_0$ and $\vartheta$ are fixed (namely independent of $\sigma$). Therefore, $\sigma$ can be chosen very small, with
the inequality being satisfied only for $k\geq k (\sigma)$. However, the arbitrariness of $\sigma$ and \eqref{e:rev_Sob2} would be incompatible with $\eta > 0$, thus leading to the required contradiction.

\newpage

\part{Rectifiability of the singular set and structure theorem} \label{part:Rect}

\section{Rectifiability of the singular set: proof of Theorem \ref{t:main10}}

We start by introducing the term ``area minimizing cones $\modp$'' for area minimizing currents $\modp$ without boundary $\modp$ which have a representative $T_0$ which is a cone in the sense of Corollary \ref{c:tangent_cones}(iii). Such cone will be called flat if it is supported
in some $m$-dimensional plane $\pi\subset \mathbb R^{m+n}$. We recall that, by Corollary \ref{c:tangent_cones}, any flat area minimizing cone $\modp$ is congruent $\modp$ to $Q \a{\pi}$, where $\pi$ is an $m$-dimensional plane and $Q$ is an integer with $0\leq Q \leq \frac{p}{2}$. For odd $p$ we then conclude that $|Q|\leq \frac{p-1}{2}$.

\smallskip

Recall the definition of $k$-symmetric cones given in Definition \ref{def:symmetric cones and stratification}. Following \cite{NV}, we introduce next the following terminology, which introduces a suitable notion of local \emph{almost} symmetry for a given integral varifold $V$. 

\begin{definition}
An $m$-dimensional integral varifold $V$ is $(k, \varepsilon)$-symmetric in the ball $\bB_\rho (x)$ if there is a $k$-symmetric cone 
$C$ such that the varifold distance between $C\res \bB_1 (0)$ and $((\iota_{x,\rho})_\sharp V)\res \bB_1 (0)$ is smaller than $\varepsilon$. 

Next, given a varifold $V$ with bounded mean curvature in an open set $U$, for every $\sigma>0$ and $\varepsilon >0$ we introduce the set
\[
\mathcal{S}^{k, \sigma}_\varepsilon (V) :=\{x\in \spt (V)\cap U : \mbox{ $V$ is not $(k+1, \varepsilon)$-symmetric in $\bB_r (x)$ for $r\in ]0, \sigma]$}\}
\]
\end{definition}

The following is then a direct corollary of Lemma \ref{l:symmetric}.

\begin{corollary}\label{c:non-m-simmetrici}
Assume that $T$ is as in Theorem \ref{t:main}, and consider the varifold $\V (T)$ induced by $T$. If $p$ is odd, then for every compact $K$ with $K\cap \spt^p (\partial T) = \emptyset$ there are constants $\varepsilon =\varepsilon (m,n,p, K)>0$ and $\sigma = \sigma (m,n,p, K) >0$ such that
\[
\sing (T) \cap K \subset \bigcup_{Q=2}^{\frac{p-1}{2}} \sing_Q (T) \cup \mathcal{S}^{m-1,\sigma}_\varepsilon (\V (T)) \cup \mathcal{S}^{m-2} (\V (T))\, . 
\]
\end{corollary}
\begin{proof}
Consider a point 
\[
q\in (\sing (T) \cap K)\setminus \left(\bigcup_{Q=2}^{\frac{p-1}{2}} \sing_Q (T)\cup  \mathcal{S}^{m-2} (\V (T))\right)\, .
\]
We then know that at least one tangent cone in $q$ is $(m-1)$-symmetric but not flat. Therefore we know from Lemma \ref{l:symmetric} that $\Theta (T, q) \geq \frac{p}{2}$. We also know that $\V (T)$ is a varifold with bounded mean curvature (the $L^\infty$ bound depending only on the second fundamental form of $\Sigma$) and that there is a $\sigma_0 (K)>0$ such that
$\dist (q, \spt^p (\partial T)) \geq \sigma_0$. In particular, by the monotonicity formula, there is a $\sigma (K, \Sigma)>0$ such that
\begin{equation}\label{e:fame}
\|\V (T)\| (\bB_r (q)) \geq \left(\frac{p}{2} - \frac{1}{4}\right)\omega_m r^m \qquad \forall r\in ]0, \sigma]\, .
\end{equation}
On the other hand, if $\V (T)$ were $(m, \varepsilon)$-symmetric in $\bB_r (q)$, then there would be a positive integer $Q$ and an oriented $m$-dimensional plane $\llbracket\pi\rrbracket$ such that the varifold distance between $((\iota_{q,r})_\sharp \V (T))\res \bB_1 (0)$ and $Q \V (\llbracket \pi \rrbracket) \res \bB_1 (0)$ is smaller than $\varepsilon$. By the compactness Proposition \ref{p:compactness} (observing that $r^{-m}\mass (T\res \bB_r (x))$ can be bounded uniformly for $x\in K$), when $\varepsilon$ is sufficiently small, $Q\a{\pi}$ must be a representative of an area minimizing current $\modp$ and as such we must have $Q\leq \frac{p-1}{2}$. In particular, if $\varepsilon$ is sufficiently small, we would conclude
\[
\|\V (T)\| (\bB_r (q)) \leq \left(\frac{p}{2} - \frac{3}{8}\right)\omega_m r^m\, .
\] 
This is however not possible because of \eqref{e:fame} and hence we deduce that $q\in \mathcal{S}^{m-1,\sigma}_\varepsilon (\V (T))$.
\end{proof}

\begin{proof}[Proof of Theorem \ref{t:main10}]
Observe that, by Almgren's stratification theorem, $\mathcal{S}^{m-2} (\V (T))$ has Hausdorff dimension at most $m-2$. Similarly, 
\[
\bigcup_{Q=2}^{\frac{p-1}{2}} \sing_Q (T) 
\]
has Hausdorff dimension at most $m-2$ by Theorem \ref{t:main3}. Since by \cite[Theorem 1.4]{NV}, $\mathcal{S}^{m-1,\sigma}_\varepsilon (\V (T))\cap K$ has finite $\mathcal{H}^{m-1}$ measure and it is $(m-1)$-rectifiable, the claim follows from Corollary \ref{c:non-m-simmetrici}. 
\end{proof}

\section{Structure theorem: proof of Corollary \ref{c:main11}}

In this section we prove Corollary \ref{c:main11}. First of all observe that each connected component $\Lambda_i$ is necessarily a regular submanifold because, by definition, it is contained in the set of regular interior points of $T$. Clearly $\Lambda_i$ is locally orientable, and it is simple to show that, since $p$ is odd, there is in fact a smooth global orientation. Clearly $T \res \Lambda_i = Q_i \a{\Lambda_i}\;\modp$ for some integer multiplicity $Q_i \in [-\frac{p}{2}, \frac{p}{2}]$ by the constancy lemma $\modp$. On the other hand we can reverse the orientation to assume that
$Q_i \in [1, \frac{p}{2}]$. Point (b) is then obvious because $T\res U = \sum_i T_i \res U \; \modp$ and in fact 
\begin{equation}\label{e:decomposition}
\|T\|\res U = \sum_i \|T_i\| \res U\, .
\end{equation} 
Now consider $U$ as in part (a) of the statement and observe that, by the monotonicity formula, there are constants $M (U)$ and $\rho (U)>0$, such that
\[
\|T\| (\bB_r (x)) \leq M r^m \qquad \forall x\in U \;\mbox{and}\; \forall r\in ]0, \rho (U)]\, .
\]
Fix a $T_i$ and note that, by \eqref{e:decomposition},
\begin{equation}\label{e:upper_bound_1000}
\|T_i\| (\bB_r (x)) \leq M r^m\, .
\end{equation}
Observe that
\[
\spt ((\partial T_i)\res U) \subset \sing (T) \cap \overline{U} =: K\, ,
\]
and that, by Theorem \ref{t:main10}, the compact set $K$ satisfies the bound
\begin{equation}\label{e:bound_Hausdorff}
\mathcal{H}^{m-1} (K)<\infty\, .
\end{equation}
We next claim that, by \eqref{e:upper_bound_1000} and \eqref{e:bound_Hausdorff},
\[
\mass ((\partial T_i)\res U)<\infty\, .
\]
First of all fix $\sigma = \frac{1}{k} < \frac{\rho (U)}{2}$ and choose a finite cover of $K$ with balls $\{B^k_j\}_j$ with radii $r_j^k$ satisfying $2\,r_j^k\leq \sigma = \frac{1}{k}$ such that
\[
\sum_j \omega_{m-1} (r_j^k)^{m-1} \leq 2 \mathcal{H}^{m-1}_\sigma (K) \leq 2 \mathcal{H}^{m-1} (K)\, .
\]
For each ball $B_j^k$ we choose a smooth cutoff function $\varphi^k_j$ which vanishes identically on $B^k_j$ and it is identically equal to $1$ on the complement of the concentric ball $2 B_j^k$ with twice the radius. We choose $\varphi^k_j$ so that $0\leq \varphi^k_j \leq 1$ and $\|d\varphi^k_j\|_0 \leq C (r^k_j)^{-1}$, where $C$ is a geometric constant. 
We then define 
\[
\varphi^k := \prod_j \varphi^k_j\, . 
\]
Recall that
\[
\mass ((\partial T_i) \res U) = \sup \{\partial T_i (\omega) : \|\omega\|_c \leq 1\,,\, \omega\in \mathcal{D}^k (U)\}\, .
\]
We therefore fix a smooth $(m-1)$-form $\omega$ with compact support in $U$ and we are interested in bounding $\partial T_i (\omega) =
T_i (d\omega)$. Observe that $\varphi^k \uparrow 1$ $\|T_i\|$-a.e. on $U$. Hence we can write
\[
T_i (d\omega) = \lim_{k\to\infty} T_i (\varphi^k d \omega)\, .
\]
On the other hand, since $\varphi^k \omega$ is supported in an open set $V\subset\subset U\setminus K$ we conclude
\[
T_i (d (\varphi^k \omega)) = \partial T_i (\varphi^k \omega) = 0\, .
\]
Hence we can estimate
\begin{align}
|T_i (\varphi^k d \omega)| &= |T_i (d\varphi^k\wedge \omega)| \leq \sum_j \left| T_i \left(\prod_{\ell\neq j} \varphi^k_\ell d\varphi^k_j \wedge \omega\right)\right|\nonumber\\
&\leq C \sum_j \|\omega\|_c\, \|d\varphi^k_j\|_0\, \|T_i\| (2 B_j^k) 
\stackrel{\eqref{e:upper_bound_1000}}{\leq} C M \|\omega\|_c \sum_j (r^k_j)^{-1} (2r^k_j)^{m}\nonumber\\
&\leq C M \|\omega\|_c \mathcal{H}^{m-1} (K)\, .
\end{align}
Letting $k\uparrow \infty$ we thus conclude 
\[
|T_i (d\omega)|\leq C M \|\omega\|_c \mathcal{H}^{m-1} (K)\, .
\]
This shows that $(\partial T_i)\res U$ has finite mass. Point (a) follows therefore from the Federer-Fleming boundary rectifiability theorem. 

In order to show (c), consider the set $K'$ of points $q\in K$ where
\begin{itemize}
\item $K$ has an approximate tangent plane $T_q K$;
\item $q$ is a Lebesgue point for all $\Theta_i$'s with $\Theta_i (q) \in \mathbb Z$. 
\end{itemize}
By a standard blow-up argument, it follows that, for every fixed $q\in K'$, any limit $S$ of the currents $(\iota_{q,r})_\sharp (T_i)$ as $r \downarrow 0$ is an area-minimizing current on $\mathbb R^{m+n}$ with boundary either $-\Theta_i (q) \a{T_q K}$ or $+\Theta_i (q) \a{T_q K}$. By the boundary monotonicity formula, 
\[
\|S\| (\bB_1 (0)) \geq \frac{|\Theta_i (q)|}{2} \omega_m\, .
\]
We therefore conclude that
\[
\liminf_{r\downarrow 0} \frac{\|T_i\| (\bB_r (q))}{r^m} \geq \omega_m \frac{|\Theta_i (q)|}{2}\, .
\]
Fix any natural number $N$. We then conclude from \eqref{e:upper_bound_1000} that
\[
M \geq \lim_{r\downarrow 0} \frac{\|T\| (\bB_r (q))}{r^m} \geq \sum_{i=1}^N \liminf_{r\downarrow 0} \frac{\|T_i\| (\bB_r (q))}{r^m} 
\geq \sum_{i=1}^N \omega_m \frac{|\Theta_i (q)|}{2}\, .
\]
In particular we conclude that
\[
\sum_{i=1}^\infty |\Theta_i (q)| \leq \frac{2M}{\omega_m} \qquad \forall q\in K'\, .
\]
This shows that
\[
\sum_i \mass ((\partial T_i)\res U) \leq  \frac{2M}{\omega_m} \mathcal{H}^{m-1} (K) < \infty\, .
\]
This completes the proof of (c) and of the structure theorem. 

\appendix

\section{Proof of Proposition \ref{SORBILLO_GRATIS_X2}} \label{appendix}

In order to reach a proof of Proposition \ref{SORBILLO_GRATIS_X2}, we will need some preliminary results. First, for a given $S \in \Rc_{1}(\R^{n})$, we say that $S$ has the property $(NC)$ (\textit{no cycles}) if there exists no $0 \neq R \in \Rc_{1}(\R^{n})$ such that $\partial R = 0$ and
\[
\M(S) = \M(R) + \M(S - R).
\]
We recall that $\In_m (\R^{m+n})$ denotes the space of $m$-dimensional integral currents in $\R^{m+n}$.

Given $S \in \In_{1}(\R^{n})$ satisfying the property $(NC)$, we call a \emph{good decomposition} of $S$ a writing
\[
S = \sum_{j=1}^{N} \theta_{j} S_{j},
\] 
where $\theta_{j} \in \N$, each $S_{j}$ is the integral current given by $S_{j} = \llbracket \gamma_{j} \rrbracket$ for $\gamma_{j}$ a simple Lipschitz curve of finite length, $S_{j} \neq S_{k}$ if $j \neq k$ and moreover
\begin{equation} \label{origano}
\M(S) = \sum_{j} \theta_{j} \M(S_{j}), \hspace{0.5cm} \M(\partial S) =  \sum_{j} \theta_{j} \M(\partial S_{j}).
\end{equation}
The existence of a good decomposition for a current $S \in \In_{1}(\R^{n})$ satisfying the property $(NC)$ is a direct consequence of \cite[4.2.25]{Federer69}. We say that a good decomposition $S = \sum_{j=1}^{N} \theta_{j} S_{j}$ has the property $(NTC)$ (\textit{no topological cycles}) if there exists no function $f \colon \{1, \dots, N \} \to \{-1, 0, 1 \}$, $f \not\equiv 0$, such that
\begin{equation} \label{doppia_mozz}
\partial \left( \sum_{j=1}^{N} f(j) S_{j} \right) = 0.
\end{equation}

\begin{lemma} \label{pummarola}
For any $S \in \In_{1}(\R^{n})$ with the property $(NC)$ there exists $S' \in \In_{1}(\R^{n})$ with the property $(NC)$ and a good decomposition of $S'$ that satisfies $\partial S' = \partial S$, $\M(S') \leq \M(S)$, and that has the property $(NTC)$. 
\end{lemma}

\begin{proof}
Let $S \in \In_{1}(\R^{n})$, and assume without loss of generality that $S \neq 0$. Among all currents $S' \in \In_{1}(\R^{n})$ with the property $(NC)$ and such that $\partial S' = \partial S$ and $\M(S') \leq \M(S)$, and among all possible good decompositions of $S'$ not satisfying the property $(NTC)$ fix a current $S'$ and a decomposition 
\[
S' = \sum_{j=1}^{N} \theta_{j}' S'_{j}
\] 
such that the quantity $N$ is minimal. Observe that necessarily $N \geq 1$.

Let $f \colon \{1, \dots, N\} \to \{ -1, 0, 1 \}$ be a function such that \eqref{doppia_mozz} holds. Define: 
\[
j_{-} \in {\rm argmin}\lbrace \theta_{j}' \, \colon \, f(j) = -1 \rbrace
\]
and 
\[
j_{+} \in {\rm argmin}\lbrace \theta_{j}' \, \colon \, f(j) = +1 \rbrace.
\]
Observe that since $S'$ has the property $(NC)$, the sets $\lbrace \theta_{j}' \, \colon \, f(j) = -1 \rbrace$ and $\lbrace \theta_{j}' \, \colon \, f(j) = +1 \rbrace$ are non-empty.

Now, consider the quantities
\[
M_{-} := \sum_{j \, \colon \, f(j) = -1} \M(S_{j}')
\]
and
\[
M_{+} := \sum_{j \, \colon \, f(j) = +1} \M(S_{j}').
\]

Clearly, if $M_{+} \geq M_{-}$ then the current 
\[
S'_{+} := S' - \theta_{j_{+}} \sum_{j} f(j) S'_{j}
\]
satisfies $\M(S'_{+}) \leq \M(S') \leq \M(S)$. If instead $M_{+} \leq M_{-}$ then the current
\[
S'_{-} := S' + \theta_{j_{-}} \sum_{j} f(j) S'_{j}
\]
satisfies $\M(S'_{-}) \leq \M(S') \leq \M(S)$. In any of the two cases, $\partial S'_{\pm} = \partial S' = \partial S$, and the obvious resulting decomposition of $S'_{\pm}$ has at most $N-1$ indexes. Hence, by minimality, the one of the two which does not increase the mass necessarily has the property $(NTC)$. This concludes the proof.
\end{proof}

\begin{lemma} \label{basilico}
Let $S \in \In_{1}(\R^{n})$ and $0 \neq Z \in \Rc_{0}(\R^{n})$ be such that:
\begin{itemize}
\item[$(H1)$] $A - B = \partial S + pZ$;
\item[$(H2)$] $S$ has the property $(NC)$ and there exists a good decomposition
\[
S = \sum_{j=1}^{N} \theta_{j} S_{j}
\]
with the property $(NTC)$. 
\end{itemize}
Then, there exists $j_{0} \in \{1, \dots, N \}$ such that $\partial S_{j_0} = \llbracket x \rrbracket - \llbracket y \rrbracket$ with $x, y \in \spt(Z)$ and $\theta_{j_0} \geq \frac{p}{2}$.
\end{lemma}

\begin{proof}
Let $S$ and $Z$ be as above. Firstly, we claim that the set of indexes $j \in \{ 1, \dots, N\}$ such that $\partial S_{j} = \llbracket x \rrbracket - \llbracket y \rrbracket$ with $x, y \in \spt(Z)$ is non-empty. We write
\[
Z = \sum_{\ell=1}^{M} \llbracket N_{\ell} \rrbracket - \sum_{\ell=1}^{M} \llbracket P_{\ell} \rrbracket,
\]
where the $N_{\ell}$'s (resp. the $P_{\ell}$'s) are not necessarily distinct, so that
\[
\partial S = \sum_{i=1}^{Q} \llbracket A_{i} \rrbracket + p \sum_{\ell=1}^{M} \llbracket P_{\ell} \rrbracket - \left( \sum_{i=1}^{Q} \llbracket B_{i} \rrbracket + p \sum_{\ell=1}^{M} \llbracket N_{\ell} \rrbracket \right).
\]
Consider any of the points $P_{\ell}$. By \eqref{origano}, the multiplicity of $\partial S$ in $P_{\ell}$ is at least $p$, and furthermore, since $Q \leq \frac{p}{2}$, there exist $j \in \{1, \dots, N\}$ and $\ell' \in \{1, \dots, M\}$ such that $\partial S_{j} = \llbracket P_{\ell} \rrbracket - \llbracket N_{\ell'} \rrbracket$, which proves our claim. 

Next, assume by contradiction that for every $j$ such that $\partial S_{j}$ is supported on $\spt(Z)$ one has $\theta_{j} < \frac{p}{2}$. Fix, for instance, the point $P_{1}$. Arguing as above, after possibly reordering the indexes (both in the family $\{S_{j}\}$ and $\{N_{\ell}\}$), we conclude that there exist $N_{1}$ and $S_{1}$ such that $\partial S_{1} = \llbracket P_{1} \rrbracket - \llbracket N_{1} \rrbracket$. Moreover, by hypothesis, $\theta_{1} < \frac{p}{2}$. This ensures that we can find $P_{2}$ and $S_{2}$ such that $\partial S_{2} = \llbracket P_{2} \rrbracket - \llbracket N_{1} \rrbracket$, and again $\theta_{2} < \frac{p}{2}$. The procedure can be iterated as long as the new points $P_{\ell + 1}$ (resp. $N_{\ell+1}$) are distinct from the previous ones. Since the decomposition of $S$ has the property $(NTC)$ by hypothesis $(H2)$, this would imply that the procedure can be iterated indefinitely, which gives the desired contradiction.   
\end{proof}

\begin{proof}[Proof of Proposition \ref{SORBILLO_GRATIS_X2}]

Let us first consider case (a), with $\sigma = 1$.

It suffices to prove that
\begin{equation} \label{1}
\Fl(A - B) \leq \Fl^{p}(A - B),
\end{equation}
because the other inequality is obvious. 

Suppose by contradiction that
\begin{equation} \label{2}
\Fl^{p}(A - B) < \Fl(A - B),
\end{equation}
and let $S \in \In_{1}(\R^{n})$ and $0 \neq Z \in \Rc_{0}(\R^{n})$ be such that
\[
A - B = \partial S + pZ \hspace{0.2cm} \mbox{ and } \hspace{0.2cm} \M(S) < \Fl(A - B).
\]

We claim that there exist currents $S^{1} \in \In_{1}(\R^{n})$ and $Z^{1} \in \Rc_{0}(\R^{n})$ such that
\begin{equation} \label{claim}
A - B = \partial S^{1} + pZ^{1}, \quad \M(S^{1}) < \Fl(A - B) \quad \mbox{ and } \quad \M(Z^{1}) = \M(Z) - 2.
\end{equation}
The conclusion trivially follows from the claim.

We proceed with the proof of \eqref{claim}. First observe that if $S$ has a cycle $R$ then the current $S' := S - R$ satisfies $A - B = \partial S' + pZ$ and $\M(S') = \M(S) - \M(R) < \Fl(A - B)$. Therefore, we can assume without loss of generality that $S$ has the property $(NC)$. Next, applying Lemma \ref{pummarola} we can also assume that $S$ has a good decomposition
\[
S = \sum_{j=1}^{N} \theta_{j} S_{j}
\]
which satisfies the property $(NTC)$. Now, by Lemma \ref{basilico} there exists $j_{0} \in \{1, \dots, N \}$ such that $\partial S_{j_0} = \llbracket x \rrbracket - \llbracket y \rrbracket$ with $x, y \in \spt(Z)$ and $\theta_{j_0} \geq \frac{p}{2}$. Let $S^{1} := S - p S_{j_0}$. We have
\[
\partial S^{1} = \partial S - p \llbracket x \rrbracket + p \llbracket y \rrbracket,
\]
and thus
\[
A - B = \partial S^{1} + pZ^{1},
\]
where $Z^{1} := Z + \llbracket x \rrbracket - \llbracket y \rrbracket$. The conclusion $\M(Z^{1}) = \M(Z) - 2$ simply follows from \eqref{origano}. Finally, we get
\[
\M(S^{1}) \leq \sum_{j \neq j_{0}} \theta_{j} \M(S_{j}) + |\theta_{j_0} - p| \M(S_{j_0}) 
\leq \sum_{j=1}^{N} \theta_{j} \M(S_j) \overset{\eqref{origano}}{=} \M(S) < \Fl(A - B),
\]
where the second inequality follows from $\theta_{j_0} \geq \frac{p}{2}$.

\smallskip

Let us now consider instead case (b), when $\sigma = -1$ and $Q = \frac{p}{2}$. We know from \eqref{basic inequality special} that
\[
\Fl^p(A+B) \leq \Fl(A+B)\,,
\]
where $\Fl(A+B)$ is defined by \eqref{reduced flat norm definition}. Assume by contradiction that $\Fl^p(A + B) < \Fl(A+B)$. That is, there exist $S \in \In_{1}(\R^n)$ and $Z \in \Rc_{0}(\R^n)$ such that
\begin{equation} \label{contraddizione}
A + B = \partial S + p Z\,, \quad \mbox{and} \quad \M(S) < \Fl(A+B).
\end{equation}

Observe that it cannot be $Z = 0$. Also, by Lemma \ref{pummarola} there is no loss of generality in assuming that $S$ admits a good decomposition
\[
S = \sum_{j=1}^{N} \theta_{j} S_{j}
\]
having the property $(NTC)$. Now, if $\M(Z) = 1$ then there exists $z \in \R^n$ such that $Z = \llbracket z \rrbracket$. In that case, if we set $R := z \cone (A+B)$ then we have
\[
\partial R = A + B - p \llbracket z \rrbracket = \partial S\,,
\]
and
\[
\begin{split}
\Fl(A + B) \leq \M(R) &= \Fl(A-Q \llbracket z \rrbracket) + \Fl(B - Q \llbracket z \rrbracket) \\
&= \sum_{i=1}^{Q} \left( |A_{i} - z| + |B_{i} - z| \right) \leq \M(S)\,,
\end{split}
\]
thus contradicting \eqref{contraddizione}. 

On the other hand, if $\M(Z) \geq 2$ (and thus in fact $\M(Z) \geq 3$) then there exists $j_0 \in \{1,\dots,N\}$ such that $\partial S_{j_0} = \llbracket x \rrbracket - \rrbracket y \rrbracket$ with $x,y \in \spt(Z)$ and $\theta_{j_0} \geq \frac{p}{2}$. Hence, setting $S^{1} := S - p S_{j_0}$ we have
\[
A + B = \partial S^1 + p Z^1\,,
\]
with $Z^1 := Z + \llbracket x \rrbracket - \llbracket y \rrbracket$, $\M(Z^1) = \M(Z) - 2$ and $\M(S^1) \leq \M(S)$. In order to complete the proof, it suffices to iterate this argument producing currents $S^{k}, Z^{k}$ until $\M(Z^k) = 1$. 
\end{proof}

\bibliographystyle{plain}
\bibliography{Biblio}

\end{document}